\documentclass[a4paper,12pt,oneside]{report}

\usepackage{eucal}
\usepackage{color}
\usepackage[all]{xy}
\usepackage[centertags]{amsmath}
\usepackage{latexsym}
\usepackage{stmaryrd}

\usepackage{amsfonts}

\usepackage{amssymb}
\usepackage{amsthm}

\usepackage{fancyhdr}
\usepackage[dvips]{epsfig}

\usepackage{graphicx}
\usepackage{t1enc}

\usepackage[latin1]{inputenc}

\usepackage{amscd}
\usepackage{graphics}
\usepackage[english]{minitoc}
\usepackage[english]{babel}

\makeatletter
\newcommand{\thechapterwords}
{ \ifcase \thechapter\or One\or Two\or Three\or Four\or
Five\or Six\or Seven \or Eight\or Nine\or Ten\or Eleven\fi}
\def\thickhrulefill{\leavevmode \leaders \hrule height 1ex \hfill \kern \z@}
\def\@makechapterhead#1{%
  \vspace*{15\p@}%
  {\parindent \z@ \centering \reset@font
        \thickhrulefill\quad
        \scshape \@chapapp{} \thechapterwords
        \quad \thickhrulefill
        \par\nobreak
        \vspace*{15\p@}%
        \interlinepenalty\@M
        \hrule
        \vspace*{15\p@}%
        \Huge \bfseries #1\par\nobreak
        \par
        \vspace*{15\p@}%
        \hrule
    \vskip 60\p@
  }}
\def\@makeschapterhead#1{%
  \vspace*{15\p@}%
  {\parindent \z@ \centering \reset@font
        \thickhrulefill
        \par\nobreak
        \vspace*{15\p@}%
        \interlinepenalty\@M
        \hrule
        \vspace*{15\p@}%
        \Huge \bfseries #1\par\nobreak
        \par
        \vspace*{15\p@}%
        \hrule
    \vskip 60\p@
  }}
  \def\@makechapterhead#1{%
  \vspace*{15\p@}%
  {\parindent \z@ \centering \reset@font
        \thickhrulefill\quad
        \scshape \@chapapp{} \thechapterwords
        \quad \thickhrulefill
        \par\nobreak
        \vspace*{15\p@}%
        \interlinepenalty\@M
        \hrule
        \vspace*{15\p@}%
        \Huge \bfseries #1\par\nobreak
        \par
        \vspace*{15\p@}%
        \hrule
    \vskip 60\p@
    }}
\usepackage{fancyhdr}
 \usepackage[dvips]{epsfig}

\fancypagestyle{plain}{ 
\fancyhead{} 

}

\usepackage{newlfont}

\newlength{\defbaselineskip}
\newcommand{\setlinespacing}[1]%
           {\setlength{\baselineskip}{#1 \defbaselineskip}}

\newtheorem{remark}{Remark}[section]

\theoremstyle{plain}
\newtheorem{lemma}{Lemma}[section]

\newtheorem{proposition}{Proposition}[section]
\newtheorem{theorem}{Theorem}[section]
\newtheorem{definition}{Definition}[section]
\newtheorem{corollary}{Corollary}[section]

\newtheorem{example}{Example}[section]

\newcommand{\Z}{\mathbb Z}
\newcommand{\R}{\mathbb R}
\newcommand{\C}{\mathbb C}
\newcommand{\N}{\mathbb N}

\newcommand{\G}{\mathcal{G}}

\newcommand{\D}{\mathcal{D}}
\newcommand{\m}{\mathfrak m}

\newcommand{\s}{\mathfrak{s}}

\newcommand{\h}{\mathfrak h}
\newcommand{\pd}{\partial}
\newcommand{\der}{\hbox{\rm der}}
\newcommand{\pder}{\hbox{\rm Pder}}
\newcommand{\paut}{\hbox{\rm Paut}}
\newcommand{\sspan}{\hbox{\rm span}}

\newcommand{\beq}{\begin{equation}}
\newcommand{\eeq}{\end{equation}}
\newcommand{\beqn}{\begin{eqnarray}}
\newcommand{\eeqn}{\end{eqnarray}}
\newcommand{\bpro}{\begin{proposition}}
\newcommand{\epro}{\end{proposition}}
\newcommand{\blem}{\begin{lemma}}
\newcommand{\elem}{\end{lemma}}
\newcommand{\bdfn}{\begin{definition}}
\newcommand{\edfn}{\end{definition}}
\newcommand{\bcor}{\begin{corollary}}
\newcommand{\ecor}{\end{corollary}}
\newcommand{\bthm}{\begin{theorem}}
\newcommand{\ethm}{\end{theorem}}
\newcommand{\bex}{\begin{example}}
\newcommand{\eex}{\end{example}}
\newcommand{\brmq}{\begin{remark}}
\newcommand{\ermq}{\end{remark}}
\newcommand{\benum}{\begin{enumerate}}
\newcommand{\eenum}{\end{enumerate}}
\newcommand{\bitem}{\begin{itemize}}
\newcommand{\eitem}{\end{itemize}}

\theoremstyle{plain}
\usepackage{graphicx}
\usepackage{t1enc}

\title{On the Geometry of Cotangent Bundles of Lie Groups }
\author{Bakary MANGA}
\date{}

\begin{document}
\begin{titlepage}
\begin{center}
\includegraphics[height=2cm,width=2.1cm]{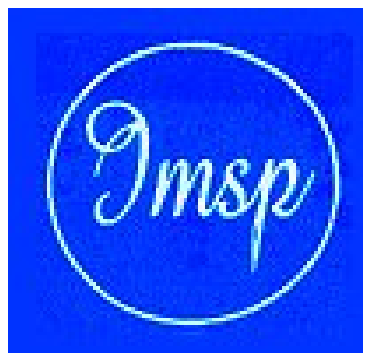} \qquad 
{\bf Universit\'e d'Abomey-Calavi (UAC), B\'enin} \qquad
\includegraphics[height=2cm,width=2.1cm]{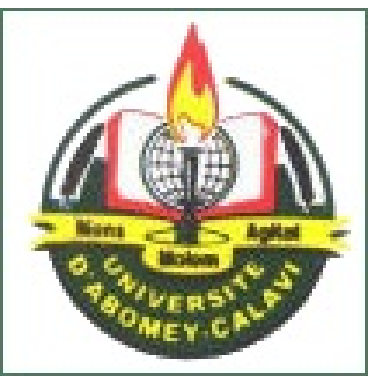}

\vspace{0.2cm}
The Abdus-Salam International Centre for Theoretical Physics (ICTP), Italy\\
\vspace{0.2cm}
{{\bf\it Institut de Math\'ematiques et de Sciences Physiques (IMSP), B\'enin\\}}
\end{center}

{\bf Order $n^0$ : 43/2010}

\vspace{1.2cm}

\begin{center}

{\bf PhD Thesis \\
\vskip 0.2cm 

Option : Differential Geometry\\}

\vspace{1.5cm}
{\bf Title :\\}
{\Large\bf On the Geometry of Cotangent Bundles of Lie Groups\\ }
\vspace{1cm}
By 

\vspace{0.5cm}

Bakary Manga\footnote{Supported by The Abdus Salam International Centre for Theoretical Physics (ICTP), Italy}\\

\vspace{1cm}

\end{center}

\vspace{1cm}

\begin{center}
{\bf Defended in June $18^{th}$, $2010$}
\end{center}

\end{titlepage}
\dominitoc

\pagenumbering{roman}

\chapter*{Dedication}
\addstarredchapter{Dedication}

{\it To My Sister Binta Manga }
\chapter*{Acknowledgements}
\addstarredchapter{Acknowledgements}

As prelude to their works, authors often display a warning, a kind of umbrella, 
fortunately opened to protect themselves from possible backslash of their statement. 
As an author, if I was allowed to do the same, I will state the following:
I acknowledge that It will be quite hard to find convenient words which can express my thanks to all those 
who helped me to achieve the dream of writing a PhD thesis. For those whose names are not written here, please be 
indulgent with me and accept my apologies.

\vskip 0.3cm

  I wish long life to the {\it Institut de Mathematiques et de Sciences Physiques (IMSP)}, 
the institute in which this thesis has been achieved.

\vskip 0.3cm

The relations between {\it Professor Jean-Pierre Ezin}, director of the IMSP (1988-2008) and me have never been always easy. 
Nevertheless, Prof. Ezin has contributed to make me become a better person by allowing my personality, even overflowing at times, 
to be expressed.  I am grateful to you for that and as well as for your strong and paramount battle related to support 
fundamental sciences in Africa. Thank you commissary \footnote{Mr. Ezin is nowadays  commissary in charge of Human Resources, 
Sciences and New Technologies of the African Union.}.

\vskip 0.3cm

{\it Professor Jo\"el Tossa}, has been for me more than an advisor.  With me, he has been always  available and patient. 
He constantly did his best in order to provide me with solutions whenever my working and living conditions were tough. 
Our relations have been intensified over years  such that it is becoming difficult for me not to have it around me.
 For sure, he will be one of those persons that I will miss  a lot once I am back in Senegal. I owe him quite a lot.

\vskip 0.3cm

I was very lucky to have {\it Professor Alberto Medina} of the University of Montpellier II (UM II) as another advisor of my thesis. 
In 2005, {\it Professor Medina} invited me for a stay in his research group whose niche is the Homogeneous spaces
(Lab. of Geometry, Topology and Algebra (GTA), Department of Mathematics, UM II). He has always been there for me and has 
boosted me every time the spectrum of doubt and abandon was in my mind.

\vskip 0.3cm 

At times, it is hard to ask to our thesis's advisers some type of questions. Such situations happen 
mainly because in our thoughts, we tend to believe that those questions are not tough enough and moreover, 
we do not want to show them our weaknesses. Hence, there is a need to have someone to whom our doubts and stupidities can 
be conveyed without any hesitation. In Doctor Andr\'e Diatta, I found that person. I was more that lucky to have him and never, 
he complained. Most of the results available in this thesis would have not been possible without his  energetic support.  
Several of the results within this thesis are obtained in collaboration with him and some of them will be published in joint papers.
Big brother, I am your servant!

\vskip 0.3cm

{\it Hermione Gandjidon} and {\it Rekiath Yasso} have had to learn the \LaTeX \, in order to become my "Secretaries". 
I thank both of you.

\vskip 0.3cm

By being at the IMSP, I met a huge number of researchers, Physicists and Mathematicians. 
 All interesting  interactions  and exchanges that we had had helped me to become a Mathematician.
 I thank you all.

\vskip 0.3cm

At the campus of the {\it  "Ecole Normale Superieure Nadjo"} where live PhDs and Engineering  students and where people 
call me {\it The Old}, I met guys funnier than me. I will miss the afternoon soccer's games of Tuesdays and Fridays 
through which I had opportunities to demonstrate to young people that {\it The Old} still has some youth in him.

\vskip 0.3cm

I spent many years in this wonderful country and I would like to emphasize that the hospitality  of the Beninese people 
 has nothing to envy to the famous Senegalese {\it "Teranga"} ("Teranga"= hospitality). Here is an opportunity to express my 
gratitude to all the people of Benin. 

\vskip 0.3cm

To the {\it Ly}'s family who adopted me and  the Senegalese community of Porto-Novo, I would like to express my sincere thanks.

\vskip 0.3cm

I would also like to convey my profound gratitude to families  Vigan, Dossa and Hounkenou all from Porto-Novo.

\vskip 0.3cm

I am grateful to {\it Professor Augustin Banyaga} who has accepted  to be a referee for my thesis. Professor L\'eonard Todjihounde
also provided a report for this thesis. Many thanks to him.

\vskip 0.3cm 

All my studies at the IMSP have been funded by the {\it  International Centre for Theoretical Physics (ICTP)} through the {\it ICAC-3 program}.
My trip and stay to Montpellier was funded by {\it SARIMA project}. Thanks a lot for your financial supports.

\vskip 0.3cm 

At last but not at least. My family has endured my absence for years. They have shown a lot of patience and devotion. 
Thank you so much dear parents, brothers and sisters. My friends from Senegal and elsewhere have been of a tremendous
and immeasurable support. At several occasions, they have represented me efficiently where I was not able to be. 
They encouraged and motivated me all along the way. Thanks to {\it A\"issatou Massaly, Mohamed Badji, Aliou Sarr, 
Anita Hounkenou, Coumba Tir\'era, Herv\'e G. Enjieu Kadji,   Youssouf Massaly, Samba Mbaye}.

\chapter*{Abstract}
\addstarredchapter{Abstract}

In this thesis we study the geometry of cotangent bundles of Lie groups as Drinfel'd double Lie groups. Lie groups
of automorphisms of cotangent bundles of Lie groups are completely characterized and interesting results are obtained. 
We give prominence to the fact that the Lie groups  of automorphisms of cotangent bundles of Lie groups are super symmetric 
Lie groups (Theorem \ref{structuretheorem}). 
In the cases of orthogonal Lie algebras, semi-simple Lie algebras and compact Lie algebras we recover by simple methods  
interesting co-homological known results (Section \ref{chap:cohomology}).

\vskip 0.3cm 

Another theme in this thesis is the study of prederivations of cotangent bundles of Lie groups. The Lie algebra of 
prederivations encompasses the one of derivations as a subalgebra. We find out that Lie algebras of cotangent Lie groups 
(which are not semi-simple) of semi-simple Lie groups have the property that all their prederivations are derivations. 
This result is an extension
of a well known result due to M\"uller (\cite{muller}). The structure of the Lie algebra of prederivations of 
Lie algebras of cotangent bundles of Lie groups is explored and we have shown that the Lie algebra of prederivations of Lie algebras of 
cotangent bundle of Lie groups are reductive Lie algebras.

\vskip 0.3cm

Prederivations are useful tools for classifying objects like pseudo-Riemannian metrics (\cite{bajo4}, \cite{muller}). 
We have studied bi-invariant metrics on cotangent bundles of Lie groups and their isometries. The Lie algebra of the Lie group of
isometries of a bi-invariant metric on a Lie group is composed with prederivations of the Lie algebra which are skew-symmetric
with respect to the induced orthogonal structure on the Lie algebra. We have shown that  the Lie 
group of isometries of any bi-invariant metric on the cotangent bundle of any semi-simple Lie groups is generated by the exponentials of 
inner derivations of the Lie cotangent algebra.

\vskip 0.3cm

Last, we have dealt with an introduction to the geometry the Lie  group of affine motions of the real line $\R$, which is a K\"ahlerian 
Lie group (see \cite{lichnerowicz-medina}). We describe, through explicit expressions, the symplectic structure, the complex structure, geodesics. 
Since the symplectic structure corresponds to a solution of the Classical Yang-Baxter equation $r$ (see \cite{di-me-cybe}), we
also study the double Lie group associated to $r$.

\chapter*{R\'esum\'e en Fran\c{c}ais}
\addstarredchapter{R\'esum\'e en Fran\c{c}ais}

Cette th\`ese est une contribution \`a l'\'etude de la g\'eom\'etrie des fibr\'es cotangents des groupes de Lie et des espaces homog\`enes.

\vskip 0.3cm 

Nous caract\'erisons compl\'etement  les groupes des automorphismes des fibr\'es cotangents des groupes de Lie 
et montrons qu'ils sont des groupes de Lie super-sym\'etriques (Theorem \ref{structuretheorem}). 
Dans le cas particulier d'un  groupe de Lie orthogonal, c'est-\`a-dire un groupe de Lie muni d'une  
m\'etrique bi-invariante, nous utilisons la m\'etrique pour r\'einterpr\'eter 
les relations et retrouver  des r\'esultats  connus de cohomologie. 

\vskip 0.3cm

Le fibr\'e cotangent $T^*G$ d'un groupe de Lie $G$ peut \^etre identifier au produit cartesien $G\times \G^*$, 
o\`u $\G^*$ est l'espace dual de l'alg\`ebre de Lie $\G$ de $G$. On peut alors munir $T^*G$ 
de la structure de groupe de Lie obtenue par produit semi-direct de $G$ et $\G^*$ via la repr\'esentation coadjointe.
Cette structure de groupe de Lie  fait de $T^*G$ un double de Drinfel'd.

\vskip 0.3cm

Nous avons \'egalement \'etudi\'e les pr\'ederivations des alg\`ebres de Lie des fibr\'es cotangents des groupes de Lie. 
Nous montrons que l'alg\`ebre des pr\'ederivations des alg\`ebres de Lie des groupes de Lie fibr\'es cotangents 
des groupes de Lie sont r\'eductives. 
M\"uller a montr\'e que toutes les pr\'ederivations d'une alg\`ebre de Lie semi-simple sont des d\'erivations (interieures). 
Nous \'etendons ce r\'esultat en montrant que si un groupe de Lie est semi-simple alors toutes les pr\'ederivations de  
l'alg\`ebre de Lie de son fibr\'e cotangent sont des d\'erivations quoi que le fibr\'e cotangent soit non semi-simple 
(Theorem \ref{thm-prederivation-semi-simple}).

\vskip 0.3cm

Un autre th\`eme abord\'e dans cette th\`ese est l'\'etude des m\'etriques bi-invariantes sur les fibr\'es cotangents des 
groupes de Lie. Nous caract\'erisons toutes les m\'etriques bi-invariantes sur les fibr\'es cotangents des groupes de Lie et 
\'etudions le groupe de leurs isom\'etries. L'alg\`ebre de Lie de ce groupe d'isom\'etries n'est  rien d'autre que l'alg\`ebre 
de toutes les pr\'ederivations de l'alg\`ebre du fibr\'e cotangent qui sont antisym\'etriques par rapport \`a la 
structure orthogonale induite sur l'alg\`ebre de Lie du fibr\'e cotangent.

\vskip 0.3cm

Enfin, nous avons fait une introduction \`a la g\'eom\'etrie du groupe de Lie des transformations affines de la droite r\'eelle. 
Nous donnons des expressions explicites d'une  forme symplectique, d'une structure affine, des g\'eod\'esiques de 
cette structure affine. La forme symplectique donnant lieu \`a une solution des \'equations Classiques de Yang-Baxter, 
nous avons \'egalement \'etudi\'e le groupe de Lie double de Drinfel'd du groupe des transformations affines de la droite r\'eelle.

\tableofcontents


\chapter*{General Introduction}
\fancyhead[L]{General Introduction}

\pagenumbering{arabic}
\addstarredchapter{General Introduction}

This thesis is a contribution to the large  task of studying the geometry of cotangent bundles of  Lie groups and 
the geometry of homogeneous spaces. 

\vskip 0.5cm

One of the greatest importance of the cotangent bundle $T^*G$ of a  Lie group $G$ is that it is a symplectic manifold on which $G$ acts 
by symplectomorphisms  with a Lagrangian orbit. A symplectic Lie group is a pair $(G,\omega)$ consisting of a Lie group $G$ and 
a closed non-degenerate $2$-form $\omega$ which is invariant under left translations of $G$. 
If identified with the Cartesian product $G\times \G^*$, the cotangent bundle $T^*G$ possesses a Lie group structure obtained by 
semi-direct product $T^*G:=G\ltimes \G^*$ using the coadjoint action of $G$ on the dual space $\G^*$ of the Lie algebra $\G$ of $G$. 
This Lie group structure  together with the Liouville form is not  a symplectic Lie group. But, if  $G$  carries a left-invariant affine structure, 
then its cotangent bundle carries a symplectic Lie group structure. This particular Lie group structure, sometimes called {\it "cotangent" }, 
is obtained by taking the semi-direct product of $G$ and the dual space $\G^*$ of its Lie algebra $\G$ by means of a natural action given 
by the affine structure on $G$. The corresponding Lie algebra is the semi-direct product via the left multiplications given by the 
left-symmetric product. The two Lie group structure on $T^*G$ defined above are not isomorphic. 

\vskip 0.5cm

It is also well known, since the works of Drinfel'd (\cite{drinfeld}), that $T^*G$ (with the Lie group structure performed by 
semi-direct product $G\ltimes \G^* $ via the coadjoint representation of $G$ on $\G^*$ ) is a particular case of the large class of the so-called 
Drinfel'd double Lie groups. As a Drinfel'd double Lie group, $T^*G$ admits a  metric which is invariant under left and right translations 
(\cite{drinfeld}). This Lie group structure on $T^*G$ does not admit a left-invariant symplectic form, except in the Abelian case.

\begin{center}
{\bf In this thesis we deal with the Lie group structure which make  $T^*G$ a Drinfel'd double Lie group of $G$.}
\end{center}


Studying the geometry of a given Lie group is to study invariant structures on it. The cotangent bundle $T^*G$ of a Lie group $G$ 
can exhibit very interesting and rich algebraic and geometric structures 
(affine, symplectic, pseudo-Riemannian, K\"ahlerian,...\cite{marle}, \cite{kronheimer}, \cite{feix}, \cite{drinfeld}, 
\cite{di-me-cybe}, \cite{bajo-benayadi-medina}). 

\vskip 0.5cm

Such structures can be better understood when one can
exhibit the group of transformations which preserve them.  This very often involves the automorphisms 
of $T^*G$, if in particular, such structures are invariant under  left or right multiplications by the  elements of $T^*G$. 
This is one of the reason for which we deal with {\it automorphisms of cotangent bundles of Lie groups} in  {\it Chapter II}  of this 
dissertation. We study the connected component of the unit of the  group of automorphisms of the Lie algebra 
$\mathcal D:=T^*\mathcal G$ of $T^*G$. Such a connected component being spanned by exponentials of derivations of $\D$, 
we often work with those derivations. Let  $\der(\G)$ stand for  the Lie algebra of derivations of $\G$, while $\mathcal J$  
denotes that of linear maps $j:\G\to\G$ satisfying $j([x,y])=[j(x),y]$, for every elements $x,y$ of $\G$. 
We give a characterization of all derivations of $T^*\G$ (Theorem \ref{derivationschar})
and  show that in particular, if $G$ has a bi-invariant Riemannian or pseudo-Riemannian metric, then every 
derivation $\phi$ of $\D$ can be expressed in terms of elements of $\der(\G)$ and $\mathcal J$ alone. Furthermore, we give prominence to the fact 
that the Lie group $Aut(\D)$ of automorphisms of $\D$ is a super symmetric Lie group and its Lie algebra $\der(\D)$
possesses Lie subalgebras which are Lie superalgebras, {\it i.e.} they are $\mathbb Z/2\mathbb Z$-graded Lie algebras with
the Lie bracket satisfying $[x,y]=-(-1)^{deg(x)deg(y)}[y,x]$ (Theorem \ref{structuretheorem}). We also consider particular 
cases (e.g. orthogonal Lie algebras, semi-simple Lie algebras, compact Lie algebras)  and recover by simple methods  interesting 
cohomological known results (Section \ref{chap:cohomology}).

\vskip 0.5cm

In {\it Chapter III} we completely characterise the space of prederivations of $\D$, that is endomorphisms $p$ of $\D$ which satisfy
$$
p\big(\big[x,[y,z]\big]\big) = \big[p(x),[y,z]\big] + \big[x,[p(y),z]\big] +
\big[x,[y,p(z)]\big],
$$
for every elements $x,y,z$ of $\D$. The Lie algebra $\der(\D)$ is a subalgebra of the Lie algebra $\pder(\D)$ of prederivations of 
$\D$. Prederivations can be used to study bi-invariant metrics on Lie groups (\cite{bajo1}, \cite{bajo4}, \cite{muller}). One of the 
important results within this chapter is: {\it If $G$ is a semi-simple Lie group with Lie algebra $\G$, then any prederivations of $T^*\G$ 
(not semi-simple) is a derivation}. This is an extension of the result of M\"uller (\cite{muller}) which states that any prederivation 
of a semi-simple Lie algebra is a derivation, hence an inner derivation. We also give a structure theorem for $\pder(\D)$ which 
states that $\pder(\D)$ decomposes into  $\pder(T^*\G) =\G_0\oplus\G_1$, where $\G_0$ is a reductive subalgebra of $\pder(\D)$, 
that is $[\G_0,\G_0]\subset \G_0$ and $[\G_0,\G_1]\subset \G_1$. Semi-simple, compact and more generally orthogonal
Lie algebras are also considered in this chapter.

\vskip 0.5cm

We study bi-invariant metrics of cotangent bundles of Lie groups in {\it Chapter IV}. In this chapter we characterise all 
orthogonal structures on $T^*\G$ (Theorem \ref{thm:orthogonal-structure}) and  their isometries. It is known  
 that if $(G,\mu)$ is a connected and simply-connected orthogonal Lie group with Lie algebra $\G$, 
then the isotropy group of the neutral element of $G$ in the group $I(G,\mu)$ of isometries of $(G,\mu)$ 
is isomorphic to the  group  of preautomorphisms of $\G$ which preserve the non-degenerate bilinear form induced 
by $\mu$ on $\G$ and whose Lie algebra is the  whole set of skew-symmetric prederivations of $\G$ (\cite{muller}). 
We characterise the isometries of bi-invariant metrics through the skew-symmetric prederivations with respect 
to the orthogonal structures induced on $T^*\G$ by the bi-invariant metrics 
(Proposition \ref{prop:skew-symmetric-prederivation}). 
In the case where $\G$ possesses an orthogonal structure, we proved that any orthogonal structure on $T^*\G$ can be 
expressed in terms of the duality pairing and endomorphisms of $\G$ which commute with all adjoint operators 
(Theorem \ref{theorem:orthogonal-structure-orthogonal-lie-algebras}).
If $\G$ is a semi-simple Lie algebra, we prove that  any prederivation of  $T^*\G$ which is skew-symmetric with respect 
to any orthogonal structure on $T^*\G$ is an inner derivation (Proposition \ref{skew-sym-prederivation-semisimple}); that is the 
connected component of the unit of the Lie group of isometries of any bi-invariant metric on $T^*G$ is spanned by exponentials 
of inner derivations of  $T^*\G$. Examples of the affine Lie group of the real line, the special linear group 
$SL(2,\R)$, the group $SO(3,\R)$ of rotations and the  $4$-dimensional oscillator group are given.

\vskip 0.5cm

The geometry of the Lie group of affine motions of the real line is explored in { \it Chapter V}. The geodesics of the 
left invariant affine structure induced by the symplectic structure are studied as well as integrale curves of left invariant vector fields. 
Since, the symplectic form considered on the affine Lie group corresponds to an invertible solution of the Classical Yang-Baxter equation, 
we have also studied the geometry of the corresponding double Lie group. The affine and complex structures on the double  introduced by Diatta 
and Medina (\cite{di-me-cybe}) are considered.

\chapter{Invariant Structures on Lie Groups}

\rhead{\textbf{\thepage}}
\lhead{\textsl{\leftmark}}

\minitoc

This chapter  is to make this thesis as self contained as possible. It defines basic notions and terminologies which might 
be useful throughout this dissertation.

\section{Orthogonal Structures on Lie Groups}

\subsection{Definition and Examples}

Let $G$ be a Lie group, $\epsilon$ its identity element, $\G$ its Lie algebra and
$TG$ its tangent bundle. Let $\G^*$ stand for the dual space of $\G$ and let $\mathbb K$ stand for the field 
$\R$ of real numbers or the field $\C$ of complex numbers.

\bdfn
A bi-invariant pseudo-metric on $G$ is a function $\mu : TG \to \mathbb K$ which is quadratic on each
fiber, nondegenerate and invariant under both left and right translations of the group $G$.
\edfn

A bi-invariant pseudo-metric on  the Lie group $G$ corresponds to a non-degenerate quadratic form
$q: \G \to \mathbb K$ for which the adjoint operators are skew-symmetric; that is, if $\mu$
also denotes the corresponding symmetric bilinear form on $\G$, we have

\beq\label{ad-invariant}
\mu ([x,y],z) + \mu(y,[x,z]) = 0
\eeq
for every $x,y,z$ in $\G$.

\vskip 0.3cm

Conversely, if $\G$ admits a non-degenerate symmetric bilinear form for which the adjoint operators
are skew-symmetric, then there exists, on every connected Lie group with Lie algebra $\G$,
a bi-invariant pseudo-metric.

\vskip 0.3cm

Throughout this work, a non-degenerate symmetric bilinear form on $\G$ for which the adjoint operators
are skew-symmetric is called simply a bi-invariant scalar product or an orthogonal structure on $\G$
as well as a bi-invariant pseudo-metric on $G$ is called a bi-invariant metric or an orthogonal
structure on $G$.

\vskip 0.3cm

Lie groups (resp. Lie algebras) with orthogonal structures are called orthogonal or quadratic
(see e.g. \cite{me-re85}, \cite{me-re93}). In \cite{astrakhantsev} orthogonal Lie algebras are called
{\it metrizable Lie algebras}.

\vskip 0.3cm

In \cite{me-re85}, Medina and Revoy have shown that every orthogonal Lie algebra is obtained
by the so-called {\it double extension } procedure.

\vskip 0.3cm

Consider the isomorphism of vector spaces $\theta : \G \to \G^*$ defined by 
\beq \label{isomorphism-theta}
\langle \theta(x),y \rangle := \mu(x,y),
\eeq 
where $\langle, \rangle$ on the left hand side, is the duality pairing $\langle x,f \rangle=f(x)$, 
between elements $x$ of $\G$ and $f$ of $\G^*$.
Then, $\theta$ is an isomorphism of $\G$-modules in
the sense that it is equivariant with respect to the adjoint and coadjoint actions of
$\G$ on $\G$ and $\G^*$ respectively; {\it i.e.}
\beq\label{suz10}
\theta \circ  ad_x =  ad^*_x \circ \theta,
\eeq
for all $x$ in $\G$. Which is also
\beq\label{suz9}
\theta^{-1} \circ ad^*_x = ad_x \circ \theta^{-1}, 
\eeq
for all $x$ in $\G$. The converse is also true. More precisely, a Lie group (resp. algebra) is orthogonal if and
only if its adjoint and coadjoint representations are isomorphic. See Theorem 1.4. of \cite{me-re93}.

\bex{\normalfont
Any Abelian Lie algebra with any scalar product is an orthogonal Lie algebra.}
\eex

\bex{\normalfont
Semi-simple Lie algebras with the Killing forms are orthogonal Lie algebras.}
\eex

\bex[Oscillator groups or diamond groups]\label{example-oscillator-group}{\normalfont
Let $\lambda=(\lambda_1, \lambda_2, \ldots, \lambda_n)$, where $0 <\lambda_1 \leq \lambda_2 \leq \cdots \leq \lambda_n$ are 
positive real numbers. On $\R^{2n+2} \equiv \R \times \R \times \C^n$, define the operation
\beqn\label{oscillator-group-operation}
(t,s\,;\,z_1,\ldots,z_n)\cdot (t',s'\,;\,z'_1,\ldots,z'_n) &=& \nonumber\\
\Big(t+t'\,,\, s+s'+\frac{1}{2}
\sum_{j=1}^nIm(\bar{z}_jz'_je^{i\lambda_jt})&\,;\,& z_1+z'_1e^{i\lambda_1t}\,,\,
\ldots\,,\, z_n+z'_ne^{i\lambda_nt}\Big),
\eeqn
where $t, t',s, s'$ are real numbers while $z_i, z_i'$, $i=1, 2,\ldots, n$, are in $\C$.
Endowed with the operation (\ref{oscillator-group-operation}) and the adjacent manifold structure, $\R^{2n+2}$
is a Lie group of dimension $2n+2$. This Lie group is noted by $G_\lambda$ and is called {\bf  Oscillator group } in \cite{bromberg-medina2004},
\cite{diaz-gadea-oubina}, \cite{gadea-oubina99}, \cite{streater}
(for dimension $n=1$), {\bf twisted Heisenberg group} in \cite{zeghib-esp-temps-hom} and
{\bf diamond group} in other literature (\cite{zeghib-esp-temps-hom}).

The Lie algebra of $G_\lambda$, called {\bf oscillator algebra} of dimension $2n+2$, is the vector space spanned by
$\{e_{-1},e_0,e_1,e_2,\ldots,e_n,\check e_1,\check e_2,\ldots,\check e_n\}$ with the following non-zero brackets:
\beq
[e_{-1},e_j] = \lambda_j \check e_j \quad ;  \quad [e_{-1},\check e_j] = -\lambda_j e_j \quad ; \quad [e_j,\check e_j] = e_0.
\eeq
The oscillator Lie algebra is noted by $\G_\lambda$. Every element $x$ of $\G_\lambda$ can be written:
\beq
x =\alpha e_{-1} + \beta e_0 + \sum_{j=1}^n\alpha_je_j + \sum_{j=1}^n\beta_j\check e_j.
\eeq
where $\alpha,\beta,\alpha_j,\beta_j$ ($1\leq j \leq n$) are real numbers. Then the following
quadratic form defines an orthogonal structure on $\G_\lambda$ (see \cite{bromberg-medina2004}):
\beq
k_\lambda(x,x) := 2\alpha \beta + \sum_{j=1}^n\frac{1}{\lambda_j}(\alpha_j^2 + \beta_j^2)
\eeq
Oscillator groups  are subject of a lot of studies (\cite{diaz-gadea-oubina}, \cite{dorr}, \cite{gadea-oubina99},
\cite{gadea-oubina02},  \cite{levichev}, \cite{me-re-oscillator}, \cite{streater} ). They appear in
many branches of Physics and Mathematical Physics and give particular solutions of Einstein-Yang-Mills equations (\cite{levichev}).
}
\eex

\subsection{Hyperbolic Lie Algebras, Manin Lie Algebras}

\bdfn
A hyperbolic plan $E$ is a $2$-dimensional linear space endowed with a non-degenerate
symmetric bilinear form $B$ such that there exists a non-zero element $v$ of $E$ with $B(v,v)=0$.
A hyperbolic space is an orthogonal sum of hyperbolic plans.
\edfn

\brmq {\normalfont
A hyperbolic space is of even dimension.}
\ermq
For more about hyperbolic spaces, see \cite{lang}.
\bdfn
A hyperbolic Lie algebra is an orthogonal Lie  algebra which contains two totally isotropic subspaces in duality for the orthogonal structure.
\edfn
\bdfn
A Manin Lie algebra is an orthogonal Lie algebra $\G$ such that:
\bitem
\item $\G$ admits two totally isotropic subalgebras $\h_1$ and $\h_2$;
\item $\h_1$ and $\h_2$ are  in duality one to the other for the orthogonal structure of $\G$.
\eitem
In this case,  $(\G,\h_1,\h_2)$ is called a Manin triple.
\edfn

\section{Poisson Structure on Lie Groups}


\subsection{Poisson Brackets, Poisson Tensors on a Manifold}
\bdfn
A $\mathcal C^\infty$-smooth {\bf \it Poisson structure (Poisson bracket)} on a $\mathcal C^\infty$-smooth
finite-dimensional manifold $M$ is an $\R$-bilinear skew-symmetric operation
\beqn
\begin{array}{rll}
\mathcal C^\infty(M) \times \mathcal C^\infty(M) & \longrightarrow & \mathcal C^\infty(M) \\
                             (f,g)               &\longmapsto     & \{f,g\}
\end{array}
\eeqn
on the space $\mathcal C^\infty(M)$ of real-valued $\mathcal C^\infty$-smooth functions on $M$,
such that
\bitem
\item $\big(\mathcal C^\infty(M),\{,\}\big)$ is a Lie algebra;
\item $\{,\}$ is a derivation in each factor, that is it verifies the Liebniz identity
\beq
\{f,gh\} = \{f,g\}h + g\{f,h\}
\eeq
for every $f,g,h$  in $\mathcal C^\infty(M)$.
\eitem
A manifold equipped with such a bracket is called a {\bf \it Poisson manifold}.
\edfn
\bex{\normalfont
Any manifold carries a trivial Poisson structure. One just has to put $\{f,g\}=0$,
for all smooth functions $f$ and $g$ on $M$.}
\eex
\bex{\normalfont
Let  $(x,y)$ denote coordinates on $\R^2$ and $p:\R^2 \to \R$ be an arbitrary smooth function.
One defines a smooth Poisson structure on $\R^2$ by putting
\beq
\{f,g\} = \left(\frac{\pd f}{\pd x}\frac{\pd g}{\pd y}-\frac{\pd f}{\pd y}\frac{\pd g}{\pd x} \right) p,
\eeq
for every $f,g$ in $\mathcal C^\infty(\R^2)$.}
\eex
\bex{\normalfont
A symplectic manifold $(M,\omega)$ is a manifold $M$ equipped with a non-degenerate closed
differential $2$-form $\omega$, called a {\bf \it symplectic form}.
If $f : M \to \R$ is a function on $(M,\omega)$, we define its
{\bf \it Hamiltonian vector field}, denoted by $X_f$, as follows:
\beq
i_{X_f}\omega :=\omega(X_f,\cdot)= -Tf,
\eeq
where $Tf$ stands for the differential map of $f$. One defines on $(M,\omega)$ a natural bracket,
called the {\bf \it Poisson bracket
of $\omega$}, as follows:
\beq
\{f,g\} = \omega(X_f,X_g) = -\langle Tf,X_g\rangle = -X_g(f) = X_f(g),
\eeq
for every $f,g \in \mathcal C^\infty(M)$. Thus, any symplectic manifold carries a Poisson
structure.
}
\eex

\vskip 0.3cm

Let $(M,\{,\})$ be a Poisson manifold. To every $f $ in $\mathcal C^\infty(M)$ corresponds
a unique vector field $X_f$ on $M$ such that
\beq
X_f(g) = \{f,g\},
\eeq
for every $g \in \mathcal C^\infty(M)$ (see e.g. \cite{marsden-ratiu}).
This is an extension of the notion of Hamiltonian vector field from the
symplectic to the Poisson context. Thus, $X_f$ is called the {\bf \it
Hamiltonian vector field} associated to $f$ as well as $f$ is
called the {\bf \it Hamiltonian function} of $X_f$.

\bdfn
A map $\phi :(M_1,\{,\}_1 ) \to (M_2,\{,\}_2)$ between two  Poisson manifolds is
said to be a Poisson morphism if it is smooth and satisfies
\beq
\{f \circ \phi,g \circ \phi\}_1 = \{f,g\}_2 \circ \phi,
\eeq
for every $f,g$ in $\mathcal C^\infty(M_2)$.
\edfn

\bex{\normalfont
If  $(M_1,\{,\}_1)$ and $(M_2,\{,\}_2)$ are Poisson manifolds, then $M_1 \times M_2$
has a Poisson structure characterized  by the following properties:
\benum
\item the projections $\pi_i:M_1 \times M_2 \to M_i$, $i=1,2$, are Poisson morphisms;
\item $\{f \circ \pi_1,g \circ \pi_2\}_1 =0$, for any $f$ in $\mathcal C^\infty(M_1)$ and
$g$ in $\mathcal C^\infty(M_2)$.
\eenum}
\eex

\bex[Kirilov-Kostant-Sauriau (KKS)]\label{linear-poisson-structure}{\normalfont
Let $\G$ be a finite dimensional Lie algebra seen as  the space of linear maps on its dual
$\G^*$; {\it i.e.} $\G \equiv (\G^*)^* \subset \mathcal C^\infty(\G^*)$.
Let $f,g$ be in $\mathcal C^\infty(\G^*)$ and $\alpha$ belongs to $\G^*$. If $T_\alpha f$ and $T_\alpha g$
denote the differentials of $f$ and $g$ at the point $\alpha$  (seen as elements of $\G$),
one defines a linear Poisson structure on $\G^*$ by putting:
\beq
\{f,g\}(\alpha) := \langle \alpha,[T_\alpha f,T_\alpha g]_\G\rangle,
\eeq
where $\langle,\rangle$ in the right hand side stands for the pairing between $\G$ and its dual.

This structure plays an important role in many domains of mathematical physics,
quantization, hyper-k\"ahlerian geometry, ...
}
\eex

\vskip 0.3cm

Let $M$ be a smooth manifold of dimension $n$ ($n \in \N^*$) and $p$ a positive integer.
Denote by $\Lambda^pTM$ the space of tangent $p$-vectors of $M$. It is a vector
bundle over $M$ whose fiber over each point $x \in M$ is the space
$\Lambda^pT_xM = \Lambda^p(T_xM)$, which is the exterior antisymmetric
product of $p$ copies of the tangent space $T_xM$. Of course $\Lambda^1TM=TM$.

\vskip 0.3cm

Let $(x^1,\ldots,x^n)$ be a local system of coordinates at $x \in M$. Then $\Lambda^pT_xM$
admits a linear basis consisting of the elements
$\frac{\pd}{\pd x^{i_1}}\wedge \ldots \wedge \frac{\pd}{\pd x^{i_p}}(x)$ with
$i_1 < i_2 < \ldots < i_p$.

\bdfn
A smooth $p$-vector field $\pi$ on $M$ is a smooth section of $\Lambda^pTM$, i.e. a map $\pi$
from $M$ to $\Lambda^pTM$, which associates to each point $x$ of $M$ a $p$-vector
$\pi(x)$ of $\Lambda^pT_xM$, in a smooth way.
\edfn

\vskip 0.3cm

Given a $2$-vector field $\pi$ on a smooth manifold $M$, one defines a tensor field, also denoted by $\pi$, by the following
formula:
\beq\label{poisson-tensor}
\{f,h\}:= \pi(f,h):=\langle Tf \wedge Th,\pi \rangle
\eeq

\bdfn
A $2$-vector field $\pi$, such that the bracket given by (\ref{poisson-tensor})
is a Poisson bracket, is called a Poisson tensor, or also a Poisson structure.
\edfn

\vskip 0.3cm

Any Poisson tensor $\pi$ arises from a $2$-vector field which we will also
denote by $\pi$.

\bex{\normalfont
The Poisson tensor corresponding to the standard symplectic structure
$\omega = \sum_{k=1}^ndx^k\wedge dy^k$ on $\R^{2n}$ is
$\sum_{k=1}^n \frac{\pd}{\pd x^k} \wedge \frac{\pd}{\pd y^k}$.
}
\eex

\subsection{Poisson-Lie Groups}

\bdfn\label{poisson-lie-structure}
Let $G$  be a Lie group. A {\bf Poisson-Lie structure} on $G$ is a Poisson tensor $\pi$ on the underlying
manifold of $G$ such that the multiplication map
\beqn
\begin{array}{ccc}
G \times G  & \to    &  G \\
(g , g')    &\mapsto & gg'
\end{array}
\eeqn
is a Poisson morphism (is grouped as said in \cite{drinfeld}); $G \times G$ being
equipped with the Poisson  tensor product $\pi \times \pi$.

A Lie group $G$ with a Poisson-Lie structure $\pi$ is called a {\bf Poisson-Lie group} and is denoted by
$(G,\pi)$.
\edfn
\noindent
The Definition \ref{poisson-lie-structure} is equivalent to the following:
\beq\label{relation:multiplicative}
\pi(gh) = T_hL_g\cdot \pi(h) + T_gR_h\cdot \pi(g),
\eeq
for every $g,h$  in $G$, where the differentials $T_gL_h$ and $T_gR_h$ of the left translation $L_g$ and the right translation $R_g$
 are naturally extended to the linear space $\Lambda^2T_gG$; {\it  i.e.}
\beqn
T_gL_h \cdot (X\wedge Y):= (T_gL_h \cdot X) \wedge (T_gL_h \cdot Y),\\
T_gR_h \cdot (X\wedge Y):= (T_gR_h \cdot X) \wedge (T_gR_h \cdot Y),
\eeqn
for any $X,Y$ in $T^*_gG$.

\bdfn
Let $G$  be a Lie group. A tensor field $\pi$ on $G$ is called multiplicative if it satisfies Relation (\ref{relation:multiplicative}).
\edfn

\bex{\normalfont
To every Lie algebra corresponds, at least, one Poisson-Lie group. Indeed, the dual space (seen as
an Abelian Lie group) of any Lie algebra, endowed with its linear Poisson structure given in
Example \ref{linear-poisson-structure} is a Poisson-Lie group.
}
\eex

Every Lie group possesses, at least, one non-trivial Poisson tensor (see \cite{desmedt}).

\subsection{Dual Lie Group, Drinfel'd Double of a Poisson-Lie Group}

Let $G$ be a Lie group with Lie algebra $\G$, $\pi$ a Poisson-Lie tensor on $G$ with corresponding
bracket $\{,\}$. One obtains a Lie algebra structure $[,]_*$ on the dual space $\G^*$ of $\G$
by setting
\beq\label{dual-bracket}
[\alpha,\beta]_*:=T_\epsilon\{f,g\},
\eeq
where $f$ and $g$ are smooth functions on $G$ such that $\alpha$ and $\beta$ are equal, respectively,
to the differentials of $f$ and $g$  at the identity element $\epsilon$ of $G$:
$T_\epsilon f=\alpha$, $T_\epsilon g = \beta$.

\bdfn
\benum
\item The bracket (\ref{dual-bracket}) is called the ''linearized'' bracket of $\pi$ at $\epsilon$ and
the pair $(\G^*,[,]_*)$ is said to be the "linearized" or the {\bf dual Lie algebra } of $(G,\pi)$ or
of $(\G,\lambda)$, where $\lambda :=T_\epsilon\pi : \G \to \Lambda^2\G$ ($\lambda$ is the transpose of
$[,]_* : \Lambda^2\G^* \to \G^*$).
\item Every Lie group, with Lie algebra $(\G^*,[,]_*)$ is called a {\bf dual Lie group} of $(G,\pi)$.
\item A subgroup $H$ of a Poisson-Lie group $(G,\pi)$ is said to be a Poisson-Lie subgroup of $(G,\pi)$
if it is also a Poisson submanifold of $(G,\pi)$.
\eenum
\edfn

\bdfn
Let $\G$ be a Lie algebra. 
\benum
\item A {\bf Lie bi-algebra structure} on $\G$ is a $1$-cocycle
$\lambda : \G \to \Lambda^2 \G$ for the adjoint representation of $\G$ such that its
transpose $ \lambda^t : \Lambda^2\G^* \to \G^*$ defines a Lie bracket on the vector space $\G^*$.
\item A Lie algebra, with a Lie bi-algebra structure, is called a Lie bi-algebra.
\eenum 
\edfn

A Lie bi-algebra will be denoted by $(\G,\lambda)$  or $(\G,\G^*)$.

\bthm(\cite{drinfeld})
Let $G$ be a simply connected Lie group. A Poisson structure, with Poisson tensor $\pi$, on $G$
bijectively corresponds to a Lie bi-algebra structure $\lambda=T_\epsilon\pi$ on the
Lie algebra $\G$ of $G$.
\ethm

\vskip 0.3cm

Let $(\G,[,])$ be a Lie algebra. Suppose $(\G,\G^*)$ is a Lie bi-algebra and let $[,]_*$ denote
the induced Lie bracket on $\G ^*$. Then, the transpose of the Lie bracket of $\G$ is a $1$-cocycle
of the Lie algebra $(\G^*,[,]_*)$. Hence, the Lie algebras $(\G,[,])$ and $(\G,[,]_*)$ act one to the other
by their respective coadjoint actions. The space $\G \times \G^*$ can be equipped with the scalar product:
\beq
\langle (x,f),(y,g) \rangle := f(y) + g(x),
\eeq
for every $x,y$ in $\G$ and every $f,g$ in $\G^*$.

We have the following result due to Drinfel'd.

\bthm(\cite{drinfeld})\label{lie-bialgebra-double}
The following are equivalent:
\benum
\item $(\G,\G^*)$ is a Lie bi-algebra;
\item $\D:= (\G \times \G^*,\langle,\rangle)$ is equipped with a unique structure $[,]_\D$ of orthogonal
Lie algebra such that:
\benum
\item $\G$  and $\G^*$ are Lie subalgebras of $\D$;
\item $\G$ and $\G^*$ are totally isotropic and in duality for the scalar product $\langle,\rangle$.
\eenum
In this case, for every $x$ in $\G$ and $g$ in $\G^*$,
\beq
[x,g]_\D=ad_x^*g - ad_g^*x
\eeq
\eenum
\ethm
\bdfn
The Lie algebra $\D$ of the Theorem \ref{lie-bialgebra-double} is called the
{\bf Drinfel'd double Lie algebra} of $(G,\pi)$ or of  $(\G,\lambda:=T_\epsilon\pi)$.
Every Lie group, with Lie algebra $\D$, is called a {\bf Drinfel'd double Lie group} of $(G,\pi)$.
\edfn
The Lie bracket on $\D$ reads:
\beq
[(x,f),(y,g)]_\D = \big([x,y] + ad_f^*y - ad_g^*x,[f,g]_* + ad_x^*g - ad_y^*f\big)
\eeq
for every $(x,f)$ and $(y,g)$ in $\D$.

\vskip 0.3cm

Since its introduction in $1983$ (\cite{drinfeld}), the notion of Drinfel'd double attracted 
many researchers (\cite{aminou-kosmann}, \cite{burciu}, \cite{ciccoli-guerra}, \cite{delvaux-vandaele04},\cite{delvaux-vandaele07},
 \cite{desmedt},\cite{di-me-cybe}, \cite{kosmann-magri}).

\section{Yang-Baxter Equation}

Let $M$ be a smooth manifold. For any integer $p$, let $\Omega_p(M)$ stand for the space of smooth sections of $\Lambda^pTM$ and
let $\Omega_*(M)$ be the direct sum of the spaces $\Omega_p(M)$, where
\bitem
\item $\Omega_p(M)=\{0\}$, if $p <0$ ;
\item $\Omega_0(M)=\mathcal C^\infty(M)$ ;
\item $\Omega_1(M)=\mathfrak X(M)$ (smooth vector fields on $M$) ;
\item $\Omega_p(M)=\{0\}$, if $p >\dim M$.
\eitem

\bthm[Schouten Bracket Theorem]
There exists a unique bilinear operation $[,]:\Omega_*(M)\times \Omega_*(M) \to \Omega_*(M)$ natural with respect to
restriction to open sets, called the {\bf Schouten Bracket}, that satisfies the following properties:
\benum
\item $[,]$ is a biderivation of degree $-1$, {i.e.} for all homogeneous elements $A$ and $B$
of $\Omega_*(M)$ and any $C$ in $\Omega_*(M)$,
\bitem
\item  $\deg [A,B]=\deg A + \deg B -1$ ; and
\item $[A,B\wedge C] = [A,B]\wedge C + (-1)^{(\deg A + 1)\deg B} B\wedge [A,C]$ ;
\eitem
\item $[,]$ is defined on $\mathcal C^\infty(M)$ and on $\mathfrak X(M)$ by
\benum
\item $[f,g] = 0$, for any $f,g$ in $\mathcal C^\infty(M)$ ;
\item $[X,f] = X\cdot f$, for all $f$ in $\mathcal C^\infty(M)$ and any  vector field $X$ on $M$ ;
\item $[X,Y]$ is the usual Jacobi-Lie bracket of vector fields if $X$ and $Y$ are in $\mathfrak X(M)$ ;
\eenum
\item $[A,B] = (-1)^{\deg A \deg B}[B,A]$.
\eenum
In addition, we have the graded Jacobi identity
\beq
(-1)^{\deg A\deg C} [[A,B],C] + (-1)^{\deg B\deg A} [[B,C],A] + (-1)^{\deg C\deg B} [[C,A],B] =0
\eeq
for all homogeneous elements $A,B,C$ of $\Omega_*(M)$.
\ethm

\vskip 0.3cm

Let $G$ be a Lie group with Lie algebra $\G$. On the $\Z$-graded vector space $\Lambda\G:=\oplus_{p\in\Z} \Lambda^p \G$
we consider the structure of graded Lie algebra obtained by the extension of the Lie bracket of $\G$ satisfying the 
properties of the definition of the Schouten's bracket.

\vskip 0.3cm

Let $r \in \Lambda^2 \G$ and note by $\eta$ the $1$-coboundary defined by
\beq
\eta(g)= Ad_g r - r, 
\eeq
for all $g$ in $G$. 

\vskip 0.3cm

Let $r^+$ and $r^-$ denote the left invariant and right invariant tensor fields associated to
$r$ respectively. One wonders whether the corresponding multiplicative tensor $\pi:= r^+-r^-$
is Poisson. The answer is given by the

\bpro(\cite{lichnerowicz77})\label{poisson-cohom}
Let $\Lambda$ be a contravariant skew-symmetric $2$-vector field on a manifold $M$.
\bitem
\item[(i)] $\Lambda$ is Poisson if and only if $[\Lambda,\Lambda]=0$ (this is equivalent to the Jacobi identity);
\item[(ii)] If $[\Lambda,\Lambda]=0$,  then the map
\beqn
\begin{array}{ccc}
\pd : \Omega_*(M) & \to     & \Omega_*(M)  \\
        P  & \mapsto & [\Lambda,P]
\end{array}
\eeqn
is an operator of cohomology, i.e. $\pd_\Lambda \circ \pd_\Lambda =0$.
\eitem
\epro

\bdfn
The cohomology defined by $(ii)$ of Proposition \ref{poisson-cohom} is called the Poisson cohomology
of the Poisson manifold $(M,\Lambda)$.
\edfn

According to the Proposition \ref{poisson-cohom}, the tensor field  $\pi$ is Poisson if
and only if the $3$-tensor field
\beq
[\pi,\pi] = [r^+,r^+]+[r^-,r^-] = [r,r]^+-[r,r]^-
\eeq
vanishes identically or equivalently, for every $g \in G$, the $3$-tensor
\beq
[\pi,\pi]_g:= T_\epsilon (Ad_g[r,r] -[r,r])
\eeq
equals zero. Hence, $\pi$ defines a Lie-Poisson tensor if and only if the $3$-vector
$[r,r]$ is $Ad_G$-invariant, {\it i.e.} for all $g \in G$,
\beq\label{GYBE}
Ad_g[r,r] = [r,r]
\eeq
\bdfn
\benum
\item Equation (\ref{GYBE}) below is called the {\bf Generalized Yang-Baxter Equation (GYBE)} and its solutions are called {\bf $r$-matrices}.
\item In the particular case where $[r,r] = 0$, one says that $r$ is a solution of the {\bf Classical Yang-Baxter Equation (CYBE)}.
\eenum
\edfn

\section{Symmetric Spaces}

Let $G$ be a connected Lie group with Lie algebra $\G$ and identity element $\epsilon$.
\bdfn
A {\bf symmetric space} for $G$ is a homogeneous space $M \equiv G/H$ such that the isotropy group
$H$ of any arbitrary point is an open subgroup of the fixed point set $G_\sigma:=\{g \in G: \sigma(g)=g\}$ of
an involution $\sigma$ of $G$.
\edfn
The involution $\sigma$ is in fact an automorphism of $G$ and fixes the identity element $\epsilon$.
Hence, the differential at $\epsilon$ of $\sigma$ is an automorphism, also denote by $\sigma$,
of the Lie algebra $\G$ with square equal to the identity mapping of $\G$: $\sigma^2=Id_\G$.
Then the eigenvalues of $\sigma$ are $+1$ and $-1$. The eigenspace associated to $+1$ is the Lie
algebra $\h$ of $H$. We denote the  $-1$ eigenspace by $\m$.
Since $\sigma$ is a automorphism of $\G$, we have the following decomposition
\beq\label{symmetric-space-decomp}
\G = \h \oplus \m
\eeq
with
\beq\label{symmetric-space-brackets}
[\h,\h] \subset \h \quad ; \quad  [\h,\m] \subset \m \quad ; \quad  [\m,\m] \subset \h.
\eeq
Conversely, given any Lie algebra $\G$ with  direct sum decomposition (\ref{symmetric-space-decomp})
satisfying (\ref{symmetric-space-brackets}), the linear map $\sigma$, equal to the identity on
$\h$ and minus the identity on $\m$, is an involutive automorphism of $\G$.

\section{Reductive  Lie algebras}
Let $\G$ be a finite dimensional Lie algebra over a field $\mathbb K$ of characteristic zero. 
\bdfn
A subalgebra $\mathfrak h$ is said to be reductive in $\G$ if $\G$ is a semisimple $\mathfrak h$-module; that is 
$\G$ is a sum of simple $\mathfrak h$-modules in the adjoint representation of $\mathfrak h$ in $\G$.
The Lie algebra $\G$ is said to be reductive if it is a reductive subalgebra of itself.
\edfn
Let $\G$ be an arbitrary finite dimensional Lie algebra over $\mathbb K$, $V$ be a semisimple $\G$-module and 
$I$ be the ideal $I=\{x\in \G : x\cdot V=\{0\}\}$. Now set $\h=\G/I$. By a result (see \cite{chevalley}) due to 
E. Cartan and N. Jacobson, we have $\h=[\h,\h]+Z(\h)$, where $Z(\h)$ is the center of $\h$, $[\h,\h]$ is semisimple, and 
$V$ is semisimple as a $Z(\h)$-module.

Now suppose that $\G$ is reductive. Since the center $Z(\G)$ of $\G$ is a $\G$-submodule of $\G$, there is an ideal $J$ of 
$\G$ such that $\G$ is the direct sum $J\oplus Z(\G)$. By the result we have cited above, we have $J=[J,J]+Z(J)$, where $Z(J)$
is the center of $J$ and $[J,J]$ is semisimple. But, of course, $[J,J]=[\G,\G]$ and $Z(J)=\{0\}$. Hence, $\G=[\G,\G]+Z(\G)$ and 
$[\G,\G]$ is semisimple. Conversely, it is clear that if a finite dimensional Lie algebra  $\G$ over $\mathbb K$ satisfies
these last conditions, then $\G$ is reductive.

\section{Lie Superalgebras}
Let $\mathbb K$ be an Abelian field of characteristic zero.

\subsection{Definition of a Lie Superalgebra}

\bdfn
A $\mathbb K$-linear space $L$ is a $\mathbb K$-superalgebra, with superbracket $[\cdot,\cdot]$, if
\bitem
\item[$a)$] it is $\Z/2\Z$-graded, {\it i.e.}
\beq
L=L_{\bar 0} \oplus L_{\bar 1} \quad ; \quad [L_{\bar 0},L_{\bar 0}] \subset L_{\bar 0}
\quad ; \quad [L_{\bar 0},L_{\bar 1}] \subset L_{\bar 1} \quad ; \quad
[L_{\bar 1},L_{\bar 1}] \subset L_{\bar 0}
\eeq
\item[$b)$] for every homogeneous elements $a,b,c$ of $L$,
\beq
(-1)^{|a|\cdot |c|}[a,[b,c]] + (-1)^{|b|\cdot |a|}[b,[c,a]] + (-1)^{|c|\cdot |b|}[c,[a,b]] = 0.
\eeq
where $|a|$ (respectively $|b|$ and $|c|$) stands for the degree of $a$ (respectively $b$ and $c$).
\eitem
\edfn 

\vskip 0.3cm

For a $\mathbb K$-superalgebra $L = L_{\bar 0} \oplus L_{\bar 1}$, $L_{\bar 0}$ is an ordinary
Lie algebra while $L_{\bar 1}$ is a $L_{\bar 0}$-module.
Each element $z$ of $L$ can be uniquely written:
\beq
z=z_{\bar 0} + z_{\bar 1}
\eeq
where $z_{\bar 0} \in L_{\bar 0}$ and $z_{\bar 1} \in L_{\bar 1}$. One says that $z_{\bar 0}$
is the component of degree $\bar 0$ of $z$ and $z_{\bar 1}$ is the component of
degree $\bar 1$ of $z$.

\subsection{Derivations of Lie Superalgebras}
See \cite{kac} for more about Lie superalgebras.
\bdfn
Let $L$ be a Lie superalgebra over $\mathbb K$.
\benum 
\item A derivation of degree $\bar 0$ of $L$ is an endomorphism $D: L \to L$ such  that
\beq
D(L_{\bar 0}) \subset L_{\bar 0} \quad ; \quad D(L_{\bar 1}) \subset L_{\bar 1} \quad ;
\quad D[a,b]=[D(a),b]+[a,D(b)]
\eeq
for every $a,b \in L$.
\item A derivation of degree $\bar 1$ of $L$ is an endomorphism $D: L \to L$ such  that
\beq
D(L_{\bar 0}) \subset L_{\bar 1} \quad ; \quad D(L_{\bar 1}) \subset L_{\bar 0} \quad ;
\quad D[a,b]=[D(a),b]+(-1)^{|a|}[a,D(b)]
\eeq
for every homogeneous element $a$ of $L$ and every element $b \in L$.
\item More generaly, a derivation of degree $r (r \in \Z_2:=\Z/2\Z)$ of a Lie superalgebra $L$ is an
endomorphism $D \in End_r(L):=\{\varphi \in End(L) : \varphi(L_s) \subset L_{r+s}\}$.
\eenum 
\edfn 

Note by $der_{\bar 0}(L)$ the set of derivations of $L$ of degree $\bar 0$ and by $der_{\bar 1}(L)$
the set of derivations of $L$ of degree $\bar 1$. Then,
\beq
der(L) = der_{\bar 0}(L) \oplus  der_{\bar 1}(L)
\eeq
Hence, to know the derivations of a Lie superalgebra $L$ it sufficies to know the derivations of degree
$\bar 0$ and the derivations of degree $\bar 1$. If $d$ is in $\der(L)$, we write
\beq
d = d_{\bar 0} + d_{\bar 1} 
\eeq
with $d_{\bar 0} \in der_{\bar 0}(L)$  and $d_{\bar 1} \in der_{\bar 1}(L)$.

\bpro\label{superalgebra-derivations}
Let $d = d_{\bar 0} + d_{\bar 1} \in der(L)$. Then,
\benum
\item About $d_{\bar 0}$:
\bitem
\item[a)] ${d_{\bar 0}}_{|L_{\bar 0}} : L_{\bar 0} \to L_{\bar 0}$ is a derivation of the Lie algebra
$L_{\bar 0}$.
\item[b)] If $z_{\bar 0} \in L_{\bar 0}$ and $z_{\bar 1} \in L_{\bar 1}$, then
\beq
d_{\bar 0}[z_{\bar 0},z_{\bar 1}] = [d_{\bar 0}(z_{\bar 0}),z_{\bar 1}] +
[z_{\bar 0},d_{\bar 0}(z_{\bar 1})];
\eeq
that is the endomorphism ${d_{\bar 0}}_{|L_{\bar 1}} : L_{\bar 1} \to L_{\bar 1}$ verifies
\beq
[d_{\bar 0},ad_{z_{\bar 0}}]_{|L_{\bar 1}} = {ad_{(d_{\bar 0}(z_{\bar 0}))}}_{|L_{\bar 1}}.
\eeq
\item[c)] For every $z_{\bar 1},z'_{\bar 1} \in L_{\bar 1}$,
\beq
d_{\bar 0}[z_{\bar 1},z'_{\bar 1}] = [d_{\bar 0}(z_{\bar 1}),z'_{\bar 1}] +
[z_{\bar 1},d_{\bar 0}(z'_{\bar 1})]
\eeq
\eitem
\item About $d_{\bar 1}$:
\bitem
\item[a)] The morphism ${d_{\bar 1}}_{|L_{\bar 0}} : L_{\bar 0} \to L_{\bar 1}$ verifies
\beq
d_{\bar 1}[z_{\bar 0},z'_{\bar 0}] = [d_{\bar 1}(z_{\bar 0}),z'_{\bar 0}]
+ [z_{\bar 0},d_{\bar 1}(z'_{\bar 0})],
\eeq
for every $z_{\bar 0},z'_{\bar 0} \in L_{\bar 0}$; {\it i.e.} ${d_{\bar 1}}_{|L_{\bar 0}} : L_{\bar 0} \to L_{\bar 1}$ is a $1$-cocycle.
\item[b)] The morphism ${d_{\bar 1}}_{|L_{\bar 1}} : L_{\bar 1} \to L_{\bar 0}$ satisfies
\beq
d_{\bar 1}[z_{\bar 0},z_{\bar 1}] = [d_{\bar 1}(z_{\bar 0}),z_{\bar 1}] + [z_{\bar 0},d_{\bar 1}(z_{\bar 1})]
\eeq
for every $z_{\bar 0} \in L_{\bar 0}$ and every $z_{\bar 1} \in L_{\bar 1}$
\eitem
\eenum
\epro

\section{Cohomology of Lie Algebras}

Let $\G$ be a Lie algebra over a field $\mathbb K$ of  characteristic zero.
\bdfn
A  $\G$-module is a linear space  $V$ of same dimension than $\G$ with a bilinear map
$\varphi : \G \times V \to V$ such that
\beq
\varphi ([x,y],v) = \varphi(x,\varphi(y,v)) - \varphi(y,\varphi(x,v)),
\eeq
\quad
for every elements  $x$, $y$ of $\G$  and every vector $v$ in $V$.
\edfn

A $\G$-module corresponds to a representation of $\G$ on the linear space $V$, 
that is a homomorphism $\Phi : \G \to gl(V)$ defined by:
\beq
\Phi(x)(v) = \varphi(x,v):= x \cdot v,
\eeq
for every $x$  in $\G$ and every $v$ in $V$.
\bdfn

Let  $V$ be a $\G$-module and  $p$ be a non-zero integer ($p\in \mathbb N^*$).
\bitem
\item A cochain of $\G$ of degree $p$ (or a $p$-cochain) with values in
 $V$ is a  skew-symmetric $p$-linear  map from $\mathfrak g^p = \overbrace{\mathfrak g \times \cdots \times
\mathfrak g}^{\mbox{ $p$ times}}$ to $V$.
\item A cochain of degree $0$ (or a $0$-cochain) of $\G$ with values in $V$ is a constant map from $\G$ to $V$.
\eitem
\edfn
Note by  $\mathcal C^p(\G,V)$ the space of $p$-cochains of $\G$ with values in $V$. We have:
\beq
\mathcal C^p(\G,V) = \left\{
\begin{array}{lll}
Hom(\Lambda^p \G,V), & \mbox{ si } p \geq 1;\cr
V                    & \mbox{ si } p = 0;\cr
\{0\}                & \mbox{ si } p < 0.
\end{array}
\right.
\eeq
The space of cochains of $\G$ with values in $V$ is denoted by
$\mathcal C^*(\G,V) := \oplus_p \mathcal C^p(\G,V)$. One endows $\mathcal C^p(\G,V)$ with
a $\G$-module structure by setting :
\beq
(x\cdot \Phi)(x_1, \ldots,x_p) = x \cdot \Phi(x_1, \ldots,x_p) -
\sum_{1 \leq i \leq p} \Phi(x_1, \ldots,x_{i-1},[x,x_i],x_{i+1}, \ldots,x_p)
\eeq
for all  $x,x_1, \ldots,x_p$ in $\G$ and all $\Phi$ in $\mathcal C^p(\G,V)$. This structure can be extended
to the space $\mathcal C^*(\G,V)$.

Let us now define the  endomorphism $d : \mathcal C^*(\G,V) \to \mathcal C^*(\G,V)$, called coboundary operator:
\bitem
\item If $\Phi$ is in $\mathcal C^0(\G,V)=V$  and $x$ is in $\G$, then
\beq
(d\Phi)(x) = d\Phi(x) = x \cdot \Phi.
\eeq
\item For $p \geq 1$, $x_1,\ldots,x_{p+1}$  in $\mathfrak g$ and 
$\Phi$ in $\mathcal C^p(\mathfrak g,V)$,
\beqn
(d\Phi)(x_1,\ldots,x_{p+1})\!\!\! & = &\!\!\! \sum_{1 \leq s \leq p+1} (-1)^{s+1}
x_s \cdot \Phi(x_1,\ldots,\hat{x}_s,\ldots,x_{p+1}) \\
&+&\!\!\!\!\!\! \sum_{1 \leq s < t \leq p+1} (-1)^{s+t}
\Phi([x_s,x_t],x_1,\ldots,\hat{x}_s,\ldots,\hat{x}_t,\ldots,x_{p+1}) \nonumber
\eeqn
\eitem

This endomorphism sends  $\mathcal C^p(\G,V)$ on $\mathcal C^{p+1}(\G,V)$.
One denotes by  $d_p$ the restriction of $d$ to the space $\mathcal C^p(\G,V)$ of $p$-cochains.
We have the
\bpro
$d^2 := d \circ d = 0$. More precisely $d_p \circ d_{p-1} = 0$, for every integer $p$.
\epro
The proof can be readed in \cite{goze-khakimdjanov}.

\vskip 0.3cm 

Set  $Z^p(\G,V):= \ker d_p$ and  $B^p(\G,V):=Im d_{p-1}$.
\bdfn
An element of  $Z^p(\G,V)$ is called a cocyle of degree $p$ (or $p$-cocycle) of $\G$ with values
in $V$ while an element of  $B^p(\G,V)$ is said to be a coboundary of  degree $p$ (or $p$-coboundary)
of  $\mathfrak g$ with values in $V$.
\edfn
Since  $d^2 = 0$, one has: $B^p(\G,V) \subset Z^p(\G,V)$. Then, we set, for
$p\geq 1$,
\beq 
H^p(\G, V):= Z^p(\G,V) / B^p(\G,V).
\eeq 
\bdfn
$H^p(\G, V)$ is the space of cohomology of $\G$ of degree $p$ (or  $p^{th}$ space of cohomology of
$\G$) with values in $V$.
\edfn

\section{Affine Lie Groups}

\bdfn
A $n$-dimensional affine manifold is smooth manifold $M$ equipped with a smooth atlas $(U_i,\Phi_i)_i$ such that
the transition functions ${\Phi_i \circ \Phi_j^{-1}}_{\mid \Phi_j(U_i\cap U_j)} : \Phi_j(U_i\cap U_j) \to \Phi_i(U_i\cap U_j)$, 
whenever they exist, are restrictions of affine transformations of $\R^n$. 
\edfn
The definition below is equivalent to the following: if $(U_{i},\Phi_{i})$
and $(U_{j},\Phi_{j})$ are two charts of the atlas satisfy $U_{i}\cap U_{j}$ then there exists an 
element $\theta_{ij}$ of the affine group $\hbox{\rm Aff}(\R^n)=GL(n,\R)\ltimes \R^n$ of $\R^n$ such that 
\beq 
{\Phi_i \circ \Phi_j^{-1}}_{\mid \Phi_j(U_i\cap U_j)} = {\theta_{ij}}_{\mid \Phi_j(U_i\cap U_j)}.
\eeq 
Recall that $\hbox{\rm Aff}(\R^n)$ is the group of diffeomorphimsms of $\R^n$ which preserve the standard connection 
$\nabla$ of $\R^n$:
\beq 
\nabla_XY:=\sum_{k=1}^n(X\cdot f_k)\frac{\pd}{\pd x_k},
\eeq 
where $Y=\sum_{k=1}^nf_k\frac{\pd}{\pd x_k}$. Every open set $U_i$ is then endowed with a connection $\nabla^i$ 
which is the reciprocal image by $\Phi_i$ of the connection induced on $\Phi_i(U_i)$ by $\nabla$. 
The connection $\nabla^i$ is torsion free and without 
curvature. Since the transition function preserve $\nabla$, the connections $\nabla^i$ can be perfectly glued into a 
zero-curvature and torsion free connection on the manifold $M$.

\bdfn
An application $f:M \to N$ between two affine manifold is called an affine transformation if
\bitem
\item it is smooth ;
\item for all charts $(U,\Phi)$ of $M$ and $(V,\Psi)$ of $N$ such that $f^{-1}(V)\cap U \neq \emptyset$, the 
function $\Psi\circ f \circ \Phi^{-1} : \Phi(f^{-1}(V)\cap U) \to \Psi(f^{-1}(V)\cap U)$ is the restriction of an element of 
$\hbox{\normalfont \rm Aff}(\R^n)$.
\eitem
\edfn

\bdfn
A Lie group $G$ endowed with a left-invariant affine structure is called an affine Lie group. In other words, an affine Lie group
is a Lie group equipped with an affine structure for which the left translations are affine transformations.
\edfn 
\noindent
The existence of affine structures is a difficult and interesting problem (\cite{helmsteter},\cite{milnor3}, 
\cite{lichnerowicz-medina88}, \cite{lichnerowicz91}).

\chapter{Automorphisms of Cotangent Bundles of Lie Groups}
\minitoc

\section{Introduction}

Let G be a Lie group whose Lie algebra $\G $ is
identified with its tangent space $T_{\epsilon}G$ at the unit
$\epsilon$. Throughout this work the cotangent bundle $T^*G$  of $G$, is
seen as a Lie group  which is obtained by the semi-direct product $G\ltimes\G^*$
of $G$ and the Abelian Lie group $\mathcal G^*$, where $G$ acts on the
dual space $\G^*$ of $\G$ via the coadjoint action. Here, using the
trivialization by left translations, the manifold underlying $T^*G$ has been
identified with the trivial  bundle $G\times \G^*$. We sometimes refer to the above Lie group structure 
on $T^*G,$ as its natural Lie group structure.
 The Lie algebra  $Lie(T^*G)=\G\ltimes
 \G^*$ of $T^*G$ will be denoted by $T^*\G$ or simply by $\D$.

\vskip 0.3cm

It is our aim in this work to fully study the connected component of the unit
of the  group of automorphisms of the Lie algebra $\mathcal D:=T^*\mathcal G$. Such a
connected component being spanned by exponentials of derivations of
$\D$, we will work with those derivations and the first
cohomology space $H^1(\D,\D)$, where $\D$ is
seen as the $\D$-module for the adjoint representation.

\vskip 0.3cm

Our motivation for this work comes from several interesting algebraic
and geometric problems.

\vskip 0.3cm

 The cotangent bundle $T^*G$ of a Lie group $G$ can exhibit very interesting and rich
algebraic and geometric structures (\cite{bajo-benayadi-medina},  \cite{di-me-cybe}, 
\cite{drinfeld}, \cite{feix}, \cite{kronheimer}, \cite{marle}).
Such structures can be better understood when one can compare, deform or
classify them. This very often involves the invertible homomorphisms
(automorphisms) of $T^*G$, if in particular, such structures are invariant under
 left or right multiplications by the  elements of $T^*G$.
The derivatives at the unit of automorphisms of the Lie group $T^*G$ are
automorphisms of the Lie algebra $\D.$ Conversely, if $G$ is connected and simply connected,
then so is $T^*G$ and every automorphism of the Lie algebra $\D$
integrates to an automorphism of the Lie group $T^*G$. A problem involving left or right  invariant
structures on a Lie group also usually transfers to one on its Lie algebra, with the Lie
algebra automorphisms used as a means to compare or classify the corresponding induced structures.

\vskip 0.3cm

In the purely algebraic point of view, finding and understanding the
derivations of a given Lie algebra, is in itself an interesting
problem  (\cite{deruiter}, \cite{dixmier57},  \cite{jacobson55}, \cite{jacobson37},
\cite{leger63}, \cite{togo67}).

\vskip 0.3cm

On the other hand, as a Lie group, the cotangent bundle $T^*G$ is a common
Drinfel'd double Lie group for all exact Poisson-Lie structures given by
solutions of the Classical Yang-Baxter Equation in $G$. See e.g. \cite{di-me-poisson}.
Double Lie algebras (resp. groups) encode information on
integrable Hamiltonian systems and Lax pairs (\cite{babelon-viallet}, \cite{bordemann}, 
\cite{drinfeld}, \cite{lu-weinstein}), Poisson homogeneous spaces of Poisson-Lie groups and 
the symplectic foliation of the corresponding Poisson structures (\cite{di-me-poisson}, \cite{drinfeld}, 
\cite{lu-weinstein}). To that extend, the complete description of the group of automorphisms of the double Lie algebra of a
Poisson-Lie structure would be a big contribution towards solving very
interesting and hard problems. See Section \ref{openproblems} for wider discussions.

\vskip 0.3cm

Interestingly, the space of derivations of  $\D$
encompasses interesting spaces of operators on  $\mathcal G$, among
which the derivations of  $\G$, the second space of the left invariant de
Rham cohomology $H^2_{inv}(G,\R)$ of $G$, bi-invariant endomorphisms,
in particular operators giving rise to complex group structure in
$G$, when they exist.

\vskip 0.3cm

Throughout this work,  $\der(\G)$ will stand for the Lie algebra of derivations of
$\G$, while $\mathcal J$ will denote that of linear maps $j:\G\to\G$ satisfying
$j([x,y])=[j(x),y]$, for every elements $x,y$ of $\G$.  We consider Lie groups and
Lie algebras over the field $\R$. However, most of the results within this chapter are
valid for any field of characteristic zero.

\medskip
We summarize some of our main results as follows.

\vskip 0.3cm

\noindent
{\bf Theorem~A.} ~ {\em   Let $G$ be a
Lie group, $\G$ its Lie algebra, $T^*G$ its cotangent bundle and $\D:=\G\ltimes \G^*=Lie(T^*G)$.
A derivation of $\D$, has the form
\beq
\phi(x,f) = \Big(\alpha(x) + \psi(f),\beta(x) + f\circ(j-\alpha)\Big), \nonumber
\eeq
for all $(x,f)$ in $\D$; where 
\bitem 
\item $\alpha: \mathcal G \to \mathcal G$ is a derivation of the Lie algebra $\mathcal G$; 
\item the linear map $j:\G\to\G$ is in $\mathcal J$; 
\item $\beta :\mathcal G \to \mathcal G^*$ is a $1$-cocycle of $\mathcal G$ with values in $\mathcal G^*$ for the coadjoint 
action of $\mathcal G$ on $\mathcal G^*$; 
\item $\psi : \mathcal G^* \to \mathcal G$ is a linear map satisfying the following conditions: 
for all  $x$ in  $\mathcal G$ and all $f,g$ in $\mathcal G^*$,
\beq
\psi \circ ad^*_x = ad_x \circ \psi \text{~ and ~}
ad^*_{\psi(f)}g = ad^*_{\psi(g)}f. \nonumber
\eeq
\eitem 

\vskip 0.3cm

 If $G$ has a bi-invariant Riemannian or pseudo-Riemannian metric, say $\mu,$ then every 
derivation $\phi$ of $\D$ can be expressed in terms of elements of $\der(\G)$ and $\mathcal J$ alone, as follows,
 \beq
 \phi(x,f) = \Big(\alpha(x) + j\circ\theta^{-1}(f),
\theta\circ\alpha'(x) + f\circ(j'-\alpha)\Big), \nonumber
\eeq
for any element  $(x,f)$ of $\D$, where $\alpha,\alpha'$ are derivations of $\G$;  the maps $j,j'$ are in
$\mathcal J$ and $\theta:\G\to\G^*$ with $\theta(x)(y):=\mu(x,y),$  for every elements $x,y $ of $\G$.
 }

\vskip 0.3cm

\noindent {\bf Theorem~B.} ~ {\em  Let $G$ be a
Lie group and $\G$ its Lie algebra. The group Aut($\D$) of
automorphisms of the Lie algebra $\D$ of the cotangent bundle $T^*G$
of $G$, is a super symmetric Lie group. More precisely, its Lie algebra
$der(\mathcal D)$ is a $\mathbb Z/2\mathbb Z$-graded symmetric (super-symmetric) Lie algebra which decomposes 
into a direct sum of vector spaces
\beq
 der(\mathcal D):=\mathcal G_0\oplus\mathcal G_1,  \text{~
 with~ } [\mathcal G_i,\mathcal G_{i'}]\subset\mathcal
 G_{i+i'}, ~~ i,i' \in \mathbb Z/2\mathbb Z = \{0, 1\} \nonumber
\eeq
 where  $\mathcal G_0$ is the Lie algebra
 \beq
 \mathcal G_0:=\Big\{ \phi:\D\to\D,  \phi (x,f)= \Big(\alpha(x)~, ~ f\circ(\alpha -j)  \Big), 
{~with~} \alpha\in der(\mathcal G) \text{~and~} j\in\mathcal J\Big\} \nonumber
 \eeq
  and
 $\mathcal G_1$ is the direct sum (as a vector space) of the space $\mathcal Q$ of 1-cocycles
 $\beta:~\mathcal G\to\mathcal G^*$ and the space $\Psi$ of equivariant maps $\psi: ~\mathcal G^*\to\mathcal G$ with respect
 to the coadjoint and the adjoint representations, satisfying 
$$ad^*_{\psi(f)}g=ad^*_{\psi(g)}f,$$
for every elements  $f,g$ of $\G^*$. 

\vskip 0.3cm 

Moreover, $\G_0\oplus\tilde \G_1$ and $\G_0\oplus\tilde \G_1'$ are subalgebras 
of $\der(\D)$ which are Lie superalgebras, i.e. they are $\mathbb Z/2\mathbb Z$-graded Lie algebras with 
the Lie bracket satisfying 
$$[x,y]=-(-1)^{deg(x)deg(y)}[y,x],$$ 
where $\tilde \G_1:=\mathcal Q$ and 
$\tilde \G_1':=\Psi$ are Abelian subalgebras of $der(\D)$ and $deg(x)=i$, if $x\in\G_i$.
}

\vskip 0.5cm

The Lie superalgebras $\G_0\oplus\tilde \G_1$ and $\G_0\oplus\tilde \G_1'$ respectively correspond to the 
subalgebras of all elements of der($\D$) which preserve the subalgebra $\G$ and the ideal $\G^*$ of $\D$.

\vskip 0.3cm

\noindent 
{\bf Theorem~C.} ~ {\em The first cohomology space $H^1(\D,\D)$ of the (Chevalley-Eilenberg) cohomology 
associated with the adjoint action of $\D$ on itself, satisfies
$$
H^1(\D,\D) \stackrel{\sim}{=} H^1(\G,\G)\oplus  \mathcal J^t\oplus H^1(\G,\G^*)\oplus \Psi,
$$  
where $H^1(\G,\G)$ and $H^1(\G,\G^*)$ are the first cohomology spaces associated with the adjoint 
and coadjoint actions of $\G$, respectively; and
$\mathcal J^t:=\{j^t, j \in \mathcal J \}$ (space of transposes of elements of $\mathcal J$).

\vskip 0.3cm

If $\G$ is semi-simple, then  $\Psi=\{0\}$ and thus $H^1(\D,\D) \stackrel{\sim}{=}  \mathcal J^t.$ Moreover, we have
$\mathcal J \stackrel{\sim}{=} \R^p$,
where $p$ is the number of the simple
 ideals $\s_i$ of $\G$ such that $\G=\s_1\oplus\cdots\oplus \s_p$. Hence, of course, $H^1(\D,\D) \stackrel{\sim}{=}   \R^p.$

\vskip 0.3cm

If $\G$ is a compact Lie algebra, with centre $Z(\G)$, we get 
\beqn
H^1(\G,\G) &\stackrel{\sim}{=}& End(Z(\G)), \nonumber \\
\mathcal J &\stackrel{\sim}{=}& \R^p\oplus End(Z(\G)), \nonumber \\
H^1(\G,\G^*) &\stackrel{\sim}{=}& L(Z(\G), Z(\G)^*), \nonumber \\
\Psi &\stackrel{\sim}{=}& L(Z(\G)^*, Z(\G)). \nonumber
\eeqn 
Hence, we get 
$$
 H^1(\D,\D) \stackrel{\sim}{=} (End(\R^k) )^4\oplus \R^p,
$$ 
where $k$ is the dimension of the center of $\G$, 
and $p$ is the number of the simple
components of the derived ideal $[\G,\G]$ of $\G$. Here, if $E,F$ are vector spaces, $L(E,F)$ is the space of 
linear maps $E\to F.$
}

\vskip 0.3cm

Traditionally, spectral sequences are used as a powerful tool for the study  of the cohomology spaces of 
extensions of Lie groups or Lie algebras and more generally, of locally trivial fiber bundles 
(see e.g. \cite{knapp}, \cite{neeb} for very interesting results and discussions). However, for the purpose 
of this investigation, we use a direct approach. Some parts of Theorem C can also be seen as a refinement, 
using our direct approach,  of some already known results (\cite{knapp}).
 Its proof is given by different lemmas and propositions, discussed in Sections \ref{chap:cohomology} and  
\ref{chap:orthogonal-algebra}.

\vskip 0.3cm

The chapter is organized as follows. In Section \ref{chap:preliminaries}, we explain some of the material and 
terminology needed to make the chapter more self contained. Sections \ref{chap:automorphisms} and 
\ref{chap:orthogonal-algebra} are the actual core of the work where the main calculations and proofs of 
theorems are carried out. In Section \ref{openproblems}, we discuss some subjects related to this work, 
as well as some of the possible applications.

\section{Preliminaries}\label{chap:preliminaries}

Although not central to the main purpose of the work within this chapter, the following material might be 
useful, at least, as regards parts of the terminology used throughout this chapter.

\subsection{The Cotangent Bundle of a Lie Group}\label{chap:notations}

Throughout this chapter, given a Lie group $G$, we will always let $G\ltimes \G^*$ stand for the Lie group 
consisting of the Cartesian product $G\times \G^*$ as its underlying manifold, together with the group 
structure obtained by semi-direct product using the coadjoint action of $G$ on $\G^*.$
Recall that  the trivialization by left translations, or simply the left trivialization of $T^*G$ is given 
by the following isomorphism $\zeta$ of vector bundles
\beq \zeta:~ T^*G \to G\times \G^*, ~~ (\sigma,\nu_\sigma)\mapsto (\sigma,\nu_\sigma\circ T_\epsilon L_\sigma),\nonumber
\eeq
where $L_\sigma$ is the left multiplication  $L_\sigma:G\to G,$ ~ $\tau\mapsto L_\sigma(\tau):=\sigma\tau$ ~ by $\sigma$
 in $G$ and $T_\epsilon L_\sigma$ is the derivative of $L_\sigma$ at the unit $\epsilon.$
In this chapter, $T^*G$ will always be endowed with the Lie group structure such that  $\zeta$ is 
an isomorphism of Lie groups. The Lie algebra of $T^*G$ is then the semi-direct product $\D:=\G\ltimes\G^*.$ 
More precisely, the Lie bracket on $\D$ reads
\beq \label{bracket_double}
[(x,f),(y,g)]:=([x,y],ad^*_xg-ad^*_yf),
\eeq
for any two elements $(x,f)$ and $(y,g)$ of $\D$.


In this work, we will refer to an object which is invariant
under both left and right translations in a Lie group $G$, as a
bi-invariant object. We discuss in this section, how $T^*G$ is naturally endowed with a 
bi-invariant  pseudo-Riemannian metric.

\vskip 0.3cm

\vskip 0.3cm

The cotangent bundle of any Lie group (with its natural Lie group
structure, as above) and in
general any element of the larger and interesting family of the so-called
Drinfel'd doubles (see Section \ref{double}),  are orthogonal Lie groups  (\cite{drinfeld}),
as explained below.

\vskip 0.3cm

 As above, let  $\mathcal D:=\G \ltimes \G^*$
be the Lie algebra of the cotangent bundle $T^*G$ of $G$, seen as the
semi-direct product of $\G$ by $\G^*$ via the coadjoint
action of $\G$ on $\G^*$, as in (\ref{bracket_double}).
Let $\mu_0$ stand for the duality pairing $\langle,
 \rangle$, that is, for all  $(x,f), \; (y,g)$ in $\D$,
\beq\label{eq:dualitypairing}
\mu_0\Big((x,f),(y,g)\Big) = f(y) + g(x).
\eeq
 Then, $\mu_0$ satisfies the property (\ref{ad-invariant})  on $\D$ and hence gives rise to a
 bi-invariant (pseudo-Riemannian) metric on $T^*G$.

\subsection{ Doubles of  Poisson-Lie Groups and Yang-Baxter Equation}\label{double}

We explain in this section how cotangent bundles of Lie groups are part of the broader family
of the so-called double Lie groups of Poisson-Lie groups.

\vskip 0.3cm

A Poisson structure on a manifold $M$ is given by a Lie bracket $\{,\}$ on the space 
$\mathcal C^\infty(M,\R)$ of smooth real-valued functions on $M,$ such that, for each
$f$ in $\mathcal C^\infty(M,\R),$ the linear operator $X_f:=\{f,.\}$ on
$\mathcal C^\infty(M,\R),$ defined by  $g\mapsto X_f\!\cdot\!g:=\{f,g\},$
is a vector field on $M$. The bracket $\{,\}$ defines a $2$-tensor, that is,
a bivector field $\pi$ which, seen as a bilinear skew-symmetric 'form' on the
space of differential $1$-forms on $M$, is given by  $\pi(df,dg):=\{f,g\}.$
The Jacobi identity for  $\{,\}$ now reads $[\pi,\pi]_S=0$, where $[,]_S$ is the
so-called Schouten bracket, which is a natural extension to all multi-vector fields,
of the natural Lie bracket of vector fields. Reciprocally, any bivector field $\pi$ on
$M$ satisfying $[\pi,\pi]_S=0,$ is a Poisson tensor, {\it i.e.} defines a Poisson structure
on $M$. See e.g.  \cite{lu-weinstein}.

\vskip 0.3cm

Recall that a Poisson-Lie structure on a Lie group $G,$ is given by a Poisson tensor $\pi$ on $G,$ 
such that, when the Cartesian product $G\times G$ is equipped with the Poisson tensor $\pi\times\pi$, 
the multiplication $m: ~(\sigma,\tau)\mapsto \sigma\tau$ is a Poisson map between the Poisson manifolds 
$(G\times G, \pi\times\pi)$ and $(G, \pi).$ In other words, the derivative $m_*$ of $m$ satisfies 
$m_*(\pi\times\pi)=\pi.$ If $f,$ $g$ are in $\G^*$ and $\bar f,\bar g$ are $\mathcal C^\infty$ 
functions on $G$ with respective derivatives $f=\bar f_{*,\epsilon},$ ~ $g = \bar g_{*,\epsilon}$ at 
the unit $\epsilon$ of $G,$ one defines another element $[ f, g]_*$ of $\G^*$ by setting 
$[f,g]_*:=(\{\bar f,\bar g\})_{*,\epsilon}.$ Then $[f,g]_*$ does not depend on the choice 
of $\bar f$ and $\bar g$ as above, and $(\G^*,[,]_*)$ is a Lie algebra. Now, there is a symmetric 
role played by the spaces $\G$ and $\G^*,$ dual to each other. Indeed, as well as acting on $\G^*$ 
via the coadjoint action, $\G$ is also acted on by $\G^*$ using the coadjoint action of $(\G^*,[,]_*).$
A lot of the most interesting properties and applications of $\pi,$ are encoded in the new Lie 
algebra $(\G\oplus\G^*, [,]_\pi),$ where
\beq
[(x,f),(y,g)]_\pi:=([x,y]+ad^*_f y - ad^*_g x, ad^*_x g - ad^*_y f+ [f,g]_*),
\eeq
for every $x,y$ in $\G$ and every $f,g$ in $\G^*.$

\vskip 0.3cm

The Lie algebras $(\G\oplus\G^*,[,]_\pi)$ and $(\G^*, [,]_*)$ are respectively called the double 
and the dual Lie algebras of the Poisson-Lie group $(G,\pi)$. Endowed with the duality pairing defined 
in (\ref{eq:dualitypairing}), the double Lie algebra of any Poisson-Lie group ($G,\pi$), is an orthogonal 
Lie algebra, such that $\G$ and $\G^*$ are maximal totally isotropic (Lagrangian) subalgebras. 
The collection $(\G\oplus\G^*,\G,\G^*) $ is then called a Manin triple. More generally, 
$(\G_1\oplus\G_2,\G_1,\G_2)$ is called a Manin triple and $(\G_1,\G_2)$ a bi-algebra or a 
Manin pair, if  $\G_1\oplus\G_2$ is an orthogonal $2n$-dimensional Lie algebra whose underlying 
adjoint-invariant pseudo-Riemannian metric is of index $(n,n)$ and $\G_1,\G_2$ are two Lagrangian complementary subalgebras.
See \cite{di-me-poisson},\cite{drinfeld}, \cite{lu-weinstein}  for wider discussions.

\vskip 0.3cm

Let $r$ be an element of the wedge product $\wedge^2\G.$ Denote by $r^+$ (resp. $r^-$) the left
(resp. right) invariant bivector field on $G$ with value $r=r^+_\epsilon$ ~ (resp. $r=r^-_\epsilon$)
at $\epsilon$. If $\pi_r:=r^+-r^-$ is a Poisson tensor, then it is a Poisson-Lie tensor and $r$
is called a solution of the Yang-Baxter Equation. If, in particular, $r^+$  is a (left invariant)
Poisson tensor on $G$, then $r$ is called a solution of the Classical Yang-Baxter Equation (CYBE)
on $G$ (or $\G$). In this latter case, the double Lie algebra  $(\G\oplus\G^*,[,]_{\pi_r})$
is isomorphic to the Lie algebra $\D$ of the cotangent bundle $T^*G$ of $G$. See e.g. \cite{di-me-cybe}.
 We may also consider the linear map $\tilde r:\G^*\to\G,$ where $\tilde r (f):=r(f,.).$
The linear map $\theta_r:~(\G\oplus\G^*,[,]_{\pi_r})\to \D$, $\theta_r(x,f):= (x+\tilde r(f),f),$
is an isomorphism of Lie algebras, between $\D$ and the double Lie algebra of any
Poisson-Lie group structure on $G,$ given by a solution $r$ of the CYBE.

\section{Group of Automorphisms of $\D:=T^*\G$}\label{chap:automorphisms}

\subsection{Derivations of $\D:=T^*\G$}

Consider a Lie group $G$ of dimension $n$, with Lie algebra $\G.$
Let us also denote by $\D$ the vector space underlying the Lie algebra
$\mathcal D$ of the cotangent bundle $T^*G$,
regarded as a $\D$-module under the adjoint action of $\D$. Consider the following
complex with the coboundary operator $\partial$,
where $\partial \circ \partial = 0$:
\beq
0 \to \D \to Hom(\mathcal D,\D) \to Hom(\Lambda^2\mathcal D,\D)  \to \cdots
\to Hom(\Lambda^{2n}\mathcal D,\D) \to 0.
\eeq
We are interested in $Hom(\mathcal D,\D):= \{\phi : \mathcal D \to \D, \phi \mbox{ linear } \}$.
The coboundary $\partial \phi$ ofthe element $\phi$ of $Hom(\mathcal D,\D)$ is the
element of $Hom(\Lambda^2\mathcal D,\D)$ defined by
\beq
\partial \phi (u,v):= ad_u\big(\phi(v)\big) - ad_v\big(\phi(u)\big) - \phi([u,v]),
\eeq
for any elements $u=(x,f)$ and $ v=(y,g)$ in $\mathcal D$. An element $\phi$ of $Hom(\mathcal D,\D)$
is a $1$-cocycle if $\partial \phi = 0$, {\it i.e. }
\beqn\label{suz1}
\phi([u,v]) &=& ad_u\Big(\phi(v)\Big) - ad_v\Big(\phi(u)\Big),\nonumber \\
            &=& [u,\phi(v)]+[\phi(u),v].
\eeqn
In other words,  $1$-cocycles are the derivations of the Lie
algebra $\mathcal D$. In Section \ref{chap:cohomology}, we will characterize the first cohomology 
space $H^1(\D,\D):=\ker(\partial^2)/Im(\partial^1)$ of the associated Chevalley-Eilenberg cohomology, 
where for clarity, we have denoted by $\partial^1$ and $\partial^2$ the following restrictions 
$\partial^1 :~ \D\to Hom(\D,\D)$ and $\partial^2:~Hom(\D,\D)\to Hom(\wedge^2\D,\D)$ of the coboudary operator $\partial.$

\begin{theorem}\label{derivationschar} 
Let $G$ be a
Lie group, $\G$ its Lie algebra, $T^*G$ its cotangent bundle and $\D:=\G\ltimes \G^*$ the Lie algebra of $T^*G$.
A $1$-cocycle (for the adjoint representation) hence a derivation of $\mathcal D$ has the following form:
\beq
\phi(x,f) = \Big(\alpha(x) + \psi(f), \beta(x) + \xi(f)\Big),
\eeq
for any $(x,f)$ in $\D$; where 
\bitem
\item[-] $\alpha: \mathcal G \to \mathcal G$ is a derivation of the Lie algebra $\mathcal G$, 
\item[-] $\beta :\mathcal G \to \mathcal G^*$ is a $1$-cocycle of $\mathcal G$ with values 
in $\mathcal G^*$ for the coadjoint action of $\mathcal G$ on $\mathcal G^*$,
\item[-] $\xi : \mathcal G^* \to \mathcal G^*$ and $\psi : \mathcal G^* \to \mathcal G$ 
are linear maps satisfying the following conditions:
\eitem 
\beqn
[\xi,ad^*_x] &=& ad^*_{\alpha(x)}, \quad  \forall \; x \in \mathcal G, \label{relation-xi-alpha} \\
\psi \circ ad^*_x &=& ad_x \circ \psi, \quad \forall \; x \in \mathcal G, \label{relation-equivariance} \\
ad^*_{\psi(f)}g &=& ad^*_{\psi(g)}f, \quad \forall \; f,g \in \mathcal G^*.\label{relation-commutation}
\eeqn

\end{theorem}

The rest of this section is dedicated to the proof of Theorem \ref{derivationschar}.

\vskip 0.3cm 
Aiming to get a simpler expression for the derivations, let us write $\phi$ in terms of
its components relative to the decomposition of $\D$
into a direct sum $\D =\mathcal G \oplus  \G^*$ of vector spaces as follows: for all $(x,f)$ in $\D$,
\beq\label{yacine}
\phi(x,f)=\Big(\phi_{11}(x)+\phi_{21}(f),\phi_{12}(x)+\phi_{22}(f)\Big),
\eeq
where $\phi_{11}:\mathcal G \to \mathcal G$, $\phi_{12}: \mathcal G \to
\mathcal G^*$, $\phi_{21}: \mathcal G^* \to \mathcal G$ and
$\phi_{22}:\mathcal G^* \to \mathcal G^*$ are all linear maps. In (\ref{yacine}) we have
made the identifications: $x=(x,0),\; f=(0,f)$ so that the element $(x,f)$ can also
be written $x+f$. Likewise, we can write
\beq
\phi(x)=(\phi_{11}(x),\phi_{12}(x)) \quad ; \quad \phi(f)=(\phi_{21}(f),\phi_{22}(f)),
\eeq
for any $x$ in $\G$ and any $f$ in $\G^*$; or simply
\beq
\phi(x)=\phi_{11}(x)+\phi_{12}(x)\quad ; \quad \phi(f)=\phi_{21}(f)+\phi_{22}(f).
\eeq

In order to find the $\phi_{ij}$'s and hence  all the derivations of
$\mathcal D$, we are now going to use the cocycle condition (\ref{suz1}).

For $x,y$ in $\mathcal G \subset \mathcal D$ we have:
\beq\label{suz2}
\phi([x,y]) = \phi_{11}([x,y]) + \phi_{12}([x,y])
\eeq
and
\beqn\label{suz3}
[\phi(x),y] + [x,\phi(y)] &=& [\phi_{11}(x) + \phi_{12}(x),y] + [x,\phi_{11}(y) + \phi_{12}(y)] \nonumber \\
                          &=& [\phi_{11}(x),y] - ad^*_y(\phi_{12}(x)) + [x,\phi_{11}(y)] +
                              ad^*_x(\phi_{12}(y)).
\eeqn
Comparing (\ref{suz2}) and (\ref{suz3}), we first get
\beq\label{phi11}
\phi_{11}([x,y]) = [\phi_{11}(x),y] + [x,\phi_{11}(y)],
\eeq
for every  $x,y$ in $\G$. This means that  $\phi_{11}$ {\em is a derivation of
the Lie algebra $\mathcal G$}.

Secondly, for all elements $x$, $y$ of $\G$, we have
\beq\label{suz5}
\phi_{12}([x,y])= ad^*_x(\phi_{12}(y)) - ad^*_y(\phi_{12}(x)).
\eeq
 Equation (\ref{suz5}) means that {\em  $\phi_{12} : \mathcal G \to
\mathcal G^*$ is a $1$-cocycle of $\mathcal G$ with
values on $\mathcal G^*$ for the coadjoint action of $\mathcal G$ on
$\mathcal G^*$}.

Now we are going to examine the following case: for all $x$ in $\mathcal G$ and all $f$ in $\mathcal G^*$,
\beqn\label{suz6}
\phi([x,f]) &=& \phi(ad^*_xf), \nonumber \\
            &=& \phi_{21}(ad^*_xf) + \phi_{22}(ad^*_xf),
\eeqn
and
\beqn\label{suz7}
[\phi(x),f] + [x,\phi(f)]&=& [\phi_{11}(x)+\phi_{12}(x),f]+[x,\phi_{21}(f) + \phi_{22}(f)],\nonumber \\
                         &=& ad^*_{\phi_{11}(x)}f + [x,\phi_{21}(f)] + ad^*_x\big(\phi_{22}(f)\big).
\eeqn
Identifying (\ref{suz6}) and (\ref{suz7}) we obtain on the one hand
\beqn
\phi_{21}(ad^*_xf) &=& [x,\phi_{21}(f)], \nonumber \\
                   &=& ad_x(\phi_{21}(f)),\nonumber
\eeqn
for every  $x$ in $\mathcal G$, and every $f$ in $\mathcal G^*$. We write the above as
\beq
\phi_{21} \circ ad^*_x = ad_x \circ \phi_{21}, \nonumber
\eeq
for all $x$ in $\mathcal G$. That is,  {\em $\phi_{21}:\mathcal G^* \to \mathcal G$ is equivariant
(commutes) with respect to the adjoint and the coadjoint actions of
$\mathcal G$
on $\mathcal G$ and $\mathcal G^*$ respectively}.

We have on the other hand
\beq\label{suz8}
\phi_{22}(ad^*_xf) = ad^*_x(\phi_{22}(f)) + ad^*_{\phi_{11}(x)}f,
\eeq
for all $x$ in $\mathcal G$ and all $f$ in $\mathcal G^*$. Formula (\ref{suz8}) can be rewritten as
\beq
\phi_{22} \circ ad^*_x - ad^*_x \circ \phi_{22} = ad^*_{\phi_{11}(x)}, \nonumber
\eeq
{\it i.e.} for any element $x$ of $\G$,
\beq
 [\phi_{22},ad^*_x] = ad^*_{\phi_{11}(x)}. \nonumber
\eeq
Last, for $f$ and $g$ in $\G^*,$ we have
\beq\label{suz11}
\phi([f,g]) = 0,
\eeq
and
\beqn\label{suz12}
[\phi(f),g] + [f,\phi(g)] &=& [\phi_{21}(f)+\phi_{22}(f),g]+[f,\phi_{21}(g)+\phi_{22}(g)], \nonumber \\
                          &=& ad^*_{\phi_{21}(f)}g - ad^*_{\phi_{21}(g)}f.
\eeqn
From (\ref{suz11}) and (\ref{suz12}) it comes that for all elements $f,g$  of $\G^*$,
\beq
ad^*_{\phi_{21}(f)}g = ad^*_{\phi_{21}(g)}f.\nonumber
\eeq

Noting $\alpha :=\phi_{11}$, $\beta:=\phi_{12}$, $\psi:= \phi_{21}$ and $\xi:=\phi_{22}$, we get 
a proof of Theorem  \ref{derivationschar}. \qed

\brmq\label{notations}({\bf Notations }) 
From now on, if $\G$ is a Lie algebra, then 
\benum 
\item $\mathcal E$ will stand for the space of linear maps $\xi : \mathcal G^* \to \mathcal G^*$  
satisfying Equation (\ref{relation-xi-alpha}), for some derivation $\alpha$ of $\mathcal G$;

\item set $$\!\! \G_0\!:=\!\Big\{\!\phi: \D\to\D, \phi (x,f) \!=\! (\alpha(x), \xi(f))\!:\! \alpha\in der(\G), 
\xi\in\mathcal E, [\xi,ad^*_x]=ad^*_{\alpha(x)}, \forall x\in\G\Big\};$$

\item  we may let $\mathcal Q$ stand for the space of 1-cocycles
$\beta \!:\! \mathcal G\to \mathcal G^*$  as in (\ref{suz5}),
whereas $\Psi$ may be used for the space of equivariant linear $\psi\!:\! \mathcal G^* \to \mathcal G$
as in (\ref{relation-equivariance}), which satisfy (\ref{relation-commutation});

\item  we will denote by $\G_1,$ the direct sum  $\G_1:=\mathcal Q\oplus\Psi$ of the vector spaces $\mathcal Q$ and $\Psi$.
\eenum 
\ermq

\brmq\label{embeddings} 
The spaces $der(\G)$ of derivations of $\G$, $\mathcal Q$ and $\Psi$, as in 
Remark \ref{notations}, are all subsets of $der(\D),$ as follows. A derivation $\alpha$ of $\G$, 
an equivariant map $\psi$ in $\Psi$, and a $1$-cocycle $\beta$ in $\mathcal Q$ are 
respectively seen as the 
elements $ \phi_{\alpha}, \phi_{\psi}, \phi_{\beta}$ of $der(\D)$, with
\beqn 
\phi_{\alpha}(x,f)&:=& (\alpha(x),-f\circ\alpha) ; \cr 
\phi_{\psi}(x,f)  &:=& (\psi(f),0) ;               \cr 
\phi_{\beta}(x,f) &:=& (0,\beta(x)), \nonumber
\eeqn 
for all $(x,f)$ in $\D$.
\ermq

\bcor 
Every derivation of $\G$ is the restriction to $\G$ of a derivation of $\D.$
\ecor

\subsection{A Structure Theorem {\tiny  for} the Group of Automorphisms of $\D$}

\blem\label{structure1}
 The space $\mathcal E,$ as in Remark \ref{notations}, is a Lie algebra.

Namely, if $\xi_1, \xi_2$ in $\mathcal E$ satisfy
\beq
[\xi_1,ad^*_x] = ad^*_{\alpha_1(x)} \text{ and  } [\xi_2,ad^*_x] =ad^*_{\alpha_2(x)}, \nonumber
\eeq
for all $x$ in $\mathcal G$ and some $\alpha_1,\alpha_2$ in der($\G$),
then their Lie bracket $[\xi_1, \xi_2]$ is in $\mathcal E$ and satisfies
\beq
[[\xi_1, \xi_2],ad^*_x] = ad^*_{[\alpha_1,\alpha_2](x)}. \nonumber
\eeq
\elem
\begin{proof} Using Jacobi identity in the Lie algebra $\G l(\G^*)$ of endomorphisms
of the vector space $\G^*$, we get: for any $x$ in $\G$
\beqn
[[\xi_1, \xi_2],ad^*_x]&=& [[\xi_1,ad^*_x], \xi_2] +  [\xi_1, [\xi_2,ad^*_x]]\nonumber\\
 &=& [ad^*_{\alpha_1(x)}, \xi_2] + [\xi_1, ad^*_{\alpha_2(x)}]\nonumber\\
 &=& - ad^*_{\alpha_2\circ\alpha_1(x)} + ad^*_{\alpha_1\circ\alpha_2(x)}\nonumber\\
 &=& ad^*_{[\alpha_1,\alpha_2](x)}.\nonumber
\eeqn
\end{proof}

\blem \label{structure2}
The space $\G_0,$ as in Remark \ref{notations}, is a Lie subalgebra of
$der(\D)$.
\elem
\begin{proof}
This is a consequence of Lemma \ref{structure1}. If $\phi_1:= (\alpha_1, \xi_1),$ and 
$\phi_2:= (\alpha_2, \xi_2)$ are in $\G_0,$ then $[\phi_1,\phi_2]=([\alpha_1,\alpha_2], [\xi_1,\xi_2]),$ 
as can easily be seen, below. For every $(x,f)$ in $\D,$ we have
\beqn
 [\phi_1,\phi_2](x,f) & = & \phi_1 \Big( \alpha_2(x),\xi_2(f) \Big) - \phi_2 \Big( \alpha_1(x),\xi_1(f) \Big)
  \nonumber\\
  & = & \Big( \alpha_1\circ\alpha_2(x)~,~ \xi_1\circ\xi_2(f)\Big) - \Big( \alpha_2\circ\alpha_1(x)~,~ \xi_2\circ\xi_1(f) \Big)\nonumber\\
 & =& \Big( [\alpha_1,\alpha_2](x), [\xi_1,\xi_2](f) \Big). \nonumber
\eeqn
\end{proof}

\blem \label{structure3}
Let $\beta $  and $\psi$ be in $\mathcal Q$ and $\Psi$, respectively. Then
$[\beta,\psi]=(-\psi\circ\beta,\beta\circ\psi)$ belongs to $\G_0,$ more precisely
$\beta\circ\psi$ is in $\mathcal E $, $\psi\circ\beta$ is in $der(\mathcal G)$ and
$[\beta\circ\psi,ad_x^*]=-ad^*_{\psi\circ \beta(x)}$, for any $x$ in $\G$.
\elem
\begin{proof}
First, $\beta$ being a 1-cocycle is equivalent to
\beq
\label{eq:cocycle}
\beta\circ ad_x (y) = ad^*_x \circ \beta (y) -
ad^*_y\circ \beta (x),
\eeq
for all $x,y$ in $\G$. Now for every  $x$ in $\G$ and every $f$ in $\G^*,$ we have
\beqn
[\beta\circ\psi ,ad_x^*](f) &= & \beta\circ\psi \circ ad_x^*(f)-
ad_x^*\circ\beta\circ\psi (f)\nonumber \\
& = & \beta\circ ad_x\circ \psi (f) - ad_x^*\circ\beta\circ\psi(f), ~\text{now take $y=\psi (f)$ in (\ref{eq:cocycle})}\nonumber\\
& = & ad_{x}^*\circ\beta\circ  \psi (f) - ad_{\psi(f)}^*\beta(x) - ad_{x}^*\circ\beta\circ\psi(f),
              \nonumber   \\
& = &  - ad_{\psi(f)}^*\beta(x),~ \text{take $g=\beta(x)$ in
 (\ref{relation-commutation})} \nonumber\\
& = &  - ad_{\psi\circ \beta(x)}^*f = ad_{\alpha(x)}^*f, ~\text{where}~ \alpha= - \psi\circ \beta.\nonumber
\eeqn
Next, the proof that $\psi\circ\beta$ is in der($\G$), is
straightforward. Indeed, for every elements $x,y$ in $\G$, we have
\beqn
\psi\circ\beta [x,y] & = & \psi\Big(ad^*_x\beta (y) - ad^*_y\beta(x)\Big)\nonumber\\
                     & = & ad_x\circ\psi\circ \beta (y) - ad_y\circ\psi\circ\beta (x)\nonumber\\
                     & = & [x,\psi\circ \beta (y)] +[\psi\circ\beta(x),y]\nonumber
\eeqn
Hence $[\beta,\psi]$ belongs to $\G_0,$ for every $\beta$ in $\mathcal Q$  and every  $\psi$ in $\Psi$.
\end{proof}

\blem
Let $\phi:=(\alpha,\xi)$ be in $\G_0$, $\beta: \mathcal G\to \mathcal G^*$  and
$\psi: \mathcal G^* \to \mathcal G$ be respectively in $\mathcal Q$ and $\Psi.$
Then both $[\phi,\beta]$ and $[\phi,\psi]$ are elements of $\G_1,$ more precisely
$[\phi,\beta]$ is in $\mathcal Q$ and $[\phi,\psi]$ is in $\Psi.$
Moreover, we have $[\mathcal Q,\mathcal Q]=0$ and $[\Psi,\Psi]=0.$
\elem
\begin{proof}
Let $\phi=(\alpha,\xi)$ be in $\G_0,$ $\beta:\G\to\G^*$ a 1-cocycle and $\psi:\G^*\to\G$ an equivariant linear map. 
Using $\phi_{\beta}$  and $\phi_{\psi}$  as  in Remark \ref{embeddings}, we obtain
\beqn
[\phi,\phi_{\beta}](x,y)&=& \phi\Big(0,\beta(x)\Big)- \phi_{\beta}\Big(\alpha(x),\xi(f)\Big)\nonumber\\
&=& \Big(0,\xi\circ \beta(x)\Big)- \Big(0,\beta\circ\alpha(x)\Big)\nonumber\\
&=& \Big( 0~,~(\xi\circ \beta-\beta\circ\alpha)(x) \Big).\nonumber
\eeqn
Now, let us show that $\tilde\beta:=\xi\circ \beta-\beta\circ\alpha : \G\to\G^*$ is a $1$-cocycle.
Indeed, on the one hand we have
\beqn\label{eq:beta1}
\xi\circ \beta([x,y])&=&\xi\Big(ad_x^*\beta(y)-ad_y^*\beta(x)\Big)\nonumber\\
&=&\Big([\xi,ad_x^*]+ad^*_x\circ\xi\Big)\big(\beta(y)\big)-\Big([\xi,ad_y^*]
+ad^*_y\circ\xi\Big)\big(\beta(x)\big)\nonumber\\
&=&ad_{\alpha(x)}^*\beta(y)+ad^*_x(\xi\circ\beta(y))-ad_{\alpha(y)}^*\beta(x)-ad^*_y(\xi\circ\beta(x))
\nonumber\\
&=&ad^*_x(\xi\circ\beta(y))-ad^*_y(\xi\circ\beta(x))+ad_{\alpha(x)}^*\beta(y)-ad_{\alpha(y)}^*\beta(x)
\eeqn
On the other hand, we also have
\beqn \label{eq:beta2}\beta\circ\alpha([x,y])&=& \beta\big([\alpha(x),y]\big)+\beta\big([x,\alpha(y)]\big)\nonumber\\
&=& ad_{\alpha(x)}^*\beta(y)-ad^*_y(\beta\circ\alpha(x))+ad^*_x(\beta\circ\alpha(y)) -ad^*_{\alpha(y)} \beta(x)
\eeqn
Subtracting (\ref{eq:beta2}) from (\ref{eq:beta1}), we see that $\tilde\beta[x,y]$ now reads
\beqn \label{eq:beta3}
\tilde\beta[x,y]&=&  ad^*_x(\xi\circ \beta-\beta\circ\alpha)(y) - ad^*_y(\xi\circ \beta-\beta\circ\alpha)(x)\nonumber\\
&=& ad^*_x\tilde\beta(y)-ad^*_y\tilde\beta(x)\nonumber
\eeqn
Hence $\tilde\beta$ is an element of $\mathcal Q.$

In the same way, we also have
\beqn [\phi,\phi_{\psi}](x,y)&=& \phi\big(\psi(f),0\big)- \phi_{\psi}\big(\alpha(x),\xi(f)\big)\nonumber\\
&=& \big(\alpha\circ\psi(f),0\big)- \big(\psi\circ\xi(f),0\big)\nonumber\\
&=& \big( (\alpha\circ\psi- \psi\circ\xi)(f)~,~0 \big)\nonumber
\eeqn
The linear map $\tilde\psi:=\alpha\circ\psi- \psi\circ\xi:\G^*\to\G$ is equivariant,
{\it i.e.} is an element of $\Psi.$ As above, this is seen by first computing, for every elements $x$ of $\G$ and $f$ of $\G^*$,
\beqn \label{eq:psi1}
\alpha\circ\psi(ad_x^*f)=\alpha\big([x,\psi(f)]\big) = [\alpha(x),\psi(f)]+[x,\alpha\circ\psi(f)]
\eeqn
and
\beqn \label{eq:psi2} \psi\circ\xi(ad_x^*f)&=&\psi\circ\big([\xi,ad_x^*]+ad_x^*\circ\xi\big)(f)\nonumber\\
&=&  \psi\big( ad^*_{\alpha(x)}f\big)+\psi\big(ad_x^*\xi(f)\big)\nonumber\\
&=&  \big[\alpha(x),\psi(f)\big]+\big[x,\psi\circ\xi(f)\big],
\eeqn
then subtracting (\ref{eq:psi1}) and (\ref{eq:psi2}).

Now we have  \beq [\phi_{\beta},\phi_{\beta'}](x,f)=\phi_{\beta}(0,\beta'(x))-\phi_{\beta'}(0,\beta(x))=0
\nonumber \eeq
and
\beq [\phi_{\psi},\phi_{\psi'}](x,f)=\phi_{\psi}(\psi'(f),0)-\phi_{\psi'}(\psi(f),0)=0,
\nonumber \eeq
for all $(x,f)$ in $\D.$ In other words, $[\mathcal Q,\mathcal Q]=0$ and $[\Psi,\Psi]=0.$
\end{proof}

We summarize all the above in the

\begin{theorem}\label{structuretheorem}
Let $G$ be a
Lie group and $\G$ its Lie algebra. The group Aut($\D$) of
automorphisms of the Lie algebra $\D$ of the cotangent bundle $T^*G$
of $G$, is a super symmetric Lie group. More precisely, its Lie algebra
$der(\mathcal D)$ is a $\mathbb Z/2\mathbb Z$-graded symmetric (supersymmetric) Lie algebra which 
decomposes into a direct sum of vector spaces
\beq
 der(\mathcal D):=\mathcal G_0\oplus\mathcal G_1,  \text{~
 with~ } [\mathcal G_i,\mathcal G_j]\subset\mathcal
 G_{i+j}, ~~ i,j \in \mathbb Z/2\mathbb Z = \{0, 1\}
\eeq
where  $\mathcal G_0$ is the Lie algebra of linear maps $\phi: \D\to\D, ~  \phi (x,f) = (\alpha(x), \xi(f))$ 
with $\alpha$ a derivation of $\G$ and the linear map
 $\xi : \mathcal G^* \to
\mathcal G^*$  satisfies,
 \beq
 [\xi,ad^*_x] = ad^*_{\alpha(x)},
\eeq
for any $x$ of $\mathcal G$; and $\mathcal G_1$ is the direct sum (as a vector space)
of the space $\mathcal Q$ of $1$-cocycles $\mathcal G\to\mathcal G^*$ and the space $\Psi$
of linear maps $\mathcal G^*\to\mathcal G$ which are equivariant with respect to the coadjoint
and the adjoint representations and satisfy (\ref{relation-commutation}).

\vskip 0.2cm
Moreover, $\G_0\oplus\tilde \G_1$ and $\G_0\oplus\tilde \G_1'$ are subalgebras of der($\D$)
which are Lie superalgebras, {\it i.e.} they are $\mathbb Z/2\mathbb Z$-graded Lie algebras with
the Lie bracket satisfying 
$$[x,y]=-(-1)^{deg(x)deg(y)}[y,x],$$ 
where $\tilde \G_1:=\mathcal Q$ and $\tilde \G_1':=\Psi$ are Abelian subalgebras 
of $der(\D)$ and $deg(x)=i$, if $x\in\G_i$.
\end{theorem}

\vskip 0.4cm

\brmq
\bitem
\item[(a)] In Propositions \ref{decomposition-xi} and \ref{decomposition-xi'}, we will prove that every
element $\xi$ of $\mathcal E$ is the transpose $\xi=(j-\alpha)^t$ of the sum of  an adjoint-invariant 
endomorphism  $j\in\mathcal J$  and a derivation $-\alpha$ of $\G.$

\item[(b)] The Lie superalgebras $\G_0\oplus\tilde \G_1$ and $\G_0\oplus\tilde \G_1'$ respectively correspond 
to the subalgebras of all elements of der($\D$) which preserve the subalgebra $\G$ and the ideal $\G^*$ 
of $\D$. The Lie superalgebra $\G_0\oplus\tilde \G_1'$ can be seen as part of the more general case of 
derivations of a semi-direct product Lie algebra $\G\ltimes \mathcal N$ which preserve the ideal $\mathcal N$ 
and which are discussed in \cite{neeb}, among other interesting results therein.
\eitem
\ermq

Let us now have a closer look at maps $\xi,$ $\psi$ and $\beta$.

\subsection{Maps $\xi$ and Bi-invariant Tensors of Type (1,1)}

\subsubsection{Adjoint-invariant Endomorphisms}\label{chap:bi-inv-endom}
Linear operators acting on vector fields of a given Lie group $G$
can be seen as fields of endomorphisms of its tangent spaces.
Bi-invariant ones correspond to endomorphisms $j:\G \to \G$ of  the
Lie algebra $\G$ of $G$, satisfying the condition $j[x,y]=[jx,y]$,
for all $x,y$ in $\mathcal G$. If we denote by $\nabla$ the connection
on $G$ given on left invariant vector fields by
$$\nabla_xy:=\frac{1}{2}[x,y],$$ 
then using the covariant derivative,
we have  $\nabla j=0$, (see e.g. \cite{tondeur}).
As above, let $$\mathcal J:= \{j:\mathcal G\to \mathcal G, \text{ linear
and } j[x,y]=[jx,y], \forall  x,y\in\mathcal G\}.$$
 Endowed with the bracket $$[j,j']:=j\circ j'-j'\circ j,$$
the space $\mathcal J$ is a Lie algebra, and indeed a subalgebra of
the Lie algebra $\G l(\G)$ of all endomorphisms of $\G.$

\vskip 0.3cm

 In the case where the dimension of $G$ is even and if in addition $j$ satisfies
$j^2\!\!=~\!\!\!-\!identity$, then $(G,j)$ is a complex Lie group.

\subsubsection{ Maps $\xi:\G^*\to\G^*$}

\bpro\label{decomposition-xi} 
Let $\mathcal G$ be a Lie algebra and $\alpha$ a derivation of
$\mathcal G$. A linear map $\xi':\mathcal G \to \mathcal G$ satisfies
$[\xi',ad_x] = ad_{\alpha(x)}$, for every element $x$ of $\G$, if and only if there
exists a linear map $j:\G\to\G$ satisfying
\beq\label{ameth}
j([x,y]) = [j(x),y] = [x,j(y)],
\eeq
for all $x,y$ in $\G$, such that $\xi'= j + \alpha$.
\epro
\begin{proof}
 Let $\alpha$ be a derivation and $\xi'$ an
endomorphism of $\G$ satisfying the hypothesis of Proposition
\ref{decomposition-xi}, that is,
$[\xi',ad_x] = ad_{\alpha(x)} = [\alpha,ad_x]$, for any $x$ in $\G$.
 We then have,
\beq\label{aby}
[\xi'-\alpha,ad_x] = 0,
\eeq
for any $x$ of $\G$. So the endomorphism $j:=\xi'-\alpha$ commutes with all adjoint operators.

Now a linear map $j:\G\to\G$ commuting with all adjoint operators, satisfies
\beq
 0=[j,ad_x](y) = j([x,y]) - [x,j(y)], \label{aby1}
\eeq
for all elements $x,y$ of $\G$.  We also have,
\beq
 0=[j,ad_y](x) = j([y,x]) - [y,j(x)], \label{aby2}
\eeq
for all $x,y$ in $\G$. From (\ref{aby1}) and (\ref{aby2}), we have
$j([x,y]) = [j(x),y] = [x,j(y)]$, for any $x,y$ in $\G$.\\
Thus, (\ref{aby}) is equivalent to $\xi'= j +\alpha$, where
$j$ satisfies (\ref{ameth}).
\end{proof}
\brmq
The above means that the space $\mathcal J$ is the centralizer of the space $ad_{\mathcal G}$
of inner derivations of $\G$ in  $\mathcal Gl (\G):= \{l:\G\to \G \text{~ linear~ }\}$, {\it i.e.}
$$\mathcal J= Z_{\mathcal Gl(\G)} (ad_{\mathcal G}):=\{j:\G  \to \G\text{~ linear
and ~} [j,ad_x]=0, \forall x\in \G\}.$$
\ermq
\bpro \label{decomposition-xi'} 
Let $\G$ be a nonabelian Lie algebra and $\mathcal S$
the space of endomorphisms  $\xi': \G\to\G$ such that there exists a derivation
$\alpha$ of $\G$ and  $[\xi', ad_x]=ad_{\alpha (x)}$ for  all $x\in \mathcal G$.
Then $ \mathcal S$ is a Lie algebra containing $\mathcal J$ and $der(\G)$ as subalgebras.
In the case where $\G$ has a trivial centre, then $ \mathcal S$ is  the
semi-direct product $\mathcal S=der(\G)\ltimes \mathcal J$ of
 $\mathcal J$ and $der(\G)$.

The following are equivalent
\bitem 
\item[(a)] The linear map $\xi:\G^*\to\G^*$ is an element of  $\mathcal E$ with $\alpha$  as the corresponding 
derivation of $\G$, i.e. $\xi$ satisfies (\ref{relation-xi-alpha}) for the derivation $\alpha$.

\item[(b)] The transpose $\xi^t$  of $\xi$ is of the form $\xi^t=j-\alpha,$ where $j$ is in $\mathcal J$ and $\alpha$ in der($\G$).

\item[(c)] $\xi^t$ is an element of $\mathcal S,$ with corresponding derivation $-\alpha.$
\eitem 
  The transposition $\xi\mapsto \xi^t$ of linear maps is an anti-isomorphism between the Lie algebras $\mathcal E$ and $\mathcal S.$
\epro
\begin{proof}
Using the same argument as in Lemma \ref{structure1}, if $[\xi'_1, ad_x]=ad_{\alpha_1 (x)}$ and 
$[\xi'_2, ad_x]=ad_{\alpha_2 (x)}$, for every $x$ in $\G,$ then $[[\xi'_1,\xi_2'], ad_x]=ad_{[\alpha_1, \alpha_2](x)}$, 
for any element $x$ of $\G$. Thus $\mathcal S$ is a Lie algebra. From Proposition \ref{decomposition-xi}, there exist 
$j_i$ in $\mathcal J$ such that $\xi'_i=\alpha_i+j_i ,$ $i=1,2.$
Obviously, $\mathcal S$ contains $\mathcal J$ and $\der(\G).$
Thus, as a vector space, $\mathcal S$ decomposes as $\mathcal S= \der(\G) +\mathcal J.$
Now, the Lie bracket in $\mathcal S$ reads
\beq \label{eq:liebraket1-S}
[\xi'_1,\xi'_2]= [\alpha_1+j_1, \alpha_2+j_2]= [\alpha_1, \alpha_2] +[\alpha_1,j_2]+[j_1, \alpha_2]+[j_1,j_2]
\eeq
Of course, $[\alpha_1, \alpha_2]$ is in $\der(\G)$. From Section \ref{chap:bi-inv-endom}, we know that $\mathcal J$ 
is a Lie algebra, hence $[j_1,j_2]$ is in $\mathcal J$. It is easy to check that
\beq
[\alpha,j]\in \mathcal J,
\eeq
for all $\alpha$ in $\der(\G)$ and for all $j$ in $\mathcal J$. Indeed, the following holds
\beqn
[\alpha,j]([x,y])&=& \alpha \big([j(x),y]\big) - j\big([\alpha(x),y]+[x,\alpha(y)]\big)\nonumber \\
                 &=& [\alpha\circ j(x),y]+ [j(x),\alpha(y)] - [j\circ \alpha(x),y]-[j(x),\alpha(y)]\nonumber\\
                 &=&  [[\alpha,j](x),y],
\eeqn
for all $x,y$ in $\G$.
The intersection der($\G$)$\cap \mathcal J$ is made of elements $j$ of $\mathcal J$ whose image Im($j$) is
 a subset of the centre $Z(\G)$ of $\G.$ Hence if $Z(\G)=0$, then $\mathcal S=$der($\G$)$ \oplus\mathcal J$ and 
as a Lie algebra,   $\mathcal S= $ der($\G$)$\ltimes \mathcal J.$ Using this  decomposition, 
we can also rewrite  (\ref{eq:liebraket1-S}) as
\beq \label{eq:liebraket2-S}
[\xi'_1,\xi'_2]= \left[(\alpha_1,j_1)~,~ (\alpha_2,j_2)\right]= \left([\alpha_1, \alpha_2]~,~   
 [j_1,j_2] +[\alpha_1,j_2]+[j_1, \alpha_2]\right)
\eeq

The equivalence between (b) and (c) comes directly from Proposition \ref{decomposition-xi}.

\vskip 0.2cm

 Now let $\xi\in\mathcal E$, with $[\xi, ad^*_x] = ad^*_{\alpha (x)}$, $\alpha \in \der(\G)$, then
\beq
- ad_{\alpha (x)} = [\xi, ad^*_x]^t = - [\xi^t, (ad^*_x)^t] = [\xi^t, ad_x]\nonumber
\eeq
Hence $\xi^t\in \mathcal S$, with $ad_{\alpha' (x)}  = [\xi^t, ad_x]$, for all $x\in \G$,
where $\alpha':=-\alpha$. Thus, (a) implies (c). From Proposition \ref{decomposition-xi}, there exist $j\in\mathcal J$ 
such that $\xi^t=-\alpha+j.$ Now it is straightforward that if  (b) $\xi^t=-\alpha+j$ with $\alpha$ a derivation and 
$j$ in $\mathcal J$, then $\xi$ satisfies $[\xi,ad^*_x]=ad^*_{\alpha(x)}$, for all $x$ in $\G$. Hence (c) implies (a).

\vskip 0.2cm

Of course, we also know that $[\xi_1,\xi_2]^t= - [\xi_1^t,\xi_2^t]$, for every $\xi_1,\xi_2\in\mathcal E$.
\end{proof}

\begin{lemma}\label{lem:decomp-semi-simpl} 
Let $\xi': \mathcal G \to \mathcal G$ be a linear map
such that there exists $\alpha : \mathcal G \to \mathcal G$ linear and $[\xi',ad_x]=ad_{\alpha(x)}$,
for all $x$ in $\mathcal G$. Then $\xi'$ preserves every ideal $\mathcal A$ of $\mathcal G$ satisfying
$[\mathcal A,\mathcal A]=\mathcal A.$  In particular, if $\mathcal G$ is semi-simple and
$ \G = \s_1\oplus\s_2\oplus\ldots \oplus \s_p$ is a decomposition of $\mathcal G$ into a sum of simple ideals 
$\s_1,\ldots , \s_p$, then $\xi'(\s_i)\subset \s_i$, for $i=  1,\ldots, p.$
\end{lemma}
\begin{proof}
The proof is straightforward. Indeed, every element $x$ of an ideal $\mathcal A$ satisfying the hypothesis of 
Lemma \ref{lem:decomp-semi-simpl}, is a finite sum of the form $x=\displaystyle \sum_i [x_i,y_i]$ where $x_i,y_i$ 
are all elements of $\mathcal A$.   But as $\mathcal A$ is an ideal,
\beqn
\xi' ([x_i,y_i])&=&\xi' \circ ad_{x_i} (y_i)=\big([\xi',ad_{x_i}] + ad_{x_i}\circ \xi' \big)(y_i)\nonumber\\
&=&\big(ad_{\alpha(x_i)} + ad_{x_i}\circ \xi'\big)(y_i) =[\alpha(x_i),y_i] +
[x_i, \xi'(y_i)]\nonumber
\eeqn
is again an element of $\mathcal A.$
Hence we have $\xi' (x) = \displaystyle \sum_i \big([\alpha(x_i),y_i] +
[x_i, \xi'(y_i)]\big)$ is in $\mathcal A$.
\end{proof}

\subsection{Equivariant Maps $\psi: \G^*\to\G $}\label{chap:equivariance}

 Let $\G$ be a Lie algebra. In this section, we would like to explore properties of the space $\Psi$ of  
linear maps $\psi:\G^* \to \G$ which are equivariant with respect to the adjoint
and the coadjoint actions of $\G$ on $\G$ and $\G^*$ respectively  and satisfy: for all $f,g$ in $\G^*$,
\beq
ad_{\psi(f)}^*g=ad_{\psi(g)}^*f.\nonumber
\eeq

\blem\label{lem:intertwining}
Let $\G$ be a Lie algebra and $\psi$ an element of $\Psi$.
Then, 
\bitem 
\item [(a)] $Im \psi$ is an Abelian ideal of $\G$ and we have $\psi(ad_{\psi (g)}^*f)=0$,
for every $f,g$ in $\G^*$; 
\item[(b)] $\psi$ sends closed forms on $\G$ in the center of $\G$;
\item[(c)] $[Im \psi,\G] \subset \ker f$, for all $f$ in $\ker \psi$;
\item[(d)] the map $\psi$ cannot be invertible if $\G$ is not Abelian.
\eitem 
\elem
\noindent
\begin{proof}
(a) For every elements $f$ of $\G^*$ and $x$ of $\G$, we
have,
\beq
[\psi(f),x]  =  -(ad_x \circ \psi)(f)
             =  -(\psi \circ ad_x^* )(f) \in Im \psi.\nonumber
\eeq
Hence $Im(\psi)$ is an ideal of $\G$.

 Now, for every $f,g$ in $\G^*$, since $\psi(f)$ and $\psi(g)$ are elements of $\G$,  we also have
$\psi \circ ad^*_{\psi(f)} = ad_{\psi(f)} \circ \psi$ and
$\psi \circ ad^*_{\psi(g)} = ad_{\psi(g)} \circ \psi$. On the one hand,
\beqn
\left(\psi \circ ad^*_{\psi(f)}\right)(g) &=&  \left(ad_{\psi(f)} \circ \psi\right)(g) \nonumber \\
\psi(ad^*_{\psi(f)}g) &=& [\psi(f),\psi(g)] \label{esso3}
\eeqn
On the other hand,
\beqn
\left(\psi \circ ad^*_{\psi(g)}\right)(f) &=&  \left(ad_{\psi(g)} \circ
\psi\right)(f) \nonumber \\
\psi(ad^*_{\psi(g)}f) &=& [\psi(g),\psi(f)] \label{esso4}
\eeqn
Using  (\ref{esso3}) and (\ref{esso4}) we get
\beq\label{esso5}
 [\psi(f),\psi(g)]=\psi(ad^*_{\psi(f)}g)=\psi(ad^*_{\psi(g)}f) = [\psi(g),\psi(f)]
\eeq
Equation (\ref{esso5}) implies the following
\beq
[\psi(f),\psi(g)]=\psi(ad^*_{\psi(g)}f) = 0,
\eeq
for all elements  $f,g$ of $\G^*$. So we have proved (a).
\newline\noindent
 (b) Let $f$ be a closed form on $\G$, that is, $f$ in $\G^*$ and $ad_x^*f=0$, for all $x$ in $\G$.
The relation (\ref{relation-commutation}) implies that $ad^*_{\psi(f)}g=0 $,
for any $g$ in $\G^*$. Thus, for any element $y$ of $\G$ and any element $g$ of $\G^*$,
\beq
g([\psi(f),y])=0, \nonumber
\eeq
\noindent
 and hence $[\psi(f),y]=0$, for all $y$ in $\G$. In other words $\psi(f)$
belongs to the center of $\G$.
\newline\noindent
(c) If $f \in \ker \psi$, then $ad^*_{\psi(g)}f = ad^*_{\psi(f)}g=0$, for any $g$ in $\G^*$,
or equivalently, for every $x$ in $\G$ and $g$ in $\G^*$, $f([\psi(g),x])=0.$
It follows that $[Im \psi,\G] \subset \ker f$, for every  $f$ of $\ker \psi$.
 \newline\noindent
(d) From (a), the map $\psi$ satisfies $\psi(ad^*_{\psi(g)}f)= 0$, for any $f,g$ in $\G^*$.
 There are two possibilities here:
 \bitem 
\item[(i)] either there exist $ f,g$  in $\G^*$ such that $ad^*_{\psi(g)}f\neq 0$, in which case $ad^*_{\psi(g)}f$ 
belongs to $\ker\psi\ne 0$ and thus  $\psi$ is not invertible;
\item[(ii)] or else, suppose $ad^*_{\psi(g)}f = 0$, for all $f,g$ in $\G^*$.
This implies that $\psi(g)$ belongs to the center of $\G$ for every
$g$ in $\G^*$. In other words, the center of $\G$ contains $Im(\psi)$.
But since $\G$ is not Abelian, the center of $\G$ is different from $\G$,
hence $\psi$ is not invertible.
\eitem
\end{proof}

\blem\label{lem:coad-inv} The space of
equivariant maps $\psi:\G^* \to \G$ bijectively corresponds to that of
$\G$-invariant bilinear forms on the $\G$-module $\G^*$ for the coadjoint representation.
\elem
\begin{proof}
Indeed, each such $\psi$ defines a unique coadjoint-invariant bilinear form $\langle,\rangle_{\psi}$ on $\G^*$ as follows:
\beq
\langle f,g\rangle_{\psi}:= \langle \psi (f), g\rangle,
\eeq
for all $f,g$ in $\G^*$, where the right hand side is the duality pairing
$\langle f,x \rangle = f(x)$, $x$ in $\G$, $f$ in $\ G^*,$ as above.
The coadjoint-invariance reads
\beq
\langle ad_x^* f,g\rangle_{\psi}+ \langle f,ad_x^* g \rangle_{\psi}= 0,
\eeq
for all $x$ in $\G$ and all $f,g$ in $\G^*$; and is due to the simple equalities
\beqn
\langle ad_x ^*f,g \rangle_{\psi} &=& \langle \psi(ad_x^*f),
g\rangle = \langle ad_x\psi(f),g \rangle\nonumber\\
 &=&- \langle \psi(f),ad_x^*g \rangle=-\langle f,ad_x^*g\rangle_{\psi}. \nonumber
\eeqn
\noindent
Conversely, every $\G$-invariant bilinear form $\langle,\rangle_1$ on $\G^*$ gives rise to a unique
linear map $\psi_1: \G^*\to \G$ which is equivariant with respect to the adjoint and coadjoint representations of $\G$, by the formula
\beq
\langle\psi_1(f),g\rangle:=\langle f,g\rangle_1.
\eeq
\end{proof}

If $\psi$ is symmetric or skew-symmetric, then so is
$\langle,\rangle_\psi$ and vice versa. Otherwise,
$\langle,\rangle_\psi$ can be decomposed into a symmetric and a
skew-symmetric parts $\langle , \rangle_{\psi,s}$ and  $\langle ,
\rangle_{\psi,a}$ respectively, defined by the following formulas:
\beqn
\langle f, g \rangle_{\psi,s} &:=&\frac{1}{2}\Big[\langle f,g\rangle_\psi +\langle g, f\rangle_\psi \Big],\\
\langle f, g \rangle_{\psi,a} &:=& \frac{1}{2}\Big[ \langle f, g
\rangle_\psi - \langle g, f\rangle_\psi \Big].
\eeqn
The symmetric and skew-symmetric parts $\langle ,  \rangle_{1,s}$ and $\langle ,
\rangle_{1,a}$ of a $\G$-invariant bilinear form $\langle ,
\rangle_{1}$, are also $\G$-invariant.
From a remark in p. 2297 of \cite{me-re93}, the radical
$Rad\langle ,  \rangle_{1}:=\{ f\in \G^*, \langle f,g  \rangle_{1}=0, \forall g\in \G^*\}$ of a $\G$-invariant 
form $\langle ,  \rangle_{1},$ contains the coadjoint orbits of all its points.

\subsection{Cocycles $\G\to\G^*.$} \label{chap:cocyles}

The 1-cocycles for the coadjoint representation of a Lie algebra $\G$ are linear maps\\
$\beta:\G\to \G^*$ satisfying the cocycle condition
$\beta ([x,y])=ad_x^*\beta (y) - ad^*_y\beta (x)$, for every elements $x,y$ of $\G.$

To any given 1-cocycle $\beta$, corresponds a bilinear form $\Omega_\beta$ on $\G$, by the formula
\beq \label{eq:cocycle-2-forms}
\Omega_\beta (x,y):=\langle \beta (x),y\rangle,
\eeq
for all $x,y$ in $\G$, where $\langle,\rangle$ is again the duality pairing between
elements of $\G$ and $\G^*$.

\vskip 0.3cm

The bilinear form $\Omega_\beta$  is skew-symmetric (resp. symmetric, nondegenerate)
if and only if $\beta$ is skew-symmetric (resp. symmetric, invertible).

\vskip 0.3cm

Skew-symmetric such cocycles $\beta$ are in bijective
correspondence with closed 2-forms in $\mathcal G$, via the formula (\ref{eq:cocycle-2-forms}). 
In this sense, the cohomology space $H^1(\D,\D)$  contains the second cohomology space $H^2(\G,\R)$ 
of $\G$ with coefficients in $\R$ for the trivial action of $\G$ on $\R$. Hence, $H^1(\D,\D)$ somehow 
contains the second space $H^2_{inv}(G,\R)$ of left invariant de Rham cohomology  $H^*_{inv}(G,\R)$ of 
any Lie group $G$ with Lie algebra $\G.$

\vskip 0.3cm

Invertible skew-symmetric ones, when they exist, are those giving rise to symplectic forms or
equivalently to invertible solutions of the Classical Yang-Baxter
Equation. The study and classification of the solutions of the Classical Yang-Baxter Equation is a still 
open problem in Geometry, Theory of integrable systems. In Geometry, they give rise to very interesting 
structures in the framework of Symplectic Geometry, Affine Geometry, Theory of
Homogeneous K\"ahler domains,  (see e.g. \cite{di-me-cybe} and references therein).

\vskip 0.3cm

If $\G$ is semi-simple, then every cocycle $\beta$ is a coboundary, that is, there exists
$f_\beta$ in $\G^*$ such that $\beta (x)= -ad^*_xf_\beta$, for any $x$ in $\G.$

\subsection{Cohomology Space $H^1(\D,\D)$} \label{chap:cohomology}

Following Remarks \ref{notations} and  \ref{embeddings}, we can embed $\der(\G)$ as a subalgebra $\der(\G)_1$ 
of $\der(\D)$, using the linear map $\alpha\mapsto \phi_\alpha,$ with $\phi_\alpha(x,f)=(\alpha(x),f\circ\alpha).$ 
In the same way, we have constructed $\mathcal Q$ and $\Psi$ as subspaces of  $\der(\D).$ Likewise, 
$\mathcal J^t:=\{j^t, \text{~where~} j\in \mathcal J\}$ is seen as a subspace of $\der(\D),$ via the linear 
map $j^t\mapsto\phi_j$, with $\phi_j(x,f)=(0,f\circ j).$

\vskip 0.3cm

We give the following summary.

\begin{theorem}\label{thm:specialcase} The first cohomology space $H^1(\D,\D)$ of the (Chevalley-Eilenberg) 
cohomology associated with the adjoint action of $\D$ on itself, satisfies
$$H^1(\D,\D) \stackrel{\sim}{=} H^1(\G,\G)\oplus  \mathcal J^t\oplus H^1(\G,\G^*)\oplus \Psi,$$  
where $H^1(\G,\G)$ and $H^1(\G,\G^*)$ are the first cohomology spaces associated with the adjoint 
and coadjoint actions of $\G$, respectively; and
$\mathcal J^t:=\{j^t, j \in \mathcal J \}$ (space of transposes of elements of $\mathcal J$).

\vskip 0.2cm

If $\G$ is semi-simple, then  $\Psi=\{0\}$ and thus 
$$H^1(\D,\D) \stackrel{\sim}{=}  \mathcal J^t.$$ 

\vskip 0.3cm

Moreover, we have $\mathcal J \stackrel{\sim}{=} \R^p$, where $p$ is the number of the simple
 ideals $\s_i$ of $\G$ such that $\G=\s_1\oplus\ldots\oplus \s_p$. Hence, of course, 
$H^1(\D,\D) \stackrel{\sim}{=}   \R^p.$

\vskip 0.3cm

If $\G$ is a compact Lie algebra, with centre $Z(\G)$, we get 
\beqn 
H^1(\G,\G)   &\stackrel{\sim}{=}& End(Z(\G)),\cr 
\mathcal J   &\stackrel{\sim}{=}& \R^p\oplus End(Z(\G)),\cr 
H^1(\G,\G^*) &\stackrel{\sim}{=}& L(Z(\G), Z(\G)^*), \cr 
\Psi         &\stackrel{\sim}{=}& L(Z(\G)^*, Z(\G)). \nonumber
\eeqn 
 Hence, we get $$ H^1(\D,\D) \stackrel{\sim}{=} (End(\R^k) )^4\oplus \R^p,$$ 
where $k$ is the dimension of the centre of $\G$, and $p$ is the number of the simple
components of the derived ideal $[\G,\G]$ of $\G$. Here, if $E,F$ are vector spaces, $L(E,F)$ is the space of linear maps $E\to F.$
\end{theorem}

The proof of Theorem \ref{thm:specialcase} is given by Proposition \ref{prop:cohomology} below and different 
lemmas and propositions, discussed in Section \ref{chap:orthogonal-algebra}.

\vskip 0.3cm

For the purpose of this investigation, we have favored a direct approach to exhibit detailed calculations of 
the first cohomology space, instead of the traditional powerful spectral sequences method commonly applied in 
the more general setting of locally trivial fiber bundles (see e.g. \cite{knapp}, \cite{neeb}).  Some parts of 
Theorem C can also be seen as a refinement, using our direct approach,  of some already known results (\cite{knapp}).

\vskip 0.3cm

As a vector space, $\der(\D)$ is isomorphic to the direct sum $\der(\G)\oplus\mathcal J^t\oplus \mathcal Q\oplus \Psi$ by
\beq \label{eq:cohomology} 
\Phi: \der(\G)\oplus\mathcal J^t\oplus \mathcal Q\oplus \Psi\to \der(\D); 
~~(\alpha,j^t,\beta,\psi)\mapsto \phi_\alpha+\phi_j+\phi_\beta+\phi_\psi.
\eeq
In this isomorphism, we have $\Phi (\der(\G)\oplus\mathcal J^t)= \der(\G)_1\oplus\mathcal J^t=\G_0$ 
and  $\Phi (\mathcal Q\oplus\Psi)=\G_1$.

\vskip 0.3cm

Now an exact derivation of $\D$, {\it i.e.} a $1$-coboundary for the Chevalley-Eilenberg cohomology associated 
with the adjoint action of $\D$ on $\D,$ is of the form $\phi_0=\partial v_0=ad_{v_0}$ for some element 
$v_0:=(x_0,f_0)$ of the $\D$-module $\D.$ That is, 
$$\phi_0(x,f)=(\alpha_0(x),\beta_0(x)+\xi_0(f)),$$ 
where 
$$
\alpha_0(x):=[x_0,x], \quad \beta_0(x)=-ad^*_xf_0, \quad \xi_0(f)=ad^*_{x_0}f.
$$ 
As we can see $\phi_0=\phi_{\alpha_0}+\phi_{\beta_0}=\Phi(\alpha_0,0,\beta_0,0)$ and

\bpro\label{prop:cohomology} 
The linear map $\Phi$ in (\ref{eq:cohomology}) induces an isomorphism   
$\bar \Phi$ in cohomology, between the spaces $H^1(\G,\G)\oplus \mathcal J^t\oplus H^1(\G,\G^*)\oplus \Psi$ and $H^1(\D,\D).$
\epro

\begin{proof}
The isomorphism in cohomology simply reads
$$\bar \Phi (class(\alpha), j^t,class(\beta),\psi)=class(\phi_\alpha+\phi_j+\phi_\beta+\phi_\psi).$$
\end{proof}

\section{Case of Orthogonal Lie Algebras}\label{chap:orthogonal-algebra}

In this section, we prove that if a Lie algebra $\G$ is orthogonal, then the Lie algebra $\der(\G)$ of 
its derivations  and the Lie algebra $\mathcal J$ of linear maps $j:\G\to \G$ satisfying $j[x,y]=[jx,y]$, 
for every $x,y$ in $\G,$ completely characterize the Lie algebra $\der(\D)$ of derivations of $\D:=\G\ltimes\G^*$, 
and hence the group of automorphisms of the cotangent bundle
of any connected Lie group with Lie algebra $\G.$ We also show that $\mathcal J$ is
isomorphic to the space of adjoint-invariant bilinear forms on $\G.$

\vskip 0.3cm

Let $(\G,\mu)$ be an orthogonal Lie algebra and consider the isomorphism $\theta : \G \to \G^*$ of $\G$-modules,  
given by $\langle \theta(x),y \rangle := \mu(x,y)$, as in Section \ref{chap:notations}.

\vskip 0.3cm

Of course, $\theta^{-1}$ is an equivariant map. But if $\G$ is not Abelian, invertible equivariant linear maps 
do not contribute to the space of derivations of $\D,$ as discussed in Lemma~\ref{lem:intertwining}.

\vskip 0.3cm

We pull coadjoint-invariant bilinear forms $B'$ on $\G^*$ back to adjoint-invariant bilinear
forms on $\G,$ as follows
$B(x,y):= B'(\theta(x),\theta(y)).$
Indeed, we have
\beqn
 B([x,y],z)  &=& B'\big(\theta([x,y]),\theta(z)\big)            \cr 
             &=& B'\big(ad^*_x\theta(y),\theta(z)\big)          \cr 
             &=& - B'\big(\theta(y),ad^*_x\theta(z)\big)        \cr
             &=& -B(y,[x,z]).\nonumber
\eeqn

\bpro
If a Lie algebra $\G$ is orthogonal, then there is an isomorphism between any two of the following vector spaces:
\bitem 
\item[(a)] the space $\mathcal J$ of linear maps $j:\G\to\G$ satisfying $j[x,y]=[jx,y]$, for every
 $x,y$ in $\G$;
\item[(b)]  the space of linear maps $\psi:\G^*\to\G$ which are equivariant with respect to the coadjoint and 
the adjoint representations of $\G$;
\item[(c)] the space of bilinear forms $B$ on $\G$ which are adjoint-invariant, {\it i.e.}
 \beq 
  B([x,y],z)+B(y,[x,z])=0,
 \eeq
for all $ x,y,z$ in $\G$;
\item[(d)] the space of bilinear forms $B'$ on $\G^*$ which are coadjoint-invariant, {\it i.e.}
\beq
 B'(ad^*_xf,g)+B'(f,ad^*_xg)=0,
\eeq
for all $x$ in $\G$,   $f,g$ in $\G^*$.
\eitem 
\epro

\begin{proof} 
$\bullet$ The linear map $\psi\mapsto \psi\circ\theta$ is an isomorphism between the space of 
equivariant linear maps $\psi:\G^*\to\G$ and the space $\mathcal J.$
Indeed, if $\psi$ is equivariant, we have
\beq
\psi\circ\theta([x,y])=-\psi(ad^*_y\theta(x))=-ad_y\psi(\theta(x))=[ \psi\circ\theta (x),y].\nonumber
\eeq
 Hence $\psi\circ\theta$ is in $\mathcal J.$ Conversely, if $j$ is in $\mathcal J$, then $j\circ\theta^{-1}$ is 
equivariant, as it satisfies
\beq
  j\circ\theta^{-1}\circ ad^*_x= j\circ ad_x\circ\theta^{-1}= ad_x \circ j\circ\theta^{-1}.\nonumber\eeq
   This correspondence is obviously linear and invertible. Hence we get the isomorphism between (a) and (b).

\vskip 0.3cm 

$\bullet$  The isomorphism between the space $\mathcal J$ of  adjoint-invariant endomorphisms and 
adjoint-invariant bilinear forms is given as follows
\beq j\in\mathcal J\mapsto B_j, \text{ where } B_j(x,y):=\mu(j(x),y).\eeq
for any $x,y$ in $\G$.
We have, for any $x,y,z$ in $\G$
\beq
B_j([x,y],z):=\mu(j([x,y]),z)=\mu([x,j(y)],z)= - \mu(j(y),[x,z])= -B_j(y,[x,z]).\nonumber\eeq
Conversely, if $B$ is an adjoint-invariant bilinear form on $\G,$ then the endomorphism $j$, defined by \beq
 \mu(j(x),y):=B(x,y)
\eeq
is an element of $\mathcal J,$ as it satisfies
\beq
\mu(j([x,y]),z):= B([x,y],z)=B(x,[y,z])= \mu(j(x) ,[y,z]) = \mu([j(x),y] ,z), \nonumber
\eeq
for all elements  $x,y,z$ of $\G$.

\vskip 0.3cm

$\bullet$ From Lemma \ref{lem:coad-inv}, the space of equivariant linear maps $\psi$ bijectively corresponds 
to that of coadjoint-invariant bilinear forms on $\G^*,$ via $\psi\mapsto \langle , \rangle_\psi$.
\end{proof}

Now, suppose $\psi$ is skew-symmetric.
 Let $\omega_\psi$ denote the corresponding skew-symmetric bilinear form  on $\G$:
\beq
\omega_\psi(x,y) := \mu(\psi\circ\theta (x),y),
\eeq
for all $x,y$ in $\G$. Then, $\omega_\psi$ is adjoint-invariant.  If we denote by $\partial$  the 
Chevalley-Eilenberg coboundary operator, that is,
\beq
(\partial\omega_\psi)(x,y,z) = - \big(\omega_\psi([x,y],z)+\omega_\psi([y,z],x)+\omega_\psi([z,x],y)\big),\nonumber
\eeq
the following formula holds true
\beq
(\partial \omega_\psi)(x,y,z)
= -\omega_\psi([x,y],z). \nonumber
\eeq
for all $x,y,z$ in $\G$.

\bcor
The following are equivalent.
\bitem
\item[(a)] $\omega_\psi$ is closed;
\item[(b)] $\psi\circ \theta([x,y])=0$, for all $x,y$ in $\G$;
\item[(c)] $Im(\psi)$ is in the centre of $\G$.
\eitem 
\noindent
In particular, if $\dim [\G,\G]\ge \dim\G-1$, then $\omega_\psi$ is closed if and only if $\psi=0.$
\ecor

\begin{proof}
 The above equality also reads
\beq \label{partial-omega}
\partial\omega_\psi(x,y,z)=-\omega_\psi([x,y],z)=-\mu(\psi\circ\theta([x,y]),z),
\eeq
for all  $x,y,z$ in $\G$; and gives the proof that (a) and (b) are equivalent. In particular,
if $\G=[\G,\G]$ then, obviously  $\partial\omega_\psi =0$ if and only if $\psi=0,$ as $\theta$  is invertible.

\vskip 0.3cm 

Now  suppose $\dim [\G,\G]= \dim\G-1$ and set $\G=\R x_0\oplus [\G,\G],$ for some element $x_0$ of $\G.$
If $\omega_\psi$  is closed, we already know that $\psi\circ \theta$ vanishes on $[\G,\G].$ Below, 
we show that, it also does on $\R x_0.$ Indeed, the formula
\beq
0=-\omega_\psi([x,y],x_0)=\omega_\psi(x_0, [x,y])=\mu(\psi\circ\theta(x_0), [x,y]),
\eeq
for all  $x,y$ in $\G$, obtained by taking $z=x_0$ in (\ref{partial-omega}),
coupled with the obvious equality $0=\omega_\psi(x_0,x_0)= \mu(\psi\circ\theta(x_0), x_0)$,
are equivalent to $\psi\circ\theta(x_0)$ satisfying $\mu(\psi\circ\theta(x_0),x)=0$
for all $x$ in $\G$. As $\mu$ is nondegenerate, this means that $\psi\circ\theta(x_0)=0$.
Hence $\psi\circ\theta=0$, or equivalently $\psi=0.$

\vskip 0.3cm

Now, as every $f$ in $\G^*$ is of the form $f=\theta(y),$ for some $y$ in $\G,$  the formula
\beq
\psi\circ\theta([x,y])= \psi\circ ad_x^*\theta(y)=  ad_x\circ \psi\circ \theta(y) = [x, \psi\circ \theta(y)],~ \forall x,y\in\G.
 \eeq
shows that $\psi\circ\theta([x,y])=0$, for all $x,y$ on $\G$ if and only if $Im(\psi)$ is a subset of the 
center of $\G.$ Thus, (b) is equivalent to (c).
\end{proof}

\medskip

Now we pull every element $\xi$ of $\mathcal E$ back to an
endomorphism $\xi'$ of $\G$ given by the formula 
$$\xi':=\theta^{-1} \circ \xi \circ \theta. $$
\bpro\label{guimbita1}
Let $(\G,\mu)$ be an orthogonal Lie algebra and $\G^*$ its dual space. Define $\theta : \G \to \G^*$ by 
$\langle \theta(x),y \rangle := \mu(x,y)$,
as in Relation (\ref{isomorphism-theta}), and let $\mathcal E$ and $\mathcal S$ stand for the same Lie algebras as above.
The linear map $Q:~\xi \mapsto \xi':=\theta^{-1} \circ
\xi \circ \theta$ is an isomorphism of Lie algebras between $\mathcal E$ and $\mathcal S.$
\epro

\begin{proof}
Let $\xi$ be in $\mathcal E,$ with $[\xi,ad^*_x]=ad^*_{\alpha(x)}$, for every $x$ in $\G$.
The image $Q(\xi)=:\xi'$ of $\xi,$ satisfies, for any  $x$ in $\G$,
\beqn
[\xi',ad_x] &:=& \xi' \circ ad_x - ad_x \circ \xi' \nonumber \\
            &:=& \theta^{-1} \circ \xi \circ \theta \circ ad_x - ad_x \circ \theta^{-1} \circ \xi \circ \theta \nonumber \\
            &=& \theta^{-1} \circ \xi \circ ad^*_x \circ \theta - \theta^{-1} \circ ad^*_x \circ \xi \circ \theta \nonumber \\
     &=& \theta^{-1}\circ (\xi \circ ad^*_x - ad^*_x \circ \xi)\circ \theta \nonumber \\
     &=& \theta^{-1} \circ ad^*_{\alpha(x)} \circ \theta, \qquad \mbox{ since } [\xi,ad^*_x] = ad^*_{\alpha(x)}. \nonumber \\
     &=& \theta^{-1} \circ \theta\circ ad_{\alpha(x)}, \qquad \mbox{ using } (\ref{suz9}). \nonumber \\
    &=& ad_{\alpha(x)}.   \nonumber
\eeqn
Now we have $[Q(\xi_1),Q(\xi_2)])=Q([\xi_1,\xi_2])$ for all $\xi_1,\xi_2$ in $\mathcal E$, as seen below.
\beqn
[Q(\xi_1),Q(\xi_2)])&:=&Q(\xi_1)Q(\xi_2) - Q(\xi_2)Q(\xi_1)\nonumber\\
&:=&\theta^{-1}\circ\xi_1\circ\theta\circ \theta^{-1}\circ\xi_2\circ\theta - 
\theta^{-1}\circ\xi_2\circ\theta\circ \theta^{-1}\circ\xi_1\circ\theta\nonumber\\
&=& \theta^{-1}\circ[\xi_1,\xi_2]\circ\theta\nonumber\\
&=& Q([\xi_1,\xi_2]).
\eeqn
\end{proof}

\bpro\label{isomorphism-cocycle-derivation}
The linear map $P:~\beta\mapsto D_\beta:=\theta^{-1}\circ \beta,$ is an isomorphism between the space of cocycles
 $\beta:\G\to\G^*$ and the space $der(\G)$ of derivations of $\G.$
\epro
\begin{proof}
  The proof is straightforward. If $\beta:\G\to\G^*$ is a cocycle, then the linear map 
$D_\beta:\G\to\G$,~ $x\mapsto\theta^{-1}(\beta(x))$ is a derivation of $\G,$ as we have
\beq
D_\beta[x,y]=\theta^{-1}\big(ad^*_x\beta(y)-ad^*_y\beta(x)\big)=[x,\theta^{-1}(\beta(y))]- [y,\theta^{-1}(\beta(x))].\nonumber
\eeq
Conversely, if $D$ is a derivation of $\G,$ then the linear map
$\beta_D:=P^{-1}(D)=\theta\circ D:\G\to\G^*,$ is 1-cocycle. Indeed we have: for every $x,y$ in $\G$
\beq
\beta_D[x,y]=\theta([Dx,y]+[x,Dy])= - ad^*_y(\theta\circ D (x)) + ad^*_x (\theta\circ D (y)). \nonumber
\eeq
\end{proof}

\subsection{Case of Semi-simple Lie Algebras}\label{chap:semisimple}

Suppose now $\G$ is semi-simple, then every derivation is inner. Thus in particular, the
derivation $\phi_{11}$ obtained in (\ref{phi11}), is of the form
\beq\label{suz4}
 \phi_{11} = ad_{x_0},
\eeq
for some  $x_0$ in $\G$. The semi-simplicity of $\G$ also implies that the $1$-cocycle $\phi_{12}$
obtained in (\ref{suz5}) is a coboundary. That is, there exists an element $f_0$ of $\G^*$ such that
\beq
\phi_{12}(x) = - ad^*_xf_0,
\eeq
for all $x$ in $\G$. Here is a direct corollary of Lemma \ref{lem:intertwining}.

\bpro\label{prop:psi-semisimple} If $\G$ is a semi-simple Lie algebra, then every linear map
$\psi:\G^*\to\G$ which is equivariant with respect to the adjoint and coadjoint actions of $\G$ and 
satisfies (\ref{relation-commutation}), is necessarily identically equal to zero.
\epro
\begin{proof}
A Lie algebra is semi-simple if and only if it contains no nonzero proper Abelian ideal. But from 
Lemma \ref{lem:intertwining}, $Im(\psi)$ must be an Abelian ideal of $\G$. So $Im(\psi)=\{0\}$ and hence $\psi=0.$
\end{proof}

\brmq
From Propositions \ref{prop:cohomology} and \ref{prop:psi-semisimple}, the cohomology space $ H^1(\D,\D)$ 
is completely determine by the space $\mathcal J$ of endomorphisms $j$ with
$j([x,y])=[j(x),y]$, for all $x,y$ in $\G,$ or equivalently, by the space of
adjoint-invariant bilinear forms on $\G.$
\ermq

\bcor
If $G$ is a semi-simple Lie group with Lie algebra $\G$, then the space of bi-invariant bilinear forms on $G$ 
is of dimension $dim H^1(\D,\D).$
\ecor


\bpro\label{guimbita2}
Suppose  $\G$ is a simple Lie algebra. Then,
\bitem 
\item[(a)] every linear map $j:\G\to\G$ in $\mathcal J$, is of the 
form $j(x)=\lambda x$, for some  $\lambda$ in $\R$;
\item[(b)] every element $\xi$ of $\mathcal E$ is of the form
\beq
\xi = ad^*_{x_0} + \lambda Id_{\G^*},
\eeq
for some  $x_0$ in $\G$ and   $\lambda$ in $\R$.
\eitem 
\epro

\begin{proof}
The part (a) is obtained from relation (\ref{aby}) and the Schur's lemma.
\newline\noindent
From Propositions \ref{decomposition-xi} and \ref{decomposition-xi'}, for every $\xi$ in
$\mathcal E,$ there exist $\alpha$ in $\der(\G)$ and $j$ in $\mathcal J$  such that
$\xi^t=\alpha + j$.
As $\G$ is simple and from part (a), there exist $x_0$ in $\G$ and $\lambda$ in $\R$ such that $\xi^t=ad_{x_0}+\lambda Id_\G.$
\end{proof}

We also have  the following.

\bpro
Let $G$ be a simple Lie  group with Lie algebra $\G$.
Let $\D:=\G\ltimes \G^*$ be the Lie algebra of the cotangent bundle $T^*G$ of $G$. Then, the first
cohomology space of $\D$ with coefficients in $\D$ is
$ H^1(\D,\D) \stackrel{\sim}{=} \R $.
\epro
\begin{proof}
Indeed, a derivation $\phi : \D \to \D$ can be written: for every element
$(x,f)$ of  $\D$,
\beq
\phi(x,f) = ([x_0,x]\; , \; ad^*_{x_0}f - ad^*_xf_0 + \lambda f),
\eeq
where $ x_0$ and $f_0$ are fixed elements in $\G$ and  $\G^*$ respectively.
The inner derivations are those with $\lambda=0$.
It follows that the first  cohomology space of $\D$ with values in $\D$ is given by
\beqn
H^1(\D,\D) &=& \{\phi : \D \to \D : \phi(x,f)=(0,\lambda f), \lambda \in \R\} \nonumber\\
           &=& \{\lambda(0, Id_{\G^*}), \lambda \in \R \} \nonumber \\
           &=& \R Id_{\G^*}  \nonumber
\eeqn
\end{proof}
As a direct consequence, we get

\bcor 
If  $\G$ is a semi-simple Lie algebra over $\R,$ then $\dim H^1(\D,\D) = p,$ where $p$ stands for the number of 
simple components of $\G,$ in its decomposition into a direct sum $\G=\s_1\oplus\cdots\oplus\s_p$ of simple 
ideals $\s_1,\ldots,\s_p.$
\ecor


Consider  a semi-simple Lie algebra $\G$  and set $\G=\s_1 \oplus \s_2 \oplus \cdots
\oplus \s_p,\; p \in \N^*$, where $\s_i,\; i=1,\ldots,p$ are simple Lie algebras. From Lemma \ref{lem:decomp-semi-simpl}, 
$\xi'$ preserves each $\s_i.$
Thus from Proposition \ref{guimbita2}, the restriction $\xi'_i$ of $\xi'$ to each $\s_i, \; i=1,2,\ldots ,p$ equals
$\xi'_i=ad_{x_{0_i}} + \lambda_i Id_{\s_i}$, for some $x_{0_i}$ in $\s_i$ and
a real number $\lambda_i$.
 Hence, $\xi'=ad_{x_0} \oplus_{i=1}^p \lambda_i Id_{\s_i}$,
where $x_0=x_{0_1}+x_{0_2}+ \cdots + x_{0_p} \in \s_1\oplus \s_2 \oplus \cdots \oplus \s_p$
and $\oplus_{i=1}^p \lambda_i Id_{\s_i}$ acts on $\s_1\oplus \s_2 \oplus \cdots \oplus \s_p$
as follows: 
$$(\oplus_{i=1}^p \lambda_i Id_{\s_i})(x_{1}+x_{2}+ \cdots + x_{p})=
\lambda_1 x_{1}+ \lambda_2 x_{2}+ \cdots + \lambda_p x_{p}.$$
In particular, we have proved

\bcor\label{j-semi-simple}
Consider the decomposition of a semi-simple Lie algebra $\G$  into a sum
$\G=\s_1 \oplus \s_2 \oplus \cdots \oplus \s_p,$  of simple Lie algebras
$\s_i,\; i=1,\ldots,p\in \N^*$. If a linear map $j:\G\to\G$ satisfies $j[x,y]=[jx,y],$
then there exist real numbers  $\lambda_1, \ldots \lambda_p$ such that
$$j=\lambda_1id_{\s_1}\oplus \cdots  \oplus \lambda_pid_{\s_p}.$$
More precisely $$j(x_1+\cdots + x_p)=\lambda_1x_1+\cdots + \lambda_px_p,$$
if $x_i$ is in $\s_i, ~ i=1,\ldots,p.$
\ecor

Now, we already know from Proposition \ref{prop:psi-semisimple}, that each $\psi$ vanishes identically.
So a $1$-cocycle of $\D$ is given by:
\beq\label{eq:cocycle-semi-simple}
\phi(x,f)=\Big([x_0,x], ad^*_{x_0}f-ad^*_xf_0 + \sum_{i=1}^p\lambda_if_i\Big)
\eeq
for every $x$ in $\G$ and every $f:=f_1 + f_2 +\cdots +f_p$ in $\s^*_1\oplus \s^*_2 \oplus
\cdots \oplus \s^*_p=\G^*$, where $x_0$ is in $\G$, $f_0$ is in $\G^*$ and $\lambda_i$,
$i=1,..,p$, are real numbers. We then have,
\bpro
Let $G$ be a semi-simple Lie group with Lie algebra $\G$ over $\R$.
Let $\D:= \G\ltimes \G^*$ be the  cotangent Lie algebra of $G$. Then, the first
 cohomology space of $\D$ with coefficients in $\D$ is given by
$ H^1(\D,\D) \stackrel{\sim}{=} \R^p$, where $p$ is the number of the simple
components of $\G$.
\epro

\subsection{Case of Compact Lie Algebras}\label{chap:compact}

It is well known that a compact Lie algebra $\G$ decomposes as the direct sum
$\G=[\G,\G]\oplus Z(\G)$ of its derived ideal $[\G,\G]$ and its centre $Z(\G)$, with $[\G,\G]$ semi-simple and 
compact. This yields a decomposition $\G^*=[\G,\G]^*\oplus Z(\G)^* $  of $\G^*$ into a direct sum of the dual 
spaces $[\G,\G]^*$, $Z(\G)^*$ of $[\G,\G]$ and $Z(\G)$ respectively, where $[\G,\G]^*$ (resp. $Z(\G)^*$) is 
identified with the space of linear forms on $\G$ which vanish on $Z(\G)$ (resp. $[\G,\G]$).
 On the other hand, $[\G,\G]$ also decomposes into as a direct sum $[\G,\G]=\s_1\oplus\ldots \oplus\s_p$ of 
simple ideals $\s_i.$
From Theorem \ref{derivationschar}, a derivation of the Lie algebra $\D := \G \ltimes \G^*$ of the cotangent 
bundle of $G$ has the following form $\phi (x,f) =\Big(\alpha(x) + \psi(f),\beta(x) + \xi(f)\Big),$ 
with conditions listed in  Theorem \ref{derivationschar}.

\vskip 0.3cm

Let us look at the equivariant maps $\psi: \G^* \to \G$ satisfying
$ad_{\psi(f)}^*g=ad_{\psi(g)}^*f$, for every $f,g$ in $\G^*$.
From Lemma \ref{lem:intertwining}, $Im(\psi)$ is an Abelian ideal of $\G$;
thus $Im(\psi) \subset Z(\G)$.
As consequence, we have $\psi (ad_x^*f)=[x,\psi(f)]=0$, for every $x$ of $\G$
and every $f$ of $\G^*$.

\blem \label{lem:semi-simple-psi}
Let $(\tilde \G,\mu)$ be an orthogonal Lie algebra  satisfying 
$\tilde\G= [\tilde\G,\tilde\G].$ Then, every $g$ in $\tilde\G^*$ is a finite sum of elements of 
the form $g_i=ad^*_{\bar x_i}\bar g_i,$ for some $\bar x_i$ in $\tilde \G,$ $\bar g_i$ in $ \tilde \G^*.$
\elem
 \begin{proof} 
Indeed, consider an isomorphism $\theta: \tilde\G \to \tilde\G^*$ of $\tilde\G$-modules. 
For every $g$ in $\tilde\G^*,$ there exists $x_g$ in $\tilde\G$ such that $g=\theta (x_g).$ But as 
$\tilde\G= [\tilde\G,\tilde\G],$ we have $ x_g=[x_1,y_1]+\ldots +[x_s,y_s]$ for some $x_i,y_i$ in $\tilde\G.$
Thus 
\beqn 
g &=& \theta ([x_1,y_1])+\ldots + \theta([x_s,y_s]) \cr 
  &=& ad^*_{x_1}\theta(y_1)+\ldots + ad^*_{x_s}\theta(y_s) \cr 
  &=& ad^*_{\bar x_1}\bar g_1 +\ldots +ad^*_{\bar x_s}\bar g_s, \nonumber
\eeqn 
where $\bar x_i=x_i$ and $ \bar g_i= \theta(y_i).$ 

\end{proof}
 A semi-simple Lie algebra being orthogonal (with, e.g.  its Killing form as $\mu$), from 
Lemma \ref{lem:semi-simple-psi} and the equality  $\psi (ad_x^*f)=0$, for all $x$ in $\G$, $f$ in $\G^*,$
each $\psi$ in $\Psi$ vanishes on $[\G,\G]^*.$ Of course, the converse is true. Every linear map 
$\psi:\G^*\to\G$ with $Im(\psi) \subset Z(\G)$ and $\psi([\G,\G]^*)=0$, is in $\Psi.$ Hence we can 
make the following identification.
\blem 
Let $\G$ be a compact Lie algebra, with centre $Z(\G).$ Then $\Psi$ is isomorphic to the 
space $L(Z(\G)^*,Z(\G))$ of linear maps $Z(\G)^*\to Z(\G).$
\elem

 The restriction of the cocycle $\beta$ to the semi-simple ideal $[\G,\G]$ is a coboundary,
that is, there exits an element $f_0$ in $\G^*$ such that for every $x_1$ in $[\G,\G]$,
$ \beta(x_1) = -ad_{x_1}^*f_0.$ Now for $x_2$ in $Z(\G)$, one has
$ 0= \beta[x_2,y]=-ad^*_y\beta(x_2)$, for all $y$ of  $\G$,  since  $x_2$ is in $Z(\G)$.
In other words, $\beta(x_2) ([y,z]) = 0$, for all $y,z$ in $\G$. That is, $\beta(x_2)$
vanishes on $[\G,\G]$ for every $x_2 \in Z(\G)$. Hence, we write
$$
\beta(x) = -ad_{x_1}^*f_0 + \eta(x_2),
$$  
for all $x:=x_1+x_2$ in $[\G,\G] \oplus Z(\G),$
where $\eta : Z(\G) \to Z(\G)^*$ is a linear map. This simply means the following.
\blem
 Let $\G$ be a compact Lie algebra, with centre $Z(\G).$ Then the first space $H^1(\G,\G^*)$ 
of the cohomology associated with the coadjoint action of $\G,$ is isomorphic to the space $L(Z(\G),Z(\G)^*)$.
\elem

We have already seen that $\xi$ is such that $\xi^t=\alpha + j$, where $\alpha$
is a derivation of $\G$ and $j$ is in $\mathcal J$. Both $\alpha$ and $j$ preserve each of
$[\G,\G]$ and $Z(\G).$ Thus we can write
$\alpha = ad_{{x_0}_1} \oplus \varphi,$  for some  ${x_0}_1 \in [\G,\G],$
where $\varphi$ is in $End(Z(\G))$. Here $\alpha$ acts on an element $x:=x_1+x_2$,
where $x_1$ is in $[\G,\G]$, $x_2$ belongs to $Z(\G),$ as follows:
\beqn 
\alpha(x)  &=& (ad_{{x_0}_1} \oplus \varphi)(x_1+x_2) \cr 
           &:=& ad_{{x_0}_1}x_1 + \varphi(x_2). \nonumber
\eeqn 

We summarize  this as
\blem\label{cohomology-compact}
If $\G$ is a compact Lie algebra with centre $Z(\G)$, then
\beq 
H^1(\G,\G)\stackrel{\sim}{=} End(Z(\G)).
\eeq 
\elem

Now, suppose for the rest of this section, that $\G$ is a compact Lie algebra.
 We write 
$$
j=\oplus_{i=1}^p \lambda_i Id_{\s_i} \oplus \rho,
$$
 where $\rho : Z(\G) \to Z(\G)$ is a linear map and $j$ acts on an element $x:=x_1 + x_2$
 as follows:
\beq
j(x) = \Big(\bigoplus_{i=1}^p \lambda_i Id_{\s_i} \oplus \rho\Big)
(x_{11} + x_{12} + \cdots + x_{1p}+x_2) = \sum_{i=1}^p \lambda_i x_{1i} + \rho(x_2) \nonumber
\eeq
 where $x_1:=x_{11} + x_{12} + \cdots + x_{1p}$ is in $[\G,\G]$, $x_2$ is in $Z(\G)$ and $ x_{1i}$ belongs to $\s_i.$
Hence,
\blem \label{j-compact}
If $\G$ is a compact Lie algebra with centre $Z(\G)$, then 
$\mathcal J\stackrel{\sim}{=} \R^p\oplus  End(Z(\G)),$ where $p$ is the number of simple components of $[\G,\G]$.
\elem
So, the expression of $\xi$ now reads
\beq
\xi = \big[-ad_{{x_0}_1}^*  + (\oplus_{i=1}^p \lambda_iId_{\s_i^*})\big]
\oplus \varphi'',  \nonumber
\eeq
with $(\varphi'')^t(x_2) = \rho(x_2) + \varphi(x_2)$, for all $x_2$ in $Z(\G)$,
where ${x_0}_1$ is in $[\G,\G]$, $\lambda_i$ is in $\R$, for all $i=1,2,\ldots,p$.

By identifying $End(Z(\G)),$ $L(Z(\G)^*,Z(\G))$ and $L(Z(\G),Z(\G)^*)$ to $End(\R^k),$ we get
\beq
H^1(\D,\D) = (End(\R^k))^4 \oplus  \R^p.
\eeq

\section{Some Possible Applications and Open Problems}\label{openproblems}

Given two left or right  invariant structures of the same `nature' (e.g. affine, symplectic, complex,
Riemannian or pseudo-Riemannian, etc) on $T^*G$, one wonders whether they are
equivalent, {\it i.e.} if there exists an automorphism of $T^*G$ mapping one to the other.
By taking the values of those structures at the unit of $T^*G,$
the problem translates to finding an automorphism of Lie algebra mapping
two structures of $\D.$ The work within this chapter may also be seen as a useful tool for the study of such structures. 
Here are some examples of problems and framework for further extension and applications of this work. 
 For more discussions about structures and problems on $T^*G$, one can have a look at \cite{bajo-benayadi-medina}, \cite{di-me-cybe}, 
\cite{drinfeld}, \cite{feix}, \cite{kronheimer}, \cite{marle}.

\subsection{Some Examples}\label{examples}

Here, we apply the above results to produce the following examples.

\bex[The Affine Lie Algebra of the Real Line]\label{chap:aff(R)}
{\normalfont
 The 2-dimensional affine Lie algebra $\G=\hbox{\rm aff}(\R)$ is solvable nonnilpotent with Lie bracket 
$$[e_1,e_2]=e_2,$$ 
in some basis ($e_1,e_2$). 
The Lie algebra $\D=T^*\G$ of the cotangent bundle of any Lie group with Lie algebra $\G,$ has a basis ($e_1,e_2,e_3,e_4$) 
with Lie bracket 
$$
[e_1,e_2]=e_2 \;,\; [e_1,e_4]=-e_4 \;,\; [e_2,e_4]=e_3,
$$ 
where $e_3:=e_1^*$ and $e_4:=e_2^*$. 
This is the semi-direct product $ \R e_1\ltimes \mathcal H_3 $ of the Heisenberg Lie algebra $\mathcal H_3=\sspan(e_2,e_3,e_4)$ 
with the line $\R e_1$, where $e_1$ acts on $\mathcal H_3$ by the restriction of the derivation of $ad_{e_1}$. 
The Lie algebra $\der(\D)$ has a basis ($\phi_1,\phi_2,\phi_3,\phi_4,\phi_5$) where 

\beq
\begin{array}{lllll}
\phi_1(e_1)=e_2  &,& \phi_1(e_4)=-e_3 &,& \phi_2(e_2)=e_2, \cr 
\phi_2(e_4)=-e_4  &,& \phi_3(e_1)=e_3 &,& \phi_4(e_1)=-e_4,\cr 
\phi_4(e_2)=e_3  &,& \phi_5(e_3)=e_3 &,& \phi_5(e_4)=e_4,
\end{array} \nonumber
\eeq
the remaining vectors $\phi_i(e_j)$ being zero, so that the Lie brackets are
$$
[\phi_2,\phi_1]=\phi_1 \; , \;  [\phi_2,\phi_3] = \phi_3 \;,\; [\phi_5,\phi_3]=\phi_3 \;,\; [\phi_5,\phi_4]= \phi_4.
$$ 
This is the semi-direct product $ \R^2\ltimes\R^3 $ of the Abelian Lie algebras $\R^3=\sspan_{\R}(\phi_1,\phi_3,\phi_4) $ 
and $\R^2=\sspan_{\R}(\phi_2,\phi_5).$ It has a contact structure (this is the Lie algebra number 18, for $p=q=1,$ in the 
list of Section 5.2.}
\eex

\bex[The Lie Algebra of the Group $SO(3)$ of Rotations]\label{chap:so(3)}
{\normalfont
Consider the Lie algebra $\mathfrak{so}(3)=\sspan(e_1,e_2,e_3)$ with 
$$ 
[e_1,e_2]=-e_3 \;,\; [e_1,e_3]=e_2 \;,\; [e_2,e_3]=-e_1.
$$
The Lie algebra $\D=T^*\G$ has $(e_1,e_2,e_3,e_4,e_5,e_6)$ with Lie bracket 
\beq
\begin{array}{llllll}
[e_1,e_5]=-e_6 &,& [e_1,e_6]=e_5  &,& [e_2,e_4]=e_6, \cr 
[e_2,e_6]=-e_4 &,& [e_3,e_4]=-e_5 &,& [e_3,e_5]=e_4, 
\end{array} \nonumber
\eeq
where $e_{3+i}=e_i^*,$ $i=1,2,3$.  The Lie algebra $\der(\D)$ is spanned by the elements $\phi_1$, $\phi_2$, $\phi_3$, $\phi_4$, 
$\phi_5$, $\phi_6$, $\phi_7$, where 
\beq
\begin{array}{llllll}
\phi_1(e_1)=-e_2 &,& \phi_1(e_2)=e_1  &,& \phi_1(e_4)=-e_5, \cr 
\phi_1(e_5)=e_4  &,& \phi_2(e_1)=-e_3 &,& \phi_2(e_3)=e_1,  \cr 
\phi_2(e_4)=-e_6 &,& \phi_2(e_6)=e_4  &,& \phi_3(e_2)=-e_3, \cr
\phi_3(e_3)=e_2  &,& \phi_3(e_5)=-e_6 &,& \phi_3(e_6)=e_5,  \cr 
\phi_4(e_1)=-e_5 &,& \phi_4(e_2)=e_4  &,& \phi_5(e_1)=-e_6, \cr 
\phi_5(e_3)=e_4  &,& \phi_6(e_4)=e_4  &,& \phi_6(e_5)=e_5,  \cr 
\phi_6(e_6)=e_6  &,& \phi_7(e_2)=-e_6 &,& \phi_7(e_3)=e_5,
\end{array} \nonumber
 \eeq
so that we have the following Lie brackets
\beq
\begin{array}{llllllll}
[\phi_1,\phi_2]=-\phi_3 &,& [\phi_1,\phi_3]=\phi_2  &,& [\phi_1,\phi_5]=-\phi_7 &,& [\phi_1,\phi_7]=\phi_5 , \cr 
[\phi_2,\phi_3]=-\phi_1 &,& [\phi_2,\phi_4]=\phi_7  &,& [\phi_2,\phi_7]=-\phi_4 &,& [\phi_3,\phi_4]=-\phi_5, \cr 
[\phi_3,\phi_5]=\phi_4  &,& [\phi_4,\phi_6]=-\phi_4 &,& [\phi_5,\phi_6]=-\phi_5 &,& [\phi_6,\phi_7]=\phi_7.   \cr
\end{array}
 \eeq 
This is the Lie algebra $\der(\D)=\mathfrak{so}(3)\ltimes \G_{id}$, where $\mathfrak{so}(3)=\sspan(\phi_1,\phi_2,\phi_3)$ and $\G_{id}$ 
is the semi-direct product $\G_{id}=\R \phi_6\ltimes \R^3$ of the abelian Lie algebras $\R^3=\sspan(\phi_4,\phi_5,\phi_7)$ 
and $\R \phi_6$ obtained by letting $\phi_6$ acts as the identity map on $\R^3.$ Thus $\der(\D)$ is also a contact Lie algebra, 
as it is the Lie algebra number 4 of Section 5.3 in \cite{diatta-contact}.
}
\eex

\bex[The Lie Algebra of the Group $SL(2,\R)$ of Spacial Linear Group]\label{chap:sl(2)}
{\normalfont
The Lie algebra  $\G=\mathfrak{sl}(2,\R)$ of $SL(2,\R)$ has a basis $(e_1,e_2,e_3)$ in which its Lie bracket reads
\beq
 [e_1 , e_2 ] = -2e_2 \quad ,\quad  [e_1 , e_3 ] = 2e_3 \quad  , \quad  [e_2 , e_3 ] =-e_1
\eeq 
 Set $e^*_1 =: e_4, e^*_2 =: e_5, e^*_3=e_6$, the Lie bracket of $\D := T^*\G$ 
in the basis $(e_1 ,e_2,e_3,e_4,e_5,e_6)$ is given by 
\beq
\begin{array}{llrlllrlllr}
[e_1, e_2] &=&-2e_2 &,& [e_1, e_3] &=& 2e_3 &,& [e_2 , e_3] &=&-e_1 , \cr 
[e_1, e_5] &=& 2e_5 &,& [e_1, e_6] &=&-2e_6 &,& [e_2 , e_4] &=& e_6 , \cr
[e_2, e_5] &=&-2e_4 &,& [e_3, e_4] &=&-e_5 &,& [e_3, e_6] &=& 2e_4. 
\end{array}
\eeq 
The Lie algebra $\der(\D)$ is $7$-dimensional. It has a basis $(\phi_1,\phi_2,\phi_3,\phi_4,\phi_5,\phi_6,\phi_7)$, where
\beq
\begin{array}{cclcccl}
\phi_1 &:=& -e_{22} + e_{33} + e_{55}-e_{66}      &,& \phi_2 &:=& e_{44} + e_{55} + e_{66}, \cr
\phi_3 &:=& -e_{12} + 2e_{31} - 2e_{46} + e_{54}  &,& \phi_4 &:=& e_{13} -2e_{21} + 2e_{45} - e_{64}, \cr
\phi_5 &:=& -e_{53} + e_{62}                      &,& \phi_6 &:=& e_{42} -e_{51} ,\cr 
 \phi_7 &:=&e_{43} -e_{61}                        & &                             
\end{array}
\eeq 
Hence, the Lie bracket of $\der(\D)$ reads
\beq
\begin{array}{ccccccccccccc}
[\phi_1,\phi_3] &=& \phi_3  &,& [\phi_1,\phi_4] &=& -\phi_4 &,& [\phi_1,\phi_6] &=& \phi_6,  \cr 
[\phi_1,\phi_7] &=& -\phi_7 &,& [\phi_2,\phi_5] &=& \phi_5  &,& [\phi_2,\phi_6] &=& \phi_6, \cr 
[\phi_2,\phi_7] &=& \phi_7  &,& [\phi_3,\phi_4] &=& 2\phi_1 &,& [\phi_3,\phi_5] &=&-2\phi_6, \cr 
[\phi_3,\phi_7] &=&-\phi_5  &,& [\phi_4,\phi_5] &=&-2\phi_7 &,& [\phi_4,\phi_6] &=& -\phi_5.
\end{array}
\eeq 
One realizes that $\der(\D) = \mathfrak{sl}(2,\R) \ltimes \G_{id}$, where $\mathfrak{sl}(2,\R) =\sspan(\phi_1,\phi_3,\phi_4)$ and as above,
$\G_{id}$ is the semi-direct product $\G_{id}= \R\phi_2 \ltimes \R^3$ of the Abelian Lie algebras $\R^3=\sspan(\phi_5,\phi_6,\phi_7)$
and $\R\phi_2$ obtained by letting $\phi_2$ act as the identity map on $\R^3$. Again, $\der(\D)$ is also a contact Lie algebra, with 
$\eta:=s\phi^*_1 + t\phi^*_5$ as an example of a contact form, where $s,t \in \R-\{0\}$.

\vskip 0.3cm

On the Lie algebra $\D=\mathfrak{sl}(2,\R)\ltimes \mathfrak{sl}(2,\R)^*$, consider the following two forms
\beqn\label{metric}%
\varphi_1\left( (x,f),(y,g) \right) &=& f(y) + g(y) \\
\varphi_2\left( (x,f),(y,g) \right) &=& f(y) + g(y) + K(x,y)
\eeqn
where $K$ stands for the Killing form of  $\mathfrak{sl}(2,\mathbb R)$.
The matrix of $\varphi_1$ has two eigenvalues, $-1$ and $+1$, both of multiplicty $3$. 
Therefore $\varphi_1$ is of signature $(3,3)$. 
Now the matrix of $\varphi_2$ 
has four eigenvalues : $4-\sqrt{17}$, $-2-\sqrt{5}$, $2-\sqrt{5}$, 
$2+\sqrt{5}$, $-2+\sqrt{5}$, $4+\sqrt{17}$. The three first eigenvalues are less than zero 
and the three last ones are positive. Hence the signature of $\varphi_2$  is $(3,3)$, too.
It is now straightforward to check that 
\beq
\der(\D) = ad_\D \oplus \mathbb R \phi_2
\eeq %
where $ad_\D:=\sspan(\phi_1,\phi_3,\phi_4,\phi_5,\phi_6,\phi_7)$ is the space of inner d\'erivations of $\D$, $\phi_2$ being, up
to multiplication by a scalar, the unique exterior derivation of $\D$. 

\vskip 0.3cm

Let us look at the restrictions of $\varphi_1$ and $\varphi_2$ to the Levi subalgebra $\mathfrak{sl}(2,\R)$ of $\D$.
Since the restriction of $\varphi_1$ to $sl(2,\R)$ is degenerate, so is the restriction of $\varphi_1$ to 
the image $\exp(\phi_i)(\mathfrak{sl}(2,\R))$ of $\mathfrak{sl}(2,\R)$ under the standard exponential map of any of the inner derivations 
$\phi_i$, $i=1,3,4,5,6,7$, of $\D$. That is because special automorphisms $\exp(\phi_i)$,  $i=1,3,4,5,6,7$, 
preserve the Levi subalgebra $sl(2,\R)$ of $\D$. Suppose that there exists a special automorphism $\exp(\phi_i)$, 
$i=1,3,4,5,6,7$, mapping $\varphi_1$ to $\varphi_2$, {\it i.e.}
\beq
\varphi_1\Big(\exp(\phi_i)(x,f)\,,\,\exp({\phi_i})(y,g)\Big) = \varphi_2\Big((x,f),(y,g)\Big)
\eeq 
for all $x,y$ in $\mathfrak{sl}(2,\R)$ and all $f,g$ in $\mathfrak{sl}(2,\R)^*$. But the restriction of $\varphi_1$ to 
$\exp(\phi_i)(\mathfrak{sl}(2,\R))$ is degenerate while the restriction of $\varphi_2$ to $\mathfrak{sl}(2,\R)$ is equal to the Killing 
form $K$ of $\mathfrak{sl}(2,\R)$, which is not degenerate. Hence, $\varphi_1$ and $\varphi_2$ are not homothetic via 
special automorphisms of $\D$. 

\vskip 0.3cm

Now, what about $\exp(\phi_2)$ ? One has $\exp(\phi_2)=e_{11} + e_{22} + e_{33} + e\phi_2$, 
where $e:=\exp(1)$. For any elements $(x,f)$ and $(y,g)$ of $\D$, we have, 
\beqn
\varphi_2\Big(\exp(\phi_2)(x,f)\,,\,\exp(\phi_2)(y,g) \Big) &=&  \varphi_2\Big((x,ef),(y,eg) \Big) \cr
                                                            &=& e\varphi_1\Big((x,f),(y,g) \Big) + K(x,y) \cr
                                                            &\neq& \varphi_1\Big((x,f),(y,g) \Big)
\eeqn 
Hence, the automorphism $\exp(\phi_2)$ too do not map $\varphi_2$ to $\varphi_1$.

}
\eex 

\subsection{Invariant Riemannian or Pseudo-Riemannian Metrics}\label{openproblems-orthogonal}

As discussed in Section \ref{chap:notations}, the Lie group $T^*G$ possesses  bi-invariant pseudo-Rieman~-~-nian metrics.

\vskip 0.3cm

Among others, one of the open problems in \cite{me-re93}, is the
question as to whether, given two bi-invariant pseudo-Riemannian metrics
$\mu_1$ and $\mu_2$ on a Lie group $\tilde G$, the two pseudo-Riemannian
manifolds $(\tilde G,\mu_1)$ and $(\tilde G,\mu_2)$  are homothetic via an
automorphism of $\tilde G.$ If this is the case, we say that $\mu_1$ and $\mu_2$
are isomorphic or equivalent.

\vskip 0.3cm

When $\tilde G =T^*G,$ the question as to how many non-isomorphic bi-invariant
pseudo-Riemannian metrics can exist on $T^*G$ is still open, in the general case.
For example, suppose $G$ itself has a  bi-invariant Riemannian or pseudo-Riemannian
metric $\mu$ and let $\mu$ stand again for the corresponding adjoint-invariant metric
in the Lie algebra $\G$ of $G.$ Then  $\mu$ induces a new adjoint-invariant metric
$\langle, \rangle_\mu$ on $\D=Lie(T^*G)$, with
\beq
\langle(x,f),(y,g)\rangle_\mu:= \langle(x,f),(y,g)\rangle + \mu(x,y),
\eeq
for all $x,y$ in $\G$ and all $f,g$ in $\G^*$, where $\langle,\rangle$ on the right
hand side is the duality pairing $\langle(x,f),(y,g)\rangle= f(y)+g(x).$
In some cases (see Section \ref{chap:sl(2)}), the two metrics can even happen to have the same index, but are still
not isomorphic via an automorphism of $\tilde G$.  If $\tilde \mu$ is a bilinear
symmetric form on $\G^*$ satisfying $\tilde \mu(ad^*_xf,g)=0$, for every
$x$ in $\G$ and every  $f,g$ in $\G^*$, then
\beq
\langle (x,f),(y,g) \rangle_{\tilde \mu}:= \langle(x,f),(y,g)\rangle +\tilde \mu(f,g),
\eeq
for all $x,y$ in $\G$, $f,g$ in $\G^*$, is also another adjoint-invariant metric on $\D$.
Now in some cases, nonzero such bilinear forms  $\tilde \mu$ may not exist, as is the
case when one of the coadjoint orbits of $G,$ is an open subset of $\G^*,$ or equivalently,
when $G$ has a left invariant exact symplectic structure.

\vskip 0.3cm

More generally, these equivalence questions can also be simply extended to all
left (resp. right) invariant Riemannian or pseudo-Riemannian structures on
cotangent bundles of Lie groups.

\subsection{Poisson-Lie Structures, Double Lie Algebras, Applications}\label{openproblems-doubles}

The classification questions of double Lie algebras, Manin pairs, Lagrangian subalgebras, ...
arising from Poisson-Lie groups, are still open problems \cite{drinfeld}, \cite{lu-weinstein}. In \cite{drinfeld}, 
Lagrangian subalgebras of double Lie algebras are used as the main tool for classifying the so-called Poisson 
Homogeneous spaces of Poisson-Lie groups.
A type of local action of those Lagrangian subalgebras is also used to describe symplectic foliations of Poisson 
Homogeneous spaces of Poisson-Lie groups in \cite{drinfeld}, \cite{di-me-poisson}.

\vskip 0.3cm

It would be interesting to extend the results within this chapter to double Lie algebras of general Poisson-Lie groups. 
It is hard to get a common substantial description valid for the
group of automorphism of the double Lie algebras of all possible
Poisson-Lie structures in a given Lie group. This is due to the
diversity of Poisson-Lie structures that can coexist in the same Lie group.

\vskip 0.3cm

Among other things, the description of the group of automorphisms of the
double Lie algebra of a
Poisson-Lie structure is a big step forward towards solving very
interesting and hard problems such as:
\bitem
\item[-]  the classification of Manin triples (\cite{lu-weinstein});
\item[-] the classification of Poisson homogeneous spaces of a Lie groups (\cite{di-me-poisson}, \cite{lu-weinstein});
\item[-] a full description and understanding   of the foliations of Poisson homogeneous spaces of Poisson Lie
groups. The  leaves of such foliations trap the   trajectories, under a Hamiltonian flow,
passing through all its points. Hence, from their knowledge, one gets a great deal of information on Hamiltonian systems.
\item[-] etc.
\eitem 

\subsection{Affine and Complex Structures on $T^*G$ }\label{openproblems-affine}

In certain cases, $T^*G$ possesses left invariant affine connections, that is,
left invariant zero curvature and torsion free linear connections. Here, the
classification problem involves Aut($T^*G$) as follows. The group Aut($T^*G$)
acts on the space of left invariant affine connections on $T^*G$,
 the orbit of each connection being the set of equivalent (isomorphic) ones.
 Recall that among other results in \cite{di-me-cybe}, the authors proved that when $G$ has an invertible solution 
of the Classical Yang-Baxter Equation (or equivalently a left invariant symplectic structure), then $T^*G$ has a 
left invariant affine connection $\nabla$ and a complex structure $J$ such that $\nabla J=0$.

\chapter{Prederivations of Lie Algebras of Cotangent Bundles of Lie Groups }
\minitoc
\section{Introduction}

In the sense of Felix Klein (\cite{klein}), studying the geometry of a  "universe" is studying its invariant 
structures under the action of a suitable Lie group. In  semi-Riemannian geometry, one of the suitable group 
used in this task  is the group  of  isometries of pseudo-Riemannian metrics. So it seems reasonable to well 
known isometries of  pseudo-Riemannian metrics. Among tools used, for instance 
in the case of bi-invariant (or orthogonal) Lie groups, there are prederivations of Lie algebras. 
M\"uller (\cite{muller}) gives an algebraic description of the group $I(G,\mu)$ of isometries of a 
connected orthogonal Lie group $(G,\mu)$. He proves that if $(G,\mu)$ is a connected and simply-connected 
orthogonal Lie group with Lie algebra $\G$, then the stabilizer of the identity element of $G$ in $I(G,\mu)$ is isomorphic to the  group of  
preautomorphisms of  $\G$ which preserve the non-degenerate bilinear form induced by $\mu$ on $\G$ and whose 
Lie algebra is the whole set of skew-symmetric prederivations of $\G$. In \cite{bajo4}, Bajo studies the 
algebra of prederivations and skew-symmetric prederivations of a direct sum of Lie algebras and this 
study allows him to generalize some results in \cite{bajo1}, \cite{bajo2} and in \cite{wu}. 

\vskip 0.3cm 

Prederivations also present an interest  in the purely algebraic point of view. As well as Jacobson (\cite{jacobson55}) 
proves that a Lie algebra admitting a non-singular derivation is necessarily nilpotent, the author quoted above 
establishes in \cite{bajo3} that a Lie algebra  possessing a non-singular prederivation is necessarily nilpotent. 
Prederivations are also useful tools for construction of affine structures on Lie algebras (see Section \ref{application-prederivations}).

\vskip 0.3cm 

In this chapter, we deal with prederivations of the Lie algebra of the cotangent bundle of a Lie group; 
which Lie algebra appears as the semi-direct sum $\G\ltimes \mathcal A$ of a Lie algebra $\G$ and an Abelian
Lie algebra $\mathcal A$. In the sequel, we will take $\G$ to be the Lie algebra of 
a Lie group $G$, $\mathcal A=\G^*$, the dual space of $\G$ and $T^*\G:=\G\ltimes \G^*$ will be the semi-direct sum 
of the Lie algebra $\G$ and the vector space $\G^*$ via the coadjoint representation.  

\vskip 0.5cm

Our aim is to explore the structure of the Lie algebra $\pder(T^*\G)$ of prederivations of $T^*\G$. 
We will have a particular attention to Lie algebras admitting quadratic or orthogonal structures. 

\vskip 0.5cm

The main results within this chapter are the following.

\vspace{0.6cm}
\noindent
{\bf Theorem A.}~{\em
A prederivation $p :T^*\G \to T^*\G$ is defined by :
\beq
p(x,f) = \Big(\alpha(x) + \psi(f),\beta(x) + \xi(f)\Big) \nonumber
\eeq
for any element $(x,f)$ of $T^*\G$, where $\alpha : \G \to \G$ is a prederivation of $\G$ and 
$\beta : \G \to \G^*$, $\psi : \G^* \to \G$ and $\xi : \G^* \to \G^*$ are linear maps satisfying 
the following four relations :
\beq
\beta\big([x,[y,z]]\big) = -ad^*_{[y,z]}\big(\beta(x)\big)   + ad^*_x \Big(ad^*_y\big(\beta(z)\big) 
-ad^*_z\big(\beta(y)\big) \Big) \nonumber
\eeq
\beq
\psi \circ ad^*_{[x,y]}    =  ad_{[x,y]} \circ \psi \nonumber
\eeq
\beq
ad^*_x  \Big(ad^*_{\psi(f)}g -ad^*_{\psi(g)}f\Big) = 0,\nonumber
\eeq
\beq
\Big[\xi,ad^*_{[x,y]}\Big]  = ad^*_{([\alpha(x),y] + [x,\alpha(y)])} \nonumber
\eeq
for every elements $x$ and $y$ of $\G$ and any elements $f$ and $g$ in $\G^*$.
}
\vspace{0.6cm}

About the structure of the Lie algebra of the Lie algebra $\pder(T^*\G)$ of prederivations of $T^*\G$, we have the 

\vspace{0.6cm}
\noindent
{\bf Theorem B.}~{\em
Let  $G$ be any finite-dimensional Lie group with Lie algebra $\G$. 
Then the Lie algebra $\pder(T^*\G)$ of prederivations of the Lie algebra $T^*\G$ of the Lie group $T^*G$ decomposes as follows :
 $\pder(T^*\G) =\G_0'\oplus\G_1'$, where $\G_0'$ is a reductive subalgebra of $\pder(T^*\G)$, 
that is $[\G_0',\G_0']\subset \G_0'$ and $[\G_0',\G_1']\subset \G_1'$.
} 

\vspace{0.6cm}

Recall that the Lie algebra of the cotangent bundle Lie group of a semi-simple Lie group is not semi-simple. 
Any way, as well as M\"uller proves that any prederivation of a semi-simple Lie algebra is a derivation we prove the following 

\vspace{0.6cm}
\noindent
{\bf Theorem C.}~{\em
Let $G$ be a semi-simple Lie group with Lie algebra $\G$. Then every prederivation of the Lie algebra $T^*\G$ of the 
cotangent bundle Lie group $T^*G$ of $G$ is a derivation.
}

\vskip 0.3cm

The chapter contains five  ($5$) sections. In Section \ref{chap:preliminaries-chap2} is explained the link between 
prederivations and isometries of bi-invariant metrics on Lie groups. In Section \ref{section:preautomorphims} 
we study the structure of the Lie algebra $\pder(T^*\G)$ of  prederivations of the Lie algebra 
$T^*\G$ of the cotangent bundle $T^*G$ of a Lie group $G$ with Lie $\G$. The particular case where the Lie group $G$
possesses a bi-invariant metric is studied in Section \ref{section:orthogonal-lie-algebra-chap2}. 
In Section \ref{section:applications-examples-chap3} we give examples and some possible applications.

\section{Preliminaries}\label{chap:preliminaries-chap2}

\subsection{Prederivations of a Lie algebra}

\bdfn
Let $\G$ be a Lie algebra. A bijective endomorphism $P:\G \to \G$ such that
\beq\label{relation-preautomorphism}
P\big(\big[x,[y,z]\big] \big) = \big[P(x),[P(y),P(z)]\big],
\eeq
 for all $x,y,z$ in $\G$, is called a {\it preautomorphism} of $\G$.
\edfn

The set of all preautomorphisms of $\G$ forms a Lie group (see \cite{muller}) which we note by $\paut(T^*\G)$.
Its Lie algebra is the subset of the set $\G l(\G)$ of endomorphisms of $\G$ consisting of elements $p$
which satisfy the following relation.
\beq
p\big(\big[x,[y,z]\big]\big) = \big[p(x),[y,z]\big] + \big[x,[p(y),z]\big] +
\big[x,[y,p(z)]\big]
\eeq
for every elements $x,y,z$ of $\G$.

\vskip 0.5cm

One can easily convince himself that any derivation of a Lie algebra is a prederivation. 
Furthermore, there exists a class
of Lie algebras that are such that any prederivation is a derivation. Semi-simple Lie algebras belong to that class
of algebras (see \cite{muller}). We will prove in Section \ref{chap:semi-simple-lie-algebras2} that
the Lie algebras (which are not semi-simple) of the cotangent Lie groups of  semi-simple Lie groups are also members of this class.

\vskip 0.5cm

As well as Jacobson (\cite{jacobson55}) proves that a Lie algebra admitting a non-singular derivation 
is necessarily nilpotent, Bajo (\cite{bajo3}) shows that a Lie algebra which possesses a non-singular 
prederivation is necessarily nilpotent. The converses of the both results are false since 
there are nilpotent Lie algebra that admit only singular derivations and prederivations (see \cite{bajo3}).

\subsection{Prederivations  and Isometries of Orthogonal Structures}

M\"uller (\cite{muller}) proves that if $(G,\mu)$ is a connected and 
simply-connected orthogonal Lie group with Lie algebra $\G$, then the isotropy group of the neutral 
element of $G$ in the group $I(G,\mu)$ of isometries of $(G,\mu)$ is isomorphic to the  subgroup of $GL(\G)$ (group of endomorphisms of $\G$) 
consisting of preautomorphisms of $\G$ which preserve the non-degenerate bilinear form induced by $\mu$ on $\G$ and whose 
Lie algebra is the  whole set of skew-symmetric prederivations of $\G$. See Section \ref{section:isometries-preautomorphisms} for wider informations.

%

\section{Preautomorphisms  of $T^*\G$}\label{section:preautomorphims}
\subsection{Prederivations of $T^*\G$}

Let $\G$ be a Lie algebra and let $T^*\G:=\G\ltimes \G^*$ stand for the semi-direct product of $\G$ with its dual via the coadjoint
representation. The Lie bracket of $T^*\G$ is given by

\beq \label{bracket_double2}
[(x,f),(y,g)]:=\Big([x,y],ad^*_xg-ad^*_yf\Big),
\eeq
for any elements $(x,f)$ and $(y,g)$ of $T^*\G$.

\vskip 0.3cm

Let $p:T^*\G \to T^*\G$ be a prederivation of $T^*\G$ and set
\beq
p(x,f) = \Big(\alpha(x) + \psi(f),\beta(x) + \xi(f)\Big),
\eeq
where $\alpha : \G \to \G$, $\psi :\G^* \to \G$, $\beta : \G \to \G^*$ and $\xi :\G^* \to \G^*$ are linear maps.

\vskip 0.3cm

Let $x,y,z$ be elements of $\G$. We have :
\beq\label{alpha1}
p\big([x,[y,z]]\big) = \underbrace{\alpha\big([x,[y,z]]\big)}_{\in \G} + \underbrace{\beta\big([x,[y,z]]\big)}_{\in \G^*}
\eeq
On the other way, since $p$ is a prederivation, then
\beqn
p\big([x,[y,z]]\big) &=& \Big[p(x),[y,z]\Big] + \Big[x,[p(y),z]\Big] + \Big[x,[y,p(z)]\Big] \cr
                     &=& \Big[\alpha(x)+\beta(x),[y,z]\Big] + \Big[x,[\alpha(y)+\beta(y),z]\Big] +\Big[x,[y,\alpha(z)+\beta(z)]\Big] \cr
                     &=&  \Big(\Big[\alpha(x),[y,z]\Big] -ad^*_{[y,z]}\big(\beta(x)\big) \Big) +\Big[x, [\alpha(y),z]
                          -ad^*_z\big(\beta(y)\big)\Big]  \cr
                     & & + \Big[x,[y,\alpha(z)]+ ad^*_y\big(\beta(z)\big)\Big]   \cr
                     &=& \Big(\Big[\alpha(x),[y,z]\Big] -ad^*_{[y,z]}\big(\beta(x)\big) \Big) + \Big(\Big[x,[\alpha(y),z]\Big]
                           -ad^*_x\circ ad^*_z\big(\beta(y)\big)   \Big) \cr
                     & & + \Big( \Big[x,[y,\alpha(z)]\Big] + ad^*_x\circ ad^*_y\big(\beta(z)\big) \Big) \cr
                     &=& \underbrace{\Big(\Big[\alpha(x),[y,z]\Big] + \Big[x,[\alpha(y),z]\Big] + \Big[x,[y,\alpha(z)]\Big]\Big)}_{\in \G} \cr
                     & & +\underbrace{\Big(-ad^*_{[y,z]}\big(\beta(x)\big) -ad^*_x\circ ad^*_z\big(\beta(y)\big)
                         + ad^*_x\circ ad^*_y\big(\beta(z)\big)\Big)}_{\in \G^*}. \label{alpha2}
\eeqn
From relations (\ref{alpha1}) and (\ref{alpha2}) we have :
\beq
\alpha\Big(\Big[x,[y,z]\Big]\Big) = \Big[\alpha(x),[y,z]\Big] + \Big[x,[\alpha(y),z]\Big] + \Big[x,[y,\alpha(z)]\Big].
\eeq
for any $x,y,z$ in $\G$ ; that is $\alpha$ is a prederivation of $\G$.
Relations (\ref{alpha1}) and (\ref{alpha2}) also give
\beq
\beta\Big(\Big[x,[y,z]\Big]\Big) = -ad^*_{[y,z]}\big(\beta(x)\big) -ad^*_x\circ ad^*_z\big(\beta(y)\big)
                         + ad^*_x\circ ad^*_y\big(\beta(z)\big).
\eeq
Now let $x$ be an element of $\G$ and $f,g$ be in $\G^*$. We have
\beq \label{psi1}
p([x,[f,g]]) = p(0) = 0.
\eeq
We also have
\beqn
p([x,[f,g]]) &=& \Big[p(x),[f,g]\Big] + \Big[x,[p(f),g]\Big] + \Big[x,[f,p(g)]\Big] \cr
             &=& 0 + \Big[x,[\psi(f)+\xi(f),g]\Big] + [x,[f,\psi(g)+\xi(g)]\Big] \cr
             &=& [x,ad^*_{\psi(f)}g] + [x,-ad^*_{\psi(g)}f] \cr
             &=& [x,ad^*_{\psi(f)}g -ad^*_{\psi(g)}f] \cr
             &=& ad^*_x  \Big(ad^*_{\psi(f)}g -ad^*_{\psi(g)}f\Big). \label{psi2}
\eeqn
From (\ref{psi1}) and (\ref{psi2}) it comes that
\beq
ad^*_x  \Big(ad^*_{\psi(f)}g -ad^*_{\psi(g)}f\Big) = 0,
\eeq
for every $x$ in $\G$ and every $f,g$ in $\G^*$. That is $ad^*_{\psi(f)}g -ad^*_{\psi(g)}f$ belongs the centralizer
$Z_{T^*\G}(\G)$ of $\G$ in $T^*\G$, for any $f$ and $g$ in $\G^*$.

\vskip 0.3cm
Let us now consider the following case. Let $x$ and $y$ be in $\G$ and $f$ be in $\G^*$.
\beq\label{psi-xi1}
p\Big(\Big[x,[y,f]\Big]\Big) = \underbrace{\psi\Big(\Big[x,[y,f]\Big]\Big)}_{\in \G} + \underbrace{\xi\Big(\Big[x,[y,f]\Big]\Big)}_{\in \G^*}
\eeq
Since $p$ is a prederivation, one has
\beqn
p\Big(\Big[x,[y,f]\Big]\Big) &=&  \Big[p(x),[y,f]\Big] + \Big[x,[p(y),f]\Big] + \Big[x,[y,p(f)]\Big] \cr
                             &=& \Big[\alpha(x)+\beta(x),[y,f]\Big] + \Big[x,[\alpha(y)+\beta(y),f]\Big]
                                + \Big[x,[y,\psi(f)+\xi(f)]\Big] \cr
                             &=& \Big[\alpha(x)+\beta(x),ad^*_yf\Big] + \Big[x,ad^*_{\alpha(y)}f\Big]
                                 + \Big[x,[y,\psi(f)]+ad^*_y\big(\xi(f)\big)\Big] \cr
                             &=& ad^*_{\alpha(x)}\circ ad^*_y(f) + ad^*_x\circ ad^*_{\alpha(y)}(f) + \big[x,[y,\psi(f)]\big]
                                + ad^*_x\circ ad^*_y\big(\xi(f)\big) \cr
                             &=& \underbrace{[x,[y,\psi(f)]}_{\in \G} + \underbrace{\Big(ad^*_{\alpha(x)}\circ ad^*_y +ad^*_x\circ ad^*_{\alpha(y)}
                                + ad^*_x\circ ad^*_y\circ\xi \Big)(f)}_{\in \G^*} \label{psi-xi2}
\eeqn
On one hand, relations (\ref{psi-xi1}) and (\ref{psi-xi2}) imply
\beqn
\psi\Big(\Big[x,[y,f]\Big]\Big)    &=& [x,[y,\psi(f)] \cr
\psi\Big(ad^*_x\circ ad^*_yf\Big)  &=& ad_x\circ ad_y \circ \psi (f), \nonumber
\eeqn
that is
\beq\label{psi-equiv1}
\psi\circ ad^*_x\circ ad^*_y  = ad_x\circ ad_y \circ \psi,
\eeq
for any $x$ and $y$ in $\G$. 
On the second hand relations (\ref{psi-xi1}) and (\ref{psi-xi2}) give
\beqn
\xi\Big(\Big[x,[y,f]\Big]\Big)  &=& \Big(ad^*_{\alpha(x)}\circ ad^*_y +ad^*_x\circ ad^*_{\alpha(y)}
                                    + ad^*_x\circ ad^*_y\circ\xi \Big)(f) \cr
\xi\circ ad^*_x\circ ad^*_y (f) &=& \Big(ad^*_{\alpha(x)}\circ ad^*_y +ad^*_x\circ ad^*_{\alpha(y)}
                                    + ad^*_x\circ ad^*_y\circ\xi \Big)(f) \cr
\Big(\xi\circ ad^*_x\circ ad^*_y -ad^*_x\circ ad^*_y\circ\xi\Big) (f) &=& \Big(ad^*_{\alpha(x)}\circ ad^*_y +ad^*_x\circ ad^*_{\alpha(y)}\Big) (f)
                   \nonumber
\eeqn
It comes that
\beq\label{xi-rel1}
[\xi,ad^*_x\circ ad^*_y] = ad^*_{\alpha(x)}\circ ad^*_y +ad^*_x\circ ad^*_{\alpha(y)}
\eeq


Let us summarize in the

\bthm\label{theorem:characterization-prederivations}
A prederivation $p :T^*\G \to T^*\G$ is defined by :
\beq
p(x,f) = \Big(\alpha(x) + \psi(f),\beta(x) + \xi(f)\Big)
\eeq
for any element $(x,f)$ of $T^*\G$, where 
\bitem
\item $\alpha : \G \to \G$ is a prederivation of $\G$ and 
\item $\beta : \G \to \G^*$, $\psi : \G^* \to \G$ and $\xi : \G^* \to \G^*$ are linear maps 
satisfying the following four relations :
\eitem 
\beq\label{beta-relation}
\beta\big([x,[y,z]]\big) = -ad^*_{[y,z]}\big(\beta(x)\big)   + ad^*_x \Big(ad^*_y\big(\beta(z)\big) 
-ad^*_z\big(\beta(y)\big) \Big)
\eeq
\beq\label{psi-relation1}
\psi\circ ad^*_x\circ ad^*_y  = ad_x\circ ad_y \circ \psi,
\eeq
\beq\label{psi-relation2}
ad^*_x  \Big(ad^*_{\psi(f)}g -ad^*_{\psi(g)}f\Big) = 0,
\eeq
\beq\label{xi-relation}
[\xi,ad^*_x\circ ad^*_y] = ad^*_{\alpha(x)}\circ ad^*_y +ad^*_x\circ ad^*_{\alpha(y)},
\eeq
for every elements $x$ and $y$ of $\G$ and any elements $f$ and $g$ in $\G^*$.
\ethm

\subsection{A Structure theorem for the Lie group $\paut(T^*\G)$}\label{section:sructure-theorem-preautomorphism}

Let us introduce the following notations :
\benum
\item $\pder(T^*\G)$ stands for the space of prederivations of $T^*\G$ ;
\item $\pder(\G)$ represents for the space of prederivations of $\G$ ;
\item $\mathcal Q'$ is the space of linear maps $\beta : \G \to \G^*$ satisfying relation (\ref{beta-relation}) ;
\item $\mathcal E'$ is the space of linear maps $\xi : \G^* \to \G^*$ such that
$$
[\xi,ad^*_x\circ ad^*_y] = ad^*_{\alpha(x)}\circ ad^*_y +ad^*_x\circ ad^*_{\alpha(y)},
$$
for some prederivation $\alpha$ of $\G$ and any elements $x$ and $y$ of $\G$.
\item $\Psi'$ stands for the space of linear maps $\psi : \G^* \to \G$ satisfying 
(\ref{psi-relation1}) and (\ref{psi-relation2}).


\item $\G_0'$ stands for the space of maps $\phi_{\alpha,\xi} :T^*\G \to T^*\G$, 
$(x,f)\mapsto \big(\alpha(x),\xi(f)\big)$, where  $\alpha$ is in $\pder(\G)$, $\xi$ in $\mathcal E'$ 
and $[\xi,ad^*_x\circ ad^*_y] = ad^*_{\alpha(x)}\circ ad^*_y +ad^*_x\circ ad^*_{\alpha(y)}$, for all  $x,y$ in $\G$ ;
\item $\G_1':=\Psi' \oplus \mathcal Q'$ (direct sum of vector spaces).
\eenum


\blem\label{lemma:lie-algebra-E'}
The space $\mathcal E'$ is a Lie algebra. Precisely, if $\xi_1,\xi_2$ in $\mathcal E'$ satisfy
\beqn 
\big[\xi_1,ad^*_x\circ ad^*_y\big] &=& ad^*_{\alpha_1(x)}\circ ad^*_y +ad^*_x\circ ad^*_{\alpha_1(y)} \nonumber \\
\big[\xi_2,ad^*_x\circ ad^*_y\big] &=& ad^*_{\alpha_2(x)}\circ ad^*_y +ad^*_x\circ ad^*_{\alpha_2(y)} \nonumber 
\eeqn 
for all $x$, $y$ in $\G$ and some $\alpha_1$, $\alpha_2$ in $\pder(\G)$, then $[\xi_1,\xi_2]$ belongs  to
$\mathcal E'$ and satisfies
\beq 
\Big[[\xi_1,\xi_2] \,,\, ad^*_x\circ ad^*_y\Big]= ad^*_{[\alpha_1,\alpha_2](x)} \circ ad^*_y + ad^*_x\circ ad^*_{[\alpha_1,\alpha_2](y)}, \nonumber
\eeq 
for all elements $x,y$ of $\G$.
\elem

\begin{proof}
Consider $\xi_1$ and $\xi_2$ as in Lemma \ref{lemma:lie-algebra-E'}. 
We have,
 \beqn
\Big[[\xi_1,\xi_2],ad^*_x\circ ad^*_y\Big] &=& \big[\xi_1 \circ \xi_2 -\xi_2 \circ \xi_1 \,,\, ad^*_x\circ ad^*_y\big] \cr
                                           &=& \big[\xi_1 \circ \xi_2 \,,\, ad^*_x\circ ad^*_y\big]
                                               -\big[\xi_2 \circ \xi_1 \,,\, ad^*_x\circ ad^*_y\big]
\eeqn
\beqn
\big[\xi_1 \circ \xi_2 \,,\, ad^*_x\circ ad^*_y\big] &=& (\xi_1 \circ \xi_2) \circ ad^*_x \circ ad^*_y
                                                       - ad^*_x \circ ad^*_y \circ (\xi_1 \circ \xi_2 ) \cr
                                                     &=& \xi_1\Big(\xi_2 \circ ad^*_x \circ ad^*_y  \Big)
                                                         -\Big( ad^*_x \circ ad^*_y \circ \xi_1\Big) \circ \xi_2 \cr
                                                     &=& \Big( \xi_1 \circ ad^*_{\alpha_2(x)} \circ ad^*_y +  \xi_1 \circ ad^*_x
                                                         \circ ad^*_{\alpha_2(y)} + \xi_1 \circ ad^*_x \circ ad^*_y \circ \xi_2 \Big)\cr
                                                     & & -\Big(\xi_1 \circ ad^*_x \circ ad^*_y -ad^*_{\alpha_1(x)}
                                                          \circ ad^*_y - ad^*_x \circ ad^*_{\alpha_1(y)} \Big) \circ \xi_2
\eeqn
Hence,
\beqn
\big[\xi_1 \circ \xi_2 \,,\, ad^*_x\circ ad^*_y\big] &=& \xi_1 \circ ad^*_{\alpha_2(x)} \circ ad^*_y +  \xi_1 \circ ad^*_x
                                                         \circ ad^*_{\alpha_2(y)} \cr
                                                     & &   + ad^*_{\alpha_1(x)} \circ ad^*_y \circ \xi_2 +   ad^*_x \circ
                                                          ad^*_{\alpha_1(y)}  \circ \xi_2
\eeqn
We also have :
\beqn
\big[\xi_2 \circ \xi_1 \,,\, ad^*_x\circ ad^*_y\big] &=& \xi_2 \circ ad^*_{\alpha_1(x)} \circ ad^*_y +  \xi_2 \circ ad^*_x
                                                         \circ ad^*_{\alpha_1(y)} \cr
                                                     & &   + ad^*_{\alpha_2(x)} \circ ad^*_y \circ \xi_1 +   ad^*_x \circ
                                                          ad^*_{\alpha_2(y)}  \circ \xi_1
\eeqn
Now we have
\beqn
\Big[[\xi_1,\xi_2] \,,\, ad^*_x\circ ad^*_y\Big] &=& \Big(\xi_1 \circ ad^*_{\alpha_2(x)} \circ ad^*_y +  \xi_1 \circ ad^*_x
                                                         \circ ad^*_{\alpha_2(y)} \cr
                                                 & &   + ad^*_{\alpha_1(x)} \circ ad^*_y \circ \xi_2 +   ad^*_x \circ
                                                          ad^*_{\alpha_1(y)}  \circ \xi_2 \Big)\cr
                                                 & & - \Big( \xi_2 \circ ad^*_{\alpha_1(x)} \circ ad^*_y +  \xi_2 \circ ad^*_x
                                                         \circ ad^*_{\alpha_1(y)} \cr
                                                 & &   + ad^*_{\alpha_2(x)} \circ ad^*_y \circ \xi_1 +   ad^*_x \circ
                                                          ad^*_{\alpha_2(y)}  \circ \xi_1 \Big) \cr
                                                 &=&\, \;\; \Big(\xi_1 \circ ad^*_{\alpha_2(x)} \circ ad^*_y
                                                       -ad^*_{\alpha_2(x)} \circ ad^*_y \circ \xi_1 \Big) \cr
                                                 & &   + \Big(\xi_1 \circ ad^*_x \circ ad^*_{\alpha_2(y)}
                                                       -  ad^*_x \circ  ad^*_{\alpha_2(y)}  \circ \xi_1\Big) \cr
                                                 & & + \Big(ad^*_{\alpha_1(x)} \circ ad^*_y \circ \xi_2
                                                      -\xi_2 \circ ad^*_{\alpha_1(x)} \circ ad^*_y\Big)  \cr
                                                 & & +  \Big(ad^*_x \circ ad^*_{\alpha_1(y)}  \circ \xi_2
                                                     -\xi_2 \circ ad^*_x \circ ad^*_{\alpha_1(y)}   \Big) \cr
                                                 &=& \,\;\;\big[\xi_1,ad^*_{\alpha_2(x)} \circ ad^*_y\big]
                                                     +\big[\xi_1,ad^*_x \circ ad^*_{\alpha_2(y)}\big] \cr
                                                 & & -\big[\xi_2,ad^*_{\alpha_1(x)} \circ ad^*_y\big]
                                                     - \big[\xi_2,ad^*_x \circ ad^*_{\alpha_1(y)}\big] \cr
                                                 &=&\,\;\; \Big(ad^*_{\alpha_1 \circ \alpha_2(x)} \circ ad^*_y
                                                     + ad^*_{\alpha_2(x)} \circ ad^*_{\alpha_1(y)}\Big)\cr
                                                 & & + \Big(ad^*_{\alpha_1(x)}  \circ ad^*_{\alpha_2(y)}
                                                     + ad^*_x \circ ad^*_{\alpha_1\circ \alpha_2(y)} \Big) \cr
                                                 & & -\Big(ad^*_{\alpha_2 \circ \alpha_1(x)}  \circ ad^*_y
                                                     + ad^*_{\alpha_1(x)} \circ ad^*_{\alpha_2(y)} \Big) \cr
                                                 & &   -\Big(ad^*_{\alpha_2(x)}  \circ ad^*_{\alpha_1(y)}
                                                     + ad^*_x \circ ad^*_{\alpha_2\circ \alpha_1(y)} \Big). \nonumber
\eeqn
We then have
\beq
\Big[[\xi_1,\xi_2] \,,\, ad^*_x\circ ad^*_y\Big]= ad^*_{[\alpha_1,\alpha_2](x)} \circ ad^*_y
                                                  + ad^*_x\circ ad^*_{[\alpha_1,\alpha_2](y)}. \nonumber
\eeq
\end{proof}

\blem
The space $\G_0'$ is a Lie subalgebra of the Lie algebra $\pder(T^*\G)$.
\elem

\begin{proof}

Let $\phi_{\alpha_1,\xi_1}$ and $\phi_{\alpha_2,\xi_2}$ be two elements of $\G_0'$.

\beqn
\Big( \phi_{\alpha_1,\xi_1} \circ \phi_{\alpha_2,\xi_2}\Big) (x,f) &=& \phi_{\alpha_1,\xi_1}\Big(\alpha_2(x),\xi_2(f)\Big) \cr
                                                                   &=&\Big(\alpha_1 \circ \alpha_2(x),\xi_1 \circ \xi_2(f)\Big) \nonumber
\eeqn
We then  have
\beq
\Big( \phi_{\alpha_1,\xi_1} \circ \phi_{\alpha_2,\xi_2}\Big) = \Big( \alpha_1 \circ \alpha_2,\xi_1 \circ \xi_2\Big). \nonumber
\eeq
By the same way, we have
\beq
\Big( \phi_{\alpha_2,\xi_2} \circ \phi_{\alpha_1,\xi_1}\Big) = \Big( \alpha_2 \circ \alpha_1,\xi_2 \circ \xi_1\Big). \nonumber
\eeq
Hence, 
\beq
[\phi_{\alpha_1,\xi_1},\phi_{\alpha_2,\xi_2}] = \Big( [\alpha_1,\alpha_2],[\xi_1,\xi_2]\Big) \in  \G_0'. \nonumber
\eeq

\end{proof}


\blem
$[\G_0',\Psi'] \subset \Psi'$ and $[\G_0',\mathcal Q'] \subset \mathcal Q'$. Hence, $[\G_0',\G_1'] \subset \G_1'$.
\elem

\begin{proof}

Let $\phi_{\alpha,\xi}$ and $\phi_{\psi,0}$ be elements of $\G_0'$ and $\Psi'$ respectively. We have
\beqn
\phi_{\alpha,\xi} \circ \phi_{\psi,0}(x,f) &=& \phi_{\alpha,\xi}\Big(\psi(f),0\Big) \cr
                                               &=& \Big(\alpha \circ \psi(f),0 \Big). \nonumber
\eeqn
\beqn
\phi_{\psi,0} \circ \phi_{\alpha,\xi}(x,f) &=& \phi_{\psi,0}\Big(\alpha(x),\xi(f) \Big) \cr
                                               &=& \Big(\psi \circ \xi(f),0 \Big). \nonumber
\eeqn
Then,
\beq
[\phi_{\alpha,\xi},\phi_{\psi,\beta}](x,f)=\Big((\alpha\circ\psi-\psi\circ\xi)(f),0\Big). \nonumber
\eeq
Let us show that $(\alpha\circ\psi - \psi\circ\xi)$ belongs to $\Psi'$. 
For any elements $x$ and $y$ in $\G$, we have
\beqn
(\alpha \circ \psi) \circ ad^*_x\circ ad^*_y &=& \alpha \circ (\psi \circ ad^*_x\circ ad^*_y) \cr
                                            &=& \alpha \circ (ad_x \circ ad_y \circ \psi)\cr
                                            &=& (\alpha \circ ad_x \circ ad_y )\circ \psi. \nonumber
\eeqn
Let $z$ be an element of $\G$. We have
\beqn
(\alpha \circ ad_x\circ ad_y)(z) &=& \alpha([x,[y,z]]) \cr
                                 &=& \Big[\alpha(x),[y,z]\Big] + \Big[x,[\alpha(y),z]\Big] + \Big[x,[y,\alpha(z)]\Big]\cr
                                 &=& \Big(ad_{\alpha(x)}\circ ad_y + ad_x \circ ad_{\alpha(y)} + ad_x \circ ad_y \circ \alpha \Big)(z).\nonumber
\eeqn
Then
\beq
\alpha \circ ad_x\circ ad_y = ad_{\alpha(x)}\circ ad_y + ad_x \circ ad_{\alpha(y)} + ad_x \circ ad_y \circ \alpha. \nonumber
\eeq
It comes that :
\beq \label{composition-alpha-psi}
(\alpha \circ \psi) \circ ad^*_x\circ ad^*_y = ad_{\alpha(x)}\circ ad_y\circ \psi + ad_x \circ ad_{\alpha(y)}\circ \psi
+ ad_x \circ ad_y \circ \alpha \circ \psi.
\eeq
Now, what about $(\psi \circ \xi) \circ ad^*_x\circ ad^*_y$ ?
\beqn
(\psi \circ \xi) \circ ad^*_x\circ ad^*_y &=& \psi \circ (\xi \circ ad^*_x\circ ad^*_y) \cr
                                          &=& \psi \circ \Big( [\xi,ad^*_x\circ ad^*_y] + ad^*_x\circ ad^*_y \circ \xi\Big) \cr
                                          &=& \psi \circ \Big( ad^*_{\alpha(x)} \circ ad^*_y + ad^*_x \circ ad^*_{\alpha(y)}
                                              + ad^*_x \circ ad^*_y \circ \xi\Big).\nonumber
\eeqn
Hence,
\beq\label{composition-psi-xi}
(\psi \circ \xi) \circ ad^*_x\circ ad^*_y = \psi \circ ad^*_{\alpha(x)} \circ ad^*_y  + \psi \circ  ad^*_x \circ ad^*_{\alpha(y)}
                                              + \psi \circ  ad^*_x \circ ad^*_y \circ \xi
\eeq
From (\ref{composition-alpha-psi}) and (\ref{composition-psi-xi}), we have :
 \beqn
(\alpha \circ \psi-\psi \circ \xi) \circ ad^*_x\circ ad^*_y &=&\,\;\; (ad_{\alpha(x)}\circ ad_y\circ \psi + ad_x \circ ad_{\alpha(y)}
                                                                \circ \psi \!+\! ad_x \circ ad_y \circ \alpha \circ \psi ) \cr
                                                            & & -(\psi \circ ad^*_{\alpha(x)} \circ ad^*_y  + \psi \circ  ad^*_x
                                                                  \circ ad^*_{\alpha(y)} + \psi \circ  ad^*_x \circ ad^*_y \circ \xi).\nonumber
 \eeqn
It comes that  $(\alpha\circ\psi - \psi\circ\xi)$ satisfies (\ref{psi-relation1}) as it verifies 
\beq\label{alpha-psi-psi-xi}
(\alpha \circ \psi-\psi \circ \xi) \circ ad^*_x\circ ad^*_y = ad_x\circ ad_y \circ (\alpha \circ \psi-\psi \circ \xi). \nonumber
\eeq
Let now  $f$ and $g$ be elements of $\G^*$.
\beqn
\Gamma &:=& ad^*_{(\alpha \circ \psi-\psi \circ \xi)(f)}g - ad^*_{(\alpha \circ \psi-\psi \circ \xi)(g)}f \cr
       &=& ad^*_{\alpha \circ \psi(f)}g-ad^*_{\psi \circ \xi(f)}g -ad^*_{\alpha \circ \psi(g)}f +ad^*_{\psi \circ \xi(g)}f \nonumber
\eeqn
For any $x$ in $\G$, one has the following
\beqn
ad^*_x \circ ad^*_{\alpha \circ \psi(f)} &=& [\xi,ad^*_x\circ ad^*_{\psi(f)}] - ad^*_{\alpha(x)} \circ ad^*_{\psi(f)} \nonumber\\
ad^*_x \circ ad^*_{\alpha \circ \psi(g)} &=& [\xi,ad^*_x\circ ad^*_{\psi(g)}] - ad^*_{\alpha(x)} \circ ad^*_{\psi(g)} \nonumber
\eeqn
Then
\beqn
ad^*_x \circ \Gamma &=& ad^*_x \circ \Big( ad^*_{(\alpha \circ \psi-\psi \circ \xi)(f)}g
                        - ad^*_{(\alpha \circ \psi-\psi \circ \xi)(g)}f \ \Big) \cr
                    &=& \xi \circ ad^*_x\circ ad^*_{\psi(f)}(g) -  ad^*_x\circ ad^*_{\psi(f)} \circ \xi(g)
                         - ad^*_{\alpha(x)} \circ ad^*_{\psi(f)}(g) \cr
                    & & - \xi\circ ad^*_x\circ ad^*_{\psi(g)}(f) + ad^*_x\circ ad^*_{\psi(g)} \circ \xi(f)
                         + ad^*_{\alpha(x)} \circ ad^*_{\psi(g)}(f) \cr
                    & & -ad^*_x\circ ad^*_{\psi \circ \xi(f)}(g) + ad^*_x\circ ad^*_{\psi \circ \xi(g)}(f) \cr
                    &=& \underbrace{\xi  \Big(ad^*_x\circ ad^*_{\psi(f)}(g)- \circ ad^*_x\circ ad^*_{\psi(g)}(f)\Big)}_{=0,
                      \text{ \tiny because of (\ref{psi-relation2}) }}
                        + \underbrace{ad^*_{\alpha(x)} \Big(ad^*_{\psi(g)}(f) -ad^*_{\psi(f)}(g)  \Big)}_{=0,
                        \text{ \tiny because of (\ref{psi-relation2}) }} \cr
                    & & + \underbrace{ad^*_x\circ ad^*_{\psi \circ \xi(g)}f- ad^*_x\circ ad^*_{\psi(\xi(g))}f}_{=0} +
                         \underbrace{ad^*_x \circ ad^*_{\psi(\xi(f))}g -ad^*_x\circ ad^*_{\psi \circ \xi(f)}g}_{=0}.\nonumber
\eeqn
Hence,
\beq
ad^*_x  \Big( ad^*_{(\alpha \circ \psi-\psi \circ \xi)(f)}g - ad^*_{(\alpha \circ \psi-\psi \circ \xi)(g)}f \ \Big)  = 0,
\eeq
for any $f,g$ in $\G^*$ and any $x$ in $\G$. We then have shown that $(\alpha \circ \psi-\psi \circ \xi)$ 
belongs to $\Psi'$. Hence, $[\G_0',\Psi'] \subset \Psi'$.

\vskip 0.3cm

Now we are going to show that $[\G_0',\mathcal Q'] \subset \mathcal Q'$.
For this goal, let $\phi_{\alpha,\xi}$ and $\phi_{0,\beta}$ be elements of $\G_0'$ and $\mathcal Q'$ respectively.
\beqn
 \phi_{\alpha,\xi} \circ \phi_{0,\beta}(x,f)    &=& \phi_{\alpha,\xi}\Big(0,\beta(x)\Big) \cr
                                                &=& \Big(0,\xi \circ \beta(x) \Big)
\eeqn
\beqn
\phi_{0,\beta} \circ \phi_{\alpha,\xi}(x,f) &=& \phi_{0,\beta}\Big(\alpha(x),\xi(f) \Big) \cr
                                               &=& \Big(0,\beta \circ \alpha (x) \Big)
\eeqn
Then,
\beq
[\phi_{\alpha,\xi},\phi_{0,\beta}](x,f)=\Big(0,(\xi\circ\beta-\beta\circ\alpha )(x)\Big)
\eeq
Does $(\xi \circ \beta - \beta \circ \alpha)$ satisfies relation (\ref{beta-relation}) ?
\beqn
(\xi \circ \beta)[x,[y,z]] &=& \xi \Big(-ad^*_{[y,z]}\big(\beta(x)\big) -ad^*_x \circ ad^*_z\big(\beta(y)\big)
                                + ad^*_x \circ ad^*_y\big(\beta(z)\big) \Big) \cr
                           &=& -\xi \circ ad^*_{[y,z]}\big(\beta(x) \big) - \xi \circ ad^*_x \circ ad^*_z\big(\beta(y)\big)
                               + \xi \circ ad^*_x \circ ad^*_y\big(\beta(z)\big) \cr
                           &=& -\xi \circ ad^*_y \circ ad^*_z\big(\beta(x)\big) + \xi \circ ad^*_z \circ ad^*_y\big(\beta(x)\big) \cr
                           & & -\xi \circ ad^*_x \circ ad^*_z\big(\beta(y)\big)  + \xi \circ ad^*_x \circ ad^*_y\big(\beta(z)\big) \cr
                           &=& -[\xi,ad^*_y \circ ad^*_z](\beta(x)) - ad^*_y \circ ad^*_z \circ \xi(\beta(x)) \cr
                           & & +[\xi,ad^*_z \circ ad^*_y](\beta(x)) + ad^*_z \circ ad^*_y \circ \xi(\beta(x)) \cr
                           & & -[\xi,ad^*_x \circ ad^*_z](\beta(y)) - ad^*_x \circ ad^*_z \circ \xi(\beta(y)) \cr
                           & & +[\xi,ad^*_x \circ ad^*_y](\beta(z)) + ad^*_x \circ ad^*_y \circ \xi(\beta(z)) \cr
                           &=& -ad^*_{\alpha(y)}\circ ad^*_z(\beta(x))-ad^*_y\circ ad^*_{\alpha(z)}(\beta(x))-ad^*_y\circ ad^*_z\circ\xi(\beta(x)) \cr
                           & & +ad^*_{\alpha(z)}\circ ad^*_y(\beta(x))+ad^*_z\circ ad^*_{\alpha(y)}(\beta(x))+ad^*_z\circ ad^*_y\circ\xi(\beta(x)) \cr
                           & & -ad^*_{\alpha(x)}\circ ad^*_z(\beta(y))-ad^*_x\circ ad^*_{\alpha(z)}(\beta(y))-ad^*_x\circ ad^*_z\circ\xi(\beta(y)) \cr
                           & & \!+ad^*_{\alpha(x)}\circ ad^*_y(\beta(z))\!+\!ad^*_x\circ ad^*_{\alpha(y)}(\beta(z))\!+\!ad^*_x\circ ad^*_y\circ\xi(\beta(z))
\eeqn
\beqn
(\beta \circ \alpha)[x,[y,z]] &=& \beta\Big( [\alpha(x),[y,z]] + [x,[\alpha(y),z]] + [x,[y,\alpha(z)]] \Big) \cr
                              &=& -ad^*_{[y,z]}\big(\beta(\alpha(x))\big) -ad^*_{\alpha(x)}\circ ad^*_z\big(\beta(y)\big)
                                  + ad^*_{\alpha(x)} \circ ad^*_y\big(\beta(z)\big) \cr
                              & & -ad^*_{[\alpha(y),z]}\big(\beta(x)\big) - ad^*_x \circ ad^*_z\big(\beta(\alpha(y)) \big)
                                  + ad^*_x \circ ad^*_{\alpha(y)}\big(\beta(z) \big) \cr
                              & & -ad^*_{[y,\alpha(z)]}\big(\beta(x)\big)\!-\! ad^*_x \circ ad^*_{\alpha(z)}\big(\beta(y)\big)
                                   \!+\! ad^*_x \circ ad^*_y\big(\beta(\alpha(z)) \big)
\eeqn
We then have
\beqn
(\xi \circ \beta - \beta \circ \alpha)[x,[y,z]] &=&\!\! -ad^*_{[y,z]}\Big((\xi \circ \beta)(x)\Big) \!-\! ad^*_{[\alpha(y),z]}(\beta(x))
                                                    \!-\! ad^*_{[y,\alpha(z)]}(\beta(x)) \cr
                                                & & \!\!- ad^*_x \circ ad^*_z\Big((\xi \circ \beta)(y)\Big)
                                                    + ad^*_x \circ ad^*_y\Big((\xi \circ \beta)(z)\Big) \cr
                                                & & \!\!+\! ad^*_{[y,z]}\Big((\beta \circ \alpha)(x)\Big) \!+\! ad^*_{[\alpha(y),z]}(\beta(x))
                                                    \!+\! ad^*_x \circ ad^*_z\Big((\beta \circ \alpha)(y)\Big) \cr
                                                & & + ad^*_{[y,\alpha(z)]}(\beta(x)) -ad^*_x \circ ad^*_y\Big((\beta \circ \alpha)(z)\Big)\cr
                                                &=&\!\! -ad^*_{[y,z]}\Big((\xi \circ \beta\! - \!\beta \circ \alpha)(x) \Big) \!-\! ad^*_x
                                                     \circ ad^*_z\Big((\xi \circ \beta \!-\! \beta \circ \alpha)(y) \Big)  \cr
                                                & & + ad^*_x \circ ad^*_y\Big((\xi \circ \beta - \beta \circ \alpha)(z) \Big)
\eeqn
That is $(\xi \circ \beta - \beta \circ \alpha)$ satisfies relation (\ref{beta-relation}) and then 
$[\G_0',\mathcal Q']\subset \mathcal Q'$. It is now clear that  $[\G_0',\G_1'] \subset \G_1'$.

\end{proof}


We summarize the Lemmas above in the 

\bthm\label{structuretheorem1}
Let  $\G$ be any finite-dimensional Lie algebra. 
Then the Lie algebra of prederivations of $T^*\G$ decomposes as follows :
 $\pder(T^*\G) =\G_0'\oplus\G_1'$, where $\G_0'$ is a reductive subalgebra of $\pder(T^*\G)$, 
that is $$[\G_0',\G_0']\subset \G_0' \text{ and } [\G_0',\G_1']\subset \G_1'.$$
\ethm 

\brmq
$\pder(T^*\G)$ is not a symmetric space as is $\der(T^*\G)$ (see Theorem \ref{structuretheorem}) since 
$[\G_1',\G_1']$ is not a subset of $\G_0'$. Precisely, let $\phi_{\psi_1,\beta_1}$ and $\phi_{\psi_2,\beta_2}$ be two elements of $\G_1'$.
Then,
\benum
\item $[\phi_{\psi_1,\beta_1}, \phi_{\psi_2,\beta_2}] = \Big((\psi_1 \circ \beta_2 - \psi_2 \circ \beta_1) \,,
\, (\beta_1 \circ \psi_2 - \beta_2 \circ \psi_1) \Big)$ ;
\item $[\phi_{\psi_1,\beta_1}, \phi_{\psi_2,\beta_2}]$ do not belong to $\G_0'$ even  $(\psi_1 \circ \beta_2\!-\!\psi_2 \circ \beta_1)$
is a prederivation of $\G$.
\item $(\beta_1 \circ \psi_2-\beta_2\circ \psi_1)$ is not linked to $(\psi_1 \circ \beta_2 - \psi_2 \circ \beta_1)$ by
(\ref{xi-relation}).
\eenum
\ermq
Indeed, let $\phi_{\psi_1,\beta_1}$ and $\phi_{\psi_2,\beta_2}$ be two elements of $\G_1'$.
\beqn
(\phi_{\psi_1,\beta_1} \circ \phi_{\psi_2,\beta_2}) (x,f) &=& \phi_{\psi_1,\beta_1} \Big(\psi_2(f),\beta_2(x) \Big) \cr
                                                          &=& \Big(\psi_1 \circ \beta_2(x),\beta_1 \circ \psi_2(f) \Big)
\eeqn
By the same way
\beq
(\phi_{\psi_2,\beta_2} \circ \phi_{\psi_1,\beta_1}) (x,f) =\Big(\psi_2 \circ \beta_1(x),\beta_2 \circ \psi_1(f) \Big)
\eeq
It comes that
\beq
[\phi_{\psi_1,\beta_1}, \phi_{\psi_2,\beta_2}](x,f) = \Big((\psi_1 \circ \beta_2 - \psi_2 \circ \beta_1)(x) \,,
\, (\beta_1 \circ \psi_2 - \beta_2 \circ \psi_1)(f) \Big)
\eeq
Let us see if $(\psi_1 \circ \beta_2 - \psi_2 \circ \beta_1)$ is a prederivation of $\G$.
\beqn
(\psi_1 \circ \beta_2)[x,[y,z]] &=& \psi_1 \Big( -ad^*_{[y,z]}\big(\beta_2(x)\big)-ad^*_x\circ ad^*_z\big(\beta_2(y)\big)
                                    + ad^*_x\circ ad^*_y\big(\beta_2(z)\big)\Big) \cr
                                &=&\!\!\! -\psi_1 \circ ad^*_{[y,z]}\big(\beta_2(x)\big)\!-\!\psi_1\circ ad^*_x\circ ad^*_z\big(\beta_2(y)\big)
                                      \! +\!\psi_1\circ  ad^*_x\circ ad^*_y\big(\beta_2(z)\big) \cr
                                &=& -ad_y \circ ad_z \circ \psi_1\big(\beta_2(x)\big) + ad_z \circ ad_y \circ \psi_1\big(\beta_2(x)\big) \cr
                                & & -ad_x \circ ad_z \circ \psi_1\big(\beta_2(y)\big) + ad_x \circ ad_y \circ \psi_1\big(\beta_2(z)\big) \cr
                                &=& -\Big[y,[z,\psi_1 \circ \beta_2(x)]\Big] + \Big[z,[y,\psi_1 \circ \beta_2(x)]\Big] \cr
                                & & -\Big[x,[z,\psi_1 \circ \beta_2(y)]\Big] + \Big[x,[y,\psi_1 \circ \beta_2(z)]\Big]
\eeqn
We also have :
\beqn
(\psi_2 \circ \beta_1)[x,[y,z]] &=& -\Big[y,[z,\psi_2 \circ \beta_1(x)]\Big] + \Big[z,[y,\psi_2 \circ \beta_1(x)]\Big] \cr
                                & & -\Big[x,[z,\psi_2 \circ \beta_1(y)]\Big] +\Big[x,[y,\psi_2 \circ \beta_1(z)]\Big]
\eeqn
Now we have :
\beqn
(\psi_1 \circ \beta_2-\psi_2 \circ \beta_1)[x,[y,z]]\!\! &=& \!\!\!-\Big[y,[z,\psi_1 \circ \beta_2(x)]\Big] \!\!+\!\! \Big[z,[y,\psi_1 \circ \beta_2(x)]\Big] \cr
                                                     & & -\Big[x,[z,\psi_1 \circ \beta_2(y)]\Big]\!\! +\!\! \Big[x,[y,\psi_1 \circ \beta_2(z)]\Big] \cr
                                                     & & \!\!+\!\!\Big[y,[z,\psi_2 \circ \beta_1(x)]\Big] \!\!-\!\! \Big[z,[y,\psi_2 \circ \beta_1(x)]\Big] \cr
                                                     & & +\Big[x,[z,\psi_2 \circ \beta_1(y)]\Big] - \Big[x,[y,\psi_2 \circ \beta_1(z)]\Big] \cr
                                                     &=&\!\!\!\!\! -\!\!\Big[y,[z,(\psi_1 \circ \beta_2 \!\!-\!\! \psi_2 \circ \beta_1)(x)]\Big]
                                                        \!\!+\!\!\Big[z,[y,(\psi_1 \circ \beta_2 \!\!-\!\! \psi_2 \circ \beta_1)(x)]\Big]            \cr
                                                     & & \!\!\!\!\! -\!\!\Big[x,[z,(\psi_1 \circ \beta_2 \!\!-\!\! \psi_2 \circ \beta_1)(y)]\Big]
                                                        \!\! +\!\!\Big[x,[y,(\psi_1 \circ \beta_2 \!\!-\!\! \psi_2 \circ \beta_1)(z)]\Big]           \cr
                                                     &=& \!\!\! \Big[(\psi_1 \circ \beta_2 \!\!-\!\! \psi_2 \circ \beta_1)(x),[y,z]\Big]
                                                         \!\!+\!\!\Big[x,[(\psi_1 \circ \beta_2 \!\!-\!\! \psi_2 \circ \beta_1)(y),z]\Big]           \cr
                                                     & & + \Big[x,[y,(\psi_1 \circ \beta_2 \!-\! \psi_2 \circ \beta_1)(z)]\Big]
\eeqn
Then $(\psi_1 \circ \beta_2-\psi_2 \circ \beta_1)$ is a prederivation of $\G$.

\vskip 0.3cm

We are now going to verify if $(\beta_1 \circ \psi_2 - \beta_2 \circ \psi_1)$ satisfies relation (\ref{xi-relation}). 
Let $x$ and $y$ be two elements of $\G$ and $f$ be in $\G^*$.
\beqn
[\beta_1 \circ \psi_2 - \beta_2 \circ \psi_1, ad^*_x \circ ad^*_y](f) &=&  [\beta_1 \circ \psi_2,ad^*_x \circ ad^*_y](f)
                                                                          -[\beta_2 \circ \psi_1,ad^*_x \circ ad^*_y](f) \cr
                                                                      &=& \Big(\beta_1 \circ \psi_2 \circ ad^*_x \circ ad^*_y
                                                                          -ad^*_x \circ ad^*_y\circ \beta_1 \circ \psi_2\Big)(f)\cr
                                                                      & & \!\!\!\! -\Big(\beta_1 \circ \psi_2 \circ ad^*_x \circ ad^*_y
                                                                          -ad^*_x \circ ad^*_y\circ \beta_1 \circ \psi_2\Big)(f) \cr
                                                                      &=& \Big(\beta_1 \circ ad_x \circ ad_y \circ \psi_2
                                                                          -ad^*_x \circ ad^*_y\circ \beta_1 \circ \psi_2\Big)(f)\cr
                                                                      & & \!\!\!\! -\Big(\beta_2 \circ ad_x \circ ad_y \circ \psi_1
                                                                          -ad^*_x \circ ad^*_y\circ \beta_2 \circ \psi_1\Big)(f) \cr
                                                                      &=& \beta_1\Big([x,[y,\psi_2(f)]]\Big)
                                                                          -ad^*_x \circ ad^*_y\Big(\beta_1\big(\psi_2(f)\big)\Big)  \cr
                                                                      & &\!\!\!\! -\beta_2\Big([x,[y,\psi_1(f)]]\Big)
                                                                          +ad^*_x \circ ad^*_y\Big(\beta_2\big(\psi_1(f)\big)\Big)  \cr
                                                                      &=& -ad^*_{[y,\psi_2(f)]}\Big(\beta_1(x)\Big)
                                                                         -ad^*_x \circ ad^*_{\psi_2(f)}\Big(\beta_1(y)\Big) \cr
                                                                      & & +ad^*_x \circ ad^*_y\Big(\beta_1\big(\psi_2(f)\big)\Big)
                                                                          -ad^*_x \circ ad^*_y\Big(\beta_1\big(\psi_2(f)\big)\Big) \cr
                                                                      & & + ad^*_{[y,\psi_1(f)]}\Big(\beta_2(x)\Big)
                                                                         +ad^*_x \circ ad^*_{\psi_1(f)}\Big(\beta_2(y)\Big) \cr
                                                                      & & -ad^*_x \circ ad^*_y\Big(\beta_2\big(\psi_1(f)\big)\Big)
                                                                          +ad^*_x \circ ad^*_y\Big(\beta_2\big(\psi_1(f)\big)\Big) \cr
                                                                      &=& -ad^*_y\circ ad^*_{\psi_2(f)}\Big(\beta_1(x)\Big)
                                                                           + ad^*_{\psi_2(f)}\circ ad^*_y\Big(\beta_1(x)\Big) \cr
                                                                      & & -ad^*_x\circ ad^*_{\psi_2(f)}\Big(\beta_1(y)\Big)
                                                                          + ad^*_y\circ ad^*_{\psi_1(f)}\Big(\beta_2(x)\Big)  \cr
                                                                      & & -ad^*_{\psi_1(f)}\circ ad^*_y\Big(\beta_2(x)\Big)
                                                                          + ad^*_x\circ ad^*_{\psi_1(f)}\Big(\beta_2(y)\Big)   \cr
                                                                      &=& -ad^*_y\circ ad^*_{\psi_2(\beta_1(x))}f
                                                                           + ad^*_{\psi_2(f)}\circ ad^*_y\Big(\beta_1(x)\Big) \cr
                                                                      & & -ad^*_x\circ ad^*_{\psi_2(\beta_1(y))}f
                                                                          + ad^*_y\circ ad^*_{\psi_1(\beta_2(x))}f  \cr
                                                                      & & -ad^*_{\psi_1(f)}\circ ad^*_y\Big(\beta_2(x)\Big)
                                                                          + ad^*_x\circ ad^*_{\psi_1(\beta_2(y))} f
\eeqn
Hence,
\beqn
[\beta_1 \circ \psi_2 - \beta_2 \circ \psi_1, ad^*_x \circ ad^*_y](f) &=& ad^*_y\circ ad^*_{(\psi_1\circ\beta_2-\psi_2\circ\beta_1)(x)}f
                                                                          +ad^*_x\circ ad^*_{(\psi_1\circ\beta_2-\psi_2\circ\beta_1)(y)}f \cr
                                                                      & & +ad^*_{\psi_2(f)}\circ ad^*_y\Big(\beta_1(x)\Big)
                                                                          -ad^*_{ \psi_1(f)}\circ ad^*_y\Big(\beta_2(x)\Big)
\eeqn
The map $(\beta_1 \circ \psi_2 - \beta_2 \circ \psi_1)$ does not satisfy the relation (\ref{xi-relation})
(with the prederivation $(\psi_1 \circ \beta_2-\psi_2 \circ \beta_1)$) because of the term
$ad^*_{\psi_2(f)}\circ ad^*_y\Big(\beta_1(x)\Big) -ad^*_{ \psi_1(f)}\circ ad^*_y\Big(\beta_2(x)\Big)$.

\subsection{Maps $\xi : \G^* \to \G^*$}

\bpro\label{proposition-xi'-caracterization}
Let $\G$ be a Lie algebra and let $\alpha$ be a prederivation of $\G$. A linear map $\xi' : \G \to \G$ satisfies
$[\xi',ad_x \circ ad_y] = -\big(ad_{\alpha(x)} \circ ad_y + ad_x\circ ad_{\alpha(y)}\big)$, for any elements $x$ and $y$ of $\G$ if and only if there exists a linear
map $j:\G \to \G$ satisfying
\beq
[j,ad_x\circ ad_y] = 0, 
\eeq
for any $x$, $y$ and $z$ in $\G$, such that $\xi' = j-\alpha$.
\epro

\begin{proof}
 Let $\xi'$ and $\alpha$ be as in the Proposition \ref{proposition-xi'-caracterization}.
 For any $x,y,z$ in $\G$, one has
\beqn
\xi'\circ ad_x \circ ad_y(z) -ad_x\circ ad_y\circ \xi'(z) &=& -ad_{\alpha(x)} \circ ad_y(z) - ad_x\circ ad_{\alpha(y)}(z) \cr              
                                                          &=&  -\big[\alpha(x),[y,z]\big] - \big[x,[\alpha(y),z]\big]  \cr 
                                                          &=& -\alpha\big([x,[y,z]]\big) +[x,[y,\alpha(z)]] \cr
                                                          &=& -\alpha \circ ad_x \circ ad_y (z) + ad_x \circ ad_y \circ \alpha(z). \nonumber
\eeqn
We then have
\beq \label{xi'-comute-adjoint-derived}
\Big[(\xi'+\alpha)\, ,\,ad_x \circ ad_y\Big] = 0,
\eeq
for any elements $x,y$ of $\G$. To finish the proof we just take $j=\xi'+\alpha$.

\end{proof}

Let us note by $\mathcal J'$ the space of linear maps $j:\G \to \G$ which satisfy 

\beq\label{rel-pre-biinvariant-tensor}
[j,ad_x\circ ad_y] = 0,
\eeq
for every $x,y,z$ in $\G$.
\bpro\label{proposition-xi-caracterization}
Let $\G$ be a non Abelian Lie algebra with dual space $\G^*$. A linear map $\xi : \G^* \to \G^*$ satisfies
$[\xi,ad^*_x \circ ad_y^*]=ad^*_{\alpha(x)} \circ ad^*_y+ ad^*_x \circ ad^*_{\alpha(y)}$, for some prederivation $\alpha$ of $\G$ and
every $x,y$ in $\G$ if and only if its transpose $\xi^t: \G \to \G$ is of the form $\xi^t = j-\alpha$, where
$j$ is in $\mathcal J'$.
\epro

\begin{proof}
Consider a linear map $\xi$ satisfying the hypothesis of Proposition \ref{proposition-xi-caracterization}. That is, for every
$x,y$ in $\G$, $[\xi,ad^*_x \circ ad_y^*]=ad^*_{\alpha(x)} \circ ad^*_y+ ad^*_x \circ ad^*_{\alpha(y)}$, for some $\alpha$ in $\pder(\G)$.
Taking transposes of the two sides one has
\beqn
[\xi,ad^*_x \circ ad_y^*]^t                      &=& \Big(ad^*_{\alpha(x)} \circ ad^*_y+ ad^*_x \circ ad^*_{\alpha(y)}\Big)^t \cr
-\Big[\xi^t,\big(ad^*_x \circ ad_y^*\big)^t\Big] &=& \Big(ad^*_{\alpha(x)} \circ ad^*_y \Big)^t + \Big( ad^*_x \circ ad^*_{\alpha(y)}\Big)^t \cr
-\Big[\xi^t,ad_y \circ ad_x\Big]                 &=& ad_y \circ ad_{\alpha(x)} +ad_{\alpha(y)} \circ ad_x  \cr
\Big[\xi^t,ad_y \circ ad_x\Big]                  &=& -\Big(ad_y \circ ad_{\alpha(x)} +ad_{\alpha(y)} \circ ad_x\Big).
\eeqn
From Proposition \ref{proposition-xi'-caracterization} we conclude that $\xi^t = j-\alpha$,
 where $j$ is in $\mathcal J'$.
\end{proof}

As a consequence, we have the following corollary of Theorem \ref{theorem:characterization-prederivations}.
\bcor\label{corollary-theorem:characterization-prederivations}
A prederivation $p :T^*\G \to T^*\G$ is defined by :
\beq
p(x,f) = \Big(\alpha(x) + \psi(f),\beta(x) + f \circ (j-\alpha)\Big)
\eeq
for any element $(x,f)$ of $T^*\G$, where $\alpha : \G \to \G$ is a prederivation of  $\G$,  $j : \G \to \G$ is in $\mathcal J'$, 
 $\beta : \G \to \G^*$ and $\psi : \G^* \to \G$  are linear maps satisfying relations (\ref{beta-relation}), 
(\ref{psi-relation1}) and  (\ref{psi-relation2}).  

\ecor

\section{Orthogonal Lie algebras}\label{section:orthogonal-lie-algebra-chap2}
\subsection{Maps $\alpha$, $\beta$, $\psi$, $\xi$ in orthogonal Lie algebras }

All over this section we consider  an orthogonal Lie algebra $(\G,\mu)$. Let $\theta : \G \to \G^*$
still stand for the isomorphism defined by $\langle \theta(x),y\rangle:=\mu(x,y)$, for all $x,y$ in $\G$.

\blem\label{isomorphism-Q'-Pder}
The map $\beta \mapsto \alpha_\beta:=\theta^{-1}\circ \beta$ is an isomorphism between the space $\mathcal Q'$ 
of linear maps $\beta$ satisfying relation (\ref{beta-relation}) and the space $\pder(\G)$ of prederivations of $\G$.
\elem 
\begin{proof}
 Let $\beta$ be an element of $\mathcal Q'$. For any $x,y,z$ in $\G$, we have
\beqn
\alpha_\beta\big(\big[x,[y,z]\big]\big)  &:=& \theta^{-1}\circ \beta\big(\big[x,[y,z]\big]\big)  \cr
                                          &=& \theta^{-1}\Big( -ad^*_{[y,z]}\big(\beta(x)\big)   +
                                              ad^*_x \circ ad^*_y\big(\beta(z)\big) - ad^*_x\circ ad^*_z\big(\beta(y)\big) \Big) \cr
                                          &=&\!\!\!\! -ad_{[y,z]} \circ \theta^{-1} \circ \beta(x) \!+\! ad_x \circ ad_y \circ \theta^{-1}
                                              \circ \beta(z) \!-\! ad_x\circ ad_z\circ\theta^{-1}\circ \beta(y) \cr
                                          &=&-\big[[y,z],\alpha_\beta(x)\big] + \big[x,[y,\alpha_\beta(z)]\big]
                                             - \big[x,[z,\alpha_\beta(z)]\big]
\eeqn
Then $\alpha_\beta$ belongs to $\pder(\G)$. Conversely, let $\alpha$ be a prederivation of $\G$ and set
$\beta_\alpha:=\theta\circ \alpha$. For any $x,y,z$ in $\G$ one has
\beqn
\beta_\alpha\big(\big[x,[y,z]\big]\big) &:=& \theta\Big( \big[\alpha(x),[y,z]\big] + \big[x,[\alpha(y),z]\big]
                                          + \big[x,[y,\alpha(z)]\big]\Big) \cr
                                        &=&-\theta\circ ad_{[x,y]}\big(\alpha(x)\big) + \theta\circ ad_x\big([\alpha(y),z]\big)
                                           + \theta \circ ad_x \big([y,\alpha(z)]\big) \cr
                                        &=&-ad^*_{[y,z]} (\theta\circ \alpha)(x) + ad^*_x\circ \theta\big(-ad_z(\alpha(y))\big)
                                           + ad^*_x\circ\theta\circ ad_y\big(\alpha(z)\big) \cr
                                        &=& -ad^*_{[y,z]} \beta_\alpha(x) - ad^*_x\circ ad^*_y\beta_\alpha(y)
                                            + ad^*_x\circ ad^*_y\beta_\alpha(z) \nonumber
\eeqn
Then $\beta_\alpha$ is an element of $\mathcal Q'$. This correspondence is obviously bijective.
\end{proof}

\vskip 0.3cm

Now we are going to look at maps $\psi's$ in the case where $\G$ is an orthogonal Lie algebra.

\blem \label{isomorphism-Psi-J'} 
The map $\psi \mapsto j_\psi:=\psi
\circ \theta$ is an isomorphism between the space of linear maps
$\psi:\G^* \to \G$ which satisfy relations (\ref{psi-relation1}) and the space $\mathcal J'$. 
\elem 

\begin{proof}
Take $\psi$  as in the Lemma \ref{isomorphism-Psi-J'}. Then for any elements $x,y,z$ in $\G$, we have
\beqn
[\psi\circ\theta,ad_x \circ ad_y] &=& \psi\circ\theta\circ  ad_x\circ ad_y -ad_x\circ ad_y \circ \psi \circ \theta \cr
                                  &=& \psi \circ ad^*_x\circ ad^*_y \circ \theta - \psi \circ ad^*_x\circ ad^*_y \circ \theta  \cr
                                  &=& 0 \nonumber
\eeqn
Then $j_\psi:=\psi\circ \theta$ is an element of $\mathcal J'$. 

Now consider an element $j$ of $\mathcal J'$ and set
$\psi_j:=j\circ \theta^{-1}$. Taking two arbitrary elements $x$ and $y$ of $\G$, we have
\beqn
\psi_j\circ ad^*_x\circ ad^*_y &=& j \circ \theta^{-1} \circ ad^*_x\circ ad^*_y \cr
                       &=& j \circ ad_x\circ ad_y \circ \theta^{-1}   \cr
                       &=&  ad_x\circ ad_y \circ j \circ \psi \cr
                       &=& ad_x\circ ad_y \circ \psi_j
\eeqn
Then $\psi_j$ satisfies relation (\ref{psi-relation1}). This correspondence is linear and invertible.
\end{proof}

\blem 
The maps $\xi \mapsto \theta^{-1}\circ \xi \circ \theta$ is an isomorphism of Lie algebras between the Lie algebra 
$\mathcal E'$ of linear maps $\xi : \G^* \to \G^*$ satisfying $[\xi,ad^*_x\circ ad^*_y] = ad^*_{\alpha(x)}\circ ad^*_y +ad^*_x\circ ad^*_{\alpha(y)}$, 
for some prederivation $\alpha$ of $\G$ and any $x$ and $y$ in $\G$, and the Lie algebra $\mathcal S'$ of linear 
maps $\xi':\G \to \G$ such that $[\xi',ad_x \circ ad_y] = ad_{\alpha(x)}\circ ad_y + ad_x \circ ad_{\alpha(y)}$, for some 
prederivation $\alpha$ of $\G$ and every elements $x,y$ of $\G$.
\elem 

\begin{proof}
 Consider an element $\xi$ of $\mathcal E'$  and set 
$\delta_\xi:=\theta^{-1}\circ \xi \circ \theta$. For $x,y,z$ in $\G$ we have
\beqn
[\delta_\xi,ad_x \circ ad_y]  &=& \delta_\xi \circ ad_x \circ ad_y - ad_x \circ ad_y \circ \delta_\xi \cr
                         &=& \theta^{-1}\circ \xi \circ \theta \circ ad_x \circ ad_y - ad_x \circ ad_y \circ \theta^{-1}\circ \xi \circ \theta \cr
                         &=& \theta^{-1}\circ \xi \circ ad^*_x \circ ad^*_y\circ \theta - \theta^{-1}\circ ad^*_x \circ ad^*_y \circ \xi \circ \theta \cr
                         &=& \theta^{-1}\circ ( \xi \circ ad^*_x \circ ad^*_y - ad^*_x \circ ad^*_y \circ \xi)\circ \theta \cr
                         &=& \theta^{-1}\circ [\xi,ad^*_x \circ ad^*_y] \circ \theta   \cr
                         &=& \theta^{-1}\circ \Big(ad^*_{\alpha(x)}\circ ad^*_y +ad^*_x\circ ad^*_{\alpha(y)}\Big) \circ \theta,
                                \text{ for some } \alpha \in \pder(\G) \cr
                         &=& \Big(ad_{\alpha(x)}\circ ad_y + ad_x \circ ad_{\alpha(y)}\Big) \circ \theta^{-1} \circ \theta \cr
                         &=& ad_{\alpha(x)}\circ ad_y + ad_x \circ ad_{\alpha(y)}.
\eeqn
Then $\delta_\xi$ is an element of $\mathcal S'$. 

\vskip 0.3cm
Now consider an element $\xi'$ of $\mathcal S'$ with $\alpha$ as
corresponding prederivation of $\G$. Then $\xi:=\theta \circ \xi' \circ \theta^{-1} $ is an element of
$\mathcal E'$ as one can see in the following. For any elements $x,y$ in $\G$ we have
\beqn
[\xi, ad^*_x \circ ad^*_y] &=& \xi \circ  ad^*_x \circ ad^*_y - ad^*_x \circ ad^*_y \circ \xi \cr
                   &=& \theta \circ \xi' \circ \theta^{-1} \circ   ad^*_x \circ ad^*_y -  ad^*_x \circ ad^*_y \circ \theta \circ \xi' \circ \theta^{-1} \cr
                   &=& \theta\circ \xi' \circ  ad_x \circ ad_y \circ \theta^{-1} - \theta \circ ad_x \circ ad_y \circ \xi' \circ \theta^{-1} \cr
                   &=& \theta \circ \big(\xi' \circ ad_x \circ ad_y-ad_x \circ ad_y \circ \xi' \big)\circ \theta^{-1} \cr
                   &=& \theta \circ [\xi',ad_x \circ ad_y] \circ \theta^{-1} \cr
                   &=& \theta \circ \big(ad_{\alpha(x)}\circ ad_y + ad_x \circ ad_{\alpha(y)}\big) \circ \theta^{-1} \cr
                   &=& \Big(ad^*_{\alpha(x)}\circ ad^*_y +ad^*_x\circ ad^*_{\alpha(y)}\Big) \circ \theta \circ \theta^{-1} \cr
                   &=& ad^*_{\alpha(x)}\circ ad^*_y +ad^*_x\circ ad^*_{\alpha(y)}.\nonumber
\eeqn
It comes that $\xi:=\theta \circ \xi' \circ \theta^{-1}$ belongs to $\mathcal E'$. 

\vspace{0.3cm}
Now note by $\delta : \mathcal E' \to \mathcal S'$, $\xi \mapsto \delta_{\xi}:=\theta^{-1} \circ \xi \circ \theta$.
For any $\xi_1$ and $\xi_2$ in $\mathcal E'$, we have :
\beqn
[\delta_{\xi_1},\delta_{\xi_2}] &=& \delta_{\xi_1}\circ \delta_{\xi_2} - \delta_{\xi_2}\circ \delta_{\xi_1}  \cr
                    &=& \theta^{-1} \circ \xi_1 \circ \theta \circ \theta^{-1} \circ \xi_2 \circ \theta -
                        \theta^{-1} \circ \xi_2 \circ \theta \circ \theta^{-1} \circ \xi_1 \circ \theta \cr
                    &=& \theta^{-1} \circ \xi_1 \circ \xi_2 \circ \theta -\theta^{-1} \circ \xi_2 \circ \xi_1 \circ \theta \cr
                    &=& \theta^{-1} \circ [\xi_1,\xi_2] \circ \theta \cr
                    &=& \delta_{[\xi_1,\xi_2]}  \nonumber
\eeqn
Thus, the map $\xi \mapsto \theta^{-1} \circ \xi \circ \theta$ is an isomorphism of Lie algebras.
\end{proof}

Now we have the following result which states that $\pder(T^*\G)$ is completely determined by $\pder(\G)$ and $\mathcal J'$.

\bpro\label{proposition:prederivation-orthogonal-lie-algebra}
Let $(\G,\mu)$ be an orthogonal Lie algebra and $\G^*$ its dual space. Consider the isomorphism 
$\theta : \G \to \G^*$ defined by $\langle \theta(x),y\rangle:=\mu(x,y)$. Any prederivation $\phi$ of $T^*\G$ 
has the following form
\beq 
\phi(x,f) = \Big( \alpha_1(x) + j_1 \circ \theta^{-1}(f) \,,\, \theta\circ \alpha_2(x) + (j_2^t-\alpha_1^t)(f)\Big),
\eeq 
for every $(x,f)$ in $T^*\G$, where 
\bitem
\item[-] $\alpha_1,\alpha_2$ are prederivations of $\G$, 
\item[-] $j_1,j_2$ are in $\mathcal J'$, with $ad^*_x\big(ad^*_{j_1\circ \theta^{-1}(f)}g-ad^*_{j_1\circ \theta^{-1}(g)}f\big)=0$, 
for all $x$ in $\G$; $f,g$ in $\G^*$,
\item[-] $j_2^t$ and $\alpha_1^t$ are the transposes of $j_2$ and $\alpha_1$ respectively.
\eitem 
\epro 

\brmq
\benum
\item Recall relation (\ref{psi-relation1}) :
$$
\psi \circ ad^*_x \circ ad^*_y = ad_x \circ ad_y \circ \psi,
$$ 
for all $x$ and $y$ in $\G$. We can also write
$$
\psi \circ ad^*_y \circ ad^*_x = ad_y \circ ad_x \circ \psi,
$$
for all $x$ and $y$ in $\G$. Substracting the above two relations we have, for all elements $x,y$ of $\G$,
$$
\psi \circ [ad^*_x,ad^*_y] = [ad_x,ad_y] \circ \psi 
$$
which can be written
$$ 
\psi \circ ad^*_{[x,y]}    =  ad_{[x,y]} \circ \psi.
$$
Hence, if $\psi$ satisfies relation (\ref{psi-relation1}) then 
\beq 
\psi \circ ad^*_{[x,y]}  =  ad_{[x,y]} \circ \psi, \label{psi-rel1}
\eeq 
for any $x$ and $y$ in $\G$.
\item By the same way, we recall relation (\ref{xi-relation}) :
$$
[\xi,ad^*_x\circ ad^*_y] = ad^*_{\alpha(x)}\circ ad^*_y +ad^*_x\circ ad^*_{\alpha(y)}
$$
for any $x$ and $y$ in $\G$. Changing the roles played by $x$ and $y$ we obtain
$$
[\xi,ad^*_y \circ ad^*_x] = ad^*_{\alpha(y)} \circ ad^*_x + ad^*_y \circ ad^*_{\alpha(x)}.
$$
Substracting again the two last relations above, we have
\beqn
\Big[\xi,[ad^*_x \,,\, ad^*_y]\Big]  &=& \big[ad^*_{\alpha(x)}, ad^*_y\big] + \big[ad^*_x,ad^*_{\alpha(y)}\big] \cr
\Big[\xi,ad^*_{[x,y]}\Big]           &=& ad^*_{[\alpha(x),y]} + ad^*_{[x,\alpha(y)]}. \nonumber
\eeqn
That is 
\beq 
\Big[\xi,ad^*_{[x,y]}\Big]  = ad^*_{([\alpha(x),y] + [x,\alpha(y)])},  \label{xi-relation2}
\eeq 
for all $x$ and $y$ in $\G$.

\eenum
\ermq

\subsection{Semi-simple Lie algebras}\label{chap:semi-simple-lie-algebras2}

A semi-simple Lie algebra is  an orthogonal Lie algebra.
Moreover, if $\G$ is semi-simple, then any prederivation is a derivation and hence an inner derivation (\cite{muller}).
In this case, we can write relation (\ref{xi-relation2}) as follows
\beq
\Big[\xi,ad^*_{[x,y]}\Big]  = ad^*_{\alpha([x,y])}
\eeq
for all $x,y$ in $\G$. Since $[\G,\G]=\G$, we can simply write
\beq
[\xi,ad^*_x]  = ad^*_{\alpha(x)}
\eeq
for any $x$ of $\G$. The following lemma is an immediate consequence of Section \ref{chap:orthogonal-algebra}.

\blem\label{xi-semi-simple}
If $\G$ is a semi-simple Lie algebra, then $\mathcal E'$ is nothing but the space of linear maps
$\xi':\G^* \to \G^*$ such that
$$
[\xi,ad^*_x] = ad^*_{\alpha(x)},
$$
for some derivation $\alpha$ of $\G$ and any $x$ in $\G$. Furthermore, if $\G$ decomposes into simple ideal as follows
$\G = \s_1 \oplus \s_2 \oplus \cdots \oplus \s_p$ ($p \in \N$), then
\beq
\xi = ad^*_{x_0} + \bigoplus_{i=1}^p\lambda_i id_{\s_i^*},
\eeq
for some $x_0$ in $\G$ and some real numbers $\lambda_1, \lambda_2, \ldots, \lambda_p$.
\elem

\vskip 0.3cm

Now we work on the maps $\beta$ and the relation (\ref{beta-relation}).
The maps $\beta \mapsto \alpha_\beta:=\theta^{-1}\circ \beta$ defined
in Lemma \ref{isomorphism-Q'-Pder} is then an isomorphism between the spaces $\der(\G)$ and $\mathcal Q'$,
since $\pder(\G)=\der(\G)$. From Proposition \ref{isomorphism-cocycle-derivation} it comes that if $\alpha_\beta$ is a derivation then
$\beta = \theta \circ \alpha_\beta$ is a $1$-cocycle of $\G$ with values in $\G^*$ for the coadjoint representation of
$\G$ on $\G^*$. We then proved the following lemma.

\blem\label{beta-semi-simple}
If $\G$ is a semi-simple Lie algebra, then any element $\beta$ of $\mathcal Q'$ is a $1$-cocycle of $\G$ with values in
$\G^*$ for the coadjoint representation of $\G$ on $\G^*$.
\elem

\blem\label{psi-semi-simple}
If $\G$ is a semi-simple Lie algebra, then $\Psi'=\{0\}$.
\elem
\begin{proof}
Relation (\ref{psi-relation2}) means that the linear form $ad^*_{\psi(f)}g - ad^*_{\psi(g)}f$ is closed,
for any linear forms $f$ and $g$ on $\G$. Since $\G$ is semi-simple, it is perfect ;  that is $\G$ is equal to its derived
ideal ($[\G,\G]=\G$). It is known that any closed form on a perfect Lie algebra is zero. Then, the closed form
$ad^*_{\psi(f)}g - ad^*_{\psi(g)}f$ is equal to zero for any $f$ and $g$ in $\G^*$, {\it i.e.}
\beq
ad^*_{\psi(f)}g - ad^*_{\psi(g)}f = 0,
\eeq
for every $f$ and $g$ in $\G$. Again because $[\G,\G]=\G$,  Relation (\ref{psi-relation1}) becomes
\beq
\psi \circ ad^*_x = ad_x \circ \psi
\eeq
Now we conclude with Proposition \ref{prop:psi-semisimple}.

\end{proof}

We summarize all the above by the following

\bthm\label{thm-prederivation-semi-simple}
Let $G$ be a finite dimensional semi-simple Lie group with Lie algebra $\G$. Then every prederivation of the Lie algebra $T^*\G$ of the 
cotangent bundle Lie group $T^*G$ of $G$ is a derivation.
\ethm

\begin{proof}
 This is direct consequence of Lemmas \ref{xi-semi-simple}, \ref{beta-semi-simple}, \ref{psi-semi-simple} and the fact that any
prederivation of a semi-simple Lie algebra is a derivation.
\end{proof}

\subsection{Compact Lie algebras}\label{section:compact-prederivation}

\blem\label{prederivation-compact}
Let $\G$ be a compact Lie algebra. Then every prederivation $\alpha$ of $\G$ is of the form $\alpha = ad_{x_0} \oplus \varphi$,
where $x_0$ belongs to the derived ideal of $[\G,\G]$ and $\varphi$ is an endomorphism of $Z(\G)$. That is for any $x_1$ in
$[\G,\G]$ and $x_2$ in $Z(\G)$,
 \beq
\alpha(x_1+x_2)=[x_0,x_1] + \varphi(x_2),
\eeq
where $x_0$ in $[\G,\G]$, where $\varphi$ is an endomorphism of $Z(\G)$.
\elem

\begin{proof}

Recall that a compact Lie algebra is a Lie algebra $\G$ which decomposes into the sum $\G=[\G,\G]\oplus Z(\G)$
of its derived ideal and its centre, with $[\G,\G]$ semi-simple and compact. A prederivation
$\alpha$ of such a Lie algebra preserves each of the factor $[\G,\G]$ and $Z(\G)$. Then it can be written as a direct sum of
a prederivation  $\alpha_1$ of $[\G,\G]$ and a prederivation $\alpha_2$ of $Z(\G)$. Since $[\G,\G]$ is semi-simple, 
the prederivation $\alpha_1$ is an inner derivation. Furthermore, the derivation $\alpha_2$ of $Z(\G)$ is just 
an endomorphism of $Z(\G)$ because
$Z(\G)$ is an Abelian ideal of $\G$. Hence, we can write
\beq
\alpha = ad_{x_0} \oplus \varphi,
\eeq
where $x_0$ is an element of the derived ideal of $\G$ and $\varphi$ is an endomorphism of $Z(\G)$. The prederivation
$\alpha$ acts on $\G$ as follows. If $x=x_1+x_2$ with $x_1$ in $[\G,\G]$ and $x_2$ in $Z(\G)$,
\beq
\alpha(x)=\alpha(x_1+x_2) = [x_0,x_1] + \varphi(x_2).
\eeq
\end{proof}

\blem\label{lemma:psi'-compact}
If $\G$ is a compact Lie algebra. Then the space $\Psi'$ is isomorphic to the space $End(Z(\G))$ of endomorphisms of the 
centre $Z(\G)$ of $\G$.
\elem

\begin{proof}

 A compact Lie algebra admits an orthogonal structure, then from Lemma \ref{isomorphism-Psi-J'}
the space of linear maps $\psi:\G^* \to \G$ which are equivariant with respect to the adjoint and coadjoint
representation of $[\G,\G]$ on $\G$ and $\G^*$ respectively is isomorphic to the space $\mathcal J'$. 
The correspondence is given by $\psi=j \circ \theta^{-1}$, for some $j$ in $\mathcal J'$.
For any element $x$ of $\G=[\G,\G]\oplus Z(\G)$, we write $x=x_1+x_2$, where $x_1$ is in the semi-simple ideal 
$[\G,\G]$ while $x_2$ belongs to the centre $Z(\G)$ of $\G$. Set 
\beq 
j(x) = \underbrace{\big(j_{11}(x_1)+j_{21}(x_2)\big)}_{\in [\G,\G]} + 
\underbrace{\big(j_{12}(x_1)+j_{22}(x_2)\big)}_{\in Z(\G)}, 
\eeq 
where $j_{11}:[\G,\G] \to [\G,\G]$, $j_{12} : [\G,\G] \to Z(\G)$, $j_{21}: Z(\G) \to [\G,\G]$ 
and $j_{22} : Z(\G) \to Z(\G)$  are linear maps.
 For any $x,y,z$ in $\G$, we have
\beqn
j\Big( \big[[x,y],z\big]\Big) &=& \big[[x,y],j(z)\big] \cr
j\Big( \big[[x_1,y_1],z_1\big]\Big) &=& \big[[x_1,y_1],j_{11}(z_1)+j_{21}(z_2)\big] \cr
                                    &=& \big[[x_1,y_1],j_{11}(z_1)\big] + \big[[x_1,y_1],j_{21}(z_2)\big] \label{hermine}
\eeqn 
If we take $z_2=0$, then $j\Big( \big[[x_1,y_1],z_1\big]\Big) = \big[[x_1,y_1],j_{11}(z_1)\big]$ for 
all $x_1,y_1,z_1$ in $[\G,\G]$. Relation (\ref{hermine}) gives $\big[[x_1,y_1],j_{21}(z_2)\big] = 0$ 
for all $x_1,y_1$ in $[\G,\G]$  and $z_2$ in $Z(\G)$; that is $j_{21}(z_2)$ belongs 
to the centre of $[\G,\G]$ which is $\{0\}$ since $[\G,\G]$ is semi-simple. 
It comes that $j_{21}\equiv 0$. On the other way 
\beqn 
j\Big( \big[[x_1,y_1],z_1\big]\Big) &=& j_{11}\Big( \big[[x_1,y_1],z_1\big]\Big) + j_{12}\Big( \big[[x_1,y_1],z_1\big]\Big) \cr
\underbrace{\big[[x_1,y_1],j_{11}(z_1)\big]}_{\in [\G,\G]}  &=& 
\underbrace{j_{11}\Big( \big[[x_1,y_1],z_1\big]\Big)}_{\in [\G,\G]} 
+ \underbrace{ j_{12}\Big( \big[[x_1,y_1],z_1\big]\Big)}_{\in Z(\G)}
\eeqn
It comes that  $j_{12}\Big( \big[[x_1,y_1],z_1\big]\Big) = 0$, for all $x_1,y_1,z_1$ in $[\G,\G]$. 
Then $J_{12}\equiv 0$ on the 
semi-simple ideal $[\G,\G]$. Now we have just 
\beq 
j(x) = j_{11}(x_1) + j_{22}(x_2),
\eeq 
where $j_{11}$ is an endomorphism of $[\G,\G]$ satisfying 
\beq\label{j-semi-simple-derived-ideal} 
j_{11}\Big( \big[[x_1,y_1],z_1\big]\Big) =\big[[x_1,y_1],j_{11}(z_1)\big],
\eeq  
for all $x_1,y_1,z_1$ in $[\G,\G]$ ; and $j_{22}$ is in $End(Z(\G))$. Since $[\G,\G]$ is perfect 
then (\ref{j-semi-simple-derived-ideal})
can be written 
\beq 
j_{11}([x_1,y_1]) = [x_1,j_{11}(y_1)],
\eeq 
for all $x_1,y_1$ in $[\G,\G]$. It comes, from Corollary \ref{j-semi-simple}, that 
\beq \label{j11compact}
j_{11}(x) = \sum_{i=1}^p\lambda_ix_{1i} + j_{22}(x_2),
\eeq 
where $x_1=x_{11}+ x_{12}+\cdots + x_{1p}$ is the decomposition of $x_1$ into elements of the simple components of $[\G,\G]$.
Now we have,
\beqn
\psi(f) &=& j \circ \theta^{-1}(f) \cr
        &=& j\Big(\sum_{k=1}^p \theta^{-1}(f_{1k}) + \theta^{-1}(f_2) \Big) \cr 
        &=& \sum_{i=1}^p\lambda_i\theta^{-1}(f_{1i}) + j_{22}\circ \theta^{-1}(f_2) \label{psi-prederivation-compact}
\eeqn 
From Lemma \ref{psi-semi-simple} the restriction of $\psi$ to the semi-simple ideal must be zero, then
\beq\label{psi-prederivation-compact1}
\psi(f) = j_{22}\circ \theta^{-1}(f_2).
\eeq 
and we are done.
%
\end{proof}

\blem\label{isomorphism-Q'-Pder2} 
If $\G$ is a compact Lie algebra, then the space $\mathcal Q'$ is isomorphic to the space
$ad_{[\G,\G]}\oplus End(Z(\G))$, where $ad_{[\G,\G]}$ stands for the
space of inner derivations of the derived ideal $[\G,\G]$ of $\G$.
\elem

\begin{proof}
 The proof is straightforward. Lemma \ref{isomorphism-Q'-Pder} asserts that $\mathcal Q'$ is isomorphic to $\pder(\G)$ and
Lemma \ref{prederivation-compact} implies that $\pder(\G)\stackrel{\sim}{=}ad_{[\G,\G]}\oplus End(Z(\G))$.
\end{proof}

\blem 
Let $\G$ be a compact Lie algebra. Then any linear map $\xi :
\G^* \to \G^*$ satisfying relation (\ref{xi-relation}) can be
written as 
\beq 
\xi = \left(\bigoplus_{i=1}^p \lambda_i id^*_{\s_i}
+ ad^*_{x_0}\right) \oplus \eta, 
\eeq 
where $\eta$ is an endomorphism of $Z(\G^*)$, $x_0$ is an element of  $[\G,\G]=\s_1\oplus\s_2\cdots\oplus\s_p$;
$\lambda_1, \lambda_2,\ldots, \lambda_p$ are real numbers and $p$ is
the number or simple components of $[\G,\G]$. More precisely, if
$f=f_1 + f_2$ is an element of $\G^*$ with $f_1 = f_{11} + f_{12} +
f_{13} + \cdots + f_{1p}$ in $[\G,\G]^*=\s^*_1\oplus \s^*_2 \oplus
\cdots \oplus \s^*_p$ and $f_2$ in $Z(\G)^*$, then 
\beq 
\xi(f) =\sum_{i=1}^p\lambda_i f_i + ad^*_{x_0}f_1 + \eta(f_2). 
\eeq 
\elem

\begin{proof}
We have already seen that the transpose $\xi^t$ of $\xi$ has the form $\xi^t=j-\alpha$, where $j$ is in 
$\mathcal J'$. From the proof of Lemma \ref{lemma:psi'-compact} (see Relation (\ref{j11compact})) we have
\beq
j = \left(\bigoplus_{i=1}^p \lambda_i id_{\s_i}\right) \oplus \rho_1, \nonumber
\eeq
where $\rho_1$ is an endomorphism of $Z(\G)$. Hence,
$$
\xi^t = \left(\bigoplus_{i=1}^p \lambda_i id_{\s_i}\right) \oplus \rho_1  -\big(ad_{x_0} \oplus \varphi_1\big),
$$
since a prederivation $\alpha$ of the compact Lie algebra $\G$ is given by $\alpha = ad_{x_0} \oplus \varphi_1$, where
$x_0$ is in $[\G,\G]$ and $\varphi_1$ is an endomorphism of $Z(\G)$ (see Lemma \ref{prederivation-compact}). We simply write
$$
\xi^t = \left(\bigoplus_{i=1}^p \lambda_i id_{\s_i}  -ad_{x_0}\right) \oplus \varphi_2,
$$
where $\varphi_2=\rho_1+\varphi_1$ belongs to $End(Z(\G))$. It comes that
$$
\xi = \left(\bigoplus_{i=1}^p \lambda_i id_{\s^*_i}  +ad^*_{x_0}\right) \oplus \eta,
$$
where $\eta = \varphi_2^t$ is the transpose of $\varphi_2$.
\end{proof}

The following Proposition holds and is a direct consequence of the Lemmas above.

\bpro
Let $\G$ be a compact Lie algebra. Let $\theta : \G \to \G^*$ be the isomorphism 
defined by $\langle \theta(x),y\rangle:=\mu(x,y)$, where $\mu$ is an arbitrary orthogonal structure on $\G$.
Then any prederivation of $T^*\G$ has the following form : 
\beqn 
\phi(x,f)  &=&  \Big(\![x_0,x_1]\! +\! \varphi_1(x_2) \!+\! \varphi_2 \circ \theta^{-1}(f_2) \,\,,\,\,
\theta([y_0,x_1]) \!+\! \theta\circ \varphi_3(x_2) \nonumber\\
\!& &+\! \sum_{i=1}^p\lambda_if_{1i} \!+\! 
ad^*_{x_0}f_1 \!+\! \eta(f_2) \!\Big), 
\eeqn 
where 
\bitem
\item[-] $x_0,y_0$ are in $[\G,\G]$; 
\item[-] $\varphi_1, \varphi_2, \varphi_3$ are endomorphisms of $Z(\G)$; 
\item[-] $\eta$ is in $End\big(Z(\G)^*\big)$;
\item[-] $\nu_i,\lambda_i$, $i=1,2,\ldots,p$ are real numbers;  
\item[-] $x=x_1+x_2=x_{11}+x_{12}+\cdots+x_{1p}+x_2 \in \G=[\G,\G]\oplus Z(\G)$ 
\item[-] and $f=f_1+f_2=f_{11}+f_{12}+\cdots+f_{1p}+f_2 \in \G^*=[\G,\G]^*\oplus Z(\G)^*$; 
\eitem 
$p$ being the number of simple components of 
$[\G,\G]=\s_1\oplus\s_2\oplus\cdots\oplus\s_p$. 
\epro 

\begin{proof}
Let $\G=[\G,\G]\oplus Z(\G)$ be a compact Lie algebra. From Proposition \ref{proposition:prederivation-orthogonal-lie-algebra} 
we have that any  prederivation $\phi$ of $\G$ has the form 
$$
\phi(x,f) = \Big( \alpha_1(x) + j_1 \circ \theta^{-1}(f) \,,\, \theta\circ \alpha_2(x) + (j_2^t-\alpha_1^t)(f)\Big),
$$
 for every $(x,f)$ in $T^*\G$, where $\alpha_1,\alpha_2$ are prederivations of $\G$, $j_1,j_2$ are in 
$\mathcal J'$, $j_2^t$ and $\alpha_1^t$ are the transposes of $j_2$ and $\alpha_1$ respectively.
 Now Lemma \ref{prederivation-compact} implies that $\alpha_1(x)=[x_0,x_1] +
\varphi_1(x_2)$ and $\alpha_2(x)=[y_0,x_1] + \varphi_3(x_2)$, for every $x=x_1+x_2 \in \G=[\G,\G]\oplus Z(\G)$
where $x_0$ and $y_0$ are fix elements of the derived ideal $[\G,\G]$. From the proof of Lemma \ref{lemma:psi'-compact}, we have
\beqn 
j_1(x)&=&\sum_{i=1}^p\nu_ix_{1i} + \varphi_2(x_2),\cr
j_2(x)&=&\sum_{i=1}^p\lambda_ix_{1i} + \varphi'_2(x_2), \nonumber
\eeqn 
where $x=x_1+x_2=x_{11}+x_{12}+ \cdots + x_{1p}+x_2 \in
[\G,\G]\oplus Z(\G)$, $\lambda, \nu_i,
i=1,2,\ldots,p$ being real numbers. Then, for any
$f=f_1+f_2=f_{11}+f_{12}+\cdots+f_{1p}+f_2$ in $[\G,\G]^*\oplus
Z(\G)^*$, we have

\beqn 
j_1\circ \theta^{-1}(f) &=& j_1\big(\theta^{-1}(f)\big) \cr
        &=&\sum_{i=1}^p\nu_i \Big(\theta^{-1}(f_{1i}) \Big) +
        \varphi_2\big(\theta^{-1}(f_2)\big).
\eeqn
The restriction of $\psi:=j_1\circ\theta^{-1}$ to the dual space of the semi-simple
ideal $[\G,\G]$ must be zero (see Proposition \ref{prop:psi-semisimple}). Then we have

\beq 
j_1\circ \theta^{-1}(f) = \varphi_2\circ \theta^{-1}(f_2), 
\eeq 
for all $f=f_1+f_2=f_{11}+f_{12}+\cdots+f_{1p}+f_2$ in $[\G,\G]^*\oplus Z(\G)^*$.

\vskip 0.3cm

 Let us, now have a look at maps $\beta:=\theta\circ \alpha_2$. We have 

\beq 
\beta(x)=\theta\big([y_0,x_1]\big) + \theta\circ \varphi_3(x_2)
\eeq 
for all $x=x_1+x_2$ in $\G=[\G,\G]\oplus Z(\G)$.
\end{proof}

\vskip 0.5cm

We finish this chapter by giving some situations where prederivations are used.

\section{Possible Applications and Examples}\label{section:applications-examples-chap3}

\subsection{Examples}
\bex[Affine Lie algebra of the real line]\label{chap2:aff(R)}{\normalfont
Let $\hbox{\rm Aff}(\R)$ stand for the affine Lie group of the real line and note by $\hbox{\rm aff}(\R)$ its Lie algebra.
We recall (see Example \ref{chap:aff(R)})  that $\D_1:=T^*\hbox{\rm aff}(\R)=\sspan(e_1, e_2, e_3)$ with the following brackets
$$
[e_1,e_2]=e_2 \;,\; [e_1,e_4]=-e_4 \;,\; [e_2,e_4]=e_3,
$$ 
One can readily verify that $\pder(\D_1)=\der(\D_1)= \R^2\ltimes\R^3$ (semi-direct product of the Abelian Lie algebras 
$\R^3=\sspan_{\R}(\phi_1,\phi_3,\phi_4)$ and $\R^2=\sspan_{\R}(\phi_2,\phi_5)$) with the following brackets (see Example \ref{chap:aff(R)}) :
$$
[\phi_2,\phi_1]=\phi_1 \; , \;  [\phi_2,\phi_3] = \phi_3 \;,\; [\phi_5,\phi_3]=\phi_3 \;,\; [\phi_5,\phi_4]= \phi_4.
$$ 

}
\eex 

\bex[The Lie Algebra of the Group $SL(2,\R)$ of Special Linear Group and the Lie Algebra of the Group $SO(3)$ of Rotations]{\normalfont
The Lie algebra $\mathfrak{sl}(2,\R)$ is simple, then $\pder(T^*\mathfrak{sl}(2,\R))=\der(T^*\mathfrak{sl}(2,\R))$ (See Example \ref{chap:sl(2)}). By the same argument,
$\pder(T^*\mathfrak{so}(3,\R))=\der(T^*\mathfrak{so}(3,\R))$ (see Example \ref{chap:so(3)}).
}
\eex 
\bex[The $4$-dimensional Oscillator Group]\label{prederivations-osciallator-group}{\normalfont In Example \ref{example-oscillator-group} we have define the Oscillator Lie group
and its Lie algebra. 
The $4$-dimensional oscillator algebra, is the space $\G_\lambda = \sspan\{e_{-1},e_0,e_1,\check{e}_1\}$ ($\lambda >0$) with the following brackets :
\beq 
[e_{-1},e_1] = \lambda \check e_1 \quad ; \quad  [e_{-1},\check e_1] = -\lambda e_1 \quad
; \quad [e_1,\check e_1] = e_0.
\eeq 
Let $(e^*_{-1},e^*_0,e^*_1,\check{e}^*_1)$ stand for  the basis of $\G_\lambda^*$ dual to $(e_{-1},e_0,e_1,\check{e}_1)$.
Then, the Lie algebra $T^*\G_\lambda = \sspan(e_{-1},e_0,e_1,\check{e}_1,e^*_{-1},e^*_0,e^*_1,\check{e}^*_1)$ with the brackets

\beq
\begin{array}{rcccrcccccc}
[e_{-1},e_1] &=& \lambda \check e_1 &;&  [e_{-1},\check e_1] &=& -\lambda e_1 &;& [e_1,\check e_1] &=& e_0   \cr
[e_{-1},e^*_1] &=& \lambda \check e^*_1 &;&  [e_{-1},\check e^*_1] &=& -\lambda e^*_1 &;& [e_1, e^*_0] &=&- \check e^*_1  \cr
[\check e_1,e^*_1] &=&-\lambda e^*_{-1} &;&  [e_1,\check e^*_1] &=& \lambda e^*_{-1} &;& [\check e_1, e^*_0] &=& e^*_1 \cr 
\end{array}
\eeq 
 
Consider the form $\mu_\lambda$ defined on $\G_\lambda$ by
\beq\label{orthogonal-structure-oscillator} 
\mu_\lambda(x,y) = x^{-1}y^0 + x^0y^{-1} + \frac{1}{\lambda}(x^1y^1+\check x^1 \check y^1)
\eeq 
for all $x=x^{-1}e_{-1}+x^0e_0+x^1e_1+\check x^1\check e_1$ and $y=y^{-1}e_{-1}+y^0e_0+y^1e_1+\check y^1\check e_1$. It is readily 
cheiked that the form (\ref{orthogonal-structure-oscillator}) defines an orthogonal structure on $\G_\lambda$ (\cite{bromberg-medina2004}). 
The isomorphism 
$\theta : \G_\lambda \to \G_\lambda$ defined by $\langle \theta (x),y\rangle = \mu_\lambda(x,y)$, for all $x,y$ in $\G_\lambda$ is given by
\beq 
\theta(e_{-1}) = e_0^* \quad ,\quad  \theta(e_0) = e_{-1}^* \quad ,\quad  
\theta(e_1)=\frac{1}{\lambda}e_1^* \quad ,\quad \theta(\check e_1)=\frac{1}{\lambda}\check e_1^*,
\eeq 
The inverse map of $\theta$ reads :

\beq
\theta^{-1}(e_{-1}^*) = e_0 \quad ,\quad  \theta^{-1}(e_0^*) = e_{-1} \quad ,\quad  
\theta^{-1}(e_1^*)=\lambda e_1 \quad ,\quad \theta^{-1}(\check e_1^*)=\lambda \check e_1.
\eeq 

Since $\G_\lambda$ is an orthogonal Lie algebra, any prederivation $\phi$ of $T^*\G_\lambda$ can be written as follows.
$$
\phi(x,f) = \Big( \alpha_1(x) + j_1 \circ \theta^{-1}(f) \,,\, \theta\circ \alpha_2(x) + (j_2^t-\alpha_1^t)(f)\Big),
$$ 
for every $(x,f)$ in $T^*\G$, where $\alpha_1,\alpha_2$ are in $\pder(\G_\lambda)$ and $j_1,j_2$ are in $\mathcal{J}'$ 
with conditions listed in Proposition \ref{proposition:prederivation-orthogonal-lie-algebra}. Now we have :
\bitem 
\item A prederivation  $\alpha$ of $\G_\lambda$ can be represented in the basis $(e_{-1},e_0,e_1,\check{e}_1)$ 
by the following matrix.
\beq 
\alpha = \left(
\begin{array}{ccrc}
0               &   0        &    0      &     0   \cr
a_{21}          & 2a_{33}    &   a_{23}  &  a_{24} \cr
-\lambda a_{23} &   0        &   a_{33}  &  a_{34} \cr
-\lambda a_{24} &   0        &  -a_{34}  &  a_{33}
\end{array}
\right),
\eeq 
where $a_{ij}$'s are reals numbers. We can put 
$$
\alpha_1 = \left(
\begin{array}{ccrc}
0               &   0        &    0      &     0   \cr
a_{21}          & 2a_{33}    &   a_{23}  &  a_{24} \cr
-\lambda a_{23} &   0        &   a_{33}  &  a_{34} \cr
-\lambda a_{24} &   0        &  -a_{34}  &  a_{33}
\end{array}
\right) \quad ;\quad  \alpha_2 = \left(
\begin{array}{ccrc}
0               &   0        &    0      &     0   \cr
b_{21}          & 2b_{33}    &   b_{23}  &  b_{24} \cr
-\lambda b_{23} &   0        &   b_{33}  &  b_{34} \cr
-\lambda b_{24} &   0        &  -b_{34}  &  b_{33}
\end{array}
\right).
$$
\item A linear map $j:\G_\lambda \to \G_\lambda$ which satisfies (\ref{rel-pre-biinvariant-tensor}) has the following form :
\beq 
j = \left(
\begin{array}{ccrc}
j_{11}          &   0        &    0      &     0   \cr
j_{21}          &  j_{11}    &    0      &     0   \cr
     0          &   0        &   j_{11}  &     0 \cr
    0           &   0        &    0      &   j_{11}
\end{array}
\right),
\eeq 
where $j_{11}$  and $j_{21}$ are real numbers. We set
$$
j_1 = \left(
\begin{array}{ccrc}
a_1          &   0     &    0   &     0   \cr
a_2          &  a_1    &    0   &     0   \cr
    0        &   0     &   a_1  &     0 \cr
    0        &   0     &    0   &   a_1
\end{array}
\right) \quad ; \quad j_2 = \left(
\begin{array}{ccrc}
b_1          &   0     &    0   &     0   \cr
b_2          &  b_1    &    0   &     0   \cr
    0        &   0     &   b_1  &     0 \cr
    0        &   0     &    0   &   b_1
\end{array}
\right).
$$
\eitem
It follows that 
$$
j_1\circ \theta^{-1} = \left(
\begin{array}{cccc}
0            &  a_1    &    0        &     0   \cr
a_1          &  a_2    &    0        &     0   \cr
    0        &   0     &\lambda a_1  &     0   \cr
    0        &   0     &    0        &\lambda a_1
\end{array}
\right) \quad ; \quad 
\theta \circ \alpha_2 = \left(
\begin{array}{cccc}
b_{21}   & 2b_{33}    &    b_{23}               &  b_{24} \cr
0        &   0        &     0                   &     0   \cr
- b_{23} &   0        &\frac{1}{\lambda}b_{33}  &\frac{1}{\lambda}b_{34} \cr
- b_{24} &   0        &-\frac{1}{\lambda}b_{34} &\frac{1}{\lambda}b_{33}
\end{array}
\right).
$$
If we consider the fact that $ad_x^*\big(ad^*_{j_1\circ\theta^{-1}(f)}g - ad_{j_1\circ\theta^{-1}(g)}f\big) = 0$, 
for all $x$ in $\G_\lambda$ and any $f,g$ in $\G^*_\lambda$, we obtain
$$
j_1\circ \theta^{-1} = \left(
\begin{array}{cccc}
    0        &   0     &    0        &     0   \cr
    0        &  a_2    &    0        &     0   \cr
    0        &   0     &    0        &     0   \cr
    0        &   0     &    0        &     0
\end{array}
\right)
$$
Now we write the matrix of a prederivation $\phi$ of $T^*\G_\lambda$.
\beq 
\phi=\left(
\begin{array}{cccccccc} 
0               &   0        &    0      &     0   &    0     &   0     &    0        &     0   \cr
&            &                         &           &     &              &  &   \cr
a_{21}          & 2a_{33}    &   a_{23}  &  a_{24} &    0     &  a_2    &    0        &     0   \cr
&            &                         &           &     &              &  &   \cr
-\lambda a_{23} &   0        &   a_{33}  &  a_{34} &    0     &   0     &    0        &     0   \cr
&            &                         &           &     &              &  &   \cr
-\lambda a_{24} &   0        &  -a_{34}  &  a_{33} &    0     &   0     &    0        &     0   \cr
&            &                         &           &     &              &  &   \cr
b_{21}          & 2b_{33}    &    b_{23}               &  b_{24}   &  b_1   & b_2-a_{21}  &\lambda a_{23}   & \lambda a_{24} \cr
&            &                         &           &     &              &  &   \cr
0               &   0        &     0                   &     0     &  0  & b_1-2a_{33} & 0 & 0   \cr
                &            &                         &           &     &              &  &      \cr    
- b_{23}        &   0        &\frac{1}{\lambda}b_{33}  &\frac{1}{\lambda}b_{34} & 0 & -a_{23} & b_1-a_{33} & a_{34} \cr
&            &                         &           &     &              &  &   \cr
- b_{24}        &   0        &-\frac{1}{\lambda}b_{34} &\frac{1}{\lambda}b_{33} & 0 & -a_{24} & -a_{34} & b_1-a_{33}
\end{array}\right)
\eeq 
It comes that $\pder(T^*\G_\lambda) = \sspan\{\phi_i, 1\leq i \leq 13\}$, where 
\beq
\begin{array}{ccl|cccl}
\hline 
\phi_1 &=& e_{21}-e_{56}                &&   \phi_2       &=& 2e_{22}+e_{33}+e_{44}-2e_{66}-e_{77}-e_{88} \cr
\hline 
\phi_3 &=&-e_{23}+\lambda e_{31}-\lambda e_{57} + e_{76} & & \phi_4 &=& e_{24}-\lambda e_{41}+\lambda e_{58}-e_{86} \cr
\hline
\phi_5 &=&e_{26}  & &\phi_6 &=& -\lambda e_{34}+\lambda e_{43} -\lambda e_{78} + \lambda e_{87} \cr
\hline 
\phi_7 &=& e_{51}  & & \phi_8 &=& 2e_{52}+\frac{1}{\lambda}e_{73}+\frac{1}{\lambda}e_{84} \cr 
\hline  
\phi_9 &=&-\lambda e_{53} +\lambda e_{71}  & & \phi_{10} &=& \lambda e_{54} - \lambda e_{81} \cr 
\hline
\phi_{11} &=& e_{55}+e_{66}+e_{77}+e_{88} & & \phi_{12} &=& e_{56} \cr 
\hline 
\phi_{13} &=& -e_{74}+e_{83}  &  &      & &       \cr
\hline 
\end{array}
\eeq 
with the following brackets
\beq 
\begin{array}{ccl|cccc|ccccc|ccc}
\hline 
[\phi_1,\phi_2] &=& -\phi_1 & & [\phi_1,\phi_8] &=& -2\phi_7 & & [\phi_2,\phi_3] &=& \phi_3 \cr
\hline 
[\phi_2,\phi_4] &=& \phi_4  & &[\phi_2,\phi_5] &=& 2\phi_5 & & [\phi_2,\phi_8] &=& -2\phi_8 \cr
\hline 
[\phi_2,\phi_9] &=&-\phi_9 & & [\phi_2,\phi_{10}] &=&-\phi_{10} & & [\phi_2,\phi_{12}] &=&2\phi_{12} \cr
\hline 
[\phi_2,\phi_{13}] &=&-2\phi_{13}  &  & [\phi_3,\phi_6] &=&\lambda\phi_4  & & [\phi_3,\phi_8] &=&-\frac{1}{\lambda}\phi_9 \cr
\hline
[\phi_3,\phi_{13}] &=&\phi_{10}  & & [\phi_4,\phi_6] &=& -\lambda\phi_3 & & [\phi_4,\phi_8] &=& -\frac{1}{\lambda}\phi_{10}  \cr
\hline 
[\phi_4,\phi_{13}] &=&-\phi_9 & & [\phi_5,\phi_8] &=&-2\phi_{12}   & & [\phi_5,\phi_{11}] &=& \phi_5  \cr
\hline
[\phi_6,\phi_9] &=&-\lambda \phi_{10}  & & [\phi_6,\phi_{10}] &=& -\lambda \phi_{10} & & [\phi_7,\phi_{11}] &=& -\phi_7   \cr
\hline 
[\phi_8,\phi_{11}] &=&-\phi_8 & & [\phi_9,\phi_{11}] &=&-\phi_9 & & [\phi_{10},\phi_{11}] &=&-\phi_{10} \cr
\hline
[\phi_{11},\phi_{12}] &=& \phi_{12} & & [\phi_{11},\phi_{13}] &=&\phi_{13}\cr
\hline 
\end{array}
\eeq 
One realizes that 
\beqn
\pder(T^*\G_\lambda) &=&ad_{T^*\G_\lambda}\ltimes\Big[\R^2\ltimes(\R^2\ltimes\R^3)\Big] \cr
                     &=&(\R^2\ltimes\R^4)\ltimes\Big[\R^2\ltimes(\R^2\ltimes\R^3)\Big] \cr
                     &=&\Big(\sspan(\phi_6,\phi_{13})\ltimes\sspan(\phi_3,\phi_4,\phi_9,\phi_{10})\Big)\cr
                     & &   \ltimes\Big[\sspan(\phi_2,\phi_{11})\ltimes\big(\sspan(\phi_1,\phi_5)
                        \ltimes\sspan(\phi_7,\phi_8,\phi_{12}\big)\Big]
\eeqn

}
\eex

\subsection{Possible Applications} \label{application-prederivations}

The existence of affine structures is a difficult and interesting problem (\cite{helmsteter},\cite{milnor3}, 
\cite{lichnerowicz-medina88}, \cite{lichnerowicz91}).

\vskip 0.5cm

Let $\G$ be a Lie algebra. If $D$ is an invertible derivation of $\G$, one defines an affine structure by the formula
\beq 
\nabla_xy = D^{-1} \circ ad_x \circ D(y),
\eeq 
for every $x$ and $y$ in $\G$. Unfortunately, there are Lie algebras that admits only singular derivations. 
Nilpotent such algebras are called characteristically nilpotent Lie algebras.

\vskip 0.5cm

If the Lie algebra $\G$ admits no regular derivation, then one uses a regular prederivation whenever it exists. The 
idea of the construction is the following (\cite{burde}). Given an invertible prederivation $P$, 
set $\omega(x,y) = P([x,y]) - [x,P(y)] - [P(x),y]$, for all $x,y \in \G$. Now we define
\beq\label{prederivation-affine-structure} 
\nabla_xy = P^{-1} \circ ad_x \circ P(y) + \frac{1}{2}P^{-1} \circ \omega(x,y),
\eeq 
for every $x,y \in \G$. In general the map $x \mapsto P^{-1} \circ ad_x \circ  + \frac{1}{2}P^{-1} \circ \omega(x,\cdot)$ might not 
be a representation. However, in many cases, (\ref{prederivation-affine-structure}) gives rise to an affine structure.

\chapter{Skew-symmetric Prederivations and Bi-invariant Metrics of Cotangent Bundles of Lie Groups}
\minitoc

\section{Introduction}

A Lie group endowed with a pseudo-Riemannian metric which is invariant under both 
left and right translations, is called an orthogonal, a bi-invariant or a quadratic Lie group. 
The corresponding Lie algebra is called an orthogonal or quadratic Lie algebra. 

\vskip 0.3cm

Such Lie groups  are important in Mathematics and in Physics. These objects are, for instance, useful in 
pseudo-Riemannian geometry, in the theory of Poisson-Lie groups, in relativity, in the theory of 
Hamiltonian systems,... (\cite{babelon-viallet}, \cite{bajo-benayadi-medina}, \cite{bordemann}, 
\cite{diaz-gadea-oubina}, \cite{drinfeld}, \cite{gadea-oubina99}, \cite{levichev}).

\vskip 0.3cm

According to the works of Medina and Revoy (\cite{me-re93},\cite{me-re85}), any orthogonal Lie algebra is obtained by 
the so-called double extension procedure. This is the case of orthogonal Lie algebras called hyperbolic and  
whose quadratic form is of signature $(n,n)$ (where $2n$ is the dimension of the concerned Lie algebra). 
One particularly interesting case is the one  where the Lie algebra admits two totally isotropic
subalgebras which are in duality. These latter Lie algebras named Manin-Lie algebras describe simply 
connected Poisson-Lie groups. 

\vskip 0.3cm

It is known that the cotangent bundle $T^*G$ of a Lie group $G$ with Lie algebra $\G$, considered with its 
Lie group structure obtained by semi-direct product of the Lie group $G$ and the Abelian Lie group
$\G^*$ (dual of $\G$) by means of the co-adjoint representation, possesses a hyperbolic  metric : 
the duality pairing. 

\vskip 0.3cm

Here we seek to characterize, up to isometric automorphisms,  all orthogonal structures on $T^*\G$
by means of the duality pairing and adjoint-invariant endomorphisms; hence,  all bi-invariant metrics on
$T^*G$; and to determine the corresponding group of isometries.  We 
will focus our study on orthogonal Lie algebras and on the more particular class of semi-simple Lie algebras.

\vskip 0.3cm

We will take $\G$ to be the Lie algebra of a Lie group $G$, $\G^*$ will be the dual space of $\G$ 
and $T^*\G:=\G\ltimes \G^*$ will be the semi-direct sum of the Lie algebra $\G$ and the vector space $\G^*$ 
via the coadjoint representation. Recall that $\pder(\G)$ stands for the Lie algebra of all prederivations 
of the Lie algebra $\G$ while $\mathcal J'$ is the set of all endomorphisms $j$ of $\G$ which satisfy 
$j(\big[[x,y],z\big])=\big[[x,y],j(z)\big]$, for any $x,y,z$ in $\G$.

\vskip 0.3cm

Among others, here are some of the important results contained in this chapter. 

\vskip 0.3cm
\noindent
{\bf Theorem~A} {\em
\benum 
\item Let $\G$ be a Lie algebra. Any orthogonal structure $\mu$ on $T^*\G$ is given by
\beq
\mu\Big((x,f),(y,g)\Big) = \big\langle g,j_{11}(x)\big\rangle + \big\langle f,j_{11}(y)\big\rangle 
+ \big\langle g,j_{21}(f)\big\rangle + \big\langle j_{12}(x),y\big\rangle,
\eeq
where $j_{11} : \G \to \G$, $j_{12} : \G \to \G^*$, $j_{21} : \G^* \to \G$ are as in 
Proposition \ref{prop:symmetric-bi-invariant-endom}.


\item If $(\G,\mu)$ is an orthogonal Lie algebra, then any orthogonal structure $\mu_D$ on $T^*\G$ has the form
\beq
\mu_\D\Big((x,f),(y,g)\Big)=\big\langle g\;,\;j_{11}(x) \big\rangle + \big\langle f\;,\;j_{11}(y)\big\rangle 
+ \big\langle g,j_2 \circ \theta^{-1}(f) \big\rangle + \big\langle \theta\circ j_1(x),y\big\rangle,
\eeq 
for all $(x,f)$, $(y,g)$ in $T^*\G$, where $\theta$ is the isomorphism 
induced by the $\mu$ through the formula (\ref{isomorphism-theta}) and $j_{11}$, $j_1$, $j_2$ satisfy 
conditions listed in Lemma \ref{lemma:symmetric-invertible-bi-invariant-endomorphism-orthogonal-lie-algebras}.


\item Let $\G$ be a semi-simple Lie algebra.  Any orthogonal structure $\mu$ on $T^*\G$ is given by 
\beq 
\mu\big((x,f),(y,g)\big) = \sum_{i=1}^p \lambda_i\big\langle(x_i,f_i),(y_i,g_i)\big\rangle_{\s_i} 
+ \sum_{k=1}^p\nu_kK_k(x_k,y_k),
\eeq 
for all $x, y \in \G=\s_1\oplus \s_2\oplus\cdots\oplus\s_p$ and $f, g\in \G^*=\s_1^*\oplus \s_2^*\oplus\cdots\oplus\s_p^*$ ;
where $\lambda_i \in \R^*$, $\nu_i\in \R$ and $K_i$ stands for the Killing form on $\s_i$ for all $i=1,2,\ldots,p$.
\eenum 
}

\vskip 0.3cm

\noindent
{\bf Theorem~B} ~ {\em
\benum 
\item Let $\G$ be a Lie algebra and let $\mu_\D$ be an orthogonal structure on $T^*\G$.  
A $\mu_\D$-skew-symmetric prederivation $\phi$ of $T^*\G$ has the form
\beq
\phi(x,f) = \Big(\alpha(x) +\psi(f)\,,\, \beta(x) + (j'^t-\alpha^t)(f) \Big),
\eeq 
where $j':\G \to \G$ is in $\mathcal J'$, $\alpha : \G \to \G$ is in $\pder(\G)$, 
$\beta : \G \to \G^*$ and  $\psi : \G^* \to \G$ are as in Theorem \ref{theorem:characterization-prederivations},   
with the additional conditions listed in Proposition \ref{prop:skew-symmetric-prederivation}.
\item Let $(\G,\mu)$ be an orthogonal Lie algebra. Then any prederivation $\phi$ of $T^*\G$ which is skew-symmetric 
with respect to any orthogonal structure $\mu_\D$ on $T^*\G$ can be written as follows
\beq 
\phi(x,f)= \Big(\alpha_1(x)+j_1'\circ \theta^{-1}(f) \;,\; \theta \circ \alpha_2(x) + ({j_2'}^t-\alpha_1^t)(f)\Big),
\eeq 
where $\alpha_1, \alpha_2$ are in $\pder(\G)$; $j_1',j_2'$ are in $\mathcal J'$ with the additional conditions
listed in Proposition \ref{prop:skew-symmetric-prederivation-orthogonal-lie-algebra}.
\item Let $\G$ a semi-simple Lie algebra. Then any prederivation of $T^*\G$ which is
skew-symmetric with respect to any orthogonal structure $\mu_\D$ on $T^*\G$ is an inner
derivation of $T^*\G$. 
\eenum

}

In Section \ref{chap3:section:preliminaries} is given some basic notions useful to make the chapter understandable. 
Section \ref{chap3:section:biinvariant-metrics} is dedicated to the characterization of orthogonal structures on 
the Lie algebras of cotangent bundles of Lie groups and their groups of isometries. In Sections 
\ref{chap3:section:orthogonal-lie-algebras} and \ref{chap3:section:semi-simple-lie-algebras} are respectively studied
the case of orthogonal Lie algebras and the case of semi-simple Lie algebras. The chapter finishes with some examples 
given in Section \ref{chap3:section:example}.

\section{Preliminaries}\label{chap3:section:preliminaries}
\subsection{Orthogonal Structures and Bi-invariant Endomorphisms}

Let $(G,\mu)$ be an orthogonal Lie group with Lie algebra $\G$. The metric $\mu$ induces on $\G$ an adjoint-invariant 
symmetric non-degenerate bilinear form $\langle,\rangle$, {\it i.e.} a symmetric and non-degenerate form $\langle,\rangle$ such that
\beq 
\langle[x,y],z\rangle + \langle y,[x,z]\rangle = 0,
\eeq 
for all $x,y$ in $\G$. The form $\langle,\rangle$ is called an orthogonal structure on $\G$.
It is well known (\cite{me-re93}) that any other 
non-degenerate symmetric bilinear form $B$ on $\G$ can be written as 
\beq\label{relation-bilinear-form-j} 
B(x,y)=\langle j(x),y\rangle,
\eeq 
for all $x,y$ in $\G$, where $j$ is a $\langle,\rangle$-symmetric automorphism of the vector space $\G$. Furthermore, 
the bilinear form $B$ is adjoint-invariant, {\it i.e.} is an orthogonal structure, if and only if $j$ commutes 
with all the adjoint operators $ad_x$, ($x \in \G$); that is 
\beq 
j([x,y]) = [j(x),y] = [x,j(y)], \label{bi-inv-endom}
\eeq 
for every $x,y$ in $\G$.

\vskip 0.3cm 

An endomorphism $j$ of $\G$ which satisfies  Relation (\ref{bi-inv-endom}) is called a 
bi-invariant or an adjoint-invariant endomorphism.

\vskip 0.3cm

From what is said above it comes that if one orthogonal structure $\langle,\rangle$ is known on $\G$ then we will be able to characterize 
all the orthogonal structures on $\G$ by characterizing all the $\langle,\rangle$-symmetric bi-invariant tensors on $\G$.

\subsection{Isometries of Bi-invariant Metrics of Lie Groups}

Let us begin by some reminders. 
\bdfn
Let $M$ be a smooth manifold equipped with a pseudo-Riemannian metric $\mu$. 
An isometry of the pseudo-Riemannian manifold $(M,\mu)$ is a diffeomorphism $f:M \to M$ such that $f^*\mu=\mu$, 
where $f^*$ stands for the "pull-back" via $f$. More precisely, for all $x$ in $M$ and any vectors $X_{|x},Y_{|x}$ in the tangent space $T_xM$,
\beq 
\mu_{f(x)}\Big(T_xf\cdot X_{|x},T_xf\cdot Y_{|x}\Big) = \mu_x\big(X_{|x},Y_{|x}\big).
\eeq 
\edfn

We denote by $I(M,\mu)$ the set of all isometries of $(M,\mu)$. Nomizu shows in \cite{nomizu} that the set $\hbox{\rm Aff}(M,\nabla)$ of all
affine transformations of the induced affine connection on $M$, with the compact-open topology, is a Lie transformation group.  
But $I(M,\mu)$ is a closed subgroup of $\hbox{\rm Aff}(M,\nabla)$. Then, equipped with the compact-open topology, $I(M,\mu)$ is a Lie transformation group.

\vskip 0.3cm

Let $(G,\mu)$ be an orthogonal Lie group with Lie algebra $\G$. We note by $I(G,\mu)=I(G)$ the set of all isometries of $(G,\mu)$, by 
$F(G)$ the set of those isometries which fix the identity element $\epsilon$ of $G$ and by $L_G$ the set of all left translations of $G$. 
We recall the following result due to M\"uller.
\blem(\cite{muller})
Let $G$ be a connected Lie group endowed with a bi-invariant metric.
\benum
\item $F(G)$ is a closed Lie subgroup of $I(G)$ ; 
\item $L_G$ is a closed connected subgroup isomorphic to $G$ ;
\item $I(G)=L_GF(G)$, with $L_G\cap F(G)=\{id\}$;  $id:G \to G$ is the identity map of $G$ ;
\item the manifolds $I(G)$ and  $G\times F(G)$ are diffeomorphic.
\eenum
\elem

\subsection{Isometries of Bi-invariant Metrics and Preautomorphisms}\label{section:isometries-preautomorphisms}

Let $(G,\mu)$ be an orthogonal Lie group with unit element $\epsilon$ and let  $\G$ be its Lie algebra. We will often
note by $\langle,\rangle :=\mu_{\mid \epsilon}$ the resulting non-degenerate adjoint-invariant bilinear form on $\G$.
The fact that $G$ is orthogonal is equivalent to the one that the geodesics through $\epsilon$ are the one-parameter
subgroups of $G$ (see \cite[Exercise 5, page $148$]{helgason}).

\vskip 0.3cm

Let $\exp : \G \to G$ denote the exponential map of the Lie group $G$. Consider an open neighborhood $U$ of $0$ in $\G$
such that the restriction $\exp_{\mid U} : U \to \exp(U)$ of $\exp$ to $U$ is a diffeomorphism and note by
$\log : \exp(U) \to U$ the inverse of $\exp_{\mid U}$.

\bdfn
A local isometry at $\epsilon$ is a diffeomorphism $\varphi:V_1 \to V_2$ between two open neighborhoods
$V_1$ and $V_2$ of $\epsilon$ in $G$ such that
\bitem
\item $\varphi$ fixes $\epsilon$, {\it i.e.} $\varphi(\epsilon)=\epsilon$ ;
\item $\varphi^*\mu = \mu$ on $V_1$, where $\varphi^*$ stands for the pull-back via $\varphi$.
\eitem
\edfn
Now, let $\varphi$ be a local isometry at $\epsilon$ and $x$  any arbitrary element of $\G$. The local 
isometry $\varphi$ maps any geodesic through $\epsilon$ onto a geodesic through $\epsilon$ on a 
neighborhood of $\epsilon$. Then there exists an element $y$ of $\G$ such that
\beq \label{geodesics-conservation}
\varphi\big(\exp(tx)\big) = \exp(ty),
\eeq
for $t$ small enough. Derivating relation (\ref{geodesics-conservation}) with respect to $t$ at $t=0$, one has
\beq\label{tangent-map-varphi}
T_\epsilon \varphi \cdot x = y,
\eeq
where $T_\epsilon\varphi$ is the tangent linear map of $\varphi$ at $\epsilon$. From relation (\ref{geodesics-conservation}) 
we have
\beq\label{y-expression}
y=\log \circ \varphi \circ \exp(x).
\eeq
Relations (\ref{geodesics-conservation}) and (\ref{tangent-map-varphi}) give the following nice formula
\beq
\varphi = \exp \circ T_\epsilon\varphi \circ \log
\eeq
on a suitable neighborhood of $\epsilon$ in $\exp(U)$.

\vskip 0.5cm

Thus if we identify two local isometries at $\epsilon$ that agree on a neighborhood of $\epsilon$, 
it comes that a local isometry at $\epsilon$ is uniquely determine by its differential at $\epsilon$. 
A local isometry at $\epsilon$ can be uniquely extended to an isometry on $G$. Indeed, it is well known 
that if $M$ and $N$ are two connected simply connected and geodesically complete pseudo-Riemannian
manifolds, then every isometry between connected open subsets of $M$ and $N$ can be uniquely extended 
to an isometry between $M$ and $N$ (\cite{kobayashi-nomizu}, \cite{barett},  \cite{wolf}). 

\vskip 0.5cm

Now we have the following theorem which establishes a link between local isometries at $\epsilon$ and 
the so-called preautomorphisms of the Lie algebra $\G$ of $G$.

\bthm(\cite{muller})
Let $(G,\mu)$ be a connected orthogonal Lie group with Lie algebra $\G$. Let $\langle , \rangle:=\mu_{\mid \epsilon}$
stands for the resulting orthogonal structure on $\G$ and let $P$ be an endomorphism of $\G$.
Then there exists a local isometry $\varphi$ of $G$ at $\epsilon$ with $T_\epsilon \varphi = P$ if and only if
$P$ satisfies
\benum
\item $\langle P(x),P(y) \rangle = \langle x,y \rangle$, for any $x,y$ in $\G$ ;
\item $P\big(\big[x,[y,z]\big] \big) = \big[P(x),[P(y),P(z)]\big]$, for all $x,y,z$ in $\G$.
\eenum
\ethm

Now, the author quoted above proves that once the Lie algebra of $F(G)$ is known one can readily construct the Lie algebra of $I(G)$. Here is his construction.

For $g$ in $G$ we define the following maps : 
\beq
\begin{array}{rclcrrrcrrc}
 L_g:G &\to    & G   & ; & R_g:G &\to    & G  &; & I_g:G &\to    & G \cr
h     &\mapsto& gh   &  &     h &\mapsto& hg & &     h &\mapsto& ghg^{-1}
\end{array}
\eeq
This maps are respectively the left transformation, the right translation and the inner automorphism by the element $g$ of $G$.

\vskip 0.3cm 

Let $\xi$ be in $\G$ and $\exp :\G \to G$ be the exponential map of $G$. 
Then, $L_{\exp(t\xi)}$, $R_{\exp(t\xi)}$, $I_{\exp(t\xi)}$ are one-parameter subgroups
of $I(G)$. Let us note by $X^{\xi,L}$, $X^{\xi,R}$ and $X^{\xi,I}$ their respective infinitesimal generators.
Now, we set
\beq 
X^{\xi,s}=X^{\xi,R} + X^{\xi,L} \quad \text{ and }\quad  X^{\xi,a}=X^{\xi,R} - X^{\xi,L}.
\eeq 
If $\{\alpha_1, \alpha_2,\ldots,\alpha_m\}$ is a basis of the Lie algebra $\mathcal F(G)$ of $F(G)$ and 
$\{\xi_1,\xi_2,\ldots,\xi_n\}$ a basis of $\G$ then $\{\alpha_1, \alpha_2,\ldots,\alpha_m,X^{\xi_1,s},X^{\xi_2,s},
\ldots,X^{\xi_n,s}\}$ is a basis of the Lie algebra $\mathcal I(G)$.

\vskip 0.3cm

It remains to compute the brackets on $\mathcal I(G)$. Let us first define some useful objects.

\vskip 0.3cm 

Let $\paut(\G)$ be the group of preautomorphisms of $\G$ and denote by $F(\G)$ the subgroup of 
$\paut(\G)$ consisting of those preautomorphisms which preserve the orthogonal structure on $\G$. Now
consider the map $T_\epsilon :F(G) \to F(\G)$ which associates
to any element $\phi$ of $F(G)$ its differential $T_\epsilon\phi$ at the neutral element $\epsilon$ of $G$.
The map  $T_\epsilon$ induces a map  $\partial_\epsilon :\mathcal F(G) \to \mathcal F(\G)$ between the Lie algebras
$\mathcal F(G)$ and $\mathcal F(\G)$ of $F(G)$ and $F(\G)$ respectively :
\beq 
\partial_\epsilon D = \left(\frac{d}{dt} T_\epsilon[\exp_{I(G)}(tD)]\right)_{|t=0}
\eeq 
Note that $\mathcal F(\G)$ consists of prederivations of $\G$ which are skew-symmetric with respect to the orthogonal
structure on $\G$. 

\vskip 0.3cm 

The following theorem gives the brackets on $\mathcal I(G)$.

\bthm(\cite{muller})\label{thm:muller-brackets-isometries-algebra}
Let $\xi$, $\eta$ be in $\G$ and $D$ be an element of $\mathcal F(G)$. Then,
\benum
\item $[X^{\xi,s},X^{\eta,s}] = [X^{\xi,a},X^{\eta,a}] = -X^{[\xi,\eta],a}$ ;
\item $[D,X^{\xi,s}] = X^{\partial_\epsilon D(\xi),s}$ ;
\eenum
\ethm

Now we conclude that if we know the prederivations of $\G$ which are skew-symmetric with respect to the 
orthogonal structure on $\G$ we can calculate the Lie algebra $\mathcal I(G)$ of the group $I(G)$ of isometries
of the orthogonal Lie group $(G,\mu)$.

\section{Bi-invariant Metrics on $T^*G$}\label{chap3:section:biinvariant-metrics}

Let $G$ be a Lie group with Lie algebra $\G$. We have already seen (Section \ref{chap:notations}) 
that the duality pairing $\langle,\rangle$ given by 
\beq 
\langle (x,f),(y,g)\rangle = f(y) + g(x),
\eeq 
for all $x,y$ in $\G$, defines an orthogonal structure on the Lie algebra $T^*\G=\G\ltimes\G^*$ of the Lie group $T^*G$.

\subsection{Bi-invariant Tensors on $T^*\G$}

Let $\langle,\rangle$ stand for the duality pairing on $T^*\G$, a bi-invariant metric on $T^*G$ is given by 
an invertible linear map
$j:T^*\mathcal G \to T^*\mathcal G$ satisfying
\begin{eqnarray}
j([u,v]) = [u,j(v)] ;\label{bi-invariant-endomorphism}\\ 
\langle j(u),v \rangle = \langle u,j(v) \rangle,  
\end{eqnarray}
for  all $u,v \in T^*\mathcal G$. 

\blem\label{lemma:invertible-bi-invariant-endomorphism}
Any linear map $j : T^*\G \to T^*\G$ satisfying (\ref{bi-invariant-endomorphism}) can be written as
\beq\label{expression:endomorphism-j}
j(x,f) = \Big(j_{11}(x) + j_{21}(f) \,,\,j_{12}(x) + j_{22}(f) \Big),
\eeq 
for any $(x,f) \in T^*\G$ ;   where $j_{11} : \G \to \G$, $j_{12} : \G \to \G^*$,
$j_{21} : \G^* \to \G$ and $j_{22}:\G^* \to \G^*$ are linear maps such that for all $x$ in $\G$, the following relations hold :  
\beqn
j_{11}\circ ad_x &=& ad_x \circ j_{11} \label{relationJ11}\\
j_{12}\circ ad_x &=& ad^*_x \circ j_{12} \label{relationJ12}\\
j_{21} \circ ad^*_x &=& ad_x \circ j_{21}=0 \label{relationJ21}\\
\big[j_{22},ad_x^*\big] &=& 0  \label{relationJ22}\\
ad^*_x\circ (j_{22}-j_{11}^t) &=& 0, \label{relationJ22-J11}
\eeqn
where $j^t_{11}$ stands for the transpose of the map $j_{11}$.
\elem 
\begin{proof}
 
If we write $u=(x,f)$ and $v=(y,g)$, then
\begin{eqnarray}
j([u,v])&=& j\big([x,y], ad^*_xg-ad^*_yf\big)\nonumber\\
        &=& \Big( j_{11}([x,y])+j_{21}(ad^*_xg-ad^*_yf)\,,\, j_{12}([x,y])+j_{22}(ad^*_xg-ad^*_yf) \Big)\label{derJ1}
\end{eqnarray}
and 
\begin{eqnarray}
[u,jv]&=& [(x,f)\,,\,\big(j_{11}(y)+j_{21}(g),j_{12}(y)+j_{22}(g)\big)]\nonumber\\
      &=& \Big([x,j_{11}(y)]+[x,j_{21}(g)]\, , \, ad^*_x\big(j_{12}(y)+j_{22}(g)\big) - ad^*_{j_{11}(y)+j_{21}(g)}f \Big)
\label{derJ2}
\end{eqnarray}
Considering the case where  $f=g=0$ the equality between (\ref{derJ2}) and (\ref{derJ3}) gives
\begin{eqnarray}
j_{11}([x,y])=[x,j_{11}(y)]  \text{~and~} j_{12}([x,y]) = ad^*_x(j_{12}(y)), \label{derJ3} \nonumber
\end{eqnarray}
for all $x,y\in\mathcal G$. The two relations above are equivalent to  (\ref{relationJ11}) 
and (\ref{relationJ12}) respectively.

\vskip 0.3cm
\noindent
The equality between (\ref{derJ1}) and (\ref{derJ2}) now gives
\begin{equation} \label{derj1}
\Big(j_{21}(ad^*_xg-ad^*_yf)\,,\, j_{22}(ad^*_xg-ad^*_yf) \Big)= \Big([x,j_{21}(g)]\,,\, 
ad^*_x\big(j_{22}(g)\big) - ad^*_{j_{11}(y)+j_{21}(g)}f \Big).
\end{equation}
Taking $f=0,$ we get on one hand
\begin{eqnarray}
j_{21}(ad^*_xg)= [x,j_{21}(g)], \nonumber
\end{eqnarray}
 for all $x\in\mathcal G$ and all $g\in\mathcal G^*$. Then Relation (\ref{relationJ21}) is proved. 
On the second hand we have
\begin{eqnarray}
j_{22}(ad^*_xg)=ad^*_x(j_{22}(g)), \nonumber
\end{eqnarray}
for every $g$ in $\G^*$ ; that is
\begin{equation}
j_{22}\circ ad^*_x=ad^*_x\circ j_{22}, \label{relationJ22-1}
\end{equation}
which is nothing but Relation (\ref{relationJ22}).
Now if we take $x=0$, we get from Eq. (\ref{derj1})
\begin{equation}
\Big(j_{21}(-ad^*_yf), j_{22}(-ad^*_yf) \Big)= \big(0, - ad^*_{j_{11}(y)+j_{21}(g)}f\big).
\end{equation}
The equality above implies the following two relations
\begin{eqnarray}
j_{21}(-ad^*_yf) &=& 0,  \nonumber \\
j_{22}(-ad^*_yf) &=& - ad^*_{j_{11}(y)+j_{21}(g)}f, \nonumber
\end{eqnarray}
for all $y\in \mathcal G$,  $f\in\mathcal G^*$. These equations are equivalent to
\begin{eqnarray}
j_{21}\circ ad^*_y &=& 0, \\
j_{22}\circ ad^*_y &=& ad^*_{j_{11}(y)}, \label{relationJ22-J11-1}\\
ad^*_{j_{21}(g)}   &=& 0, 
\end{eqnarray}
for all $y\in \mathcal G$ and  $g\in\mathcal G^*$. The relations
\begin{eqnarray}
ad^*_y\circ   j_{22} &=& j_{22}\circ ad^*_y \cr
                     &=& ad^*_{j_{11}(y)}  \cr
                     &=& ad^*_y\circ j_{11}^t, \label{guimbi}
\end{eqnarray}
for all $y\in \mathcal G$, coming from (\ref{relationJ22-1}) and  (\ref{relationJ22-J11-1}), give
\begin{eqnarray} 
ad^*_y\circ (  j_{22}-j_{11}^t)=0,
\end{eqnarray}
for all $y\in \mathcal G$ ;  which means that $(j_{22}-j_{11}^t)f$ is a closed $1$-form, for every $f\in\mathcal G^*$.
\end{proof}

\brmq \label{rmq:symmetric-bi-invariant-endom}
\benum
\item The equality $ad^*_{j_{21}(g)}=0$, for all $g$ is equivalent to
\begin{equation}
Im(j_{21}) \subset Z(\mathcal G),
\end{equation}
where  $Z(\mathcal G)$ is the centre of  $\mathcal G$. So if $Z(\mathcal G)=\{0\}$ then $j_{21}=0.$
 \item Relation (\ref{guimbi})  also gives
\begin{eqnarray}
-\langle j_{22}(f),[y,x]\rangle &=& \langle j_{22}(ad^*_y(f)),x \rangle \cr 
                                &=& -\langle f,[j_{11}(y),x] \rangle  \cr 
                                &=& -\langle f\circ j_{11},[y,x]\rangle,
\end{eqnarray}
for all $f$ in $\G^*$ and any $x,y$ in $\G$.
\eenum 
\ermq

\blem\label{lemma:endomorphism-j-symmetric}
Any linear map $j : T^*\G \to T^*\G$ which is $\langle,\rangle$-symmetric  can be written as
\beq \label{expression:endomorphism-j-symmetric}
j(x,f) = \Big(j_{11}(x) + j_{21}(f) \,,\,j_{12}(x) + j^t_{11}(f) \Big),
\eeq 
for any $(x,f) \in T^*\G$ ;   where $j_{11} : \G \to \G$, $j_{12} : \G \to \G^*$,
$j_{21} : \G^* \to \G$ are linear maps  with the following relations valid for all $x, y$ in $\G$ and all
$f, g$ in $\G^*$.
\beqn
\langle j_{12}(x),y \rangle &=& \langle x,j_{12}(y) \rangle \label{symJ12}\\
\langle j_{21}(f),g \rangle &=& \langle f,j_{21}(g) \rangle. \label{symJ21}
\eeqn
\elem
\begin{proof}
 Let $u=(x,f)$ and $v=(y,g)$ be two elements of $T^*\G$.
\begin{eqnarray}
\langle ju,v \rangle &=& \Big\langle \Big(j_{11}(x)+j_{21}(f),j_{12}(x)+j_{22}(f)\Big),(y,g) \Big\rangle \nonumber\\
                     &=& \langle j_{11}(x),g \rangle + \langle j_{21}(f),g \rangle + \langle j_{12}(x),y \rangle
                         + \langle j_{22}(f),y\rangle \label{derj2}
\end{eqnarray}
\begin{eqnarray}
\langle u,jv \rangle &=& \Big\langle (x,f), \Big(j_{11}(y)+j_{21}(g),j_{12}(y)+j_{22}(g) \Big) \Big\rangle \nonumber \\
                     &=& \langle f,j_{11}(y)\rangle + \langle f,j_{21}(g) \rangle
                        + \langle x,j_{12}(y) \rangle + \langle x,j_{22}(g)\rangle \label{derj3}
\end{eqnarray}
For $f=g=0$, the equality between (\ref{derj2}) and (\ref{derj3}) implies that for all $x,y$ in $\G$, 
\begin{equation}
 \langle j_{12}(x),y \rangle = \langle x,j_{12}(y) \rangle. \nonumber
\end{equation}
So (\ref{symJ12}) is established. By the same way, for $x=y=0$, Relation (\ref{symJ21}) is obtained, as
\begin{equation}
\langle j_{21}(f),g \rangle =\langle f,j_{21}(g) \rangle. \nonumber
\end{equation}
Now the equality between (\ref{derj2}) and (\ref{derj3}) can be written 
\begin{equation}
\langle j_{11}(x),g \rangle + \langle j_{22}(f),y \rangle = \langle f,j_{11}(y) \rangle + \langle x,j_{22}(g) \rangle. \nonumber
\end{equation}
We take $y=0$  to obtain
\begin{equation}
\langle x,j_{11}^t(g) \rangle = \langle j_{11}(x),g \rangle = \langle x,j_{22}(g) \rangle \nonumber
\end{equation}
for all $x$ in $\mathcal G$ and  $g$ in $\mathcal \G^*$; or equivalently
\begin{equation}
j_{11}^t=j_{22}  \nonumber 
\end{equation}
\end{proof}

\brmq
Relations (\ref{symJ12}) and (\ref{symJ21}) mean that $j_{12}$ and $j_{21}$ are symmetric with respect to
the duality.
\ermq

\blem
Let $\G$ be a Lie algebra without centre, that is $Z(\G)=\{0\}$. A $\langle,\rangle$-symmetric bi-invariant endomorphism $j:T^*\G \to T^*\G$ defined 
by (\ref{expression:endomorphism-j-symmetric}) is invertible if and only if $j_{11}$ is invertible.
\elem 

\begin{proof}
Let $j:T^*\G \to T^*\G$ be defined by (\ref{expression:endomorphism-j-symmetric}) 
with conditions listed in Lemma \ref{lemma:invertible-bi-invariant-endomorphism}.
We have seen in Remark \ref{rmq:symmetric-bi-invariant-endom} that if $Z(\G)=\{0\}$ then 
$j_{21}=0$. The determinant of $j$ is 
\beq 
\det(j) = \left|
\begin{array}{cc}
j_{11} & j_{21} \cr
j_{12} & j_{11}^t 
\end{array}
\right| = \left|
\begin{array}{cc}
j_{11} & 0 \cr
j_{12} & j_{11}^t 
\end{array}
\right| = \big(\det(j_{11})\big)^2.
\eeq
%
\end{proof}

Let us summarize the above three Lemmas in the 

\bpro\label{prop:symmetric-bi-invariant-endom}
Let $\G$ be a Lie algebra. Any $\langle,\rangle$-symmetric and adjoint-invariant tensor $j : T^*\G \to T^*\G$ has the form
\beq\label{expression:symmetric-bi-invariant-invertible-endomorphism}
j(x,f) = \Big(j_{11}(x) + j_{21}(f) \,,\,j_{12}(x) + j^t_{11}(f) \Big),
\eeq
for any $(x,f) \in T^*\G$ ;   where $j_{11} : \G \to \G$, $j_{12} : \G \to \G^*$,
$j_{21} : \G^* \to \G$ are linear maps  with the following relations valid for all $x, y$ in $\G$ and all
$f, g$ in $\G^*$.
\beqn
j_{11}\circ ad_x &=& ad_x \circ j_{11} \\
j_{12}\circ ad_x &=& ad^*_x \circ j_{21} \\
j_{21} \circ ad^*_x &=& ad_x \circ j_{21} =0 \label{equiv-j21}\\
\langle j_{12}(x),y \rangle &=& \langle x,j_{12}(y) \label{symJ12-1}\rangle \\
\langle j_{21}(f),g \rangle &=& \langle f,j_{21}(g) \label{symJ21-1}\rangle.
\eeqn
If, in addition, the centre $Z(\G)$ of the Lie algebra is $\{0\}$, then the endomorphism $j$ is 
invertible if and only if $j_{11}$ is invertible. 
\epro

\subsection{Orthogonal Structures on $T^*\G$}

In the sequel we will use the duality pairing on $T^*\G$ and Relation (\ref{relation-bilinear-form-j}) in order 
to determine all non-degenerate symmetric and adjoint-invariant form on $T^*\G$, hence all bi-invariant
metrics on $T^*G$. Precisely, any other orthogonal structure $\mu$ on $T^*\G$ is linked to the duality pairing $\langle,\rangle$ by
\beq 
\mu\Big((x,f),(y,g) \Big) = \Big\langle j(x,f),(y,g)\Big\rangle, 
\eeq 
for every $(x,f),(y,g)$ in $T^*\G$; where $j$ is an endomorphism of $T^*\G$ such that :
\bitem 
\item $j$ is invertible;
\item $j$ is symmetric with respect to the duality, {\it i.e. } $\langle j(u),v\rangle = \langle u,j(v)\rangle$, for all $u,v \in T^*\G$;
\item $j([u,v])=[j(u),v]=[u,j(v)]$, for all $u,v \in T^*\G$.
\eitem 

\vskip 0.3cm

We have the following characterization of all orthogonal structures on $T^*\G$.

\bthm\label{thm:orthogonal-structure}
Let $\G$ be a Lie algebra. Any orthogonal structure $\mu$ on $T^*\G$ is given by
\beq\label{expression:orthogonal-structure}
\mu\Big((x,f),(y,g)\Big) = \big\langle g,j_{11}(x)\big\rangle + \big\langle f,j_{11}(y)\big\rangle 
+ \big\langle g,j_{21}(f)\big\rangle + \big\langle j_{12}(x),y\big\rangle,
\eeq
where $j_{11} : \G \to \G$, $j_{12} : \G \to \G^*$, $j_{21} : \G^* \to \G$ are as in 
Proposition \ref{prop:symmetric-bi-invariant-endom}.
\ethm 

\begin{proof}
 Indeed, an orthogonal structure $\mu$ on $T^*\G$ is defined by 
\beq
\mu\Big((x,f),(y,g) \Big) = \Big\langle j(x,f),(y,g)\Big\rangle, \nonumber
\eeq 
where $j$ is given by (\ref{expression:symmetric-bi-invariant-invertible-endomorphism}) and 
condition listed in Proposition \ref{prop:symmetric-bi-invariant-endom}. We then have
\beqn
\mu\Big((x,f),(y,g) \Big) &=& \Big\langle \Big(j_{11}(x) + j_{21}(f) \,,\,j_{12}(x) + j^t_{11}(f) \Big),(y,g)\Big\rangle \cr
                          &=& \Big\langle g,j_{11}(x) + j_{21}(f)\Big\rangle + \Big\langle j_{12}(x) + j^t_{11}(f),y\Big\rangle \cr
                          &=& \Big\langle g,j_{11}(x)\Big\rangle + \Big\langle g,j_{21}(f)\Big\rangle + \Big\langle j_{12}(x),y\Big\rangle
                             + \Big\langle j^t_{11}(f),y\Big\rangle \cr 
                          &=& \Big\langle g,j_{11}(x)\Big\rangle + \Big\langle g,j_{21}(f)\Big\rangle + \Big\langle j_{12}(x),y\Big\rangle
                             + \Big\langle f,j_{11}(y)\Big\rangle \nonumber
\eeqn 
\end{proof}

\subsection{Skew-symmetric Prederivations on $T^*\G$}

Let $\G$ be a Lie algebra. On $T^*\G$ we consider the orthogonal structure $\mu$ defined 
by (\ref{expression:orthogonal-structure}). 
We seek  to characterise all prederivations of $T^*\G$ which are skew-symmetric with respect to t
he orthogonal structure  $\mu$.

\vskip 0.3cm

Let us first recall the space $\mathcal J'=\{j':\G \to \G \text{ linear } :[j',ad_x\circ ad_y]=0,\forall x,y \in \G\}$.

\bpro\label{prop:skew-symmetric-prederivation}
A $\mu$-skew-symmetric prederivation $\phi$ of $T^*\G$ has the form
\beq
\phi(x,f) = \Big(\alpha(x) +\psi(f)\,,\, \beta(x) + (j'^t-\alpha^t)(f) \Big),
\eeq 
where $j':\G \to \G$ is in $\mathcal J'$, $\alpha : \G \to \G$ is a prederivation of $\G$, 
$\beta : \G \to \G^*$ and  $\psi : \G^* \to \G$ are as in 
Theorem \ref{theorem:characterization-prederivations},   with the additional conditions
\beqn
\langle \beta(x),j_{11}(y)\rangle + \langle \beta(y),j_{11}(x) \rangle 
+ \langle j_{12}\circ \alpha(x),y\rangle + \langle j_{12}\circ \alpha(y),x\rangle &=& 0 \label{first-relation}\\
\langle f,j_{11}\circ \psi(g)\rangle + \langle g,j_{11}\circ \psi(f)\rangle + \langle f,j'\circ j_{21}(g)\rangle \cr
+\langle g,j'\circ j_{21}(f)\rangle - \langle f,\alpha\circ j_{21}(g)\rangle 
- \langle g,\alpha\circ j_{21}(f)\rangle &=& 0 \label{second-relation}\\
\big[ j_{11},\alpha\big]  + j_{21} \circ \beta + j'\circ j_{11} + \psi^t\circ j_{12} &=& 0 \label{third-relation}
\eeqn 
for every elements $x$ and $y$ of $\G$ and any elements $f$ and $g$ in $\G^*$.
\epro 

\begin{proof}
 Recall that according to Corollary \ref{corollary-theorem:characterization-prederivations} a prederivation $\phi$ of $T^*\G$ can be written as 
\beq
\phi(x,f) = \Big(\alpha(x) +\psi(f)\,,\, \beta(x) + (j'^t-\alpha^t)(f) \Big), \nonumber
\eeq 
where $j'$ belongs to $\mathcal J'$ and $\alpha, \beta, \psi$ are as in 
Theorem \ref{theorem:characterization-prederivations}. The prederivation $\phi$ is $\mu$-skew-symmetric 
if and only if
\beq 
\mu\big(\phi(x,f),(y,g)\big) + \mu\big((x,f),\phi(y,g)\big) =0, \label{orth-structure1}
\eeq 
for all $(x,f)$ and $(y,g)$ in $T^*\G$. On one hand, we have
\beqn
\mu\big(\phi(x,f),(y,g)\big) &=& \mu\Big(\big(\alpha(x) +\psi(f)\,,\, \beta(x) + (j'^t-\alpha^t)(f) \big)\;,\;(y,g)\Big) \cr
                             &=& \Big\langle g\;,\;j_{11}\big(\alpha(x)+\psi(f)\big) \Big\rangle 
                               + \Big\langle  \beta(x) + (j'^t-\alpha^t)(f)\;,\;j_{11}(y)\Big\rangle \cr
                             & & \Big\langle g\;,\;j_{21}\big(\beta(x) + (j'^t-\alpha^t)(f) \big)\Big\rangle  + 
                                \Big\langle j_{12}\big(\alpha(x) +\psi(f)\big)\;,\;y\Big\rangle \cr
                             &=& \big\langle g\;,\;j_{11}\circ \alpha(x)\big\rangle \!+\! 
                                  \big\langle g\;,\; j_{11}\circ \psi(f)\big\rangle \!
                                 +\! \big\langle \beta(x)\;,\; j_{11}(y)\big\rangle \!
                                 +\! \big\langle f\;,\; j'\circ j_{11}(y)\big\rangle \cr
                             & & -\big\langle f\;,\; \alpha\circ j_{11}(y)\big\rangle 
                                 + \big\langle g\;,\; j_{21} \circ \beta(x)\big\rangle 
                                 + \big\langle g\;,\; j_{21}\circ j'^t(f)\big\rangle  \cr
                             & &  - \big\langle g\;,\; j_{21}\circ \alpha^t(f)\big\rangle 
                                 + \big\langle j_{12}\circ \alpha(x),y\big\rangle 
                                 + \big\langle j_{12}\circ \psi(f),y\big\rangle  \label{first-part}
\eeqn 

\vskip 0.3cm
\noindent
On the other hand,
\beqn 
\mu\big((x,f),\phi(y,g)\big) &=& \mu\Big((x,f)\;,\;\big(\alpha(y) +\psi(g)\,,\, \beta(y) + (j'^t-\alpha^t)(g) \big)\Big) \cr
                             &=&  \Big\langle \beta(y) + (j'^t-\alpha^t)(g)\;,\; j_{11}(x) \Big\rangle 
                                  + \Big\langle f\;,\; j_{11}\big(\alpha(y) +\psi(g) \big)\Big\rangle \cr
                             & &  + \Big\langle  \beta(y) + (j'^t-\alpha^t)(g)\;,\; j_{21}(f) \Big\rangle  
                                  + \Big\langle j_{12}(x),  \alpha(y) +\psi(g) \Big\rangle  \cr
                             &=& \big\langle f\;,\;j_{11}\circ \alpha(y)\big\rangle + 
                                    \big\langle f\;,\; j_{11}\circ \psi(g)\big\rangle    
                                 + \big\langle  \beta(y)\;,\;j_{11}(x) \big\rangle \cr
                             & &   +\big\langle g \;,\; j'\circ j_{11}(x)\big\rangle -\big\langle g\;,\;\alpha\circ j_{11}(x)\big\rangle 
                                   + \big\langle \beta(y) \;,\; j_{21}(f)\big\rangle \cr
                             & & +  \big\langle g\;,\; j'\circ j_{21}(f)\big\rangle - \big\langle g,\alpha\circ j_{21}(f) \big\rangle 
                                 + \big\langle j_{12}\circ \alpha(y)\;,\; x \big\rangle \cr
                             & & + \big\langle j_{12}\circ \psi(g)\;,\; x \big\rangle \label{second-part}, 
                                 \text{ since  $j_{21}$  is $\langle, \rangle$-symmetric }
\eeqn 
Summing (\ref{first-part}) and (\ref{second-part}) and taking $f=g=0$, Relation (\ref{orth-structure1}) becomes
$$
 \big\langle \beta(x)\;,\; j_{11}(y)\big\rangle + \big\langle j_{12}\circ \alpha(x),y\big\rangle + \big\langle  \beta(y)\;,\;j_{11}(x) \big\rangle
 + \big\langle j_{12}\circ \alpha(y)\;,\; x \big\rangle =0
$$
So (\ref{first-relation}) is proved. 

\vskip 0.3cm
\noindent
Now we rewrite (\ref{orth-structure1}) and take $x=y=0$ to obtain (\ref{second-relation}).

\vskip 0.3cm
\noindent
Last, (\ref{third-relation}) is obtained by summing (\ref{first-part}) and (\ref{second-part}), 
simplifying and taking $y=0$.

\end{proof}

\section{Case of Orthogonal Lie Algebras}\label{chap3:section:orthogonal-lie-algebras}

Let $\G$ be a Lie algebra. We recall that an orthogonal structure $\mu$ on $\G$ induces 
an isomorphism $\theta : \G \to \G^*$ by the formula $\langle \theta(x),y \rangle = \mu(x,y)$, for all $x,y$
in $\G$. We also recall the set $\mathcal J=\{j:\G \to \G : j([x,y])=[j(x),y], \forall x,y \in \G\}$.

\subsection{Bi-invariant Metrics On  Cotangent Bundles of Orthogonal Lie groups}
\blem\label{lemma:symmetric-invertible-bi-invariant-endomorphism-orthogonal-lie-algebras}
Let $(\G,\mu)$ be an orthogonal Lie algebra. Any bi-invariant $\langle,\rangle$-symmetric invertible endomorphism 
$j:T^*\G \to T^*\G$ is given by
\beq \label{expression:symmetric-invertible-bi-invariant-endomorphism-orthogonal-lie-algebras}
j(x,f) = \Big(j_{11}(x) + j_2\circ \theta^{-1}(f)\;,\; \theta\circ j_1(x)+j_{11}^t(f) \Big),
\eeq 
for any $(x,f)$ in $T^*\G$; where  $j_{11}, j_1,j_2$ are elements  of 
$\mathcal J$ such that 
\bitem 
\item $j_{11}$ is invertible,  
\item $j_1$ is $\mu$-symmetric 
\item and $j_2$ satisfies the following two relations
\beqn
j_2 \circ ad_x &=& 0,\\
\langle j_2\circ \theta^{-1}(f)\;,\;g\rangle &=& \langle f\;,\;j_2\circ \theta^{-1}(g)\rangle
\eeqn 
for all $x$ in $\G$ and all $f,g$ in $\G^*$.
\eitem 
\elem 

\begin{proof}
 We have already seen that any bi-invariant $\langle,\rangle$-symmetric invertible endomorphism of $T^*\G$ is given by 
$$
j(x,f) = \Big(j_{11}(x) + j_{21}(f) \,,\,j_{12}(x) + j^t_{11}(f) \Big),
$$
with conditions listed in Proposition \ref{prop:symmetric-bi-invariant-endom}. 

\vskip 0.3cm

The linear maps $\theta^{-1}\circ j_{12}:\G \to \G$ and $j_{21}\circ \theta:\G \to \G$ 
commutes with all adjoint operators of $\G$; 
that is there exists $j_1$, $j_2$ in $\mathcal J$ such  that $\theta^{-1}\circ j_{12}=j_1$ and $j_{21} \circ \theta=j_2$. 
It comes that $j_{12} = \theta \circ j_1$ and $j_{21}=j_2 \circ \theta^{-1}$. 

\vskip 0.3cm 
\noindent
Now, (\ref{equiv-j21}) is equivalent to $j_2 \circ \theta^{-1}\circ ad_x^*=0$, for all $x$ in $\G$. Since
$\theta^{-1} \circ ad_x^*=ad_x \circ \theta^{-1}$, for all $x$ in $\G$, we have  $j_2 \circ ad_x \circ \theta^{-1}=0$, 
for any $x$ in $\G$. It comes that $j_2 \circ ad_x=0$, for all  $x$ in $\G$, since $\theta^{-1}$ is an isomorphism.

\vskip 0.3cm
\noindent
Relation (\ref{symJ12-1}) becomes $\langle\theta\circ j_1(x),y\rangle=\langle x,\theta\circ j_1(y)\rangle$, 
for any $x,y$ in $\G$. The latter can be written $\mu\big(j_1(x),y\big)=\mu\big(x,j_1(y)\big)$, for any $x,y$ in $\G$.

\vskip 0.3cm
\noindent
Last, Relation (\ref{symJ21-1}) reads $\langle j_2 \circ \theta^{-1}(f),g\rangle = \langle f,j_2\circ \theta^{-1}(g)\rangle$,
 for all $f,g$ in $\G^*$.
\end{proof}
\brmq\label{j2-perfect-lie-algebras}
The relation $j_2\circ ad_x=0$, for all $x$ in $\G$ means that $j_2$ vanishes identically on the derived 
ideal $[\G,\G]$ of $\G$. So, if $\G$ is perfect, then $j_2$ vanishes identically on all $\G$.
\ermq 

The following comes immediately from Lemma \ref{lemma:symmetric-invertible-bi-invariant-endomorphism-orthogonal-lie-algebras} 
and Remark \ref{j2-perfect-lie-algebras}
\bcor\label{invertible-biinvariant-endom-perfect-algebra}
Let $(\G,\mu)$ be an orthogonal and perfect Lie algebra. Then, any bi-invariant $\langle,\rangle$-symmetric invertible 
endomorphism $j:T^*\G \to T^*\G$ is given by
\beq 
j(x,f) = \Big(j_{11}(x) \;,\; \theta\circ j_1(x)+j_{11}^t(f) \Big),
\eeq 
for any $(x,f)$ in $T^*\G$; where  $j_{11}, j_1$ are bi-invariant tensors $\G$ such that  $j_1$ is $\mu$-symmetric and  $j_{11}$ is invertible.
\ecor 

\bthm\label{theorem:orthogonal-structure-orthogonal-lie-algebras}
Let $(\G,\mu)$ be an orthogonal Lie algebra. Any orthogonal structure $\mu_\D$ on $T^*\G$ is given by
\beq \label{expression:orthogonal-structure-orthogonal-lie-algebras}
\mu_\D\Big((x,f),(y,g)\Big)=\big\langle g\;,\;j_{11}(x) \big\rangle + \big\langle f\;,\;j_{11}(y)\big\rangle 
+ \big\langle g,j_2 \circ \theta^{-1}(f) \big\rangle + \big\langle \theta\circ j_1(x),y\big\rangle,
\eeq 
for all $(x,f)$, $(y,g)$ in $T^*\G$, where $j_{11}$, $j_1$ and $j_2$ satisfy conditions listed in 
Lemma \ref{lemma:symmetric-invertible-bi-invariant-endomorphism-orthogonal-lie-algebras}.
\ethm 
\begin{proof}
Any orthogonal structure $\mu_\D$ on $T^*\G$ is linked to the duality pairing by means of a map 
$j:T^*\G \to T^*\G$ defined by (\ref{expression:symmetric-invertible-bi-invariant-endomorphism-orthogonal-lie-algebras}) 
through the following formula valid for all  $(x,f)$, $(y,g)$ in $T^*\G$ : 
\beqn 
\mu_\D\Big((x,f),(y,g) \Big) &=& \langle j(x,f)\;,\; (y,g) \rangle. \cr
                             &=&\Big\langle \Big(j_{11}(x) + j_2\circ \theta^{-1}(f)\;,\; \theta\circ j_1(x)+j_{11}^t(f) \Big) \;,\; (y,g)\Big\rangle \cr
                             &=& \big\langle g \;,\; j_{11}(x) + j_2\circ \theta^{-1}(f) \big\rangle 
                               + \big\langle \theta\circ j_1(x)+j_{11}^t(f) \;,\; y \big\rangle \cr
                             &=& \big\langle g\;,\; j_{11}(x)\big\rangle +\big\langle g\;,\;  j_2\circ \theta^{-1}(f)\big\rangle
                                 + \big\langle \theta\circ j_1(x)\;,\;y\big\rangle + \big\langle j_{11}^t(f) \;,\; y\big\rangle \cr 
                             &=& \big\langle g\;,\; j_{11}(x)\big\rangle +\big\langle g\;,\;  j_2\circ \theta^{-1}(f)\big\rangle
                                 + \big\langle \theta\circ j_1(x)\;,\;y\big\rangle + \big\langle f \;,\; j_{11}(y)\big\rangle \nonumber
\eeqn

\end{proof}

\vskip 0.3cm

We have the following consequence coming from Theorem \ref{theorem:orthogonal-structure-orthogonal-lie-algebras} and 
Remark \ref{j2-perfect-lie-algebras}.
\bcor
Let $(\G,\mu)$ be an orthogonal and perfect Lie algebra. Any orthogonal structure $\mu_\D$ on $T^*\G$ is given by
\beq 
\mu_{\D}\Big((x,f),(y,g)\Big)=\big\langle g\;,\;j_{11}(x) \big\rangle + \big\langle f\;,\;j_{11}(y)\big\rangle 
 + \big\langle \theta\circ j_1(x)\;,\;y\big\rangle,
\eeq 
for all $(x,f)$, $(y,g)$ in $T^*\G$, where $j_{11}$ and $j_1$ are are bi-invariant tensors of $\G$ such that $j_{11}$ is invertible,  $j_1$ is $\mu$-symmetric.
\ecor

\vskip 0.3cm
Let us now characterise skew-symmetric prederivations of $T^*\G$ in the case where $\G$ is an orthogonal Lie algebra.
\vskip 0.3cm

\subsection{Skew-symmetric Prederivations}

\bpro\label{prop:skew-symmetric-prederivation-orthogonal-lie-algebra}
Let $(\G,\mu)$ be an orthogonal Lie algebra. Then any prederivation $\phi$ of $T^*\G$ which is skew-symmetric 
with respect to the orthogonal structure $\mu_\D$ defined by (\ref{expression:orthogonal-structure-orthogonal-lie-algebras}) 
can be written as 
\beq 
\phi(x,f)= \Big(\alpha_1(x)+j_1'\circ \theta^{-1}(f) \;,\; \theta \circ \alpha_2(x) + ({j_2'}^t-\alpha_1^t)(f)\Big),
\eeq 
where $\alpha_1, \alpha_2$ are in $\pder(\G)$; $j_1',j_2'$ are in $\mathcal J'$ with the following additional conditions:
\beqn 
\theta \circ j_1 \circ \alpha_1 + \alpha_1^t \circ \theta \circ j_1 + \alpha_2^t \circ \theta \circ j_{11} &=& 0 \\
(j_{11}\circ j_1' + j_2'\circ j_2 -\alpha_1 \circ j_2) \circ \theta^{-1} + j_2 \circ \theta^{-1}\circ({j_2'}^t-\alpha_1^t) 
+ \theta^{-t} \circ {j_1'}^t\circ j_{11}^t &=& 0 \\
\big[j_{11},\alpha_1\big] + j_2 \circ \alpha_2 + j_2'\circ j_{11} + \theta^{-t} \circ {j_1'}^t \circ \theta \circ j_1 &=& 0,
\eeqn 
where $\theta^{-t}$ is the transpose of $\theta^{-1}$.
\epro

\begin{proof}
From Proposition \ref{proposition:prederivation-orthogonal-lie-algebra}, if $\G$ is an orthogonal Lie algebra, 
then any prederivation of $T^*\G$ is given by 
$$
\phi(x,f) = \Big( \alpha_1(x) + j_1' \circ \theta^{-1}(f) \,,\, \theta\circ \alpha_2(x) + ({j_2'}^t-\alpha_1^t)(f)\Big),
$$ 
for every $(x,f)$ in $T^*\G$, where $\alpha_1,\alpha_2$ are prederivations of $\G$, $j_1',j_2'$ are in $\mathcal J'$. The prederivation 
$\phi$ is $\mu_D$-skew-symmetric means that
\beq \label{phi-skew-symmetric}
\mu_\D\Big(\phi(x,f),(y,g)\Big) + \mu_\D\Big((x,f),\phi(y,g)\Big) = 0,
\eeq 
for all $(x,f),(y,g)$ in $T^*\G$. Firstly, we compute $\mu_\D\big(\phi(x,f),(y,g)\big)$.
\beqn
\mu_\D\Big(\phi(x,f),(y,g)\Big) &=& \mu_\D\Big( \big( \alpha_1(x) + j_1' \circ \theta^{-1}(f) \,,\, \theta\circ \alpha_2(x) 
                                    + ({j_2'}^t-\alpha_1^t)(f)\big)\; , \; (y,g)\Big) \cr
                                &=& \Big\langle g\;,\;j_{11}\Big(\alpha_1(x) + j_1' \circ \theta^{-1}(f) \Big)\Big\rangle
                                   + \langle f , j_{11}(y)\rangle \cr 
                                & & + \Big\langle g \;,\; j_{2}\circ \theta^{-1}\Big(\theta\circ \alpha_2(x) 
                                    + ({j_2'}^t-\alpha_1^t)(f) \Big)  \Big\rangle \cr 
                                & &   + \Big\langle \theta \circ j_1\Big( \alpha_1(x) + j_1' \circ \theta^{-1}(f)\Big)\; , \; y\Big\rangle \cr
                                &=& \Big\langle g\;,\;j_{11}\circ \alpha_1(x)\Big\rangle 
                                   + \Big\langle g\;,\;j_{11} \circ j_1' \circ \theta^{-1}(f)\Big\rangle +  \langle f , j_{11}(y)\rangle \cr
                                & & + \Big\langle g \;,\;j_{2}\circ\alpha_2(x)\Big\rangle 
                                    + \Big\langle g \;,\; j_{2}\circ\theta^{-1}\circ {j_2'}^t(f) \Big\rangle \cr
                                & & -\Big\langle g \; , \; j_{2}\circ\theta^{-1}\circ \alpha_1^t(f) \Big\rangle 
                                    + \Big\langle \theta \circ j_1 \circ \alpha_1(x)\;,\;y\Big\rangle \cr 
                                & &  + \Big\langle \theta \circ j_1 \circ j_1'\circ \theta^{-1}(f)\;,\; y \Big\rangle.
\eeqn 
Secondly,
\beqn 
\mu_\D\Big((x,f),\phi(y,g)\Big) &=& \mu_\D\Big((x,f)\; , \; \big( \alpha_1(y) + j_1' \circ \theta^{-1}(g) \,,\, \theta\circ \alpha_2(y) 
                                    + ({j_2'}^t-\alpha_1^t)(g)\big) \Big) \cr 
                                &=& \Big\langle  \theta\circ \alpha_2(y)\! +\! ({j_2'}^t-\alpha_1^t)(g) \;,\; j_{11}(x) \Big\rangle
                                    \!+\! \Big\langle f\;,\; j_{11}\big(\alpha_1(y) \!+\! j_1' \circ \theta^{-1}(g)\big)\Big\rangle \cr
                                & & + \Big\langle \theta\circ \alpha_2(y)  + ({j_2'}^t-\alpha_1^t)(g)\;,\; j_2\circ \theta^{-1}(f)\Big\rangle \cr
                                & & + \Big\langle \theta\circ j_1(x) \;,\; \alpha_1(y) + j_1' \circ \theta^{-1}(g)\Big\rangle \cr 
                                &=& \Big\langle  \theta\circ \alpha_2(y)\;,\; j_{11}(x)\Big\rangle + \Big\langle {j_2'}^t(g)\;,\;j_{11}(x)\Big\rangle
                                    -\Big\langle \alpha_1^t(g) \;,\; j_{11}(x)\Big\rangle \cr
                                & & +\Big\langle f \;,\; j_{11}\circ \alpha_1(y) \Big\rangle 
                                    + \Big\langle f \;,\; j_{11}\circ j_1' \circ \theta^{-1}(g) \Big\rangle \cr 
                                & & + \Big\langle  \theta\circ \alpha_2(y) \;,\; j_2 \circ \theta^{-1}(f)\Big\rangle 
                                    + \Big\langle {j_2'}^t(g)\;,\; j_2 \circ \theta^{-1}(f)\Big\rangle \cr
                                & & -\Big\langle \alpha_1^t(g) \;,\; j_2 \circ \theta^{-1}(f)\Big\rangle 
                                    + \Big\langle \theta \circ j_1(x) \;,\; \alpha_1(y) \Big\rangle \cr
                                & & + \Big\langle \theta \circ j_1(x)\;,\; j_1'\circ \theta^{-1}(g)\Big\rangle.
\eeqn
Now, if we take $f=g=0$, (\ref{phi-skew-symmetric}) gives 
\beq
\Big\langle \theta \circ j_1 \circ \alpha_1(x)\;,\;y\Big\rangle + \Big\langle  \theta\circ \alpha_2(y)\;,\; j_{11}(x)\Big\rangle 
+ \Big\langle \theta \circ j_1(x) \;,\; \alpha_1(y) \Big\rangle = 0 \nonumber
\eeq 
which can be written 
\beq 
\Big\langle \theta \circ j_1 \circ \alpha_1(x)\;,\;y\Big\rangle + \Big\langle  y\;,\; \alpha_2^t\circ \theta\circ j_{11}(x)\Big\rangle 
+ \Big\langle \alpha_1^t\circ \theta \circ j_1(x) \;,\; y \Big\rangle = 0 \nonumber
\eeq
for all $x,y$ in $\G$. Then, 
$$
\theta \circ j_1 \circ \alpha_1 + \alpha_1^t\circ \theta \circ j_1 + \alpha_2^t\circ \theta \circ j_{11}=0.
$$

We can rewrite (\ref{phi-skew-symmetric}) as follows

\beqn
\begin{array}{lll}
& & \big\langle g\;,\;j_{11}\circ \alpha_1(x)\big\rangle \!+\! \big\langle g\;,\;j_{11} \circ j_1' \circ \theta^{-1}(f)\big\rangle \!+\! \langle f , j_{11}(y)\rangle  \cr 
& & + \big\langle g \;,\;j_{2}\circ\alpha_2(x)\big\rangle \!+\! \big\langle g \;,\; j_{2}\circ\theta^{-1}\circ {j_2'}^t(f) \big\rangle 
\!-\!\big\langle g \; , \; j_{2}\circ\theta^{-1}\circ \alpha_1^t(f) \big\rangle \cr 
& & + \big\langle \theta \circ j_1 \circ j_1'\circ \theta^{-1}(f)\;,\; y \big\rangle \! +\! \big\langle {j_2'}^t(g)\;,\;j_{11}(x)\big\rangle 
\!-\!\big\langle \alpha_1^t(g) \;,\; j_{11}(x)\big\rangle \cr  
& &+ \big\langle f \;,\; j_{11}\circ \alpha_1(y) \big\rangle \!+\! \big\langle f \;,\; j_{11}\circ j_1' \circ \theta^{-1}(g) \big\rangle  
  \!+\! \big\langle  \theta\circ \alpha_2(y) \;,\; j_2 \circ \theta^{-1}(f)\big\rangle \cr 
& &+ \big\langle {j_2'}^t(g)\;,\; j_2 \circ \theta^{-1}(f)\big\rangle 
  \!-\!\big\langle \alpha_1^t(g) \;,\; j_2 \circ \theta^{-1}(f)\big\rangle  
   \!+\! \big\langle \theta \circ j_1(x)\;,\; j_1'\circ \theta^{-1}(g)\big\rangle \!=\! 0. \label{phi-skew-symmetric-samba}
\end{array}
\eeqn 
If we take $x=y=0$, (\ref{phi-skew-symmetric-samba}) becomes
\beqn
\begin{array}{l}
\Big\langle g\;,\;j_{11} \circ j_1' \circ \theta^{-1}(f)\Big\rangle + \Big\langle g \;,\; j_{2}\circ\theta^{-1}\circ {j_2'}^t(f) \Big\rangle
-\Big\langle g \; , \; j_{2}\circ\theta^{-1}\circ \alpha_1^t(f) \Big\rangle \cr
+  \Big\langle f \;,\; j_{11}\circ j_1' \circ \theta^{-1}(g) \Big\rangle
 + \Big\langle {j_2'}^t(g)\;,\; j_2 \circ \theta^{-1}(f)\Big\rangle -\Big\langle \alpha_1^t(g) \;,\; j_2 \circ \theta^{-1}(f)\Big\rangle  = 0 
\end{array}
\eeqn
which is  
\beqn
\begin{array}{l}
\Big\langle g\;,\;j_{11} \circ j_1' \circ \theta^{-1}(f)\Big\rangle + \Big\langle g \;,\; j_{2}\circ\theta^{-1}\circ {j_2'}^t(f) \Big\rangle
-\Big\langle g \; , \; j_{2}\circ\theta^{-1}\circ \alpha_1^t(f) \Big\rangle \cr
+ \Big\langle \theta^{-t} \circ {j_1'}^t \circ j_{11}^t (f) \;,\;  g \Big\rangle
 + \Big\langle g\;,\;j_2'\circ j_2 \circ \theta^{-1}(f)\Big\rangle -\Big\langle g \;,\;\alpha_1\circ j_2 \circ \theta^{-1}(f)\Big\rangle  = 0, \nonumber
\end{array}
\eeqn 
for all $f,g$ in $\G^*$, where $\theta^{-t}$ is the transpose of $\theta^{-1}$.  This latter equality is equivalent to
$$
(j_{11}\circ j_1' + j_2'\circ j_2 -\alpha_1 \circ j_2) \circ \theta^{-1} + j_2 \circ \theta^{-1}\circ({j_2'}^t-\alpha_1^t) 
+ \theta^{-t} \circ {j_1'}^t\circ j_{11}^t = 0
$$
Last, (\ref{phi-skew-symmetric}) have a more simple expression :

\beqn
\begin{array}{lll}
& & \big\langle g\;,\;j_{11}\circ \alpha_1(x)\big\rangle  \!+\! \langle f , j_{11}(y)\rangle + \big\langle g \;,\;j_{2}\circ\alpha_2(x)\big\rangle   \cr 
& &+ \big\langle \theta \circ j_1 \circ j_1'\circ \theta^{-1}(f)\;,\; y \big\rangle + \big\langle {j_2'}^t(g)\;,\;j_{11}(x)\big\rangle \!-\!\big\langle \alpha_1^t(g) \;,\; j_{11}(x)\big\rangle 
 \cr 
& & + \big\langle f \;,\; j_{11}\circ \alpha_1(y) \big\rangle + \big\langle  \theta\circ \alpha_2(y) \;,\; j_2 \circ \theta^{-1}(f)\big\rangle 
 \!+\! \big\langle \theta \circ j_1(x)\;,\; j_1'\circ \theta^{-1}(g)\big\rangle \!=\! 0. \label{phi-skew-symmetric-samba2}
\end{array}
\eeqn 
We take $y=0$ and obtain
\beq 
\langle g \;,\; j_{11}\circ \alpha_1(x) \rangle \!+\! \langle g \;,\; j_2\circ \alpha_2(x) \rangle 
\!+\! \langle g \;,\; j_2' \circ j_{11}(x)\rangle \!-\!\langle g \;,\; \alpha_1 \circ j_{11}(x)\rangle 
\!+\! \langle \theta \circ j_1(x)\;,\;j_1'\circ \theta^{-1}(g) \rangle = 0 
\eeq 
The latter equality is equivalent to 
$$
[j_{11},\alpha_1] + j_2\circ \alpha_2 + j_2' \circ j_{11} + \theta^{-t}\circ {j_1'}^t\circ \theta \circ j_1 =0.
$$
\end{proof}

\section{Case of Semi-simple Lie Algebras}\label{chap3:section:semi-simple-lie-algebras}

Let $\G=\s_1\oplus \s_2\oplus\cdots\oplus\s_p$ be a semi-simple Lie algebra, where $\s_i$, $i=1,2,\ldots,p$
($p \in \N^*$) are simple ideals. 

\subsection{Bi-invariant metrics On Cotangent bundles of Semi-simple Lie groups}

\blem\label{endomorphism-j-semi-simple}
Any $\langle,\rangle$-symmetric and invertible bi-invariant  tensor  $j:T^*\G \to T^*\G$ 
is defined by 
\beq \label{expression:endomorphism-j-semi-simple}
j(x,f) = \Big(\sum_{i=1}^p \lambda_i x_i\; , \; \sum_{k=1}^p\nu_k\theta(x_k) + \sum_{i=1}^p \lambda_i f_i\Big),
\eeq 
where $x=x_1+x_2+\cdots+x_p \in \s_1\oplus \s_2\oplus\cdots\oplus\s_p$, $f=f_1+f_2+\cdots+f_p \in \s_1^*\oplus \s_2^*\oplus\cdots\oplus\s_p^*$,
$\lambda_i$, ($i=1,2,\ldots,p$) are non-zero real numbers, $\nu_i$  ($i=1,2,\ldots,p$) are any real numbers and $\theta:\G\to\G^*$ is the isomorphism induced by any 
orthogonal structure on $\G$ through the formula (\ref{isomorphism-theta}).
\elem

\begin{proof}
According to Corollary \ref{invertible-biinvariant-endom-perfect-algebra}, any bi-invariant, 
$\langle,\rangle$-symmetric and invertible endomorphism $j:T^*\G \to T^*\G$ has the form 
\beq 
j(x,f) = \Big(j_{11}(x), \theta \circ j_1(x) + j_{11}^t(f) \Big),
\eeq 
for all $(x,f)$ in $T^*\G$, where $j_{11}$ and $j_1$ are bi-invariant endomorphisms of $\G$ such that
\bitem 
\item $j_{11}$ is invertible, 
\item $j_1$ is symmetric with respect to any orthogonal structure on $\G$.
\eitem 
Now we have seen in Section \ref{chap:semisimple} that the map $j_{11}$ must have the form 
\beq 
j_{11}(x) = \sum_{i=1}^p \lambda_i x_i,
\eeq
if $x=x_1+x_2+\cdots+x_p \in \s_1\oplus \s_2\oplus\cdots\oplus\s_p$, where $\lambda_i, i=1,2,\ldots,p$ are real numbers. It comes that 
\beq
j_{11}^t(f) = \sum_{i=1}^p \lambda_i f_i,
\eeq 
if $f=f_1+f_2+\cdots+f_p \in \s_1^*\oplus \s_2^*\oplus\cdots\oplus\s_p^*$.

\vskip 0.3cm

We also have 
\beq
j_1(x) = \sum_{k=1}^p\nu_kx_k,
\eeq 
where $\nu_k$, $k=1,2,\ldots,p$, are real numbers. So we have 
\beqn 
\theta \circ j_1(x) &=& \theta\Big( \sum_{k=1}^p\nu_kx_k \Big) \cr
                    &=& \sum_{k=1}^p\nu_k\theta(x_k)
\eeqn  
We then can write $j$ as follows
\beq 
j(x,f) = \Big(\sum_{i=1}^p \lambda_i x_i, \sum_{k=1}^p\nu_k\theta(x_k) + \sum_{i=1}^p \lambda_i f_i \Big)
\eeq 
Now $j$ is invertible if and only if $j_{11} = \sum_{i=1}^p \lambda_i Id_{\s_i}$ is invertible.

For all $i=1,2,\ldots,p$, note by $n_i$ the dimension of the simple ideal $\s_i$ ($n_i:=\dim\s_i$). 
The matrix of $j_{11}$, in some basis 
$\{e_{11},e_{12},\ldots,e_{1n_1};e_{21},e_{22},\ldots,e_{2n_2};\ldots;e_{p1},e_{p2},\ldots,e_{pn_p}\}$,  
is the following

\beq
\left(
\begin{array}{cccccccccccc}
 A_1   &   O     & \ldots &  O    \cr
  O    &   A_2   & \ldots &  O    \cr  
\vdots & \vdots  & \ddots &  O    \cr
  O    &  O      & \ldots & A_p   \cr 
\end{array}
\right) 
\eeq 
where $A_i$ is the $n_i\times n_i$-matrix defined as follows
\beq 
A_i=diag(\underbrace{\lambda_i,\lambda_i,\ldots,\lambda_i}_{n_i \, times }) = 
\left( 
\begin{array}{cccccccccccc}
\lambda_i   &   O      & \ldots &  O    \cr
  O         &\lambda_i & \ldots &  O    \cr  
\vdots      & \vdots   & \ddots &  O    \cr
  O         &  O       & \ldots &\lambda_i   \cr 
\end{array}
\right)
\eeq 
The determinant of $j_{11}$ is then 
\beq 
\det(j_{11})=\prod_{k=1}^p\lambda_k^{n_k}.
\eeq 
Then, $j_{11}$ is invertible if and only if each of $\lambda_k$ is non-zero.
\end{proof}

\brmq
The matrix of $j$ in the basis $\{e_{11},e_{12},\ldots,e_{1n_1};e_{21},e_{22},\ldots,e_{2n_2};
\ldots; \\e_{p1},e_{p2},\ldots,e_{pn_p};     e_{11}^*,e_{12}^*,\ldots,e_{1n_1}^*;e_{21}^*,e_{22}^*,\ldots,e_{2n_2}^*;
\ldots;e_{p1}^*,e_{p2}^*,\ldots,e_{pn_p}^* \}$ is given by
\beq
\left(
\begin{array}{cccccccccccc}
 A_1   &   O     & \ldots &  O    &   O   &   O   &\ldots&  O          \cr
  O    &   A_2   & \ldots &  O    &   O   &   O   &\ldots&  O         \cr  
\vdots & \vdots  & \ddots &  O    &\vdots &\vdots &\vdots& \vdots     \cr
  O    &  O      & \ldots & A_p   &   O   &   O   &\ldots&  O         \cr 
       &         &        &       &   A_1 &   O   &\ldots&  O         \cr 
       &         &        &       &   O   & A_2   &\ldots&  O         \cr 
       &         &        &       &\vdots &\vdots &\ddots& \vdots     \cr 
       &         &        &       &   O   &   O   &\ldots& A_p        \cr 
\end{array}
\right) 
\eeq 
and the determinant of $j$ is 
\beq 
\det(j)=\prod_{k=1}^p\lambda_k^{2n_k}.
\eeq 
\ermq

\vskip 0.3cm

Note by $\langle,\rangle_{\s_i}$ the duality pairing between $\s_i$ and its dual space $\s_i^*$. 
The following theorem characterizes all orthogonal structures on $T^*\G$.

\vskip 0.3cm

\bthm\label{theorem:orthogonal-structure-semi-simple-lie-algebras}
Let $\G$ be a semi-simple Lie algebra.  Any orthogonal structure $\mu$ on $T^*\G$ is given by 
\beq\label{expression:metric-semi-simple} 
\mu\big((x,f),(y,g)\big) = \sum_{i=1}^p \lambda_i\big\langle(x_i,f_i),(y_i,g_i)\big\rangle_{\s_i} 
+ \sum_{k=1}^p\nu_k\langle \theta(x_k),y_k\rangle_{\s_k},
\eeq 
for all $x, y$ in $ \G=\s_1\oplus \s_2\oplus\cdots\oplus\s_p$ and $f, g$ in 
$\G^*=\s_1^*\oplus \s_2^*\oplus\cdots\oplus\s_p^*$ ;
where for all $i=1,2,\ldots,p$,  $(\lambda_i,\nu_i)$ belongs to $\R^*\times \R$.

\vskip 0.3cm
 
In particular, if $\G$ is simple, then any orthogonal structure can be written
\beq \label{expression:metric-simple}
\mu\big((x,f),(y,g)\big) = \lambda\big\langle(x,f),(y,g)\big\rangle + \nu\langle \theta(x),y\rangle,
\eeq 
for all $(x,f), (y,g)$ in $T^*\G$, where $(\lambda,\nu)$ is in $\R^*\times \R$.
\ethm 
\begin{proof}
Any orthogonal structure on $T^*\G$ is given by an adjoint-invariant invertible and 
$\langle,\rangle$-symmetric endomorphism $j$ through the formula

\beq
\mu_j\big((x,f),(y,g)\big) = \langle j(x,f)\,,\;(y,g) \rangle, 
\eeq 
We have seen in Lemma \ref{endomorphism-j-semi-simple} that 
$$
j(x,f) = \Big(\sum_{i=1}^p \lambda_i x_i\; , \; \sum_{k=1}^p\nu_k\theta(x_k) + \sum_{i=1}^p \lambda_i f_i\Big),
$$
where $x=x_1+x_2+\cdots+x_p \in \s_1\oplus \s_2\oplus\cdots\oplus\s_p$, $f=f_1+f_2+\cdots+f_p \in \s_1^*\oplus \s_2^*\oplus\cdots\oplus\s_p^*$,
$\lambda_i,\nu_i$, ($i=1,2,\ldots,p$) are real numbers, $\lambda_i\neq 0$, for all $i=1,2,\ldots,p$. We can write
\beqn
\mu_j\big((x,f),(y,g)\big) &=& \sum_{i=1}^p \lambda_i g(x_i) + \sum_{k=1}^p\nu_k\theta(x_k)(y) + \sum_{i=1}^p \lambda_if_i(y) \cr
                           &=& \sum_{i=1}^p \lambda_i \Big(g(x_i) +f_i(y)\Big) + \sum_{k=1}^p\nu_k\theta(x_k)(y) \cr
                           &=& \sum_{i=1}^p \lambda_i \Big(g_i(x_i) +f_i(y_i)\Big) + \sum_{k=1}^p\nu_k\theta(x_k)(y_k) \cr
                           &=& \sum_{i=1}^p \lambda_i \big\langle (x_i,f_i),(y_i,g_i)\big\rangle_{\s_i} 
                                   + \sum_{k=1}^p\nu_k\langle \theta(x_k),y_k\rangle_{\s_k} \label{metric-semi-simple}
\eeqn 
So, the first part of the theorem is proved. 

\vskip 0.3cm

Now, if $\G$ is simple, then $\G=\s_1$, $x=x_1$, $f=f_1$ and Relation (\ref{metric-semi-simple}) becomes
\beq 
\mu_j\big((x,f),(y,g)\big) =  \lambda \big\langle (x,f),(y,g)\big\rangle 
                                   + \nu\langle \theta(x),y\rangle \label{metric-simple}
\eeq 
where $\lambda \in \R^*$ and $\nu \in \R$.
\end{proof}

\brmq
If $\G=\s_1\oplus\s_2\oplus\cdots\oplus\s_p$ is a semi-simple Lie algebra, where $\s_i, i=1,2,\ldots,p$
are simple ideals, then on any of the ideals $\s_i$ the Killing form $K_{\s_i}$ defines an orthogonal
structure. Thus, Relation (\ref{expression:metric-semi-simple}) can be written 
\beq\label{expression:metric-semi-simple-killing}
\mu\big((x,f),(y,g)\big) = \sum_{i=1}^p \lambda_i\big\langle(x_i,f_i),(y_i,g_i)\big\rangle_{\s_i} 
+ \sum_{k=1}^p\nu_kK_{\s_k}(x_k,y_k)
\eeq 
and Relation (\ref{expression:metric-simple}) can be expressed as
\beq 
\mu_j\big((x,f),(y,g)\big) =  \lambda \big\langle (x,f),(y,g)\big\rangle + \nu K(x,y),
\eeq 
where $K$ stands for the Killing form on $\G$.
\ermq 

\vskip 0.3cm

Let us now study the group of isometries of bi-invariant metrics on $T^*G$, when $G$ is a semi-simple Lie group.

\subsection{Skew-symmetric Prederivations}

\bpro\label{skew-sym-prederivation-semisimple} 
Let $\G$ be a semi-simple Lie. Then any prederivation of $T^*\G$ which is
skew-symmetric with respect to any orthogonal structure $\mu$ on $T^*\G$ is an inner
derivation of $T^*\G$;  that is if $\phi$ is a $\mu$-skew-symmetric prederivation of  $T^*\G$, then
there exist $x_0$ in $\G$ and $f_0$ in $\G^*$ such that 
\beq
\phi(x,f)=([x_0,x], ad^*_{x_0}f-ad^*_xf_0 ), 
\eeq 
for every $x$ and $f$ in $\G$ and $\G^*$ respectively. 
\epro
\begin{proof}
We have shown in Theorem \ref{thm-prederivation-semi-simple} that if $\G$ is a semi-simple Lie algebra, then  every prederivation
of $T^*\G$ is a derivation. From Relation (\ref{eq:cocycle-semi-simple}) we have that, if $\G$ is semi-simple,  any derivation of $T^*\G$ has the
form
\beq
\phi(x,f)=([x_0,x], ad^*_{x_0}f-ad^*_xf_0 + \sum_{i=1}^p\gamma_if_i)
\eeq
for every $x$ in $\G$ and every $f:=f_1 + f_2 +\cdots +f_p$ in $\s^*_1\oplus \s^*_2 \oplus
\cdots \oplus \s^*_p=\G^*$, where $x_0$ is in $\G$, $f_0$ is in $\G^*$ and $\gamma_i$,
$i=1,..,p$, are real numbers. To prove the Proposition it suffices to prove that $\phi$ is  $\mu$-skew-symmetric if and 
only if $\gamma_i=0$, for all $i=1,2,\ldots,p$.

\beqn
\mu\Big(\phi(x,f),(y,g)\Big) &=& \mu\Big(\big([x_0,x], ad^*_{x_0}f-ad^*_xf_0 + \sum_{i=1}^p\gamma_if_i\big)\,,\,(y,g) \Big) \cr
                             &=& \sum_{i=1}^p\lambda_i\Big\langle \Big([x_0,x]_i, \big(ad^*_{x_0}f-ad^*_xf_0 
                                   + \sum_{k=1}^p\gamma_kf_k\big)_i\Big)\,,\,(y_i,g_i)\Big\rangle_{\s_i} \cr
                             & &    + \sum_{k=1}^p\nu_k\Big\langle \theta(y_k),[x_0,x]_k\Big\rangle_{\s_k} \nonumber
\eeqn
That is
\beqn
\mu\Big(\phi(x,f),(y,g)\Big) &=& \sum_{i=1}^p \lambda_i\left[g_i([x_0,x]_i) +(ad^*_{x_0}f)_i(y_i) - (ad^*_xf_0)_i(y_i)
                                 +\left(\sum_{k=1}^p\gamma_kf_k\right)_i(y_i)  \right] \cr 
                             & & + \sum_{k=1}^p\nu_k\langle \theta(y_k),[x_0,x]_k\rangle \label{skew-sym1}
\eeqn 
By the same way we obtain
\beqn
\mu\Big((x,f),\phi(y,g)\Big) &=& \sum_{i=1}^p \lambda_i\left[f_i([x_0,y]_i) + (ad^*_{x_0}g)_i(x_i) - (ad^*_yf_0)_i(x_i) 
                                  + \left(\sum_{k=1}^p\gamma_kg_k\right)_i(x_i)\right]  \cr
                               & & + \sum_{k=1}^p\nu_k\langle \theta([x_0,y]_k),x_k\rangle. \label{skew-sym2}
\eeqn 
Now $\phi$ is $\mu$-symmetric means that
\beqn 
0 &=& \mu\Big(\phi(x,f),(y,g)\Big) + \mu\Big((x,f),\phi(y,g)\Big) \cr
  &=& -\sum_{i=1}^p \lambda_i \Big[(ad^*_xf_0)_i(y_i) + (ad^*_yf_0)(x_i)\Big] +    
\sum_{k=1}^p\nu_k\Big[\langle \theta(y_k),[x_0,x]_k\rangle +  \langle \theta([x_0,y]_k),x_k\rangle \Big] \cr
  & & + \sum_{i=1}^p\lambda_ig_i([x_0,x]_i) + \sum_{i=1}^p\lambda_i(ad_{x_0}^*f)_i(y_i) 
      + \sum_{i=1}^p\lambda_i\left(\sum_{k=1}^p\gamma_kf_k\right)_i(y_i) \cr 
  & & + \sum_{i=1}^p\lambda_if_i([x_0,y]_i) + \sum_{i=1}^p\lambda_i(ad_{x_0}^*g)_i(x_i) 
      + \sum_{i=1}^p\lambda_i\left(\sum_{k=1}^p\gamma_kg_k\right)_i(x_i) \cr 
  &=& \sum_{i=1}^p\lambda_if_0\Big(\underbrace{[x_i,y_i]+[y_i,x_i]}_{=0}\Big) 
      + \sum_{k=1}^p\nu_k\Big[\underbrace{\langle \theta(y_k),[x_0,x_k]\rangle 
      - \langle\theta(y_k),[x_0,x_k]}_{=0}\rangle\Big] \cr
  & & + \sum_{i=1}^p\lambda_i \left[\underbrace{g_i([x_0,x]_i) + (ad_{x_0}^*g)_i(x_i)}_{=0}\right]  
      + \sum_{i=1}^p\left[\underbrace{f([x_0,y]_i) + (ad_{x_0}^*f)_i(y_i)}_{=0} \right] \cr
  & & + \sum_{i=1}\lambda_i\left[\left(\sum_{k=1}^p\gamma_kf_k\right)_i(y_i) 
      + \left(\sum_{k=1}^p\gamma_kg_k\right)_i(x_i) \right].
\eeqn 
We then have
\beqn 
\sum_{i=1}^p\lambda_i\left[\left(\sum_{k=1}^p\gamma_kf_k\right)_i(y_i) 
      + \left(\sum_{k=1}^p\gamma_kg_k\right)_i(x_i) \right] &=& 0 \cr
\sum_{i=1}^p\lambda_i\Big[\gamma_if_i(y_i) + \gamma_ig_i(x_i)\Big] &=& 0      
\eeqn 
for all $x,y$ in $\G$ and all $f,g$ in $\G^*$. Now if we take $y=0$, we obtain 
\beq 
\sum_{i=1}^p\lambda_i\left(\sum_{k=1}^p\gamma_kg_k\right)_i(x_i) = 0,
\eeq 
for all $x,y$ in $\G$ and all $f,g$ in $\G^*$.
That is 
\beq 
\sum_{i=1}^p \lambda_i\gamma_ig_i(x_i) = 0,
\eeq 
for all $x=x_1+x_2+\cdots+x_p \in \G$ and for all $g=g_1+g_2+\cdots+g_p$ in $\G^*$. 
Since, $\lambda_i \neq 0$, for any $i=1,2,\ldots,p$, then we have 
$\gamma_i=0$, for any $i=1,2,\ldots,p$. 

\end{proof}

\section{Examples}\label{chap3:section:example}


\subsection{The Affine Lie Group of the Real Line $\R$}

Let $G:=\hbox{\rm Aff}(\R)$ be the Lie group of affine motions of $\R$ and let  $\G:=\hbox{\rm aff}(\R)$ its Lie algebra. 
We note by  $T^*\G:=T^*\hbox{\rm aff}(\R)=\G\ltimes\G^*$ 
the Lie algebra of the cotangent bundle of the Lie group $G$. If $\G:=span\{e_1,e_2\}$ and $\{e_3,e_4\}$ denotes the dual basis of $\{e_1,e_2\}$, 
we have the following brackets:
\beq
[e_1,e_2]=e_2, \quad [e_1,e_4]=-e_4, \quad [e_2,e_4]=e_3
\eeq

\bpro
Let $\langle,\rangle$ denote the duality pairing between $\G$ and $\G^*$.
\benum
\item Any orthogonal structure $\mu$ on $T^*\G$ is given by the following expression:
\beq  
\mu\Big((x,f),(y,g) \Big) = a\langle (x,f),(y,g) \rangle + bx_1y_1, \label{orthogonal-structure-aff(R)}
\eeq
for all $(x,f)=x_1e_1+x_2e_2+f_3e_3+f_4e_4$ and $(y,g)=y_1e_1+y_2e_2+g_3e_3+g_4e_4$ in $T^*\G$, where $a \in \R^*$ and $b\in \R$.
\item Any orthogonal structure $\mu$ on $T^*\G$ is of signature $(2,2)$.
\eenum
\epro

\begin{proof}
On the basis $(e_1,e_2,e_3,e_4)$ of $T^*\G$ an invertible bi-invariant tensor $j :T^*\G \to T^*\G$
has the following matrix 
\beq\label{bi-invariant-endom-affine} 
\left(
\begin{array}{cccc}
a & 0 & 0 & 0 \\
0 & a & 0 & 0 \\
b & 0 & a & 0 \\
0 & 0 & 0 & a
\end{array}
\right), \, a \in \R^*,\, b\in \R.
\eeq 
We note  $j_{a,b}:=j$ if $j$ is defined by (\ref{bi-invariant-endom-affine}). One can verify that $j_{a,b}$   
is $\langle,\rangle$-symmetric.
 
\vskip 0.3cm
\noindent
Now any adjoint-invariant scalar product  on $T^*\G$ is given by an adjoint-invariant 
invertible and $\langle,\rangle$-symmetric  endomorphism $j_{a,b}$. 
We note it by $\mu_{j_{a,b}}$ or simply by  $\mu_{a,b}$. We have
\beq 
\mu_{a,b}\Big((x,f),(y,g) \Big) = \langle j_{a,b}(x,f),(y,g) \rangle, 
\eeq  
for all $(x,f)$ and $(y,g)$ in $T^*\G$. Now if we set $(x,f)=x_1e_1+x_2e_2+f_3e_3+f_4e_4$ 
and $(y,g)=y_1e_1+y_2e_2+g_3e_3+g_4e_4$, then we have
\beqn 
\mu_{a,b}\Big((x,f),(y,g) \Big) &=& \big\langle ax_1e_1+ax_2e_2 + (bx_1+af_3)e_3+af_4 \;,\; y_1e_1+y_2e_2+g_3e_3+g_4e_4\big\rangle \cr
                                &=&  ax_1g_3 + ax_2g_4 +(bx_1+af_3)y_1 + ay_2f_4 \cr 
                                &=& a(x_1g_3 + x_2g_4 + y_1f_3 + y_2f_4) + bx_1y_1. \nonumber
\eeqn
That is 
\beq  
\mu_{a,b}\Big((x,f),(y,g) \Big) = a\langle (x,f),(y,g) \rangle + bx_1y_1, \nonumber
\eeq
for all $(x,f)$ and $(y,g)$ in $T^*\G$, where $a \in \R^*$ and $b\in \R$.

\vskip 0.3cm 

Let us study the signature of the scalar product $\mu_{a,b}$. The matrix of $\mu_{a,b}$ on 
the basis $(e_1,e_2,e_3,e_4)$ of $T^*\G$ is given by
\beq 
M_{a,b} = \left(
\begin{array}{cccc}
b & 0 & a & 0 \cr
0 & 0 & 0 & a \cr
a & 0 & 0 & 0 \cr
0 & a & 0 & 0 \cr
\end{array}
\right).
\eeq 
The characteristic polynomial of $M_{a,b}$ is 
\beq 
P_{a,b}(t) = (t-a)(t+a)(t^2-bt-a^2),
\eeq  
for any $t \in \R$. It is now a little matter to check that $M_{a,b}$ has four eigenvalues : 
$$
-a \quad , \quad a \quad , \quad \dfrac{b-\sqrt{b^2+4a^2}}{2} \quad ,\quad \dfrac{b+\sqrt{b^2+4a^2}}{2}.
$$
Two of these eigenvalues are less than zero and two are greater than zero. Hence, $\mu_{a,b}$ is of signature $(2,2)$ 
for all $(a,b) \in \R^*\times \R$. 
\end{proof}

\vskip 0.3cm

Let $G_1$ stand for the identity connected component of $G$. We also note by $\mu_{a,b}$ the bi-invariant metric induced on the Lie group $T^*G_1$ 
by the orthogonal structure $\mu_{a,b}$ given by (\ref{orthogonal-structure-aff(R)}).

\bpro\label{aff(R):skew-symmetric-prederivations}
Any $\mu_{a,b}$-skew-symmetric prederivation $\phi$ of  $T^*\G$ is an inner derivation of $T^*\G$. 
\epro 

\begin{proof}
 A prederivation $\alpha$ of $T^*\G$ has the following matrix on the basis $(e_1,e_2,e_3,e_4)$ of $T^*\G$ 
(see Examples \ref{chap:aff(R)} and \ref{chap2:aff(R)}):
\beq 
\begin{pmatrix}
 0          &        0   &    0      &        0    \cr
\alpha_{21} &\alpha_{22} &    0      & \alpha_{24} \cr
\alpha_{31} &\alpha_{32} &\alpha_{33}&-\alpha_{21}  \cr
-\alpha_{32} &    0      &    0      &\alpha_{33}-\alpha_{22}  
\end{pmatrix}
\eeq 

A $\mu_{a,b}$-skew-symmetric prederivation $\alpha$ of $T^*\G$ is characterised by 
\beq \label{skew-symmetry-prederivation}
\mu_{a,b}\Big(\alpha(x,f),(y,g) \Big) = -\mu_{a,b}\Big((x,f),\alpha(y,g) \Big),
\eeq 
for all $(x,f)$ and $(y,g)$ in $T^*\G$. The left hand side of the equality above is
\beqn
\mu_{a,b}\Big(\alpha(x,f),(y,g) \Big)\!\!\! &=&\!\!\! \mu_{a,b}\Big( (\alpha_{21}x_1 \!+\!\alpha_{22}x_2\!+\!\alpha_{24}f_4)e_2 
                                               \!+\! (\alpha_{31}x_1+\alpha_{32}x_2 \!+\!\alpha_{33}f_3\!-\!\alpha_{21}f_4)e_3 \cr
                                      & & + (-\alpha_{32}x_1\!+\!(\alpha_{33}-\alpha_{22})f_4)e_4 \;\,,\;\,
                                           y_1e_1+y_2e_2+g_3e_3+g_4e_4\Big) \cr
                                      &=&\!\!\! a\Big\langle (\alpha_{21}x_1 \!+\!\alpha_{22}x_2\!+\!\alpha_{24}f_4)e_2 
                                               \!+\! (\alpha_{31}x_1+\alpha_{32}x_2 \!+\!\alpha_{33}f_3\!-\!\alpha_{21}f_4)e_3 \cr
                                      & & + (-\alpha_{32}x_1\! +\!(\alpha_{33}-\alpha_{22})f_4)e_4 \;\,,\;\,
                                           y_1e_1+y_2e_2+g_3e_3+g_4e_4 \Big\rangle  \cr
                                      &=&\!\!\! a\Big[(\alpha_{21}x_1 \!+\!\alpha_{22}x_2\!+\!\alpha_{24}f_4)g_4 \!+\! 
                                         (\alpha_{31}x_1+\alpha_{32}x_2 \!+\!\alpha_{33}f_3\!-\!\alpha_{21}f_4)y_1 \cr 
                                      & & \!\!+\! (-\alpha_{32}x_1\!+\!(\alpha_{33}-\alpha_{22})f_4)y_2   \Big] \cr
                                      &=& a\Big[\alpha_{21}(x_1g_4-y_1f_4) + \alpha_{22}x_2g_4 + \alpha_{24}f_4g_4 + \alpha_{31}x_1y_1 \cr 
                                      & &   + \alpha_{32}(x_2y_1-x_1y_2) + \alpha_{33}y_1f_3 + \alpha_{42}x_2y_2 + \alpha_{44}y_2f_4\Big] 
\eeqn 
The right hand side gives
\beqn
-\mu_{a,b}\Big((x,f),\alpha(y,g) \Big)\!\!\!\!\! &=&\!\!\!\! -\mu_{a,b}\Big( x_1e_1+x_2e_2+f_3e_3+f_4e_4  \;\,,\;\,
                                                (\alpha_{21}y_1 \!+\!\alpha_{22}y_2\!+\!\alpha_{24}g_4)e_2 \cr
                                             & &\!\!\! \! \!\!+\! (\alpha_{31}y_1+\alpha_{32}y_2 \!+\!\alpha_{33}g_3\!-\!\alpha_{21}g_4)e_3 
                                                 \!+\! (-\alpha_{32}y_1\! +\! \alpha_{42}y_2\!+\!\alpha_{44}g_4)e_4 \Big) \cr 
                                             &=&\!\!\!\!\! -a\Big[(\alpha_{21}y_1 \!+\!\alpha_{22}y_2\!+\!\alpha_{24}g_4)f_4 \!+\! 
                                                 x_1(\alpha_{31}y_1+\alpha_{32}y_2 \!+\!\alpha_{33}g_3\!-\!\alpha_{21}g_4) \cr
                                             & &\!\! \! +\! x_2(-\alpha_{32}y_1\! +\! \alpha_{42}y_2\!+\!\alpha_{44}g_4)\Big] \cr
                                             &=& -a\Big[\alpha_{21}(-x_1g_4+y_1f_4) + \alpha_{22}y_2f_4 + \alpha_{24}f_4g_4 \!+\! \alpha_{31}x_1y_1 \cr 
                                      & &   + \alpha_{32}(-x_2y_1+x_1y_2) + \alpha_{33}x_1g_3 + \alpha_{42}x_2y_2 + \alpha_{44}x_2g_4\Big] 
\eeqn 
The equality (\ref{skew-symmetry-prederivation}) then implies that 
 $\alpha_{24}=\alpha_{31}=\alpha_{33} = 0$. In this case $\alpha$ can be written 
$\alpha=\alpha_{21}\phi_1  + \alpha_{22}\phi_2 + \alpha_{32}\phi_4 $. One can readily check that $\phi_1=-ad_{e_2}$, $\phi_2=ad_{e_1}$, $\phi_4=-ad_{e_4}$.
\end{proof}
Let us now focus our attention on isometries of bi-invariant metrics on $T^*G_1$. Recall first the following materials.
\bitem
\item The operation law of $T^*G_1$ is given by (see  Proposition \ref{prop:double-affine}) :
\beq
x\cdot y= \Big(x_1+y_1\;,\; x_2 + y_2e^{x_1}\;,\; x_3+ y_3 + x_2y_4e^{-x_1}\;,\;x_4+y_4e^{-x_1}\Big).
\eeq 
The unit element is $\epsilon =(0,0,0,0)$ and the inverse of an element $x=(x_1,x_2,x_3,x_4)$ is the element
$x^{-1}=(-x_1,-x_2e^{-x_1},-x_3+x_2x_4,-x_4e^{x_1})$.
\item The exponential map of $T^*G_1$ is defined as follows (see Chapter \ref{chapter:affine-lie-group}, Corollary \ref{corollary:exponential-double-affine}) : 
let $\xi =\xi_1e_1+\xi_2e_2+ \xi_3e_3+\xi_4e_4$ be in $T^*\G$,

\beq 
\exp(\xi)\!\! = \!\!\left\{\!\!\!\!\!\!\!\!\!\!
\begin{array}{ll}
& \left(0\;,\;\xi_2\;,\;\dfrac{1}{2}\xi_2\xi_4 + \xi_3\;,\;\xi_4\right), \qquad \qquad \qquad \qquad  \qquad \qquad \quad \quad \quad \;\text{ if } \xi_1 = 0 \cr 
&  \cr
&\left(\!\xi_1 , \dfrac{\xi_2}{\xi_1}[\exp(\xi_1)\!\!-\!\!1], \xi_3 \!\!+\!\! \dfrac{\xi_2\xi_4}{\xi_1} 
\!\!+\!\!\dfrac{\xi_2\xi_4}{\xi_1^2}[\exp(-\xi_1)\!\!-\!\!1], \dfrac{\xi_4}{\xi_1}[1\!\!-\!\!\exp(-\xi_1)]\!\right) \text{ if }  \xi_1 \neq 0.
\end{array}\right. \nonumber
\eeq 
\eitem
For $i=1,2,3,4$, note by $X^{iL}$, $X^{iR}$ the infinitesimal generators of the one-parameter subgroups $L_{exp(te_i)}$ and $R_{exp(te_i)}$ respectively.
It is readily checked that for all $g=(x_1,x_2,x_3,x_4)$ in $T^*G_1$, we have :
\beq 
\begin{array}{lll}
X^{1,L}_{\mid g}=(1,x_2,0,-x_4) &,& X^{1,R}_{\mid g}=(1,0,0,0)        \cr 
                                & &                                   \cr
X^{2,L}_{\mid g}=(0,1,x_4,0)    &,& X^{2,R}_{\mid g}=(0,e^{x_1},0,0)  \cr 
                                & &                                   \cr
X^{3,L}_{\mid g}=(0,0,1,0)      &,& X^{3,R}_{\mid g}=(0,0,1,0)        \cr
                                & &                                   \cr
X^{4,L}_{\mid g}=(0,0,0,1)      &,& X^{4,R}_{\mid g}=(0,0,x_2e^{-x_1},0).  
\end{array}
\eeq 
It is now easy to obtain
\beq 
\begin{array}{lll}
X^{1,s}_{\mid g}=(2,x_2,0,-x_4)               &,& X^{1,a}_{\mid g}=(0,-x_2,0,x_4) \cr 
                                & &                                   \cr
X^{2,s}_{\mid g}=(0,1+e^{x_1},x_4,0)          &,& X^{2,a}_{\mid g}=(0,e^{x_1}-1,-x_4,0)\cr 
                                & &                                   \cr
X^{3,s}_{\mid g}=(0,0,2,0)                    &,& X^{3,a}_{\mid g}=(0,0,0,0) \cr 
                                & &                                   \cr
X^{4,s}_{\mid g}=(0,0,x_2e^{-x_1},1+e^{-x_1}) &,& X^{4,a}_{\mid g}=(0,0,x_2e^{-x_1},e^{-x_1}-1) \cr 
\end{array}
\eeq 
Now $\{\phi_1,\phi_2,\phi_4,X^{1,s},X^{2,s},X^{3,s},X^{4,s}\}$ is a basis of the Lie algebra $\mathcal I(G_1)$ of the Lie group $I(G_1)$ of
isometries of any bi-invariant metric on $T^*G_1$.  Now we just have to compute the brackets on  $\mathcal I(G_1)$ by  
Theorem \ref{thm:muller-brackets-isometries-algebra}. The non-vanishing brackets are the following :
\beq
\begin{array}{llrlllrlllr}
[\phi_1,\phi_2]  &=&-\phi_1   &,& [\phi_1,X^{1,s}] &=& X^{2,s}  &,& [\phi_1,X^{4,s}] &=&-X^{1,s} \cr 
[\phi_2,X^{2,s}] &=& X^{2,s}  &,& [\phi_2,X^{4,s}] &=& -X^{4,s} &,& [\phi_4,X^{1,s}] &=& X^{4,s} \cr
[\phi_4,X^{2,s}] &=&-X^{3,s}  &,& [X^{1,s},X^{2,s}]&=&-X^{2,s}  &,& [X^{1,s},X^{4,s}]&=& X^{4,s} \cr
[X^{2,s},X^{4,s}] &=& -X^{3,s}. &&                              &&                  
\end{array}
\eeq

%
%

\subsection{The Special Linear Group $SL(2,\R)$}


Let $G_2$ denote the special linear group $SL(2,\R)$. 
Set $\G_2:=\mathfrak{sl}(2,\R)=\sspan\{e_1,e_2,e_3\}$ and $T^*\G_2:=\G_2\ltimes\G_2^*=\sspan\{e_1,e_2,e_3,e_4,e_5,e_6\}$.
We have the following brackets (see Example \ref{chap:sl(2)}):
\beq
\begin{array}{ccccccccccc}
[e_1,e_2] &=& -2e_2  & & [e_1,e_3] &=& 2e_3  & & [e_1,e_5] &=& 2e_5  \cr 
[e_1,e_6] &=&-2e_6   & & [e_2,e_3] &=&-e_1   & & [e_2,e_4] &=& e_6 \cr
[e_2,e_5] &=&-2e_4   & & [e_3,e_4] &=&-e_5   & & [e_3,e_6] &=& 2e_4
\end{array}
\eeq
An element $(x,f)$ of $T^*\G_2$ can be written $(x,f)=x_1e_1+x_2e_2+x_3e_3+f_4e_4+f_5e_5+f_6e_6$.
\bpro
Let $\langle,\rangle$ stand for the duality pairing between $\G_2$ and $\G_2^*$ and $\mu$ be any orthogonal structure on $T^*\G_2$. Then, 
\benum
\item there exist $(a,b)$ in $\R^*\times\R$ such that 
\beq\label{orthogonal-structures-sl(2)} 
\mu\big((x,f),(y,g)\big) = a\big\langle(x,f),(y,g)\big\rangle + 4b(2x_1y_1+x_2y_3+x_3y_2),
\eeq 
for all  $(x,f)$ and $(y,g)$ in $T^*\G$.
\item The orthogonal structure $\mu$ on $T^*\G_2$ is of signature $(3,3)$.
\eenum
\epro

\begin{proof}
\benum
\item Since $\G_2$ is a simple Lie algebra,  Theorem \ref{theorem:orthogonal-structure-semi-simple-lie-algebras} asserts that 
any orthogonal structure $\mu_{a,b}$ on $T^*\G_2$ is given by
\beq 
\mu_{a,b}\big((x,f),(y,g)\big) = a\big\langle(x,f),(y,g)\big\rangle + b\langle \theta(x),y\rangle,
\eeq 
for all $(x,f), (y,g)$ in $T^*\G_2$, where $(a,b) \in \R^*\times \R$. Because of the simplicity of 
$\G_2$ every orthogonal structure on $\G_2$ is a multiple of the Killing form. So, we let $\theta$
be induced by the Killing form $K$ of $\G_2$, {\it i.e. } $\langle \theta(x),y\rangle=K(x,y)$, for all $x,y$ in $\G_2$. Then, 
\beq 
\mu\big((x,f),(y,g)\big) = a\big\langle(x,f),(y,g)\big\rangle + b K(x,y), \nonumber
\eeq 
for all $(x,f), (y,g)$ in $T^*\G_2$.

\vskip 0.3cm
\noindent
By definition, the Killing form is given by
\beq 
K(x,y) = trace(ad_x\circ ad_y),
\eeq 
for all $x,y$ in $\G_2$. On the basis of $(e_1, e_2, e_3)$ of $\G_2$, we have the following matrix of $K$
\beq 
\left( 
\begin{array}{ccc}
8  &  0  &  0  \cr
0  &  0  &  4  \cr
0  &  4  &  0
\end{array}
\right) \nonumber
\eeq 
Now it is a little matter to checked that 
\beq 
\mu_{a,b}\big((x,f),(y,g)\big) = a\big\langle(x,f),(y,g)\big\rangle + b(8x_1y_1+4x_2y_3+4x_3y_2) \nonumber
\eeq 

\item The matrix of $\mu_{a,b}$ on the basis of $T^*\G_2$ is 
\beq 
M_{a,b}=
\left(
\begin{array}{cccccc}
8b & 0 & 0 & a & 0 & 0 \\
0 & 0 & 4b & 0 & a & 0 \\
0 & 4b & 0 & 0 & 0 & a \\
a & 0 & 0 & 0 & 0 & 0 \\
0 & a & 0 & 0 & 0 & 0 \\
0 & 0 & a & 0 & 0 & 0
\end{array}
\right).
\eeq 
The characteristic polynomial of $M_{a,b}$ is 
\beq 
P_{a,b}(t) = (t^2-8bt-a^2)\big[t^4-2(a^2+8b^2)t^2+a^4\big],
\eeq 
for all $t\in \R$. The roots of $P_{a,b}$ are
\beq
\begin{array}{cclcccl}
t_1 &=& 4b-\sqrt{a^2+16b^2}                              &;& t_2 &=& 4b+\sqrt{a^2+16b^2} \cr   
    & &                                                  & &     & &                  \cr
t_3 &=&-\sqrt{a^2+8b^2-4|b|\sqrt{a^2+4b^2}}              &;& t_4 &=&+\sqrt{a^2+8b^2-4|b|\sqrt{a^2+4b^2}} \cr
    & &                                                  & &     & &                                              \cr                  
t_5 &=&-\sqrt{a^2+8b^2+4|b|\sqrt{a^2+4b^2}}              &;& t_6 &=&+\sqrt{a^2+8b^2+4|b|\sqrt{a^2+4b^2}}
\end{array} \nonumber
\eeq
The roots $t_1,t_3$ and $t_5$ are negative while $t_2,t_4$ and $t_6$ are positive. Then, the signature of the 
bilinear form  $\mu_{a,b}$ is $(3,3)$, for all $(a,b) \in \R^* \times \R$.
\eenum
\end{proof} 
%

 From Theorem \ref{thm-prederivation-semi-simple}, any prederivation of $T^*\G_2$ is a derivation. Then the space of 
prederivations of $T^*\G_2$ is $\pder\big(T^*\G_2\big)=\der(T^*\G_2)=\sspan(\phi_1,\phi_2,\phi_3,\phi_4,\phi_5,\phi_6,\phi_7)$, 
where $\phi_i$, $i=1,2,3,4,5,6,7$ are defined  in Example \ref{chap:sl(2)}. 
If $\phi$ is a prederivation  of $T^*\G_2$,  
\beq 
\phi  = \alpha_1 \phi_1 + \alpha_2 \phi_2+ \alpha_3 \phi_3 + \alpha_4 \phi_4 + \alpha_5 \phi_5 + \alpha_6 \phi_6 + \alpha_7 \phi_7,
\eeq 
where $\alpha_i$, $i=1,2,3,4,5,6,7$, are real numbers. Now, since $\G_2$ is simple then $\phi$ is skew-symmetric 
with respect to the form (\ref{orthogonal-structures-sl(2)}) if and only if it is an inner derivation 
(Proposition \ref{skew-sym-prederivation-semisimple}). Hence, $\alpha_2=0$ and 
then any $\mu_{a,b}$-skew-symmetric prederivation of $T^*\G_2$ has the form 
\beq 
\phi=\alpha_1 \phi_1 + \alpha_3 \phi_3 + 
\alpha_4 \phi_4 + \alpha_5 \phi_5 + \alpha_6 \phi_6 + 
\alpha_7 \phi_7,
\eeq 
where $\alpha_i$, $i=1,3,4,5,6,7$, are real numbers. 
In the basis of $T^*\G_2$, such prederivation has the following matrix
\beq 
\left(
\begin{array}{cccccc}
       0   &-\alpha_3 & \alpha_4  &   0      &    0    &    0    \cr
-2\alpha_4 &-\alpha_1 &   0       &   0      &    0    &    0    \cr
2\alpha_3  &     0    &\alpha_1   &   0      &    0    &    0    \cr
      0    &\alpha_6  &  \alpha_7 &   0 &2\alpha_4& -2\alpha_3   \cr
-\alpha_6  &    0     &-\alpha_5  & \alpha_3 &\alpha_1 &    0          \cr
-\alpha_7  &\alpha_5  &   0       &-\alpha_4 &   0     &-\alpha_1
\end{array}
\right)
\eeq 

\subsection{The $4$-dimensional Oscillator Lie Group}
In Example \ref{prederivations-osciallator-group} we have defined the $4$-dimensional osciallator Lie group $G_\lambda$ 
and its Lie algebra $\G_\lambda$.

\bpro
Any bi-invariant metric on the $4$-dimensional oscillator group  is Lorentzian and its induced orthogonal structure 
$\langle,\rangle_\lambda$ on the Lie algebra  $\G_\lambda$ has the following form :
\beq 
\langle x,y\rangle_\lambda = k\mu_\lambda(x,y) + mx^{-1}y^{-1}, \label{general-orthogonal-structure-osciallator}
\eeq 
for any $x=x^{-1}e_{-1}+x^0e_0+x^1e_1+\check x^1\check e_1$ and $y=y^{-1}e_{-1}+y^0e_0+y^1e_1+\check y^1\check e_1$ in $\G_\lambda$, where 
$\mu_\lambda$ is the orthogonal structure on $\G_\lambda$ given by (\ref{orthogonal-structure-oscillator}) and $(k,m)$ is in $\R^*\times \R$.
\epro
\begin{proof}
We have already seen that the form $\mu_\lambda$ given by (\ref{orthogonal-structure-oscillator}) is an orthogonal structure on $\G_\lambda$. 
Then, any other orthogonal structure $\langle,\rangle_\lambda$ on $\G_\lambda$ is given by $\langle x,y\rangle_\lambda = \mu_\lambda\big(j(x),y\big)$,
for all $x,y$ in $\G_\lambda$, where $j:\G_\lambda \to \G_\lambda$ in an $\mu_\lambda$-symmetric and invertible bi-invariant tensor. Now, one can check that
such map $j$ is given by $j(x)=kx^{-1}e_{-1}+(mx^{-1}+kx^{0})e_0+kx^1e_1+k\check x^1\check e_1$, where $(k,m)$ is in $\R^*\times \R$. So, we have
\beqn
\langle x,y\rangle_\lambda &=& \mu_\lambda\Big(kx^{-1}e_{-1}+(mx^{-1}+kx^{0})e_0+kx^1e_1+k\check x^1\check e_1 \,,\, y^{-1}e_{-1}+y^0e_0+y^1e_1+\check y^1\check e_1\Big) \cr
                           &=& kx^{-1}y^0 + (mx^{-1}+kx^{0})y^{-1} + \frac{1}{\lambda}(kx^1y^1+k\check x^1\check y^1) \cr
                           &=& kx^{-1}y^0 + mx^{-1}y^{-1}+kx^{0}y^{-1} + \frac{1}{\lambda}(kx^1y^1+k\check x^1\check y^1) \cr
                           &=& k\Big[x^{-1}y^0 + x^{0}y^{-1} + \frac{1}{\lambda}(x^1y^1+\check x^1\check y^1) \Big] + mx^{-1}y^{-1} \cr
                           &=& k\mu_\lambda(x,y) + mx^{-1}y^{-1}. \nonumber
\eeqn 
Now the matrix of $\langle,\rangle_\lambda$ on the basis ($e_{-1}, e_0, e_1, \check e_1$) of $\G_\lambda$ is given by
\beq
M_\lambda=\left(
\begin{array}{cccc}
m  &  k  &          0        &  0               \cr
k  &  0  &          0        &  0               \cr
0  &  0  &\frac{k}{\lambda}  &  0               \cr
0  &  0  &          0        &\frac{k}{\lambda}
\end{array} \right)
\eeq
The characteristic polynomial of $M_\lambda$ is $P_\lambda(t)=\left(t-\frac{k}{\lambda}\right)^2(t^2-mt-k^2)$ and its roots are
\beq 
t_1 =\frac{k}{\lambda} \quad ; \quad t_2=\frac{1}{2}\big(m+\sqrt{m^2+4k^2}\big) \quad ; \quad t_3=\frac{1}{2}\big(m-\sqrt{m^2+4k^2}\big)
\eeq
$t_2$ and $t_3$  are simple roots while $t_1$ is of multiplicity $2$. It is clear that $t_2 \geq 0$, $t_3 \leq 0$ and $t_1$ has the same sign as $k$.
\end{proof}

\bpro
Any orthogonal structure $\mu_\lambda^*$  on $T^*\G_\lambda$ can be written as
\beq 
\mu_\lambda^*\Big((x,f),(y,g)\Big)\!\! =\!\! A\langle (x,f),(y,g)\rangle \!+\! B\mu_\lambda(x,y) 
                                           +\! C(x^{-1}g^0 \!+ y^{-1}f^0) \!+\! Dx^{-1}y^{-1} \!+\! Ef^0g^0, \label{orthogonal-structure-cotangent-osciallator}
\eeq 
for all elements $(x,f)=x^{-1}e_{-1}+x^0e_0+x^1e_1+\check x^1\check e_1+f^{-1}e_{-1}^*+f^0e_0^*+f^1e_1^*+\check f^1\check e_1^*$ and 
$(y,g)=y^{-1}e_{-1}+y^0e_0+y^1e_1+\check y^1\check e_1+g^{-1}e_{-1}^*+g^0e_0^*+g^1e_1^*+\check g^1\check e_1^*$ in $T^*\G_\lambda$,
 where $(A;B,C,D,E)$ is in  $\R^*\times\R^4$ and $\langle,\rangle$ stands for the duality pairing between $\G_\lambda$ and $\G_\lambda^*$.
\epro
\begin{proof}
 Since $\G_\lambda$ is an orthogonal Lie algebra, then (see Theorem \ref{theorem:orthogonal-structure-orthogonal-lie-algebras})
 any orthogonal structure $\mu_\lambda^*$ on $T^*\G_\lambda$ is given by
\beq
\mu_\lambda^*\Big((x,f),(y,g)\Big)=\big\langle g\;,\;j_{11}(x) \big\rangle + \big\langle f\;,\;j_{11}(y)\big\rangle 
+ \big\langle g,j_2 \circ \theta^{-1}(f) \big\rangle + \big\langle \theta\circ j_1(x),y\big\rangle, \nonumber
\eeq 
for all $(x,f)$, $(y,g)$ in $T^*\G_\lambda$, where $\theta : \G_\lambda \to \G_\lambda^*$, $\langle\theta(x),y\rangle = \mu_\lambda(x,y)$ 
for all $x,y$ in $\G_\lambda$;  $j_{11}$, $j_1$ and $j_2$ commute with all adjoint operators of $\G$ and satisfy the following conditions : 
\benum
\item $j_{11}$ is invertible ; 
\item $j_1$ is $\mu_\lambda$-symmetric \big($\langle,\rangle_\lambda$ being defined by (\ref{general-orthogonal-structure-osciallator})\big);
\item $j_2 \circ ad_x =0$ and $\langle j_2 \circ \theta^{-1}(f),g\rangle =\langle f,j_2 \circ \theta^{-1}(g)\rangle$, for all $x$ in $\G$ and 
all $f,g$ in $\G^*$. 
\eenum 
One can easily establish that the endomorphisms $j_{11}$, $j_1$ and $j_2$ have the following matrices on the basis $(e_{-1}, e_0, e_1, \check e_1)$ of $\G_\lambda$ : 
\beq 
j_{11} = \left(
\begin{array}{cccc}
a_{11} &  0  &  0  &  0 \cr
a_{21} &  a_{11}  &  0  &  0 \cr
0 &  0  &  a_{11}  &  0 \cr
0 &  0  &  0  & a_{11} \cr
\end{array}
\right) \, ; \, j_1= \left(
\begin{array}{cccc}
b_{11} &  0  &  0  &  0 \cr
b_{21} &  b_{11}  &  0  &  0 \cr
0 &  0  &  b_{11}  &  0 \cr
0 &  0  &  0  & b_{11} \cr
\end{array}
\right) \, ; \, j_2= \left(
\begin{array}{cccc}
0 &  0  &  0  &  0 \cr
c_{21} &  0  &  0  &  0 \cr
0 &  0  &  0  &  0 \cr
0 &  0  &  0  & 0 \cr
\end{array}
\right) \nonumber
\eeq 
where $a_{11} \neq 0$, $a_{21}$, $b_{11}$, $b_{21}$, $c_{21}$ are real numbers. 
It is also readily checked that, with respect to the basis $(e_{-1}, e_0, e_1, \check e_1)$ 
and ($e_{-1}, e_0, e_1, \check e_1$) of $\G_\lambda$ and $\G_\lambda^*$ respectively, $\theta$ and $\theta^{-1}$ have the following matrices :
\beq 
\theta= \left(
\begin{array}{cccc}
0 &  1  &  0  &  0 \cr
1 &  0  &  0  &  0 \cr
0 &  0  &\frac{1}{\lambda}    &  0 \cr
0 &  0  &  0  & \frac{1}{\lambda} \cr
\end{array}
\right) \,;\, \theta^{-1}= \left(
\begin{array}{crcc}
0 &  1  &  0  &  0 \cr
1 & 0  &  0  &  0 \cr
0 &  0  &\lambda    &  0 \cr
0 &  0  &  0  & \lambda \cr
\end{array}
\right)
\eeq 
It is now a little matter to establish that

\beqn
\mu_\lambda^*\Big((x,f),(y,g)\Big) &=& a_{11}\langle (x,f),(y,g)\rangle + b_{11}\mu_\lambda(x,y) \cr
                                         & &  + a_{21}(x^{-1}g^0 + y^{-1}f^0) + b_{21}x^{-1}y^{-1} + c_{21}f^0g^0.
\eeqn
This latter equality is nothing but  (\ref{orthogonal-structure-cotangent-osciallator}); one just has to put 
$A=a_{11}$, $B=b_{11}$, $C=a_{21}$, $D=b_{21}$, $E=c_{21}$. 
\end{proof}

\chapter{ K\"ahlerian Structure On The Lie Group of Affine Motions of the Real Line $\R$ }\label{chapter:affine-lie-group}
\minitoc

\section{Introduction}

An almost  complex structure on a Lie algebra $\G$, when it exists, is defined by a linear endomorphism $J:\G \to \G$, $J^2=-Id$
($Id$ is the identity map of $\G$). If, in addition, $J$ satisfies the condition 
$$J[x,y] -[Jx,y]-[x,Jy] - J[Jx,Jy] =0,$$ 
for all $x$ and $y$ in $\G$, we will say that $J$ is integrable. Now, an integrable almost complex structure is called a complex structure.
In this case the pair $(\G,J)$ is called a complex algebra.

\vskip 0.3cm

A metric $\mu$ on a complex algebra $(\G,J)$ is called K\"ahlerian if it is hermitian, that is, 
$$
\mu(Jx,Jy) =\mu(x,y),
$$ for all vectors $x$ and $y$ in $\G$, 
and if $J$ is a parallel tensor  with respect to the  connection arising from $\mu$.
Likewise, given a Lie algebra with metric $\mu$, we shall say that a complex structure $J$ on $\G$ is
K\"aahlerian if $\mu$ is K\"ahlerian with respect to $J$ in the above sense. Such a pair
$(J,\mu)$ defines a K\"ahlerian structure on $\G$.

\vskip 0.3cm

A Poisson-Lie structure on a Lie group $G,$ is given by a Poisson tensor $\pi$ on $G,$ 
such that, when the Cartesian product $G\times G$ is equipped with the Poisson tensor $\pi\times\pi$, 
the multiplication $m: ~(\sigma,\tau)\mapsto \sigma\tau$ is a Poisson map between the Poisson manifolds 
$(G\times G, \pi\times\pi)$ and $(G, \pi)$. If $f,$ $g$ are in $\G^*$ and $\bar f,\bar g$ are $\mathcal C^\infty$ 
functions on $G$ with respective derivatives $f=\bar f_{*,\epsilon},$ ~ $g = \bar g_{*,\epsilon}$ at 
the unit $\epsilon$ of $G,$ one defines another element $[ f, g]_*$ of $\G^*$ by setting 
$$
[f,g]_*:=(\{\bar f,\bar g\})_{*,\epsilon}.
$$ 
Then $[f,g]_*$ does not depend on the choice 
of $\bar f$ and $\bar g$ as above, and $(\G^*,[,]_*)$ is a Lie algebra. Now, there is a symmetric 
role played by the spaces $\G$ and $\G^*,$ dual to each other. Indeed, as well as acting on $\G^*$ 
via the coadjoint action, $\G$ is also acted on by $\G^*$ using the coadjoint action of $(\G^*,[,]_*).$
A lot of the most interesting properties and applications of $\pi,$ are encoded in the new Lie 
algebra $(\G\oplus\G^*, [,]_\pi),$ where
\beq
[(x,f),(y,g)]_\pi:=([x,y]+ad^*_f y - ad^*_g x, ad^*_x g - ad^*_y f+ [f,g]_*),
\eeq
for every $x,y$ in $\G$ and every $f,g$ in $\G^*.$

\vskip 0.3cm

The Lie algebras $(\G\oplus\G^*,[,]_\pi)$ and $(\G^*, [,]_*)$ are respectively called the double 
and the dual Lie algebras of the Poisson-Lie group $(G,\pi)$. Endowed with the duality pairing defined 
in (\ref{eq:dualitypairing}), the double Lie algebra of any Poisson-Lie group ($G,\pi$), is an orthogonal 
Lie algebra, such that $\G$ and $\G^*$ are maximal totally isotropic (Lagrangian) subalgebras. 

\vskip 0.3cm

Let $r$ be an element of the wedge product $\wedge^2\G.$ Denote by $r^+$ (resp. $r^-$) the left
(resp. right) invariant bivector field on $G$ with value $r=r^+_\epsilon$ ~ (resp. $r=r^-_\epsilon$)
at $\epsilon$. If $\pi_r:=r^+-r^-$ is a Poisson tensor, then it is a Poisson-Lie tensor and $r$
is called a solution of the Yang-Baxter Equation. If, in particular, $r^+$  is a (left invariant)
Poisson tensor on $G$, then $r$ is called a solution of the Classical Yang-Baxter Equation (CYBE)
on $G$ (or $\G$). In this latter case, the double Lie algebra  $(\G\oplus\G^*,[,]_{\pi_r})$
is isomorphic to the Lie algebra $\D$ of the cotangent bundle $T^*G$ of $G$. See e.g. \cite{di-me-cybe}.
 We may also consider the linear map $\tilde r:\G^*\to\G,$ where $\tilde r (f):=r(f,.).$
The linear map $\theta_r:~(\G\oplus\G^*,[,]_{\pi_r})\to \D$, 
$$\theta_r(x,f):= (x+\tilde r(f),f),$$
is an isomorphism of Lie algebras, between $\D$ and the double Lie algebra of any
Poisson-Lie group structure on $G,$ given by a solution $r$ of the CYBE.

\vskip 0.3cm

Let $G=\hbox{\rm Aff}(\R)$ denote the group of affine motions of the real line $\R$. It possesses a lot of interesting structures : 
symplectic (\cite{agaoka},\cite{bajo-benayadi-medina}), 
complex, affine (\cite{bordeman-medina-ouadfel}), K\"alerian (\cite{lichnerowicz-medina}). Note by $\G=\hbox{\rm aff}(\R)$ its Lie algebra. 
We wish to study the geometry of $G$ as a K\"ahlerian Lie group.
This supposes to describe the symplectic structures, the complex structures, the affine transformations and transformations 
which preserve those structures. Furthermore, the symplectic structure corresponds to a solution of the Classical Yang-Baxter equation $r$
(see \cite{di-me-cybe}). So we will also study the double Lie group $\mathcal D(G,r)$ of $G$ associated to $r$.

\vskip 0.3cm

In Section \ref{section:affine-lie-group} we show how can man construct a sympectic structure on $G$, determine the induced left-invariant
 affine structure,  study the geodesics of this affine structure and compare these geodesics with the integral curves of left-invariant
vector fields on $G$.  In Section \ref{section:double-affine-lie-group} we deal with the geometry of a double Lie group of $G$ associated to
a solution of the Classical Yang-Baxter equation. As explained below, the Lie algebra $\mathcal D(\G,r):=\G\ltimes \G^*(r)$ of any double Lie group
$\mathcal D(G,r)$ of $G$ is isomorphic to the Lie algebra $T^*\G$ of the cotangent  bundle $T^*G:=G\ltimes \G^*$(\cite{di-me-cybe}). The Lie group structure
on $T^*G$ considered here is the one obtained by semi-direct product via the coadjoint action of $G$ on the dual space $\G^*$ of $\G$, considered as 
an Abelian Lie group. We construct a double Lie group $\mathcal D(G,r)$ with Lie algebra $\mathcal D(\G,r)\simeq T^*\G$. The double Lie group
 $\mathcal D(G,r)$ admits an affine structure and a complex structure (\cite{di-me-cybe}). We study the both structures.

\section{The Affine Lie Group of the Real Line}\label{section:affine-lie-group}

\subsection{The Affine Lie Group of $\R$ and its Lie algebra}

An affine transformation of $\R$ is a function $\R \to \R$, $x \mapsto ax + b$, where 
$a$ and $b$ are real numbers ($a \neq 0$). Let $\hbox{\rm Aff}(\R)$  denote the set of all such transformations.
We identify the sets $\hbox{\rm Aff}(\R)$ and $\R^* \times \R$ by putting $f=(a,b)$ if $f: x \mapsto ax + b$.
%
Now consider the operation
\beq\label{group operation}
(a_1,b_1)(a_2,b_2) := (a_1a_2, a_1b_2 + b_1), 
\eeq
for all $(a_1,b_1),(a_2,b_2)$ in $\hbox{\rm Aff}(\R)$.
It is readily checked that, endowed with the composition rule (\ref{group operation})
$\hbox{\rm Aff}(\R)$ is a group.  The identity element is $\epsilon:=(1,0)$ and the inverse of the 
element $(a,b)\in G$ is given by $(a,b)^{-1} = (\dfrac{1}{a},- \dfrac{b}{a})$. 
Considering the underlying manifold $\R^* \times \R$, $G$ is a Lie group. We will note it by $G$.
\brmq
The operation law (\ref{group operation}) is nothing but the composition law of maps.
\ermq



The Lie algebra of $G$ is $\G = \R^2$ (as a set). 
The bracket on $\G$ is given by 
\beq\label{bracket}
[(u',v'),(u,v)] = (0,-uv' + vu').
\eeq
In some basis $(e_1,e_2)$ of $\G$ we have the following :
\beq 
[e_1,e_2] = e_2
\eeq 


\subsection{Symplectic and Affine Structures on the Affine Lie group}

We first recall the affine Lie group of $\R^n$ ($n\in \N$) and how can man  construct 
a symplectic form on it (see \cite{agaoka}). The affine group of $\R^n$ is the even dimensional Lie group
\beq
\hbox{\rm Aff}(\R^n) =\left\{ \left(
                               \begin{array}{cc}
                                A & v \\
                                0 & 1
                               \end{array} \right), \; A \in GL(n,\R), \; v \in \R^n\right\}
\eeq
Its Lie algebra is
\beq
\hbox{\rm aff}(\R^n) :=\left\{ \left(
                               \begin{array}{cc}
                                A & v \\
                                0 & 0
                               \end{array} \right), \; A \in \mathfrak{gl}(n,\R), \; v \in \R^n\right\}
\eeq
Let $e_{ij}$ be the matrix such that the $(i-j)$-entry equals $1$ and the other entries are zero.
$(e_{ij})_{\stackrel{1\leq i\leq n}{1 \leq j\leq n+1} }$ forms a basis of $\hbox{\rm aff}(n,\R)$.
Denote by $(e^*_{ij})_{\stackrel{1\leq i\leq n}{1 \leq j\leq n+1} }$ its dual basis and set
\beq
\alpha_0 := \sum_{k=1}^n e^*_{k,k+1} \quad  \mbox{ and } \quad  \omega_0:= -d\alpha_0.
\eeq
Hence, if $x,y$ belong to $\hbox{\rm aff}(n,\R)$,
\beqn
\omega_0(x,y) &=& -d\alpha_0(x,y), \nonumber \\
              &=& \alpha_0([x,y])  \nonumber
\eeqn



One can check that the $2$-form $\omega_0$ is non-degenerate and  
gives a left invariant symplectic structure $\omega$ on $\hbox{\rm Aff}(n,\R)$ by the following formula :
\beq
\omega_g(X_{|g},Y_{|g}) :=\omega_0(T_gL_{g^{-1}}\cdot X_{|g},T_gL_{g^{-1}}\cdot Y_{|g}), \label{sympectic-structure-induction}
\eeq
for all   $g$ in $\hbox{\rm Aff}(n,\R)$ and all vectors  $X_{|g},Y_{|g}$ in $T_g\hbox{\rm Aff}(n,\R)$, 
where $L_g:\hbox{\rm Aff}(n,\R) \to \hbox{\rm Aff}(n,\R)$, $h\mapsto g\cdot h$ 
stands for the left translation by $g$ in $G$. 
\vskip 0.3cm

From now on we focus our  attention on $G=\hbox{\rm Aff}(\R)$. The  Lie algebra of $G$ can be written 
\beq 
\G:= \hbox{\rm aff}(\R) =\{ \left(
                               \begin{array}{cc}
                                u & v \\
                                0 & 0
                               \end{array} \right); u, v \in \R \}
\eeq 
Put
$$e_{11}:=\left(
\begin{array}{cc}
1 & 0 \\
0 & 0
\end{array}
\right) \quad ; \quad e_{12}:= \left(
\begin{array}{cc}
0 & 1 \\
0 & 0
\end{array} \right).
$$
Then $\G = vect\{e_{11},e_{12}\}$. We rename the vectors as follow : $e_{11}=e_1,e_{12}=e_2$.
We have
\beq
[e_1,e_2] = e_2.
\eeq
Now we put 
$$
\alpha_0 = e^*_{12} = e^*_2  \text{ and } \omega_0 = -d\alpha_0 = -d e^*_2,
$$
that is 
\beq
\omega_0(x,y) = e^*_2([x,y]),
\eeq
for all $x,y$ in $\G$. We then have
\beq
\omega_0(e_1,e_1) = 0 \quad ; \quad \omega_0(e_1,e_2) = 1 \quad  ; \quad \omega_0(e_2,e_2) = 0. \label{symplectic-form-affine}
\eeq
$\omega_0$ is a non-degenerate $2$-form on $\G$. It induces a symplectic form $\omega$ on $G$ by relation  (\ref{sympectic-structure-induction}).
%

\vskip 0.3cm

For any $\xi$  in $\G$, let $X^\xi$ denote the associated left invariant vector field. That is 
\beq 
X^\xi_{\mid g}:=T_\epsilon L_g \cdot \xi,
\eeq 
for any $g$ in $G$, where $\epsilon$ is the identity element of $G$. One defines an affine connection $\nabla$ on $G$ by the following formula (see \cite{chu}) : 
$\forall \; \xi,\eta, \sigma \in \G$,
\beq
\omega(\nabla_{X^\xi}X^\eta,X^\sigma) = - \omega(X^\eta,[X^\xi,X^\sigma]).
\eeq
\noindent
We obtain 
\beq
\nabla_{e_1}e_1 = -e_1 \quad ; \quad \nabla_{e_1}e_2 = 0 \quad ; \quad \nabla_{e_2}e_1 = -e_2 \quad ;
\quad \nabla_{e_2}e_2 = 0. \label{affine-structure-aff}
\eeq
\noindent
Note by $\Gamma_{ij}^k$ the symbols of Christoffel, that is $\nabla_{e_i}e_j=\Gamma_{ij}^ke_k$ (Einstein's summation). We then have 
the following symbols of Christoffel
\beqn
\begin{array}{llrlllrllllllll}
\Gamma_{11}^1 &=& -1 & ; & \Gamma_{11}^2 &=& 0 & ; &  \Gamma_{12}^1 &=& 0 & ; & \Gamma_{12}^2 &=& 0 \cr
\Gamma_{21}^1 &=& 0  & ; & \Gamma_{21}^2 &=&-1 & ; &  \Gamma_{22}^1 &=& 0 & ; & \Gamma_{22}^2 &=& 0
\end{array}
\eeqn  

\vskip 0.3cm 

In the sequel, we will consider the connected component of the unit $\epsilon$ of $G$. Let us note it by $G_0$. That is 
$G_0=\R^*_+\times \R=\{(x,y) \in \R\times \R, x>0\}$. We endow $G_0$ with the connection (also denote by $\nabla$) induced on $G_0$ by the connection
$\nabla$ defined by relation (\ref{affine-structure-aff}).

\subsection{Geodesics of $(G_0,\nabla)$ at the unit}

%
%
Let $(\xi_0,\eta_0)$ be an element of $\G$ and let  $\gamma=(\gamma_1,\gamma_2)$ be an auto-parallel ($\nabla_{\dot \gamma(t)}\dot \gamma(t)=0$, for all $t$) 
curve such that
\beq 
\gamma(0)=\epsilon=(1,0) \quad ; \quad \dot\gamma(0)=(\xi_0,\eta_0). \label{initial-conditions1}
\eeq 
The curve $\gamma$ satisfies the equations :
\beqn
{\ddot\gamma}_1 -\dot\gamma_1^2=0. \label{equation-geodesic1}\\
{\ddot\gamma}_2 -\dot\gamma_1\dot\gamma_2=0. \label{equation-geodesic2}
\eeqn 
%
%
Now we consider the following two cases.

\vskip 0.2cm
\noindent 
{\bf Case 1. } $\gamma(0)=(1,0)$ and $\dot\gamma(0)=(0,\eta_0)$.

\noindent 
If $\gamma_1$ is not constant, then for any $t \in \R$ such that $\gamma_1(t) \neq 0$, we can 
solve equation (\ref{equation-geodesic1}) as follows :
\beqn
{\ddot\gamma}_1 -\dot\gamma_1^2=0 &\Leftrightarrow& \frac{{\ddot\gamma}_1}{\dot\gamma_1^2} = 1 \cr 
                                  &\Leftrightarrow&  -\frac{1}{\dot\gamma_1} = t + c,  \quad c \in \R \cr
                                  &\Leftrightarrow&   \dot\gamma_1(t)  = -\frac{1}{t+c} .
\eeqn  
From the condition $\dot\gamma_1(0)=0$, we obtain : $-\frac{1}{c}=0$, which is not possible, 
then $\dot\gamma_1(t)=0$, for all $t \in \R$. 
We then have : 
\bitem
\item $\gamma_1(t)=constante=a$, for all real number $t$ ;
\item Equation (\ref{equation-geodesic1}) gives $\gamma_2(t) = bt + d$, for all $t$ in $\R$, where $b$ and $d$ belong to $\R$.
\eitem 
From conditions $\gamma(0)=(1,0)$ and $\dot\gamma(0)=(0,\eta_0)$, we have : 
\beq 
a=1 \quad ; \quad d=0 \quad ; \quad b=\eta_0.
\eeq 
So, for any $(0,\eta_0)$ in $\G$ the geodesic through $\epsilon$ with velocity $(0,\eta_0)$ is given by 
\beq 
\gamma(t) = (1,\eta_0 t),
\eeq 
for all $t \in \R$.
%

%

\vskip 0.2cm 
\noindent
{\bf Case 2.} Now we consider an element $(\xi_0,\eta_0)$ of $\hbox{\rm aff}(\R)$, with $\xi_0 \neq 0$. 
Equation (\ref{equation-geodesic1}) can be solve as above and  
\beq
\dot\gamma_1(t)  = -\frac{1}{t+c}.
\eeq  
From the condition $\dot\gamma_1(0)=\xi_0$ we obtain : $-\dfrac{1}{c}=\xi_0$, {\it i.e.} $c=-\dfrac{1}{\xi_0}$. Now we have :
\beqn
\dot\gamma_1(t)  &=& -\frac{1}{t-\frac{1}{\xi_0}} =-\frac{\xi_0}{\xi_0t-1}, \; t\neq \frac{1}{\xi_0}\cr  
    \gamma_1(t)  &=& -\ln|\xi_0t-1| + cste. 
\eeqn  
Since $\gamma_1(0)=1$, we have $cste=1$. It comes that :
\beq 
\gamma_1(t)=1-\ln|\xi_0t-1|,
\eeq 
for any $\displaystyle t \neq \frac{1}{\xi_0}$ and $\displaystyle \frac{1-e}{\xi_0} < t < \frac{1+e}{\xi_0}$, 
since $\gamma_1(t)>0$, if it is defined. Equation (\ref{equation-geodesic2}) now becomes 
\beqn
{\ddot\gamma}_2 +\frac{\xi_0}{\xi_0t-1}\dot\gamma_2=0.
\eeqn
So we have
\beq 
\dot\gamma_2(t) = \frac{C_1}{|\xi_0t-1|},
\eeq 
where $C_1$ is a real number. Since $\dot\gamma_2(0)=\eta_0$, we have $C_1=\eta_0$ and 
\beq 
\dot\gamma_2(t) = \frac{\eta_0}{|\xi_0t-1|}.
\eeq 
Integrating the latter we have :
\beq 
\gamma_2(t) = \frac{\eta_0}{\xi_0}\ln(k|\xi_0t-1|),
\eeq 
where $k >0$, $\displaystyle t \neq \frac{1}{\xi_0}$ and $\displaystyle \frac{1-e}{\xi_0} < t < \frac{1+e}{\xi_0}$.
Since $\gamma_2(0)=0$, we have $k = 1$ and 
\beq 
\gamma_2(t)=\frac{\eta_0}{\xi_0}\ln|\xi_0t-1|,
\eeq  
for any $\displaystyle t \neq \frac{1}{\xi_0}$ and $\displaystyle \frac{1-e}{\xi_0} < t < \frac{1+e}{\xi_0}$. 
Hence, for any element $(\xi_0,\eta_0)$, with $\xi_0 \neq 0$, the geodesic through $\epsilon$ with velocity $(\xi_0,\eta_0)$ is given by :
\beq 
\gamma(t) = \Big(1-\ln|\xi_0t-1|\;,\;\frac{\eta_0}{\xi_0}\ln|\xi_0t-1|\Big), 
\eeq 
for all $\displaystyle t \neq \frac{1}{\xi_0}$ and $\displaystyle \frac{1-e}{\xi_0} < t < \frac{1+e}{\xi_0}$.

\vskip 0.3cm
\noindent
Now we can summarize. 
\bpro
A geodesic through the unit element of $(G_0,\nabla)$ with velocity $(\xi,\eta)$ in $\G$ is given by 
\beq\label{geodesics-affine-structure} 
\gamma(t) = \left\{ 
\begin{array}{lcl}
(1,\eta t )       &\mbox{ if }& \xi =0, \cr
                  &           &          \cr 
(1-\ln|\xi t-1|\;,\;\dfrac{\eta}{\xi}\ln|\xi t -1|) &\mbox{ if }& \xi \neq 0.
\end{array}
\right.
\eeq 
for all $\displaystyle t \neq \frac{1}{\xi}$ and $\displaystyle \frac{1-e}{\xi} < t < \frac{1+e}{\xi}$.
\epro
In order to represent them, let us write Cartesian equations of the geodesics. Set 
\beq 
\gamma(t) = (y(t),x(t)).
\eeq 
\bitem
\item If $\xi=0$, the Cartesian equation of the geodesic through $\epsilon$ with velocity $(0,\eta)$ reads 
\beq 
y=1.
\eeq 
This is the unique complete geodesic and is just a horizontal line through $\epsilon$. 
\item If $\xi \neq 0$ and  $\eta=0$, the Cartesian equation of the geodesic through $\epsilon$ with velocity $(\xi,0)$ is given by 
\beq 
\left\{
\begin{array}{ccc}
x  &=& 0  \cr
y  &>& 0.       
\end{array}
\right.
\eeq 
This geodesic is not complete and is the vertical line through $\epsilon$ contained in the half-plan  $\{(y,x) \in \R^2: y>0\}$.
\item If $\xi\neq 0$ and $\eta \neq 0$, the Cartesian equation of the geodesic through $\epsilon$ with velocity $(\xi,\eta)$ is 
$y=-\dfrac{\xi}{\eta}x+1$ and $y>0$. We can simply write 
\beq 
\left\{
\begin{array}{ccc}
y  &=& ax+1  \cr
y  &>& 0.       
\end{array}
\right. \qquad  a \neq \R.
\eeq
These geodesics are oblique lines through $\epsilon$. They are not complete. 
\eitem 
The Figure \ref{fig-geodesics-affine} represents the geodesics of $(G_0,\nabla)$ through $\epsilon$.

%
%
%

\begin{figure}[htbp]
\centering
\includegraphics[height=10cm,width=13cm]{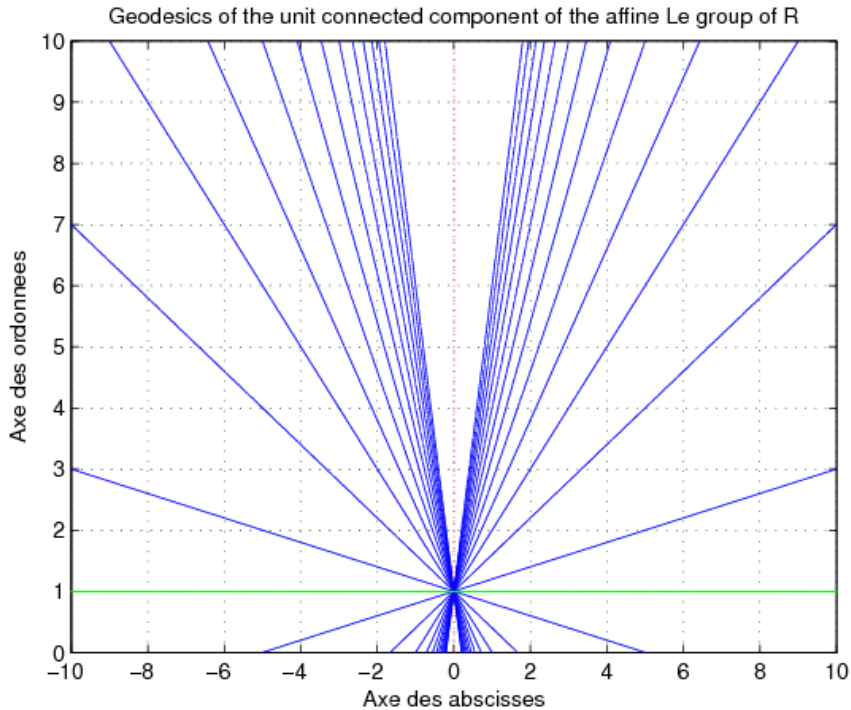}
\caption{Geodesics of $G_0$ through $\epsilon$}
\label{fig-geodesics-affine}
\end{figure}

We just take the geodesics at $t=1$, whenever it is possible, to get the
\bcor
The exponential map of the affine structure $\nabla$ on $G_0$ is defined for any   
$(\xi,\eta)$  in $\G$ such that $\xi \neq 1$ and $1-e< \xi < 1+e$, by 
\beq 
Exp_\epsilon(\xi,\eta) = \left\{ 
\begin{array}{lcl}
(1,\eta)       &\mbox{ if }& \xi =0, \cr
               &           &          \cr 
(1-\ln|\xi-1|\;,\;\frac{\eta}{\xi}\ln|\xi-1|) &\mbox{ if }& \xi \neq 0.
\end{array}
\right.
\eeq 
\ecor 

\vskip 0.3cm

Let us now look at the inverse of $Exp_\epsilon$, whenever it exists. Let $(y,x)$ be an element 
of $G_\epsilon$. We wish to find $(\xi,\eta)$ in $\G$ such that $Exp_\epsilon(\xi,\eta)=(y,x)$.
\benum
\item If $y=1$, then the unique solution is $(\xi,\eta) = (0,x)$.
\item Suppose now $y \neq 1$, then the equation $Exp_\epsilon(\xi,\eta)=(y,x)$ gives the following two equations :
\beqn
1 - \ln|\xi -1|              &=& y \label{a}\\
\dfrac{\eta}{\xi} \ln|\xi-1| &=& x \label{b}
\eeqn 
Equation (\ref{a}) is equivalent to
\beqn
\ln|\xi -1|              &=& 1 -y  \cr
|\xi-1|                  &=& \exp(1-y) \cr
                         & &           \cr
\xi-1                    &=& \left\{
\begin{array}{ccc}
 \exp(1-y)  &\text{ if } & \xi > 1  \cr
-\exp(1-y)  &\text{ if } & \xi < 1
\end{array}
\right.\cr
                        & &           \cr
\xi                     &=& \left\{
\begin{array}{ccc}
1+ \exp(1-y)  &\text{ if } & \xi > 1  \cr
              &            &          \cr
1-\exp(1-y)  &\text{ if }  & \xi < 1
\end{array}
\right.
\eeqn 
Relation (\ref{b})  becomes
\beqn
\dfrac{\eta}{\xi} (1-y) &=& x  \cr
\eta                    &=& \frac{\xi}{1-y}x \cr   
                        & &                  \cr
\eta                    &=& \left\{
\begin{array}{ccc}
\dfrac{1+ \exp(1-y)}{1-y} x &\text{ if }& \xi > 1  \cr
                           &           &          \cr
\dfrac{1-\exp(1-y)}{1-y}x  &\text{ if }& \xi < 1
\end{array}
\right.
\eeqn
\eenum
It comes that the logarithm map of $(G_0,\nabla)$ which is the inverse of $Exp_\epsilon$, is defined only on 
the subset $\{(1,x), x\in \R\}$ of $G_0$.
\bpro
The Logarithm map of $(G_0,\nabla)$ is defined on $\{(1,x), x\in \R\}$ and is given by
\beq 
Log_\epsilon(1,x)=(0,x).
\eeq 
\epro
\subsection{Integral curves of left invariant vector fields on $G$}

If $g=(y,x)$ et $h=(y',x')$ are two elements of $G_0$, the left translation $L_g$ by $g$ acts on  $h$ as follows :
\beq 
L_gh =(yy', yx'+x).
\eeq 
It comes that the differential map at $\epsilon$ of  $L_g$ on the basis $(e_1,e_2)$ has the following matrix :
\beq 
T_\epsilon L_g = \left(
\begin{array}{cc}
y   &    0   \\
0   &    y
\end{array}
\right).\nonumber
\eeq 
Let $\xi = \xi_1e_1 + \xi_2e_2$ be an element of $\G$. We wish to compute the exponential $\exp_G(\xi)$ of $\xi$.
Let $X^\xi$ stand for the left invariant vector field associated to $\xi$. We have :
\beq 
X^\xi_{\mid g}:=T_\epsilon L_g \cdot \xi = \left(
\begin{array}{cc}
y   &    0   \\
0   &    y
\end{array}
\right)\left( \begin{array}{c}
       \xi_1 \\
       \xi_2
        \end{array}
\right) = y\xi_1\frac{\partial}{\partial y} +  y\xi_2\frac{\partial}{\partial x}.
\eeq 
Now let $\gamma=(\gamma_1,\gamma_2)$ be the curve such that 
\beq 
\gamma(0)=\epsilon=(1,0) \text{ and } \dot\gamma(t)=X^\xi_{\mid \gamma(t)}.
\eeq 
We then have the following equations :
\beq \label{equation-geodesic3}
\left\{\begin{array}{lcl}
        \dot \gamma_1(t) &=& \gamma_1(t)\xi_1; \\
        \dot \gamma_2(t) &=& \gamma_1(t)\xi_2.
       \end{array}
\right.
\eeq 
The first relation of the  system (\ref{equation-geodesic3}) has the following solution : 
\beq 
\gamma_1(t) = k \exp(\xi_1 t),\;\; t \in \R,
\eeq 
where $k >0$. Since $\gamma_1(0)=1$, we have $k=1$ and 
\beq 
\gamma_1(t) =\exp(\xi_1 t), \;\; t \in \R.
\eeq 
Now the second equation of the system (\ref{equation-geodesic3}) becomes :
\beq
\dot \gamma_2(t) = \xi_2\exp(\xi_1 t). \label{second-composante-geodesic} 
\eeq
We consider two cases :
\benum
\item[(i)] If $\xi_1=0$, then we have 
\beqn
\dot\gamma_2(t) &=& \xi_2 \cr
\gamma_2(t)     &=& \xi_2 t + q, 
\eeqn 
 where $q$ is a real number. But $\gamma_2(0)=0$, then $q=0$ and 
\beq 
\gamma_2(t) = \xi_2 t, \;\; t \in \R.
\eeq 
\item[(ii)] If $\xi_1 \neq 0$, then the equation (\ref{second-composante-geodesic}) integrates to 
\beq 
\gamma_2(t) = \frac{\xi_2}{\xi_1}\exp(\xi_1t) + r,
\eeq
for some real number  $r$. Since, $\gamma_2(0)=0$, then $\dfrac{\xi_2}{\xi_1}+r=0$, {\it i.e.} $r=-\dfrac{\xi_2}{\xi_1}$.
Hence, we have :
\beq 
\gamma_2(t) = \frac{\xi_2}{\xi_1}\Big[\exp(\xi_1t)-1\Big],
\eeq
for all $t$ in $\R$. 
\eenum
We then summarize  in the
\bpro
For any element $\xi=\xi_1e_1 + \xi_2 e_2$ in $\G$, the integral curve $\gamma_\xi$ of the left-invariant vector field associated to $\xi$
is defined by
\beq 
\gamma(t) = \left\{ 
\begin{array}{ccc}
                         (1\;,\;\xi_2t)                          &\mbox{ if }& \xi_1 = 0;  \cr 
                                                                 &           &             \cr 
 \Big(\exp(\xi_1t) \;,\;\frac{\xi_2}{\xi_1}[\exp(\xi_1t)-1]\Big) &\mbox{ if }& \xi_1 \neq 0.
\end{array}
\right.
\eeq 
for all $t$ in $\R$.
\epro 
It is readily checked that the Cartesian equation of these integral curves are given as follows. Set $\gamma_\xi(t) = (y,x)$.
\bitem 
\item If $\xi_1=0$, then the Cartesian equation of $\gamma_\xi$ is $y=1$. These curve is complete.
\item If $\xi_1\neq 0$ and $\xi_2=0$, then $x=0$ and $y>0$. This is a non-complete integral curve.
\item If $\xi_1\neq 0$ and $\xi_2 \neq 0$, we have the equation : $y = \dfrac{\xi_1}{\xi_2}x+1$ and  $y>0$ 
or simply $y=ax+1$ and  $y>0$, where $a$ is a non-zero real number.
\eitem  
\brmq
The integral curves of the left-invariant vector fields associated to the elements of $\G$ globally coincides with the geodesics
through $\epsilon$ of $G_0$ obtained in (\ref{geodesics-affine-structure}).

\ermq
\bcor
The exponential map of the group $G_0$ is defined by 
\beq 
\exp_G(\xi)=\left\{ 
\begin{array}{ccc}
                         (1\;,\;\xi_2)                           &\mbox{ if }& \xi_1 = 0;  \cr 
                                                                 &           &             \cr 
 \Big(\exp(\xi_1) \;,\;\frac{\xi_2}{\xi_1}[\exp(\xi_1)-1]\Big)   &\mbox{ if }& \xi_1 \neq 0.
\end{array}
\right.
\eeq 
for all $\xi=\xi_1e_1+\xi_2e_2$ in $\G$.
\ecor 
Let us now define the Logarithm map of $G_0$.
Let $(y,x) \in G=\R^*_+\times \R$. We want to find $\xi \in \G$ such that $\exp_G(\xi)=(y,x)$. 
\benum 
\item If $y=1$, then the unique solution is $\xi=(0,x)$.
\item Suppose $y\neq 1$, then 
\beq 
\Big(\exp(\xi_1) \;,\;\dfrac{\xi_2}{\xi_1}[\exp(\xi_1)-1]\Big) = (y,x). 
\eeq 
That is 
\beq
\left\{\begin{array}{ccc}
\exp(\xi_1)                        &=& y \cr
                                   &  &   \cr
\dfrac{\xi_2}{\xi_1}[\exp(\xi_1)-1] &=& x
\end{array}
\right. 
\eeq 
and then
\beq 
 \left\{\begin{array}{ccc}
\xi_1   &=& \ln y \cr
        & &          \cr
\xi_2 &=& \dfrac{x}{y-1}\ln y.
\end{array}
\right.       
\eeq
\eenum
We get,
\bpro
The map $\exp_G$ is invertible and its inverse is the map $Log_G : G \to \G$ defined by
\beq 
Log_G(y,x) = \left\{\begin{array}{ccc}
(0,x)                                          &\mbox{ if }&  y=1 \cr
                                               &           &       \cr
\left(\ln y \, , \, \dfrac{x}{y-1}\ln y\right) &\mbox{ if }& y \neq 1
\end{array}
\right. 
\eeq 
\epro

\section{Double Lie groups  of the affine Lie group of $\R$}\label{section:double-affine-lie-group}

\subsection{Double Lie group of the affine Lie group of $\R$}

Let $G:=\hbox{\rm Aff}(\R)$ be the affine Lie group of $\R$ and let $\G:=\hbox{\rm aff}(\R)$ stand for the affine Lie algebra. It is shown 
(see \cite{di-me-cybe}) that the double Lie algebra $\mathcal D(\G,r)$ (where $r:\G^* \to \G$ is a solution of the Classical
Yang-Baxter Equation), of a double Lie group $\mathcal D(G,r)$ of $G$, is isomorphic to  the Lie algebra $T^*\G$ of the 
cotangent bundle $T^*G=G \ltimes \G^*$ endowed with the Lie group structure obtained by semi-direct
product via the coadjoint action of the Lie group $G$ and the Abelian Lie group $\G^*$. The Lie bracket of $T^*\G$, on some basis
$(e_1,e_2,e_3,e_4)$, reads :
\beq 
[e_1,e_2]=e_2 \quad , \quad [e_1,e_4]=-e_4 \quad , \quad [e_2,e_4]=e_3.
\eeq 
We write $T^*\G=\R e_1\ltimes \mathcal H_3$ (see Example \ref{chap:aff(R)}),  where $e_1$ acts on the Heisenberg Lie algebra 
$\mathcal H_3 = span(e_2,e_3,e_4)$ by the restriction of $ad_{e_1}$. Set  $D:=ad_{e_1}=diag(1,0,-1)$. 
One Lie group of  $T^*\G$ is then the group $\R \ltimes \mathbb H_3$, where $\mathbb H_3$ is the $3$-dimensional 
Heisenberg group ($Lie(\mathbb H_3)=\mathcal H_3$)  and $\R$ acts on $\mathbb H_3$  via the standard 
exponential of the endomorphism $D$ by  $\rho : \R \to Aut(\mathbb H_3)$. Precisely, 
\beq 
\rho_{t} : y \mapsto \exp_{\mathbb H_3}\Big(Exp(tD)Y\Big),
\eeq 
where  $\exp_{\mathbb H_3}$ stands for the exponential map of the Lie group $\mathbb H_3$, $Exp$ is the standard exponential 
of endomorphisms and  $Y$ is an element of $\G$ such that $\exp_{\mathbb H_3}(Y)=y$. The product on $\R \ltimes \mathbb H_3$ reads :
\beq 
(t,x)(s,y) = (t+s,x\cdot \rho_t(y)),
\eeq 
where  $x\cdot \rho_t(y)$ is the product in the Heisenberg group of the  elements $x$ and  $\rho_t(y)$. 
We then need to know the exponential and the logarithm map of the Heisenberg group.
\blem
Let $\exp_{\mathbb H_3}$ and  $\ln_{\mathbb H_3}$ be the exponential and the logarithm maps of the 
Heisenberg group $\mathbb H_3$. Then,
\benum
\item the exponential map is defined by : for all $\xi=(\xi_2,\xi_3,\xi_4)$ in $\mathcal H_3$,
\beq 
\exp_{\mathbb H_3}(\xi) = (\xi_2\;,\;\xi_3 + \frac{1}{2}\xi_2\xi_4\;,\; \xi_4)
\eeq 
\item while the logarithm is given by 
\beq 
\ln_{\mathbb H_3}(y)=(y_2\;,\;y_3-\frac{1}{2}y_2y_4\;,\;y_4),
\eeq 
for any $(y_2,y_3,y_4)$ in $\mathbb H_3$.
\eenum
\elem  
\begin{proof}
 Recall that the multiplication of $ \mathbb H_3$ is given by : 
\beq 
x\cdot y = (x_2,x_3,x_4)\cdot (y_2,y_3,y_4) = (x_2+y_2,x_3+y_3 + x_2y_4,x_4+y_4)
\eeq
and the unit element of $\mathbb H_3$ is $\epsilon_{\mathbb H_3}=(0,0,0)$. The differential map at $\epsilon_{\mathbb H_3}$ of the left translation by 
an element $x$ has the following matrix on the basis $(e_2,e_3,e_4)$ : 
\beq 
T_{\epsilon_{\mathbb H_3}} L_x = \left( 
\begin{array}{ccc}
1 & 0 & 0 \cr
0 & 1 & x_2 \cr 
0 & 0 & 1
\end{array}
\right).
\eeq 
Let us compute the exponential map $\exp_{\mathbb H_3}$. Consider an element  $\xi =\xi_2e_2+ \xi_3e_3+\xi_4e_4$ of $\mathcal H_3$ 
and let $X^\xi$ be the associated left invariant vector field. We note 
by $\gamma_\xi$ the integral curve of $X^\xi$ with initial conditions $\gamma_\xi(0)=\epsilon_{\mathbb H_3}$ and  
$\dot\gamma_\xi(0)=\xi$. We have :
\beq 
X^\xi_{\mid (x_2,x_3,x_4)} = T_\epsilon L_{(x_2,x_3,x_4)} \cdot \xi = (\xi_2,\xi_3 + x_2\xi_4,\xi_4)
\eeq 
and  
\beq 
\dot\gamma_\xi(l) = X^\xi_{\mid \gamma_\xi(l)}.
\eeq 
If we set $\gamma_\xi=(\gamma_\xi^2,\gamma_\xi^3,\gamma_\xi^4)$, we obtain :
\beqn
\dot\gamma_\xi^2(l) &=& \xi_2 \\
\dot\gamma_\xi^3(l) &=& \xi_3 +\gamma_\xi^2(l)\xi_4 \\
\dot\gamma_\xi^4(l) &=& \xi_4.
\eeqn
With the initial conditions we obtain the following solution :
\beqn
\gamma_\xi^2(l) &=& \xi_2l \\
\gamma_\xi^3(l) &=& \xi_3l + \frac{1}{2}\xi_2\xi_4l^2\\
\gamma_\xi^3(l) &=& \xi_4l.
\eeqn
It comes that  
\beq 
\exp_{\mathbb H_3}(\xi) = (\xi_2\;,\;\xi_3 + \frac{1}{2}\xi_2\xi_4\;,\; \xi_4).
\eeq 
Let now $Y=(Y_2,Y_3,Y_4)$ be an element of $\mathcal H_3$ such that $\exp_{\mathbb H_3}(Y)=y$. Then 
$$
(Y_2,Y_3+\frac{1}{2}Y_2Y_4,Y_4) = (y_2,y_3,y_4).
$$ 
We obtain that 
\beq 
Y_2=y_2 \quad ;\quad Y_3=y_3 - \frac{1}{2}y_2y_4 \quad ; \quad  Y_4=y_4.
\eeq
That is 
\beq 
\ln_{\mathbb H_3}(y)=(y_2\;,\;y_3-\frac{1}{2}y_2y_4\;,\;y_4),
\eeq 
where $\ln_{\mathbb H_3}$ is the inverse map of $\exp_{\mathbb H_3}$.
\end{proof}
Now we have :
\beq 
Exp(tD) = diag(e^t,1,e^{-t}) = \left(
\begin{array}{ccl}
e^t  &  0   &  0 \cr 
0    &  1   &  0 \cr
0    &  0   & e^{-t}
\end{array}
 \right).
\eeq 
 Hence, 
\beqn  
\rho_t(y) &=& \exp_{\mathbb H_3}\Big(Exp(tD)Y\Big) \cr 
          &=& \exp_{\mathbb H_3}\Big(y_2e^t\;,\; y_3 - \frac{1}{2}y_2y_4\;,\; y_4e^{-t}\Big) \cr 
          &=& \Big(y_2e^t \;,\; y_3 - \frac{1}{2}y_2y_4 + \frac{1}{2}y_2e^t\times y_4e^{-t}\;,\;y_4e^{-t}\Big) \cr 
          &=& \Big(y_2e^t \;,\; y_3 \;,\;y_4e^{-t}\Big)
\eeqn   
Now we can write the multplication on $\R\ltimes \mathbb H_3$ as follows : 
\beqn  
(t,x)\cdot(s,y) &=& (t,x_2,x_3,x_4)\cdot(s,y_2,y_3,y_4) \cr 
           &=& \Big(t+s\;,\; (x_2,x_3,x_4)\cdot \rho_t(y_2,y_3,y_4)\Big) \cr 
           &=& \Big(t+s\;,\; (x_2,x_3,x_4)\cdot (y_2e^t,y_3,y_4e^{-t})\Big) \cr 
           &=& \Big(t+s\;,\; x_2 + y_2e^t\;,\; x_3+ y_3 + x_2y_4e^{-t}\;,\;x_4+y_4e^{-t}\Big).
\eeqn 
Hence, the product of two elements $x=(x_1,x_2,x_3,x_4)$,  $y=(y_1,y_2,y_3,y_4)$ of $\R\times \mathbb H_3$ is given by 
\beq\label{product-double-affine} 
x\cdot y= \Big(x_1+y_1\;,\; x_2 + y_2e^{x_1}\;,\; x_3+ y_3 + x_2y_4e^{-x_1}\;,\;x_4+y_4e^{-x_1}\Big).
\eeq 
The unit element is $\epsilon =(0,0,0,0)$ and the inverse of an element $x=(x_1,x_2,x_3,x_4)$ is the element
$x^{-1}=(-x_1,-x_2e^{-x_1},-x_3+x_2x_4,-x_4e^{x_1})$.

\vskip 0.3cm

We have prove the following
\bpro\label{prop:double-affine}
Endowed with the product (\ref{product-double-affine}), $\R\times \mathbb H_3$ is a double Lie group of the affine Lie group of $\R$.
\epro
Let $\D(G,r)$ denote the double Lie group of $G$ defined in Proposition \ref{prop:double-affine}.

\subsection{Connection on the double of the affine Lie group }
Among a lot of possibilities, we are interested in the left invariant affine 
connection on $\mathcal D(G,r)$, introduced on any double Lie group of a symplectic Lie group by 
Diatta and Medina in \cite{di-me-cybe} :
\beq 
\nabla_{(x,\alpha)^+}(y,\beta)^+ = (x\cdot y +ad_\alpha y  , ad^*_{r(\alpha)}\beta + ad_x^*\beta)^+,
\eeq 
where $x\cdot y$ is the product induced by the symplectic structure on $\G$ through the formula 
\beq 
\omega(x\cdot y,z) = -\omega(y,[x,z]),
\eeq 
for all $x,y,z \in \G$. 
\bpro
On the basis of $\mathcal D(\G,r)$, the connection is given by
 
\beq 
\begin{array}{rrlrrrlrrrlrrrl}
\nabla_{e_1}e_1 &=& -e_1      &;& \nabla_{e_1}e_2 &=& 0   &;& \nabla_{e_1}e_3 &=& e_2 &;& \nabla_{e_1}e_4 &=& -e_1-e_4 \cr 
\nabla_{e_2}e_1 &=& -e_2      &;& \nabla_{e_2}e_2 &=& 0   &;& \nabla_{e_2}e_3 &=& 0   &;& \nabla_{e_2}e_4 &=& e_3      \cr 
\nabla_{e_3}e_1 &=& e_2       &;& \nabla_{e_3}e_2 &=& 0   &;& \nabla_{e_3}e_3 &=& 0   &;& \nabla_{e_3}e_4 &=& -e_3     \cr
\nabla_{e_4}e_1 &=& -e_1-e_4  &;& \nabla_{e_4}e_2 &=& e_3 &;& \nabla_{e_4}e_3 &=& 0   &;& \nabla_{e_4}e_4 &=& -e_4      
\end{array} \nonumber
\eeq

\epro

\begin{proof}
 We have (see Relation (\ref{symplectic-form-affine}))
\beq 
\omega(e_1,e_2) = 1
\eeq 
and (see Relation (\ref{affine-structure-aff}))
\beq 
e_1 \cdot e_1 = -e_1 \quad ; \quad e_2\cdot e_1 = -e_2.
\eeq 
Let us compute the bracket $[e_3,e_4]_*$ on $\G^*(r)$.
\beqn
[e_3,e_4]_* = ad^*_{r(e_3)}e_4 - ad^*_{r(e_4)}e_3,
\eeqn 
where $r:=q^{-1}:\mathcal{G}^* \mapsto \mathcal{G}$, with   $\langle q(x),y\rangle =\omega(x,y)$.
\beqn 
 \langle q(e_1),e_1\rangle &=&\omega(e_1,e_1) =0. \\
 \langle q(e_1),e_2\rangle &=&\omega(e_1,e_2) =1.\label{1}
\eeqn
Then,  
$ q(e_1) = e^*_2 = e_4 $.
\beqn
 \langle q(e_2),e_1\rangle &=&\omega(e_2,e_1) = -1.\\
 \langle q(e_2),e_2 \rangle &=& \omega(e_2,e_2) = 0.
\eeqn 
Hence,  $q(e_2) = -e^*_1 = -e_3 $. We then have the matrix of $q$ on the basis $(e_1,e_2)$ and $(e_3,e_4)$ of  $\G$ and  $\G^*$ respectively :
\beq 
M= \left(
             \begin{array}{cc}
              0 & -1 \\
             1 & 0 
            \end{array}
\right).
\eeq
The  matrix  $M$ est invertible and its inverse reads :  
\beq 
M^{-1}=
\left(
                 \begin{array}{cc}
                 0 & 1\\
                 -1 & 0
                \end{array}
\right).
\eeq 
It comes  that $r(e_3)=-e_2$ et $r(e_4)=e_1$. Hence,
\begin{eqnarray}
 [e_3,e_4]_* &=& -ad^*_{e_2} e_4 -ad^*_{e_1} e_3. \cr
             &=& + e_4\circ ad_{e_2} +e_3 \circ ad_{e_1}. 
\end{eqnarray}
\beqn
e_4 \circ ad_{e_2} (e_i)=
\left\{\begin{array}{ccc}
          -1 & \mbox{ si } &   i=1\\
           0 & \mbox{ si } & i=2  
         \end{array} 
\right.\quad \Rightarrow e_4\circ ad_{e_2}=-e_3.
\eeqn
\beqn
e_3 \circ ad_{e_1} (e_i)=
\left \{
         \begin{array}{ccc}
          0 & \mbox{ si } &   i=1\\
          0 & \mbox{ si } &   i=2
         \end{array}
\right. \Rightarrow e_3 \circ ad_{e_1}=0.
\eeqn
We then have $[e_3,e_4]_*=-e_3.$ 

\vskip 0.3cm

We can now compute the connection on the basis of $\mathcal D(\G,r)$.

\beqn 
\nabla_{e_1}e_1 &=& e_1\cdot e_1 = -e_1. \\
\nabla_{e_1}e_2 &=& e_1\cdot e_2 = 0. \\
\nabla_{e_1}e_3 &=& ad^*_{e_3} e_1 + ad^*_{e_1}e_3 = ad^*_{e_3} e_1 =-e_1 \circ {ad_{e_3}}_{\mid \G^*(r)} = e_2 \\
\nabla_{e_1}e_4 &=& ad^*_{e_4} e_1 + ad^*_{e_1}e_4 =-e_1 \circ {ad_{e_4}}_{\mid \G^*(r)} + -e_4 \circ {ad_{e_1}}_{\mid \G} =-e_1-e_4. \\
\nabla_{e_2}e_1 &=& e_2 \cdot e_1 = -e_2. \\
\nabla_{e_2}e_2 &=& e_2 \cdot e_2 = 0. \\
\nabla_{e_2}e_3 &=& ad^*_{e_3} e_2 + ad^*_{e_2}e_3 = 0. \\ 
\nabla_{e_2}e_4 &=& ad^*_{e_4} e_2 + ad^*_{e_2}e_4 = e_3. \\
\nabla_{e_3}e_1 &=& \nabla_{e_1}e_3 = e_2. \\ 
\nabla_{e_3}e_2 &=&\nabla_{e_2}e_3 = 0. \\
\nabla_{e_3}e_3 &=& ad^*_{r(e_3)}e_3 =-ad^*_{e_2}e_3=0. \\ 
\nabla_{e_3}e_4 &=&ad^*_{r(e_3)}e_4 =-ad^*_{e_2}e_4=-e_3. \\
\nabla_{e_4}e_1 &=& \nabla_{e_1}e_4 = -e_1-e_4.\\
\nabla_{e_4}e_2 &=&\nabla_{e_2}e_4 = e_3. \\
\nabla_{e_4}e_3 &=& ad^*_{r(e_4)}e_3 =ad^*_{e_1}e_3=0. \cr
\nabla_{e_4}e_4 &=& ad^*_{r(e_4)}e_4 =ad^*_{e_1}e_4=-e_4.
\eeqn 

Thus, if we set $\nabla_{e_i}e_j=\Gamma_{ij}^ke_k$ (Einstein's summation), the Christoffel's symbols $\Gamma_{ij}^k$ 
for this connection are :
\beq
\begin{array}{rclcccccccccccl}
\Gamma_{11}^1 &=&-1  &;& \Gamma_{11}^2 &=& 0 &;& \Gamma_{11}^3 &=& 0 &;& \Gamma_{11}^4 &=& 0 \cr
\Gamma_{12}^1 &=& 0  &;& \Gamma_{12}^2 &=& 0 &;& \Gamma_{12}^3 &=& 0 &;& \Gamma_{12}^4 &=& 0 \cr
\Gamma_{13}^1 &=& 0  &;& \Gamma_{13}^2 &=& 1 &;& \Gamma_{13}^3 &=& 0 &;& \Gamma_{13}^4 &=& 0 \cr
\Gamma_{14}^1 &=&-1  &;& \Gamma_{14}^2 &=& 0 &;& \Gamma_{14}^3 &=& 0 &;& \Gamma_{14}^4 &=&-1 \cr

\Gamma_{21}^1 &=& 0  &;& \Gamma_{21}^2 &=&-1 &;& \Gamma_{21}^3 &=& 0 &;& \Gamma_{21}^4 &=& 0 \cr
\Gamma_{22}^1 &=& 0  &;& \Gamma_{22}^2 &=& 0 &;& \Gamma_{22}^3 &=& 0 &;& \Gamma_{22}^4 &=& 0 \cr
\Gamma_{23}^1 &=& 0  &;& \Gamma_{23}^2 &=& 0 &;& \Gamma_{23}^3 &=& 0 &;& \Gamma_{23}^4 &=& 0 \cr
\Gamma_{24}^1 &=& 0  &;& \Gamma_{24}^2 &=& 0 &;& \Gamma_{24}^3 &=& 1 &;& \Gamma_{24}^4 &=& 0 \cr

\Gamma_{31}^1 &=& 0  &;& \Gamma_{31}^2 &=& 1 &;& \Gamma_{31}^3 &=& 0 &;& \Gamma_{31}^4 &=& 0 \cr
\Gamma_{32}^1 &=& 0  &;& \Gamma_{32}^2 &=& 0 &;& \Gamma_{32}^3 &=& 0 &;& \Gamma_{32}^4 &=& 0 \cr
\Gamma_{33}^1 &=& 0  &;& \Gamma_{33}^2 &=& 0 &;& \Gamma_{33}^3 &=& 0 &;& \Gamma_{33}^4 &=& 0 \cr
\Gamma_{34}^1 &=& 0  &;& \Gamma_{34}^2 &=& 0 &;& \Gamma_{34}^3 &=&-1 &;& \Gamma_{34}^4 &=& 0 \cr

\Gamma_{41}^1 &=&-1  &;& \Gamma_{41}^2 &=& 0 &;& \Gamma_{41}^3 &=& 0 &;& \Gamma_{41}^4 &=&-1 \cr
\Gamma_{42}^1 &=& 0  &;& \Gamma_{42}^2 &=& 0 &;& \Gamma_{42}^3 &=& 1 &;& \Gamma_{42}^4 &=& 0 \cr
\Gamma_{43}^1 &=& 0  &;& \Gamma_{43}^2 &=& 0 &;& \Gamma_{43}^3 &=& 0 &;& \Gamma_{43}^4 &=& 0 \cr
\Gamma_{44}^1 &=& 0  &;& \Gamma_{44}^2 &=& 0 &;& \Gamma_{44}^3 &=& 0 &;& \Gamma_{44}^4 &=&-1 \cr
\end{array}\nonumber
\eeq 
The only non vanishing symbols are :
\beq
\begin{array}{rclcccccccccccl}
\Gamma_{11}^1 &=&-1 &;& \Gamma_{13}^2 &=& 1 &;& \Gamma_{14}^1 &=&-1 &;& \Gamma_{14}^4 &=&-1 \cr
\Gamma_{21}^2 &=&-1 &;& \Gamma_{24}^3 &=& 1 &;& \Gamma_{31}^2 &=& 1 &;& \Gamma_{34}^3 &=&-1 \cr
\Gamma_{41}^1 &=&-1 &;& \Gamma_{41}^4 &=&-1 &;& \Gamma_{42}^3 &=& 1 &;& \Gamma_{44}^4 &=&-1 \cr
\end{array}\nonumber
\eeq 

\end{proof}

\subsection{Geodesics of $(\D(G,r),\nabla)$}

Now let $\gamma(t) =\big(\gamma_1(t),\gamma_2(t),\gamma_3(t),\gamma_4(t)\big)$ be a geodesic such that 
$\gamma(0) = (0,0,0,0)$ in $T^*G\simeq \D(G,r)$ and $\dot\gamma(0)=(\xi_1,\xi_2,\xi_3,\xi_4) \in \mathcal D(\G,r)$. 
We have the following non-linear differential equations which are the equations of geodesics of $\D(G,r)$.
\beqn
\ddot\gamma_1 - \dot\gamma_1^2 - \dot\gamma_1\dot\gamma_4 - \dot\gamma_1\dot\gamma_4 =0, \cr  
\ddot\gamma_2 + \dot\gamma_1\dot\gamma_3 - \dot\gamma_1\dot\gamma_2 + \dot\gamma_1\dot\gamma_3 =0, \cr
\ddot\gamma_3 + \dot\gamma_2\dot\gamma_4 - \dot\gamma_3\dot\gamma_4 + \dot\gamma_2\dot\gamma_4 =0, \cr
\ddot\gamma_4 - \dot\gamma_4^2 - \dot\gamma_1\dot\gamma_4 - \dot\gamma_1\dot\gamma_4 =0. \nonumber
\eeqn 
We rearrange the latter equations as follows : 
\beqn
\ddot\gamma_1 - \dot\gamma_1^2 - 2\dot\gamma_1\dot\gamma_4 =0, \\  
\ddot\gamma_2 - \dot\gamma_1\dot\gamma_2 + 2\dot\gamma_1\dot\gamma_3 =0, \\
\ddot\gamma_3 - \dot\gamma_3\dot\gamma_4 + 2\dot\gamma_2\dot\gamma_4 =0, \\
\ddot\gamma_4 - \dot\gamma_4^2 -  2\dot\gamma_1\dot\gamma_4 =0. 
\eeqn 
Unfortunately we do not yet have a solution for this system.

\subsection{Integral curves of left-invariant vector fields on  the double Lie group of the affine Lie group of $\R$}

Let  $\xi =\xi_1e_1+\xi_2e_2+ \xi_3e_3+\xi_4e_4$ be in $\mathcal D(\G,r)$. 
The left invariant vector field $X^\xi$ associated to $\xi$ is given by 
\beqn 
X^\xi_{\mid (x_1,x_2,x_3,x_4)} &=& T_\epsilon L_{(x_1,x_2,x_3,x_4)}\cdot \xi \cr 
                               &=& \left(\begin{array}{cccc}
                                          1   &      0   &  0  &  0 \cr 
                                          0   & e^{x_1}  &  0  &  0 \cr
                                          0   &      0   &  1  & x_2e^{-x_1} \cr 
                                          0   &      0   &  0  & e^{-x_1}
                                         \end{array}
 \right)\left( \begin{array}{c}
\xi_1 \cr 
\xi_2 \cr 
\xi_3 \cr 
\xi_4
\end{array}
\right) \cr  
                          & &                                        \cr
                          &=& (\xi_1\;,\;\xi_2e^{x_1}\;,\; \xi_3 + x_2\xi_4e^{-x_1}\;,\;\xi_4e^{-x_1}).
\eeqn 
Let $\gamma_\xi$ the unique curve such that $\gamma_\xi(0)=\epsilon$, $\dot\gamma_\xi(0)=\xi$ and 
\beq 
\dot\gamma_\xi(t)=X^\xi_{\mid \gamma_\xi(t)}. \label{eq:integral-curve}
\eeq 
If $\gamma_\xi(t)=(\gamma_\xi^1(t),\gamma_\xi^2(t),\gamma_\xi^3(t),\gamma_\xi^4(t))$, we have the following equations coming from  (\ref{eq:integral-curve}) :
\beqn 
\dot \gamma_\xi^1(t) &=& \xi_1 \label{integrale-curve1}\\
\dot \gamma_\xi^2(t) &=& \xi_2\exp[\gamma_\xi^1(t)] \label{integrale-curve2}\\
\dot \gamma_\xi^3(t) &=& \xi_3 + \xi_4\gamma_\xi^2(t)\exp[-\gamma_\xi^1(t)] \label{integrale-curve3}\\
\dot \gamma_\xi^4(t) &=& \xi_4\exp[-\gamma_\xi^1(t)]. \label{integrale-curve4}
\eeqn 
Resolving relation (\ref{integrale-curve1}) and taking care with the initial conditions we have :
\beq 
\gamma_\xi^1(t) = \xi_1t.
\eeq 
Relation (\ref{integrale-curve2}) is resolved as follows.
\bitem
\item If $\xi_1=0$, then 
\beqn
\gamma_\xi^2(t) &=& \xi_2t + b, \quad b \in \R \cr 
               &=& \xi_2t \quad  \mbox{ since } \gamma_\xi^2(0)=0.
\eeqn
\item If $\xi_1\neq 0$, then we have 

\beqn
\gamma_\xi^2(t) &=& \frac{\xi_2}{\xi_1} \exp(\xi_1t) + b_1, \quad b_1 \in \R ; \cr 
                &=& \frac{\xi_2}{\xi_1}[\exp(\xi_1t)-1], \quad \mbox{ since } \gamma_\xi^2(0) = 0.
\eeqn 
\eitem
Now we come to the relation (\ref{integrale-curve3}) and again consider two cases.
\bitem
\item If $\xi_1 = 0$, the equation (\ref{integrale-curve3}) can be written as  
\beq 
\dot \gamma_\xi^3(t) = \xi_3 + \xi_4\xi_2t.
\eeq 
The latter equation is solved as follows : 
\beqn
\gamma_\xi^3(t) &=& \xi_3t +\frac{1}{2} \xi_4\xi_2t^2 + c_1, \quad c_1 \in \R; \cr 
                &=& \xi_3t +\frac{1}{2} \xi_4\xi_2t^2, \quad \mbox{ as } \gamma_\xi^3(0)=0.
\eeqn
\item If $\xi_1\neq 0$, the relation (\ref{integrale-curve3}) can be written as 
\beqn 
\dot \gamma_\xi^3(t) &=& \xi_3 + \xi_4\frac{\xi_2}{\xi_1}[\exp(\xi_1t)-1]\exp(-\xi_1t) \cr 
                     &=& \xi_3 + \frac{\xi_2\xi_4}{\xi_1}[1-\exp(-\xi_1t)].
\eeqn
and  is integrated as 
\beqn
\gamma_\xi^3(t) &=& \xi_3t +\frac{\xi_2\xi_4}{\xi_1}\left[t +\frac{1}{\xi_1}\exp(-\xi_1t)\right] + c_2 \cr 
                &=& \xi_3t +\frac{\xi_2\xi_4}{\xi_1}\left[t +\frac{1}{\xi_1}\exp(-\xi_1t)\right] -\frac{\xi_2\xi_4}{\xi_1^2}, \; \mbox{ since } \gamma_\xi^3(0)=0. \cr 
                &=& \left(\xi_3 + \frac{\xi_2\xi_4}{\xi_1}\right)t +\frac{\xi_2\xi_4}{\xi_1^2}\left[\exp(-\xi_1t) -1\right].
\eeqn
\eitem
Let us now solve equation (\ref{integrale-curve4}). 
\bitem 
\item If $\xi_1 =0$, then (with the condition $\gamma_\xi^4(0)=0$)
\beq 
\gamma_\xi^4(t) = \xi_4t.
\eeq 
\item If $\xi_1 \neq 0$, then 
\beqn
\gamma_\xi^4(t) &=& -\frac{\xi_4}{\xi_1}\exp(-\xi_1t) + d, \quad d \in \R; \cr 
              &=& \frac{\xi_4}{\xi_1}[1-\exp(-\xi_1t)], \mbox{ since } \gamma_\xi^4(0)=0.
\eeqn
\eitem
Let us summarize all the above. 
\bpro
The integral curve of the left invariant vector field associated to any element $\xi =\xi_1e_1+\xi_2e_2+ \xi_3e_3+\xi_4e_4$ 
of  $\mathcal D(\G,r)$ is defined by  
\beq                                                                    
\gamma_\xi(t)=\left(0\;,\;\xi_2t\;,\;\dfrac{1}{2}\xi_2\xi_4t^2 + \xi_3t\;,\;\xi_4t\right),   
\eeq 
if  $\xi_1=0$;  and by 
\beq 
\gamma_\xi(t)\!\!=\!\!  
\left(\!\xi_1t , \dfrac{\xi_2}{\xi_1}[\exp(\xi_1t)\!-\!1] , \,\left(\xi_3\!\! +\!\! \dfrac{\xi_2\xi_4}{\xi_1}\right)t 
\!\!+\!\!\dfrac{\xi_2\xi_4}{\xi_1^2}[\exp(-\xi_1t)\!\!-\!\!1], \dfrac{\xi_4}{\xi_1}[1\!\!-\!\!\exp(-\xi_1t)]\right)
\eeq 
if $\xi_1 \neq 0$.
\epro 
As an immediate consequence, we have the 
\bcor\label{corollary:exponential-double-affine} 
The exponential map of the double Lie group $\mathcal D(G,r)$ of the affine Lie group of $\R$ is defined as follows. For any 
$\xi =\xi_1e_1+\xi_2e_2+ \xi_3e_3+\xi_4e_4$,
\beq 
\exp_{\D(G,r)}(\xi)\!\! = \!\!\left\{\!\!\!\!\!\!\!\!\!\!
\begin{array}{ll}
& \left(0\;,\;\xi_2\;,\;\dfrac{1}{2}\xi_2\xi_4 + \xi_3\;,\;\xi_4\right), \qquad \qquad \qquad \qquad  \qquad \qquad \quad \quad \quad \;\text{ if } \xi_1 = 0 \cr 
&  \cr
&\left(\!\xi_1 , \dfrac{\xi_2}{\xi_1}[\exp(\xi_1)\!\!-\!\!1], \xi_3 \!\!+\!\! \dfrac{\xi_2\xi_4}{\xi_1} 
\!\!+\!\!\dfrac{\xi_2\xi_4}{\xi_1^2}[\exp(-\xi_1)\!\!-\!\!1], \dfrac{\xi_4}{\xi_1}[1\!\!-\!\!\exp(-\xi_1)]\!\right) \text{ if }  \xi_1 \neq 0.
\end{array}\right. \nonumber
\eeq 
\ecor 

\vskip 0.3cm 

We are now going to deal with the invertibility of the exponential map above.
%
%
Let $(x,y,z,t)$ be an arbitrary element of $\D(G,r)$. Our goal is to find $\xi$ in $T^*\G$ such   that $\exp_{\D(G,r)}(\xi)=(x,y,z,t)$. 
\benum 
\item If $x=0$, then  
\beqn 
\left(0,\xi_2,\dfrac{1}{2}\xi_2\xi_4 + \xi_3,\xi_4\right) = (x,y,z,t) \cr
\left\{\begin{array}{ccc}
 0                            &=& x \cr
                              &  &   \cr
\xi_2                         &=& y \cr
                              &  &   \cr
\frac{1}{2}\xi_2\xi_4 + \xi_3 &=& z \cr
                              &  &   \cr
 \xi_4                        &=& t
\end{array}
\right. \Longleftrightarrow 
\left\{\begin{array}{lcl}
 0                            &=& x \cr
                              &  &   \cr
\xi_2                         &=& y \cr
                              &  &   \cr
 \xi_3                        &=& z -\dfrac{1}{2}yt \cr
                              &  &   \cr
 \xi_4                        &=& t
\end{array}
\right. 
\eeqn 

We then have that 
\blem 
The restriction $\exp_{\mid \{0\}\times \R^3}$ of the exponential map of $\mathcal D(G,r)$ 
to the subset $\{0\}\times \R^3$ of $\mathcal D(G,r)$  is invertible and its inverse is given by
\beq
(0,y,z,t) \mapsto  (0,y,z-\frac{1}{2}yt,t) 
\eeq 
\elem 

\item Suppose $x \neq 0$, then we have
\beq 
\left(\xi_1 , \dfrac{\xi_2}{\xi_1}[e^{\xi_1}-1], \xi_3\! +\! \dfrac{\xi_2\xi_4}{\xi_1} \!
+\! \dfrac{\xi_2\xi_4}{\xi_1^2}[e^{-\xi_1}\!-\!1], \dfrac{\xi_4}{\xi_1}[1\!-\!e^{-\xi_1}]\right)=(x,y,z,t) 
\eeq 
\beq
\left\{\begin{array}{lcl}
\xi_1                                                                                      &=& x \cr
                              &  &   \cr
\dfrac{\xi_2}{\xi_1}[e^{\xi_1}-1]                                                          &=& y \cr
                              &  &   \cr
\xi_3\! +\! \dfrac{\xi_2\xi_4}{\xi_1} \!+\! \dfrac{\xi_2\xi_4}{\xi_1^2}[e^{-\xi_1}\!-\!1]  &=& z \cr
                              &  &   \cr
\dfrac{\xi_4}{\xi_1}[1\!-\!e^{-\xi_1}]                                                     &=& t
\end{array}
\right.  \Leftrightarrow 
\left\{\begin{array}{lcl}
\xi_1                                                                                      &=& x \cr
                              &  &   \cr
\xi_2                                                                                      &=& \dfrac{xy}{e^x-1} \cr
                              &  &   \cr
\xi_3\!                                                                      &=& z+\dfrac{xyt}{(e^x-1)(e^{-x}-1)} + \dfrac{yt}{e^x-1} \cr
                              &  &   \cr
\xi_4                                                                        &=& \dfrac{xt}{1-e^{-x}}
\end{array}
\right. \nonumber 
\eeq 
Hence,
\blem
The restriction $\exp_{\mid \R^*\times \R^3}$ of the exponential map of the Lie group $\mathcal D(G,r)$ to the subset 
$\R^*\times \R^3$ is invertible and its inverse is the map defined as :

\beq 
(x,y,z,t)\mapsto \left(x  , \dfrac{xy}{e^x-1},z\!+\!\dfrac{yt}{e^x-1}\Big[\dfrac{x}{e^{-x}\!-\!1}\!+\! 1\Big],\dfrac{xt}{1\!-\!e^{-x}}\right)
\eeq 
\elem 
\eenum 

The precedent two Lemmas imply
\bpro
The exponential map of the Lie group $\mathcal D(G,r)$ is invertible ans its inverse is the map 
$Log_{\mathcal D(G,r)} :\mathcal D(G,r) \to \mathcal D(\G,r) $ defined as follows :

\beq 
Log_{\mathcal D(G,r)}(x,y,z,t) = \left\{\begin{array}{ccc}
(0,y,z-\frac{1}{2}yt,t)                        &\mbox{ if }&  x=0  \cr
                                               &           &       \cr
\left(x  , \dfrac{xy}{e^x-1},z\!+\!\dfrac{yt}{e^x-1}\Big[\dfrac{x}{e^{-x}\!-\!1}\!+\! 1\Big],\dfrac{xt}{1\!-\!e^{-x}}\right) &\mbox{ if }& x \neq 0
\end{array}
\right. 
\eeq 
\epro

\subsection{A Left Invariant Complex Structure On The Double of The Affine Lie group of $\R$}

From \cite{di-me-cybe}, the following formula defines a left-invariant complex structure $J$ on any Lie group with Lie 
algebra  $\mathcal D(\G,r)$ :
\beq 
J\big((x,\alpha)^+\big):= \big(-r(\alpha),q(x)\big)^+,
\eeq
for any $(x,\alpha)$ in $\mathcal D(\G,r)$.
We have
\beqn
Je_1^+ &=& (q(e_1))^+ = e_4^+ = \left(
\begin{array}{cccc}
1  &    0    &      0    &     0       \cr
0  &  e^{x_1}    &      0    &     0       \cr
0  &  0      &      1    &   x_2e^{-x_1} \cr
0  &  0      &      0    &    e^{-x_1}
\end{array}
\right)\left(
\begin{array}{c}
0  \cr
0  \cr
0  \cr
1
\end{array}
 \right)=
\left(
\begin{array}{c}
0  \cr
0  \cr
x_2e^{-x_1}  \cr
e^{-x_1}
\end{array}
 \right) \\
Je_2^+ &=& (q(e_2))^+ = -e_3^+ = -\left(
\begin{array}{cccc}
1  &    0    &      0    &     0       \cr
0  &e^{x_1}  &      0    &     0       \cr
0  &  0      &      1    &   x_2e^{-x_1} \cr
0  &  0      &      0    &    e^{-x_1}
\end{array}
\right)\left(
\begin{array}{c}
0  \cr
0  \cr
1  \cr
0
\end{array}
 \right) = -
\left(
\begin{array}{c}
0  \cr
0  \cr
1  \cr
0
\end{array}
 \right) \\
Je_3^+ &=& (-r(e_3))^+ = e_2^+ = \left(
\begin{array}{cccc}
1  &    0    &      0    &     0       \cr
0  &e^{x_1}  &      0    &     0       \cr
0  &  0      &      1    &   x_2e^{-x_1} \cr
0  &  0      &      0    &    e^{-x_1}
\end{array}
\right)\left(
\begin{array}{c}
0  \cr
1  \cr
0  \cr
0
\end{array}
 \right)=
\left(
\begin{array}{c}
0  \cr
e^{x_1}  \cr
0  \cr
0
\end{array}
 \right) \\
Je_4^+ &=& (-r(e_4))^+ = -e_1^+ = -\left(
\begin{array}{cccc}
1  &    0    &      0    &     0       \cr
0  &e^{x_1}  &      0    &     0       \cr
0  &  0      &      1    &   x_2e^{-x_1} \cr
0  &  0      &      0    &    e^{-x_1}
\end{array}
\right)\left(
\begin{array}{c}
1  \cr
0  \cr
0  \cr
0
\end{array}
 \right)= -
\left(
\begin{array}{c}
1  \cr
0  \cr
0  \cr
0
\end{array}
 \right) 
\eeqn
Let us summarize
\beqn 
Je_1^+ &=& x_2e^{-x_1}\frac{\partial}{\partial x_3} + e^{-x_1}\frac{\partial}{\partial x_4} \\
Je_2^+ &=& -\frac{\partial}{\partial x_3}  \\
Je_3^+ &=& e^{x_1}\frac{\partial }{\partial x_2}  \\
Je_4^+ &=& -\frac{\partial}{\partial x_1}
\eeqn 

This tensor is not bi-invariant, that is it does not commute with the adjoint operators of $\G$. Indeed, if $j:=J_{\mid \epsilon}$, we have
\bitem
\item[$\bullet$] $j[e_1,e_2] = je_2 = -e_3$ ;
\item[$\bullet$] $[e_1,je_2] = [e_1,-e_3] = 0$,
\eitem
then $j[e_1,e_2] \neq [e_1,je_2]$.

\chapter*{General Conclusion}
\fancyhead[L]{General Conclusion}
\addstarredchapter{General Conclusion}


In this thesis we study some aspect of the geometry of cotangent bundles of Lie groups as Drinfel'd double Lie groups. 
Automorphisms of cotangent bundles of Lie groups are completely characterized and interesting results are obtained. 
We give prominence to the fact that the Lie groups  of automorphisms of cotangent bundles of Lie groups are super-symmetric 
Lie groups (Theorem \ref{structuretheorem}). 
In the cases of semi-simple Lie algebras, compact Lie algebras and more generally orthogonal Lie algebras,  we recover by simple methods  
interesting co-homological known results (Section \ref{chap:cohomology}).

\vskip 0.3cm 

Another theme in this thesis is the study of prederivations of cotangent bundles of Lie groups. The Lie algebra of 
prederivations encompasses the one of derivations as a subalgebra. We find out that Lie algebras of cotangent Lie groups 
(which are not semi-simple) of semi-simple Lie groups have the property that all their prederivations are derivations. This result is an extension
of a well known result due to M\"uller (\cite{muller}). The structure of the Lie algebra of prederivations of 
Lie algebras of cotangent bundles of Lie groups is explore and we have shown that the Lie algebra of prederivations of Lie algebras of 
cotangent bundle of Lie groups are reductive Lie algebras.

\vskip 0.3cm

Prederivations are useful tools for classifying objects like pseudo-Riemannian metrics (\cite{bajo4}, \cite{muller}). 
We have studied bi-invariant metrics on cotangent bundles of Lie groups and their isometries. The Lie algebra of the Lie group of
isometries of a bi-invariant metric on a Lie group is entirely determine by prederivations of the Lie algebra which are skew-symmetric
with respect to the induced orthogonal structure on the Lie algebra. We have shown that  the Lie 
group of isometries of any bi-invariant metric on the cotangent bundle of any semi-simple Lie groups is given by inner derivations of the Lie
algebra of the  cotangent bundle.

\vskip 0.3cm

Last, we have dealt with an introduction to the geometry of the Lie  group of affine motions of the real line $\R$, which is a K\"ahlerian 
Lie group (see \cite{lichnerowicz-medina}). We describe, through explicit expressions, a symplectic structure, a complex structure, geodesics. 
Since the symplectic structure corresponds to a solution $r$ of the Classical Yang-Baxter equation  (see \cite{di-me-cybe}), we
also study the double Lie group associated to $r$.

\vskip 0.5cm  

Admittedly, questions remain. Can it be otherwise ?

\vskip 0.3cm 

Let $\G$ be a Lie algebra. We have said that the Lie algebra $\pder(T^*\G)$ of prederivations of $T^*\G$ contains the one $\der(T^*\G)$ 
of derivations as a Lie subalgebra. It would be interesting to know :
\bitem 
\item if $\pder(T^*\G)$ can be decomposed into a semi-direct sum of $\der(T^*\G)$ and a subspace $\mathfrak \h$
of $\pder(T^*\G)$ : $\pder(T^*\G) = \der(T^*\G) \ltimes \mathfrak{h}$ ; 
\item if $\der(T^*\G)$ is an ideal of $\pder(T^*\G)$.
\eitem 
We have seen that if $\G$ is semi-simple, then $\pder(T^*\G)=\der(T^*\G)$. It would be a great step to find necessary and sufficient 
 conditions for which the latter equality holds.

\vskip 0.3cm

Another open question is the following. As the cotangent bundle of any Lie group is a Drinfel'd double Lie group, it seems to be a 
good idea to extend the results within this thesis to the class of Lie groups which are double Lie groups of Poisson-Lie groups. 

\vskip 0.3cm

One other step is to list all orthogonal Lie groups of low-dimension using the double extension procedure of Medina and Revoy.
Studying the isomorphic classes, up to isometric automorphisms, of bi-invariant metrics on cotangent bundles of Lie groups is also 
an interesting subject. It seems to be reasonable to begin by simple Lie algebras.



\end{document}